%% file: convergence-gordon.tex
\documentclass[11pt,twoside]{article}

\usepackage{geometry}
\input{command-template}

\begin{document}

\begin{center}

  {\bf{\LARGE{ \mbox{Sharp global convergence guarantees for iterative} \\ \mbox{nonconvex optimization: A Gaussian process perspective}}}}

\vspace*{.2in}

{\large{
\begin{tabular}{ccc}
Kabir Aladin Chandrasekher$^{\dagger}$, Ashwin Pananjady$^{\star}$, Christos Thrampoulidis$^{\diamondsuit}$
\end{tabular}
}}
\vspace*{.2in}

\begin{tabular}{c}
$^{\dagger}$Department of Electrical Engineering, Stanford University\\
$^\star$Schools of Industrial \& Systems Engineering and Electrical \& Computer Engineering, \\
Georgia Tech \\
$^{\diamondsuit}$Department of Electrical \& Computer Engineering, University of British Columbia
\end{tabular}

\vspace*{.2in}

\today

\vspace*{.2in}

\begin{abstract}
We consider a general class of regression models with normally distributed covariates, and the associated nonconvex problem of fitting these models from data. We develop a general recipe for analyzing the convergence of iterative algorithms for this task from a random initialization. In particular, provided each iteration can be written as the solution to a convex optimization problem satisfying some natural 
conditions, we leverage Gaussian comparison theorems to derive a deterministic sequence that 
provides sharp upper and lower bounds on the error of the algorithm with sample-splitting.
Crucially, this deterministic sequence accurately captures both the convergence rate of the algorithm and the eventual error floor in the finite-sample regime, and is distinct from the commonly used ``population” sequence that results from taking the infinite-sample limit.
We apply our general framework to derive several concrete consequences for parameter estimation in popular statistical models including phase retrieval and mixtures of regressions. Provided the sample size scales near-linearly in the dimension, we show sharp global convergence rates for both higher-order algorithms based on alternating updates and first-order algorithms based on subgradient descent. These corollaries, in turn, yield multiple consequences, including:
\smallskip

\noindent (a) Proof that higher-order algorithms can converge significantly faster than their first-order counterparts (and sometimes super-linearly), \mbox{even if the two share the same population update;} \\

\noindent (b) Intricacies in super-linear convergence behavior for higher-order algorithms, which can be nonstandard (e.g., with exponent $3/2$) and sensitive to the noise level in the problem.
\smallskip

\noindent We complement these results with extensive numerical experiments, which show excellent agreement with our theoretical predictions.
\end{abstract}
\end{center}

\small
\addtocontents{toc}{\protect\setcounter{tocdepth}{2}}
\tableofcontents

\normalsize

\section{Introduction}
\label{sec:introduction}
\input{sections/introduction/introduction.tex}

\section{Background and illustrative examples}
\label{sec:setup}
\input{sections/background/setup.tex}

\section{Recipe and main result: The Gordon state evolution update}
\label{sec:heuristic}
\input{sections/heuristic-derivation/heuristic-derivation_alt.tex}

\section{Consequences for some concrete statistical models}
\label{sec:main-results}
\input{sections/main-results/main-results_alt.tex}

\section{Numerical illustrations}
\label{sec:numerical-illustrations}
\input{sections/numerical-illustrations/numerical-illustrations.tex}

\input{sections/discussion/discussion.tex}

\section{Proof of general results, part (a): Gordon update and deviation bounds}
\label{sec:gordon}
\input{sections/gordon/gordon-alt/gordon-alt2.tex}

\section{Proof of general results, part (b): Tighter bounds on parallel component}
\label{sec:random-init-signal}
\input{sections/random-init/random-init.tex}

\input{sections/deterministic-convergence/deterministic-convergence_alt.tex}

\newpage
\small

\subsection*{Acknowledgments}

Part of this work was performed when the authors were participants in the program on Probability, Geometry, and Computation in High Dimensions hosted at the Simons Institute for the Theory of Computing. KAC was supported in part by a National Science Foundation Graduate Research Fellowship and the Sony Stanford Graduate Fellowship. AP was supported in part by a research fellowship from the Simons Institute and National Science Foundation grant CCF-2107455. CT was supported in part by the National Science Foundation grant CCF-2009030, by an NSERC Discovery Grant, and by a research grant from KAUST.

\bibliographystyle{abbrvnat}
\bibliography{intro-refs,convergence-gordon}
\normalsize

\newpage
\appendix

\begin{center}
\bf{\LARGE{Appendix}}
\end{center}

\input{sections/appendix/aux_heuristic}

\input{sections/appendix/aux-gordon.tex}

\input{sections/appendix/aux-random-init.tex}

\input{sections/appendix/aux_alt.tex}

\input{sections/appendix/aux-elementary-lemmas.tex}

\end{document}

%% file: sections/introduction/introduction.tex

In many modern statistical estimation problems involving nonlinear observations, latent variables, or missing data, the log-likelihood---when viewed as a function of the parameters of interest---is nonconcave. Accordingly, even though the maximum likelihood estimator enjoys favorable statistical properties in many of these problems, the more practically relevant question is one at the intersection of statistics and optimization: Can we, in modern problems where the dimension is typically comparable to the sample size, optimize the likelihood in a computationally efficient manner to produce statistically useful estimates? When viewed in isolation, many of these nonconvex model-fitting problems can be shown to be NP-hard. However, the statistically relevant setting---in which data are drawn i.i.d. from a suitably ``nice” distribution---gives rise to random ensembles of optimization problems that are often amenable to iterative algorithms. Following a decade of intense activity, a unifying picture of nonconvex optimization in statistical models has begun to emerge.

Given the rapid development of the field, a few distinct methods now exist for the analysis of these nonconvex procedures. Let us briefly discuss two salient and natural approaches. The first approach is to directly work with iterates of the algorithm. In particular, one can view each iteration as a random operator mapping the parameter space to itself, and study its properties. A typical example of this approach involves tracking some error metric between the iterates and the ``ground truth'' parameter, and showing that applying the operator reduces the error at a certain rate, possibly up to an additive correction to accommodate noise in the observations~\citep[see, e.g., some of the early papers][]{jain2013low,loh2012high}.
The second approach looks instead at the landscape of the loss function that the iterative algorithm is designed to minimize, showing that once the sample size exceeds a threshold, this landscape has favorable properties that make it amenable to iterative optimization.
Several such properties have been established in particular problems, 
including, but not limited to, properties of the loss in a neighborhood of its optimum~\citep[e.g.,][]{loh2015regularized,candes2015phase} and the absence of local minima with high probability~\citep[e.g.,][]{sun2018geometric,ge2016matrix}.

In both recipes alluded to above, we are interested in characterizing properties of random objects: the sample-based operator in the first case and the sample-dependent loss function in the second. A general-purpose tool to carry out both recipes is to first understand deterministic, \emph{population} analogs of these (random) objects in the infinite-sample limit. For instance, taking the sample size to infinity in these problems yields a \emph{population operator} in the first case and a \emph{population likelihood} in the second, and allows one to analyze algorithms deterministically in this limit. Post this point, tools from empirical process theory~\citep[e.g.,][]{balakrishnan2017statistical,mei2018landscape} or more refined leave-one-out techniques~\citep{ma2020implicit,chen2019gradient} can be used to argue that sample-based versions (of the operator/loss) behave similarly. In particular, the first program---showing convergence of the population operator and relating this to its sample-based analog---has established convergence of several algorithms in a variety of settings~\citep[e.g.,][]{balakrishnan2017statistical, daskalakis2017ten,tian2017analytical,chen2019gradient,dwivedi2020singularity,wu2019randomly,xu2018EM,ho2020instability}. The overall style of the analysis is appealing for several reasons: (a)~It applies (in principle) to any iterative algorithm run on any model-fitting problem and, (b)~In contrast to the direct sample-based approach, it does not require the analysis of a complex recursion involving highly nonlinear functions of the random data. In addition, decomposing the analysis into a deterministic optimization-theoretic component applied to the population operator and a stochastic component that captures the eventual statistical neighborhood of convergence provides a natural two-step approach. But does the population operator always provide a reliable prediction of convergence behavior in modern, high dimensional settings?

\subsection{Motivation: Accurate deterministic predictions of convergence behavior}

Toward answering the question posed above,
we run a simulation on what is arguably the simplest nonlinear model resulting in a nonconvex fitting problem: \emph{phase retrieval} with a real signal. This is a regression model in which a scalar response $y$ is related to a $d$-dimensional covariate $\bx$ via $\EE[y|\bx] = | \inprod{\bx}{\thetastar} |$, and the task is to estimate $\thetastar$ from i.i.d. observations $(\bx_i, y_i)$. Two popular algorithms to optimize the nonconcave log-likelihood in this problem---described in detail in Section~\ref{sec:it_alg} to follow---are given by:
\begin{itemize}
\item[(i)] Alternating minimization, an algorithm that dates back to~\citet{gerchberg1972practical} and~\citet{fienup1982phase}. Each (tuning-free) iteration is based on fixing the latent ``signs" according to the current parameter and solving a least squares problem. 
\item[(ii)] Subgradient descent with stepsize $\eta$. This is a simple first-order method on the negative log-likelihood that also goes by the name of reshaped Wirtinger flow~\citep{zhang2017nonconvex}. \\
\end{itemize}

Before running our simulation, we emphasize two aspects of it that form recurring themes throughout the paper. First, we use a sample-splitting device: each iteration of the algorithm is executed using $n$ fresh observations of the model, drawn independently of past iterations. This device has been used extensively in the analysis of iterative algorithms as a simplifying assumption~\citep[e.g.,][]{jain2013low,hardt2014fast,netrapalli2015phase,kwon2019global}, and forms a natural starting point for our investigations. Second, over and above tracking the $\ell_2$ error of parameter estimation, we track a more expressive statistic over iterations. In particular, we associate each parameter $\bt \in \real^d$ with a two-dimensional \emph{state}
\begin{align} \label{eq:def-state}
\parcomp(\bt) = \| \proj_{\thetastar} \bt \|_2 \qquad \text{ and } \qquad \perpcomp(\bt) = \| \proj^{\perp}_{\thetastar} \bt \|_2,
\end{align}
where $\proj_{\thetastar}$ denotes the projection matrix onto the one-dimensional subspace spanned by $\thetastar$ and $\proj^{\perp}_{\thetastar}$ denotes the projection matrix onto the orthogonal complement of this subspace. In words, these two scalars measure the component of $\bt$ parallel to $\thetastar$ and perpendicular to $\thetastar$, respectively. Iterates $\bt_t$ of the algorithm then give rise to a two-dimensional \emph{state evolution} $(\parcomp_t, \perpcomp_t)$, where $\parcomp_t = \parcomp(\bt_t)$ and $\perpcomp_t = \perpcomp(\bt_t)$. As several papers in this space have pointed out~\citep{chen2019gradient,wu2019randomly,tan2019online}, tracking the state evolution instead of the evolving $d$-dimensional parameter provides a useful summary statistic of the algorithm's behavior, and natural losses such as the $\ell_2$ or angular loss of parameter estimation can be expressed solely in terms of the state evolution.
\begin{figure*}[ht!]
	\centering
	\begin{subfigure}[b]{0.48\textwidth}
		\centering
		\input{sections/introduction/intro-figs/pop-dist.tex}
		\caption{The $\ell_2$ error of the population update vs. $\ell_2$ error of the empirics.}    
		\label{subfig:dist-pop}
	\end{subfigure}
	\hfill
	\begin{subfigure}[b]{0.48\textwidth}  
		\centering 
		\input{sections/introduction/intro-figs/pop-alpha-beta.tex}
		\caption{State evolution: $\alpha$ and $\beta$ components~\eqref{eq:def-state} for the population update and empirics.}
		\label{subfig:alpha-beta-pop}
	\end{subfigure}
	\caption{Plots of behavior over iterations for both alternating minimization (AM) and subgradient descent (GD) with step-size $1/2$, alongside the population update, which is identical for both algorithms (see Section~\ref{sec:main-results}). Both algorithms are initialized at $\bt_0 = 0.2 \cdot \btstar + \sqrt{1 - 0.2^2} \proj_{\btstar}^{\perp} \bgam$, with 
	$\bgam$ uniformly distributed on the unit sphere.
	Additionally, the observations are noisy with noise standard deviation $10^{-8}$.  Shaded envelopes around the empirics denote 95\% confidence bands over $100$ independent trials.} 
	\label{fig:intro-pop}
\end{figure*}
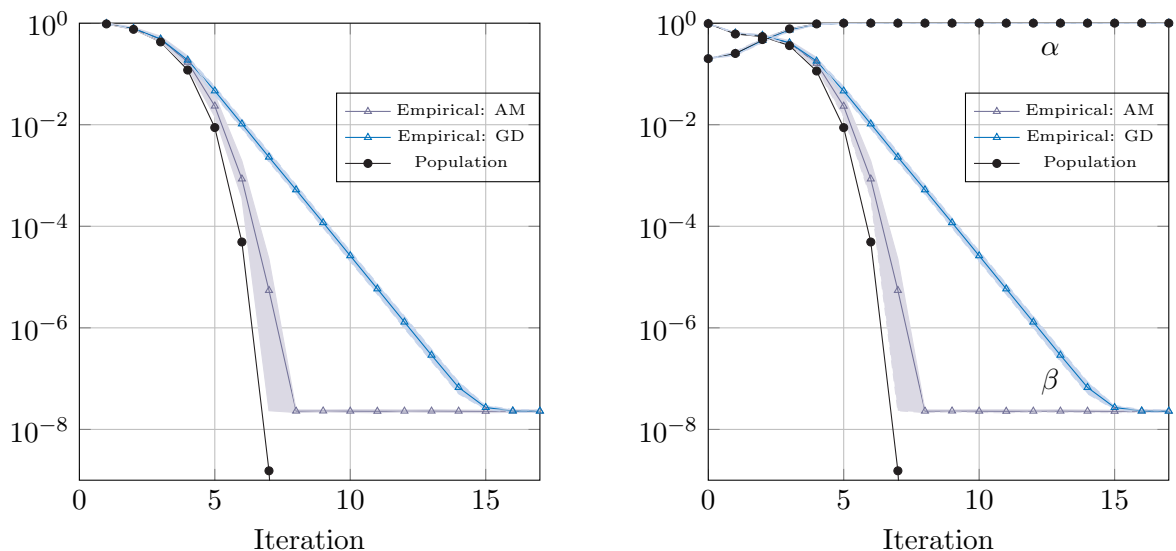

We run our simulation in dimension $d = 600$, and use $n = 12,000$ observations per iteration of the algorithm. We set $\thetastar$ to be the first standard basis vector in $\real^d$, and suppose that the covariates $\bx_i \in \real^{d}$ follow a standard Gaussian distribution. We generate the $i$-th response via $y_i = |\inprod{\bx_i}{\thetastar}| + \epsilon_i$, where $\epsilon_i \sim \NORMAL(0, \sigma^2)$.  In our simulation, we consider $\sigma = 10^{-8}$.  Consider the two iterative algorithms above initialized at a randomly chosen  point \sloppy\mbox{$\bt_0 = 0.2 \cdot \btstar + \sqrt{1 - 0.2^2} \proj_{\btstar}^{\perp} \bgam$}---with 
$\bgam$ uniformly distributed on the unit sphere---and choosing the stepsize $\eta$ to ensure that the population operators of both algorithms coincide\footnote{See Section~\ref{sec:main-results} for the concrete setting, and explicit evaluations of the population update.}.
Figure~\ref{fig:intro-pop} plots both the (random) $\ell_2$-error of parameter estimation and the state evolution for both algorithms along with the analogous quantities for the (deterministic) population update. As we make clear shortly, the population update in the latter case takes the form of a \emph{state evolution} update, in that the state of the next population iterate can be computed as a deterministic function of the state of the current iterate. Two conclusions can be drawn immediately from Figure~\ref{fig:intro-pop}. First, the population update is overly optimistic when predicting the convergence rates of both algorithms. Second, algorithms with the same population update can exhibit very different convergence behaviors.
Looking more closely at Figure~\ref{fig:intro-pop}(b), we see that while the state evolution predictions are comparable to empirical behavior towards the left of the plot (immediately after initialization), the $\perpcomp$ prediction is no longer accurate towards the right, in a local neighborhood of $\thetastar$. 
In Figure~\ref{fig:stepsize_popvgor} in Section~\ref{sec:numerical-illustrations}, we exhibit even more drastic situations in which the population update predicts convergence to $\thetastar$ whereas the empirical iterates stay bounded away from it.

As our simple experiment demonstrates, the population operator is not, at least in general, a very reliable predictor of convergence behavior. The underlying reason is simply that the problem is high dimensional: it is too simplistic to assume that the algorithm's finite-sample behavior will resemble the case when the sample size goes to infinity. This observation naturally leads to the principal question that we attempt to answer in this paper:
\smallskip
\begin{center}
\emph{Is there a more faithful deterministic prediction for the empirical behavior of iterative algorithms in the high dimensional setting?}
\end{center}
\smallskip
To be more specific, we would like such a deterministic update to satisfy several desiderata. First and foremost, we should be able to accurately predict the error of parameter estimation after running one step of the update from any point, allowing us to distinguish cases in which the algorithm gets closer to the ground truth parameter (thereby suggesting convergence) from otherwise. Second, the update should give us \emph{sharp} predictions of convergence behavior that differentiate, for instance, between linear and superlinear convergence. Such a sharp prediction for the iteration complexity can be used in conjunction with a characterization of the per-step computational cost of the algorithm to guide the choice of the fastest procedure to implement for any given task. Third, and related to the previous point, we would like deterministic recursions that provide both upper and lower bounds on the error of the algorithm, at least in a local neighborhood of the solution. This would allow us to rigorously compare and delineate algorithms in terms of their convergence behavior, instead of simply comparing upper bounds with upper bounds.

\begin{figure*}[!htbp]
	\centering
	\begin{subfigure}[b]{0.48\textwidth}   
		\centering 
		\input{sections/introduction/intro-figs/gor-dist.tex}
		\caption{The  $\ell_2$ error of the Gordon update vs. $\ell_2$ error of the empirics.}
		\label{subfig:dist-gor}
	\end{subfigure}
	\hfill
	\begin{subfigure}[b]{0.48\textwidth}   
		\centering 
		\input{sections/introduction/intro-figs/gor-alpha-beta.tex}
		\caption{State evolution: $\alpha$  and $\beta$  components~\eqref{eq:def-state} for the Gordon update and the empirics.}    
		\label{subfig:alpha-beta-gor}
	\end{subfigure}
	\caption{Plots of behavior over iterations for both alternating minimization and subgradient descent with step-size $1/2$, in a simulation identical to that of Figure~\ref{fig:intro-pop}. Overlaid is the prediction from their respective Gordon state evolution updates. As is evident from the occlusion of the triangle markers, the Gordon prediction exactly tracks behavior in both cases.} 
	\label{fig:intro-Gordon}
\end{figure*}
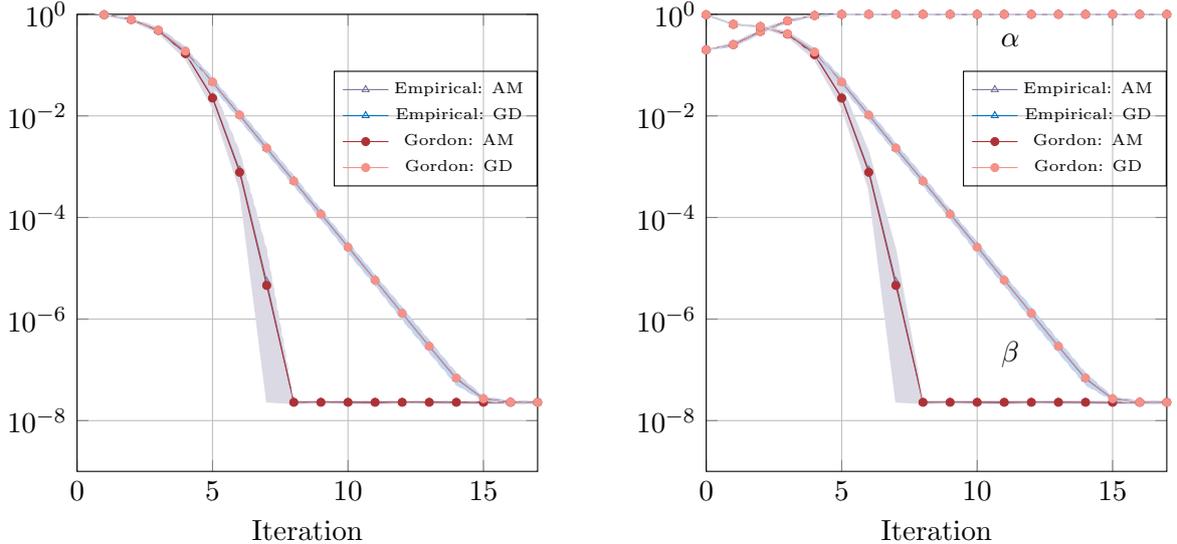

\subsection{Contributions and roadmap}

We consider a general class of regression models with Gaussian covariates (to be introduced precisely in Section~\ref{sec:setup}), and analyze the convergence behavior of iterative algorithms run with sample-splitting. Our contributions are summarized below.
\begin{enumerate}
\item \textbf{Gordon state evolution update:} 
Our main contribution is to use the machinery of Gaussian comparison inequalities---in particular, the convex Gaussian minmax theorem, or CGMT for short~\citep{thrampoulidis2015regularized}---to derive a deterministic~\emph{Gordon state evolution update} (or Gordon update for short) that 
satisfies the desiderata laid out above. This update applies provided each iteration of the algorithm can be written as the solution to a convex optimization problem satisfying some mild assumptions\footnote{In particular, although all of the scenarios we consider in this paper involve optimizing a nonconcave log-likelihood function, our recipe can also be applied to provide sharp convergence guarantees for the iterative minimization of convex loss functions.}.

The Gordon state evolution update is distinct from  the population state evolution update in that it involves an additive
correction term, which is nonzero whenever the sample size is finite. In particular, suppose that $\kappa = n/d$ denotes the oversampling ratio used to implement one step of the algorithm; this is defined precisely in Section~\ref{sec:it_alg}. Then the perpendicular component of the Gordon update takes the form
\begin{align*}
\widebar{\perpcomp}_{t + 1} = \perpcomp^{\pop}_{t + 1} + \order\left( \frac{1}{\kappa} \right) \cdot  D( \alpha_t, \beta_t),
\end{align*}
where $\perpcomp^{\pop}_{t + 1}$ is the analogous component of the population state evolution update run from the point $(\parcomp_t, \perpcomp_t)$, and
$D: \real^2 \to [0, \infty)$ is some function that takes only nonnegative values. 
Thus, taking the sample size (and hence $\kappa$) to infinity, we recover the population state evolution update directly from the Gordon update. However, as we will see shortly, the finite-sample behavior of the update is often dominated by the term  $\order\left( \frac{1}{\kappa} \right) \cdot  D( \alpha, \beta)$, and in these scenarios the population update is a poor predictor of convergence behavior. We showcase the general recipe involved in deriving the Gordon update in Section~\ref{sec:heuristic}.

Our recipe provides not only a deterministic update but also a finite-sample concentration bound, showing that the empirical state evolution concentrates sharply around the point predicted by the Gordon update. This is illustrated for the phase retrieval simulation in Figure~\ref{fig:intro-Gordon}. As is clear from this figure, the Gordon update provides a sharp prediction of convergence behavior, allowing us to distinguish different types of convergence and also providing near-exact predictions of the eventual error floor. As we explore further in Figure~\ref{fig:stepsize_popvgor} (see Section~\ref{sec:numerical-illustrations}), the Gordon update accurately captures behavior even in situations where the empirical iterates do not converge.

\item \textbf{Results for concrete models:} While the machinery that we develop is general, we showcase its utility by deriving global convergence guarantees (i.e., from a random initialization) for both higher-order and first-order algorithms in two statistical models: phase retrieval and mixtures of regressions. 
Some salient takeaways are collected in Table~\ref{table:summary}.
\begin{table}[ht!]
\hspace{5mm}\begin{tabular}{|c|c|c|c|}
\hline
\textbf{Algorithm}       & \textbf{Model}         & \textbf{Metric} & \textbf{Local convergence rate} \\ \hline
Alternating minimization & Phase retrieval        & $\ell_2$        & Superlinear, exponent $3/2$      \\ \hline
Subgradient descent      & Phase retrieval        & $\ell_2$        & Linear                          \\ \hline
Alternating minimization & Mixture of regressions & Angular         & Linear                          \\ \hline
Subgradient AM           & Mixture of regressions & Angular         & Linear                          \\ \hline
\end{tabular}
\hspace{5mm}\caption{Summary of results for specific models and algorithms. In all cases, we provide global convergence guarantees showing
that with high probability, convergence to the local neighborhood of the ground-truth parameter takes places after a number of iterations that is logarithmic in the dimension of the problem. Convergence rates within this neighborhood, as predicted by the Gordon update, are listed above. These exactly match empirical behavior in all cases.} \label{table:summary}
\end{table}

To summarize, for the phase retrieval model, our primary contribution is to make quantitative the behavior observed in Figures~\ref{fig:intro-pop} and~\ref{fig:intro-Gordon}. 
While
the population update predicts quadratic convergence (i.e., superlinear convergence with exponent $2$), we show that both alternating minimization and subgradient descent behave differently from this prediction. The former algorithm does converge superlinearly but with a nonstandard exponent $3/2$, while the latter converges linearly at best.
	For the mixture of regressions model, 
	we propose a first-order method termed \emph{subgradient AM}, which is inspired by the closely related gradient EM update~\citep{dempster1977maximum,neal1998view}. We study it alongside alternating minimization, and 
 show that while both algorithms exhibit linear convergence in the angular metric, they are inconsistent in the $\ell_2$ metric for any nonzero noise level. 
	We exhibit regimes in which the first-order method is competitive (in terms of its iteration complexity) with its higher-order counterpart, suggesting that the first-order method should be preferred in these regimes given its lower per-iteration cost.

\item \textbf{Techniques of independent interest:} Over the course of proving our results, we develop some techniques that may be of broader interest, three of which we highlight below. 
\begin{itemize}
	\item In proving finite-sample concentration bounds around the deterministic Gordon updates, we handle a family of loss functions that is strictly more general than those used for proving analogous results in linear models~\citep{oymak2013squared,miolane2018distribution}. Our techniques are based on arguing about carefully chosen growth properties of these loss functions, and may prove useful in other non-asymptotic instantiations of the CGMT machinery.
	\item Characterizing algorithmic behavior near a random initialization requires a sharper bound on the deviation of the parallel component than what is provided by the general technique alluded to above. We develop a refined bound---applicable to higher-order updates that involve a matrix inversion in each iteration---by using a leave-one-out device. This characterization allows us to replace a polylogarithmic factor in the sample complexity bound with a doubly-iterated logarithm, and the technique may prove more broadly useful in analyzing other higher-order updates from a random initialization. 
	\item Finally, our local convergence analysis for particular algorithms relies on a first-order expansion of the Gordon update. In particular, we show that the Gordon update is contractive in a local neighborhood of the ground truth $\thetastar$, and combine this structural characterization with our refined concentration bounds on the sample state evolution to show deterministic upper \emph{and} lower bounds, i.e., a high-probability envelope around, the error of the empirical trajectory of the algorithm. Such a technique may prove more broadly useful in producing sharp characterizations of convergence behavior in other classes of iterative algorithms.
\end{itemize}
\end{enumerate}

The rest of the paper is organized as follows. In Section~\ref{sec:setup} to follow, we present the formal problem setup and background on the models and algorithms that we use to illustrate the Gordon state evolution machinery. This section also introduces the ``subgradient AM" update for mixtures of regressions. In Section~\ref{sec:heuristic}, we provide the recipe itself, starting with a high level overview of the steps and a heuristic derivation in a special case before stating our main results in Theorems~\ref{thm:one_step_main_HO} and~\ref{thm:one_step_main_FO}. Section~\ref{sec:main-results} collects consequences for two models, phase retrieval and mixtures of linear regressions, and on each we employ two algorithms, one based on alternating projections and another on subgradient descent. Theorems~\ref{thm:AM-PR}-\ref{thm:GD-MLR} establish global convergence results for all of these cases. In Section~\ref{sec:numerical-illustrations}, we present numerical experiments to corroborate our theoretical findings. We discuss future directions in Section~\ref{sec:discussion}, and then turn to our proofs. Theorems~\ref{thm:one_step_main_HO} and~\ref{thm:one_step_main_FO} both have two parts; we present the proof of parts (a) in Sections~\ref{sec:gordon} and the proof of parts (b) in Section~\ref{sec:random-init-signal}. Proofs of Theorems~\ref{thm:AM-PR} through~\ref{thm:GD-MLR} are presented in a unified fashion in Section~\ref{sec:proofs-specific}. 
Our appendices collect proofs of auxiliary technical lemmas.

\subsection{Related work}
The literature on nonconvex optimization in statistical settings is vast, and we cannot hope to cover all of it here. We refer the reader to a few recent monographs~\citep{jain2017non,chen2018harnessing,chi2019nonconvex,zhang2020symmetry} for surveys, and the webpage~\citep{sun2021} for an ever-expanding list of relevant references. We focus in this subsection on describing a few papers that are most closely related to our contributions, categorized for convenience under three broad headings.

\paragraph{Predictions 
in random optimization problems:} As alluded to before, the population update has proven useful in analyzing many algorithms in a variety of settings including Gaussian mixture models~\citep{balakrishnan2017statistical, daskalakis2017ten,xu2016global}, mixtures of regressions~\citep{balakrishnan2017statistical,kwon2019global,klusowski2019estimating}, phase retrieval~\citep{chen2019gradient}, mixtures of experts~\citep{makkuva2019breaking}, and neural networks~\citep{tian2017analytical}. In addition to providing local convergence guarantees, it has enabled researchers to study the more challenging setting with random initialization~\citep{chen2019gradient,dwivedi2020singularity,wu2019randomly}, and also revealed several surprising phenomena related to overparameterization and stability~\citep{xu2018EM,ho2020instability}. The Gordon update that we derive is a much sharper deterministic predictor of convergence behavior than its population counterpart, and we hope that other surprising phenomena---over and above those that we present in the current paper---can be uncovered by making use of it. 

In addition to papers that characterize the random loss landscape by utilizing properties of the population loss~\citep[e.g.,][]{mei2018landscape,davis2020nonsmooth,hand2019global}, we mention another line of inquiry---rooted in the literature on statistical physics---that leads to deterministic predictions. 
This framework is especially appealing when a prior on the underlying parameter is assumed, and employs the approximate message passing (AMP) algorithm~\citep{donoho2009message,donoho2011noise,bayati2011lasso,montanari2013statistical}. AMP is carefully designed to satisfy certain (approximate) independence properties across iterates and leads to a simple state evolution without sample-splitting; see the recent tutorial by~\citet{feng2021unifying} for an introduction. 
The analysis framework has recently been used to explore the (sub-)optimality of first-order methods in terms of their eventual parameter estimation error~\citep{celentano2020estimation}, to predict computational barriers in a variety of problems including phase retrieval in high dimensions~\citep{maillard2020phase}, and to demonstrate that logistic regression is biased in high-dimensions, thereby suggesting an asymptotic correction~\citep{sur2019modern}.
In contrast to our motivation, predictions in this family are not designed with the dual goal of characterizing the (optimization-theoretic) rate of convergence of various algorithms as well as the statistical error of the eventual solution. Instead, they focus on producing a single algorithm that eventually attains statistical optimality, which is typically a member of the AMP family.  

Finally, we note that~\citet{oymak2017sharp} focused on showing sharp time-data tradeoffs in linear inverse problems. In particular, they considered random design linear regression where the underlying parameter was constrained to an arbitrary (possibly nonconvex) set, and showed that employing projected gradient descent on the square loss with a particular choice of stepsize enjoys a linear rate of convergence to an order-optimal neighborhood of the true parameter. They also showed that a linear rate is the best achievable when the constraint set is convex. 
In follow-up work and for the same optimization algorithms, \citet{oymak2016fast}, obtained similar results for  single-index model estimation. Specifically, their measurement model allows a nonlinear link function, but their algorithm assumes a linear one (thus, is agnostic to the nonlinearity) following the paradigm of \citet{brillinger2012generalized,plan2016generalized}.
While these results are compelling, they are restricted to the analysis of a single algorithm, do not provide sharp iterate-by-iterate predictions, and their primary focus is on exploiting structure in the underlying parameter. For comparison and on the one hand, we do not explicitly model structure in the parameter of interest, and also require that each iteration of the algorithm solves a convex program. On the other hand, we allow for arbitrary nonlinear models, and our machinery allows us to derive sharp tradeoffs applying to a broad class of iterative algorithms that go beyond first-order methods {for linear regression}.

\paragraph{Convergence guarantees for iterative algorithms beyond first-order updates:} As made clear shortly, the Gordon state evolution recipe is particularly powerful when dealing with iterative algorithms that go beyond first-order updates, and consequently involve highly non-linear functions of the random data. There are several ``direct'' analyses of such \emph{higher-order} updates in the literature on matrix factorization, mixture models, neural networks, and index models, including for alternating projections~\citep{jain2013low,gunasekar2013noisy,hardt2014fast,yi2014alternating,agarwal2016learning,sun2016guaranteed,waldspurger2018phase,jagatap2017fast,zhang2020phase,ghosh2019max,pananjady2021single}, composite optimization~\citep{duchi2019solving,charisopoulos2021low}, and Gauss--Newton methods~\citep{gao2017phaseless}. For the expectation maximization (EM) algorithm and its Newton (i.e., second-order) analog, the population update has been widely used to prove parameter estimation guarantees~\citep{balakrishnan2017statistical,xu2016global,ho2020instability}, although convergence in function value can be shown via other means~\citep{xu1996convergence,kunstner2021homeomorphic}.
All of the analyses mentioned here are only able to provide upper bounds on the parameter estimation error over iterations, and we expect that employing our recipe in these settings would yield either matching lower bounds or sharper convergence rates.

\paragraph{Gordon's Gaussian comparison theorem in statistical models:} Gordon proved his celebrated minmax theorem for doubly-indexed Gaussian processes in the 1980s~\citep{gordon1985some,gordon1988milman}, which was popularized in statistical signal processing by~\citet{rudelson2006sparse,stojnic2009various}. Following a line of work~\citep{stojnic2013framework,stojnic2013upper,stojnic2013regularly,amelunxen2014living,oymak2013squared}, a sharp version of Gordon's result in the presence of convexity---providing both upper and lower bounds on the minmax value---was formalized in~\citet{thrampoulidis2015regularized}; see~\citet{thrampoulidis2016recovering} for broader historical context. Since then, the \emph{convex} Gaussian minmax theorem (or CGMT for short) has been used to provide sharp performance guarantees for several convex programs with Gaussian data, including regularized M-estimators \citep{thrampoulidis2018precise,thrampoulidis2018symbol}, one-bit compressed sensing \citep{thrampoulidis2015lasso}, the Phase-Max program for phase-retrieval \citep{dhifallah2018phase,salehi2018precise}, regularized logistic regression~\citep{salehi2019impact,taheri2020sharp,taheri2021fundamental,aubin2020generalization,dhifallah2020precise}, adversarial training for linear regression and classification~\citep{javanmard2020preciseregression,javanmard2020precise,taheri2020asymptotic}, max-margin linear classifiers~\citep{montanari2019generalization,Deng2021Model,kammoun2021precise}, distributional characterization of minimum norm linear interpolators \citep{chang2021provable}, and minimum $\ell_1$ norm interpolation and boosting~\citep{liang2020precise}. While this line of work typically uses the Gordon machinery to provide a one-step---and asymptotic---guarantee, the results of our paper are obtained by using the CGMT in each step of the iterative algorithm, which requires a non-asymptotic characterization. 
Having said that, we note that non-asymptotic bounds have been obtained using the CGMT in the context
of the LASSO~\citep{oymak2013squared,miolane2018distribution,celentano2020lasso} and SLOPE~\citep{wang2019does}, 
but existing guarantees of this form appear to have been restricted to the study of sparse linear regression.

\subsection{General notation}

We use boldface small letters to denote vectors and boldface capital letters to denote matrices. We let $\sign(v)$ denote the sign of a scalar $v$, with the convention that $\sign(0) = 1$. We use $\sign(\bv)$ to denote the sign function applied entrywise to a vector $\bv$. Let $\ind{\cdot}$ denote the
indicator function. 
For $p \geq 1$, let $\mathbb{B}_p(\bv; t) = \{ \bx: \| \bx - \bv \|_p \leq t \}$ denote the closed $\ell_p$ ball of 
radius $t$ around a point $\bv$, with the shorthand $\mathbb{B}_p(t) = \mathbb{B}_p(\bm{0}; t)$; the dimension will usually be clear from context. Analogously, let $\mathbb{B}_p(S; t) = \{\bx: \| \bx - \bv \|_p \leq t \text{ for some } \bv \in S \}$
denote the $t$-fattening of a set $S$ in $\ell_p$-norm.
For an operator $\mathcal{A}: \mathbb{S} \to \mathbb{S}$, let $\mathcal{A}^t := \underbrace{ \mathcal{A} \otimes \cdots \otimes \mathcal{A} }_{t \text{ times}}$ denote the operator obtained by $t$ repeated applications of $\mathcal{A}$.

For two sequences of non-negative reals $\{f_n\}_{n
\geq 1}$ and $\{g_n \}_{n \geq 1}$, we use $f_n \lesssim g_n$ to indicate that
there is a universal positive constant $C$ such that $f_n \leq C g_n$ for all
$n \geq 1$. The relation $f_n \gtrsim g_n$ indicates that $g_n \lesssim f_n$,
and we say that $f_n \asymp g_n$ if both $f_n \lesssim g_n$ and $f_n \gtrsim
g_n$ hold simultaneously. We also use standard order notation $f_n = \order
(g_n)$ to indicate that $f_n \lesssim g_n$ and $f_n = \ordertil(g_n)$ to
indicate that $f_n \lesssim
g_n \log^c n$, for a universal constant $c>0$. We say that $f_n = \Omega(g_n)$ (resp. $f_n = \widetilde{\Omega}(g_n)$) if $g_n = \order(f_n)$ (resp. $g_n = \ordertil(f_n)$). The notation $f_n = o(g_n)$ is
used when $\lim_{n \to \infty} f_n / g_n = 0$, and $f_n =
\omega(g_n)$ when $g_n = o(f_n)$. Throughout, we use $c, C$ to denote universal
positive constants, and their values may change from line to line. All logarithms are to the natural base unless
otherwise stated.

 We denote by $\NORMAL(\bm{\mu}, \bSig)$ a normal distribution with mean $\bm{\mu}$ and covariance matrix $\bSig$. Let $\mathsf{Unif}(S)$ denote the uniform distribution on a set $S$, where the distinction between a discrete and continuous distribution can be made from context. We say that $X \overset{(d)}{=} Y$ for two random variables $X$ and $Y$ that are equal in distribution. For $q \geq 1$ and a random variable $X$ taking values in $\real^d$, we write $\| X \|_q = (\EE[ |X|^q ])^{1/q}$ for its $L^{q}$ norm. Finally, for a real valued random variable $X$ and a strictly increasing convex function $\psi: \real_{\geq 0} \to \real_{\geq 0}$ satisfying $\psi(0) = 0$, we write $\| X \|_{\psi} = \inf\{ t > 0 \; \mid \; \EE[\psi(t^{-1} |X| )] \leq 1\}$ for its $\psi$-Orlicz norm. We make particular use of the $\psi_q$-Orlicz norm for $\psi_q(u) = \exp(|u|^q) - 1$. We say that $X$ is sub-Gaussian if $\| X \|_{\psi_2}$ is finite
 and that $X$ is sub-exponential if $\| X \|_{\psi_1}$ is finite.

%% file: sections/introduction/intro-figs/pop-dist.tex
	\begin{tikzpicture}
		\begin{axis}[
			width=\textwidth,
			height=\textwidth,
			xlabel={Iteration},
			ymode=log,
xmin=0, xmax=17,
ymin=0.000000001, ymax=1,
legend style={font=\tiny, fill=none, at={(1,0.75)},anchor=east},
ymajorgrids=true,
grid style=grid,
			]
			
			\addplot[
			color=CadetBlue,
			mark=triangle,
			mark size=1.5pt   
			]
			table[x=iter,y=avgAM] {sections/introduction/intro-data/intro-data.dat};
			\legend{Empirical: AM}
			
			\addplot[
			color=RoyalBlue,
			mark=triangle,
			mark size=1.5pt   
			]
			table[x=iter,y=avgGD] {sections/introduction/intro-data/intro-data.dat};
			\addlegendentry{Empirical: GD}
			
			\addplot[
			color=Black,
			mark=*,
			mark size=1.5pt
			]
			table[x=iter,y=popPR] {sections/introduction/intro-data/intro-data.dat};
			\addlegendentry{Population}
			
%
			
			\addplot+[name path=A-AM,CadetBlue!20, no markers] table[x=iter,y=lowerAM] {sections/introduction/intro-data/intro-data.dat};
			\addplot+[name path=B-AM,CadetBlue!20, no markers] table[x=iter,y=upperAM] {sections/introduction/intro-data/intro-data.dat};
			
			\addplot[CadetBlue!20] fill between[of=A-AM and B-AM];
			
			\addplot+[name path=A-GD,RoyalBlue!20, no markers] table[x=iter,y=lowerGD] {sections/introduction/intro-data/intro-data.dat};
			\addplot+[name path=B-GD,RoyalBlue!20, no markers] table[x=iter,y=upperGD] {sections/introduction/intro-data/intro-data.dat};
			
			\addplot[RoyalBlue!20] fill between[of=A-GD and B-GD];
		\end{axis}
	\end{tikzpicture}

%% file: sections/introduction/intro-figs/pop-alpha-beta.tex
	\begin{tikzpicture}
		\node[above, black] (sourcebeta) at (4.5,1) {$\beta$};
		\node[above, black] (sourcealpha) at (4.5, 5.5) {$\alpha$};
		\begin{axis}[
			width=\textwidth,
			height=\textwidth,
			xlabel={Iteration},
			ymode=log,
			xmin=0, xmax=17,
			ymin=0.000000001, ymax=1,
			legend style={font=\tiny, fill=none, at={(1,0.75)},anchor=east},
			ymajorgrids=true,
			grid style=grid,
			]
			
			\addplot[
			color=CadetBlue,
			mark=triangle,
			mark size=1.5pt   
			]
			table[x=iter,y=avgalphaAM] {sections/introduction/intro-data/intro-data.dat};
			\legend{Empirical: AM}
			
			\addplot[
			color=RoyalBlue,
			mark=triangle,
			mark size=1.5pt   
			]
			table[x=iter,y=avgalphaGD] {sections/introduction/intro-data/intro-data.dat};
			\addlegendentry{Empirical: GD}
			
			\addplot[
			color=Black,
			mark=*,
			mark size=1.5pt
			]
			table[x=iter,y=alphaPR] {sections/introduction/intro-data/intro-data.dat};
			\addlegendentry{Population}
			
			\addplot[
			color=CadetBlue,
			mark=triangle,
			mark size=1.5pt   
			]
			table[x=iter,y=avgbetaAM] {sections/introduction/intro-data/intro-data.dat};
			
			\addplot[
			color=RoyalBlue,
			mark=triangle,
			mark size=1.5pt   
			]
			table[x=iter,y=avgbetaGD] {sections/introduction/intro-data/intro-data.dat};

			\addplot[
			color=Black,
			mark=*,
			mark size=1.5pt
			]
			table[x=iter,y=betaPR] {sections/introduction/intro-data/intro-data.dat};
			
%
			
			\addplot+[name path=Aalpha-AM,CadetBlue!20, no markers] table[x=iter,y=loweralphaAM] {sections/introduction/intro-data/intro-data.dat};
			\addplot+[name path=Balpha-AM,CadetBlue!20, no markers] table[x=iter,y=upperalphaAM] {sections/introduction/intro-data/intro-data.dat};
			
			\addplot[CadetBlue!20] fill between[of=Aalpha-AM and Balpha-AM];
			
			\addplot+[name path=Abeta-AM,CadetBlue!20, no markers] table[x=iter,y=lowerbetaAM] {sections/introduction/intro-data/intro-data.dat};
			\addplot+[name path=Bbeta-AM,CadetBlue!20, no markers] table[x=iter,y=upperbetaAM] {sections/introduction/intro-data/intro-data.dat};
			
			\addplot[CadetBlue!20] fill between[of=Abeta-AM and Bbeta-AM];
			
			\addplot+[name path=Aalpha-GD,RoyalBlue!20, no markers] table[x=iter,y=loweralphaGD] {sections/introduction/intro-data/intro-data.dat};
			\addplot+[name path=Balpha-GD,RoyalBlue!20, no markers] table[x=iter,y=upperalphaGD] {sections/introduction/intro-data/intro-data.dat};
			
			\addplot[RoyalBlue!20] fill between[of=Aalpha-GD and Balpha-GD];
			
			\addplot+[name path=Abeta-GD,RoyalBlue!20, no markers] table[x=iter,y=lowerbetaGD] {sections/introduction/intro-data/intro-data.dat};
			\addplot+[name path=Bbeta-GD,RoyalBlue!20, no markers] table[x=iter,y=upperbetaGD] {sections/introduction/intro-data/intro-data.dat};
			
			\addplot[RoyalBlue!20] fill between[of=Abeta-GD and Bbeta-GD];
		\end{axis}
	\end{tikzpicture}

%% file: sections/introduction/intro-figs/gor-dist.tex
	\begin{tikzpicture}
		\begin{axis}[
			width=\textwidth,
			height=\textwidth,
			xlabel={Iteration},
			ymode=log,
xmin=0, xmax=17,
ymin=0.000000001, ymax=1,
legend style={font=\tiny, fill=none, at={(1,0.75)},anchor=east},
ymajorgrids=true,
grid style=grid,
			]
			
			\addplot[
			color=CadetBlue,
			mark=triangle,
			mark size=1.5pt   
			]
			table[x=iter,y=avgAM] {sections/introduction/intro-data/intro-data.dat};
			\legend{Empirical: AM}
			
			\addplot[
			color=RoyalBlue,
			mark=triangle,
			mark size=1.5pt   
			]
			table[x=iter,y=avgGD] {sections/introduction/intro-data/intro-data.dat};
			\addlegendentry{Empirical: GD}
			
			
			\addplot[
			color=Maroon,
			mark=*,
			mark size=1.5pt   
			]
			table[x=iter,y=gorAM] {sections/introduction/intro-data/intro-data.dat};
			\addlegendentry{Gordon: AM}
			
			\addplot[
			color=Salmon,
			mark=*,
			mark size=1.5pt   
			]
			table[x=iter,y=gorGD] {sections/introduction/intro-data/intro-data.dat};
			\addlegendentry{Gordon: GD}
			
			\addplot+[name path=A-AM,CadetBlue!20, no markers] table[x=iter,y=lowerAM] {sections/introduction/intro-data/intro-data.dat};
			\addplot+[name path=B-AM,CadetBlue!20, no markers] table[x=iter,y=upperAM] {sections/introduction/intro-data/intro-data.dat};
			
			\addplot[CadetBlue!20] fill between[of=A-AM and B-AM];
			
			\addplot+[name path=A-GD,RoyalBlue!20, no markers] table[x=iter,y=lowerGD] {sections/introduction/intro-data/intro-data.dat};
			\addplot+[name path=B-GD,RoyalBlue!20, no markers] table[x=iter,y=upperGD] {sections/introduction/intro-data/intro-data.dat};
			
			\addplot[RoyalBlue!20] fill between[of=A-GD and B-GD];
		\end{axis}
	\end{tikzpicture}

%% file: sections/introduction/intro-figs/gor-alpha-beta.tex
	\begin{tikzpicture}
		\node[above, black] (sourcebeta) at (4, 1.25) {$\beta$};
		\node[above, black] (sourcealpha) at (4, 5.5) {$\alpha$};
		\begin{axis}[
			width=\textwidth,
			height=\textwidth,
			xlabel={Iteration},
			ymode=log,
			xmin=0, xmax=17,
			ymin=0.000000001, ymax=1,
			legend style={font=\tiny, fill=none, at={(1,0.75)},anchor=east},
			ymajorgrids=true,
			grid style=grid,
			]
			
			\addplot[
			color=CadetBlue,
			mark=triangle,
			mark size=1.5pt   
			]
			table[x=iter,y=avgalphaAM] {sections/introduction/intro-data/intro-data.dat};
			\legend{Empirical: AM}
			
			\addplot[
			color=RoyalBlue,
			mark=triangle,
			mark size=1.5pt   
			]
			table[x=iter,y=avgalphaGD] {sections/introduction/intro-data/intro-data.dat};
			\addlegendentry{Empirical: GD}
			
			\addplot[
			color=Maroon,
			mark=*,
			mark size=1.5pt   
			]
			table[x=iter,y=alphagorAM]{sections/introduction/intro-data/intro-data.dat};
			\addlegendentry{Gordon: AM}
			
			\addplot[
			color=Salmon,
			mark=*,
			mark size=1.5pt   
			]
			table[x=iter,y=alphagorGD] {sections/introduction/intro-data/intro-data.dat};
			\addlegendentry{Gordon: GD}
			
			\addplot[
			color=CadetBlue,
			mark=triangle,
			mark size=1.5pt   
			]
			table[x=iter,y=avgbetaAM] {sections/introduction/intro-data/intro-data.dat};
			
			\addplot[
			color=RoyalBlue,
			mark=triangle,
			mark size=1.5pt   
			]
			table[x=iter,y=avgbetaGD] {sections/introduction/intro-data/intro-data.dat};
			
%

			\addplot[
			color=Maroon,
			mark=*,
			mark size=1.5pt   
			]
			table[x=iter,y=betagorAM]{sections/introduction/intro-data/intro-data.dat};

			\addplot[
			color=Salmon,
			mark=*,
			mark size=1.5pt   
			]
			table[x=iter,y=betagorGD] {sections/introduction/intro-data/intro-data.dat};
			
			\addplot+[name path=Aalpha-AM,CadetBlue!20, no markers] table[x=iter,y=loweralphaAM] {sections/introduction/intro-data/intro-data.dat};
			\addplot+[name path=Balpha-AM,CadetBlue!20, no markers] table[x=iter,y=upperalphaAM] {sections/introduction/intro-data/intro-data.dat};
			
			\addplot[CadetBlue!20] fill between[of=Aalpha-AM and Balpha-AM];
			
			\addplot+[name path=Abeta-AM,CadetBlue!20, no markers] table[x=iter,y=lowerbetaAM] {sections/introduction/intro-data/intro-data.dat};
			\addplot+[name path=Bbeta-AM,CadetBlue!20, no markers] table[x=iter,y=upperbetaAM] {sections/introduction/intro-data/intro-data.dat};
			
			\addplot[CadetBlue!20] fill between[of=Abeta-AM and Bbeta-AM];
			
			\addplot+[name path=Aalpha-GD,RoyalBlue!20, no markers] table[x=iter,y=loweralphaGD] {sections/introduction/intro-data/intro-data.dat};
			\addplot+[name path=Balpha-GD,RoyalBlue!20, no markers] table[x=iter,y=upperalphaGD] {sections/introduction/intro-data/intro-data.dat};
			
			\addplot[RoyalBlue!20] fill between[of=Aalpha-GD and Balpha-GD];
			
			\addplot+[name path=Abeta-GD,RoyalBlue!20, no markers] table[x=iter,y=lowerbetaGD] {sections/introduction/intro-data/intro-data.dat};
			\addplot+[name path=Bbeta-GD,RoyalBlue!20, no markers] table[x=iter,y=upperbetaGD] {sections/introduction/intro-data/intro-data.dat};
			
			\addplot[RoyalBlue!20] fill between[of=Abeta-GD and Bbeta-GD];
		\end{axis}
	\end{tikzpicture}

%% file: sections/background/setup.tex

In this section, we set up our formal observation model, and a general form for the iterative algorithms that we will study. 

\subsection{Observation model}
Suppose that we observe i.i.d. covariate response pairs $(\bd{x}_i, y_i)$ generated according to the model
\begin{align}\label{eq:model}
y_i = f(\langle \bd{x}_i, \thetastar\rangle; q_i) + \epsilon_i, \qquad i = 1, 2, \ldots.
\end{align} 
The covariates $\bx_i$ are assumed to be $d$-dimensional and drawn i.i.d. from the standard normal distribution $\NORMAL(0, \bd{I}_{d})$, and the function $f$ is some known \emph{link function}. The random variable $q_i \sim \mathbb{Q}$ represents a possible \emph{latent variable}, i.e., some source of auxiliary randomness that is unobserved, and $\epsilon_i$ represents additive noise drawn from the distribution $\NORMAL(0, \sigma^2)$; both of these are drawn i.i.d.
Our goal is to use observations of pairs $(\bx_i, y_i)_{i \geq 1}$ to estimate the unknown $d$-dimensional parameter $\thetastar$.
For the rest of this paper, we will make the assumption that $\| \thetastar \|_2 = 1$ in order to simplify statements of our theoretical results.\footnote{This assumption can be removed by straightforward means; in particular, the algorithms that we study will not make explicit use of the fact that $\thetastar$ is unit-norm.} Before proceeding, let us give two canonical examples of the observation model~\eqref{eq:model} that will form the focus of this paper, illustrating why maximum likelihood estimation in these models can be computationally challenging.

\paragraph{Example: Phase retrieval with a real-valued signal.}  Here, there is no auxiliary latent variable, and the function $f$ depends solely on its first argument.  In the \emph{nonsmooth} version, it is given by $f(t; q) = \lvert t \rvert$, so that our model for the $i$-th observation takes the form
\begin{align} \label{eq:PR-model}
y_i = |\langle \bd{x}_i, \thetastar\rangle| + \epsilon_i.
\end{align}
Our goal is to estimate the real-valued \emph{signal} $\thetastar$ from these covariate-response pairs. 
Note that the negative log-likelihood of our observations $(\bX, \by)$ is the shifted least squares objective
\begin{align} \label{eq:LL-PR}
- \log p(\bt; \bX, \by) = \frac{1}{n} \sum_{i = 1}^n \left( y_i - |\langle \bd{x}_i, \bt\rangle| \right)^2 + c_0,
\end{align}
where $c_0$ is a scalar independent of $\bt$. This is a nonconvex function of $\bt \in \real^d$.
\hfill $\clubsuit$

\paragraph{Example: Symmetric mixture of linear regressions.} Here, the latent variables $q_i$ are chosen i.i.d. from a Rademacher distribution $\mathsf{Unif}(\{\pm 1 \})$ and the function $f$ is specified by $f(t; q) = q\cdot t$. This leads to the observation model
\begin{align} \label{eq:MoR-model}
y_i = q_i \cdot \langle \bd{x}_i, \thetastar\rangle + \epsilon_i
\end{align}
for the $i$-th observation. 
The negative log-likelihood of our observations $(\bX, \by)$ is given by
\begin{align} \label{eq:LL-MoR}
-\log p(\bt; \bX, \by) = -\frac{1}{n} \sum_{i = 1}^n  \log \left( \exp\left\{ - \frac{( y_i - \langle \bd{x}_i, \bt\rangle)^2}{2\sigma^2}  \right\} + \exp\left\{ - \frac{( y_i + \langle \bd{x}_i, \bt\rangle)^2}{2\sigma^2} \right\} \right) + c_0,
\end{align}
where $c_0$ is a scalar independent of $\bt$. Clearly, this is a nonconvex function of $\bt$.
\hfill $\clubsuit$
\medskip

An important feature of estimation under the general observation model~\eqref{eq:PR-model} that is exemplified by the specific cases above is that the negative log-likelihood, when viewed as a function of the parameter of interest, is nonconvex. 
Nevertheless, it is common to run iterative algorithms---beginning either from a random initialization or a carefully designed spectral initialization---to attempt to optimize the negative log-likelihood. Our focus will be on studying two such canonical families of iterative algorithms from a random initialization, which we introduce next and under a general framework.

\subsection{Iterative algorithms}\label{sec:it_alg}

We study iterative algorithms designed to recover $\thetastar$ in the observation model~\eqref{eq:model} when run with sample-splitting. In particular, suppose that at each iteration, we form a \emph{fresh} batch\footnote{Owing to sample-splitting, the pair $(\bX, \by)$ can also be thought of as depending on the iteration number $t$, but we suppress this dependence and opt for more manageable notation.} of $n$ observations by collecting the covariates in a matrix $\bX \in \real^{n \times d}$ and the responses in a vector $\by \in \real^n$. By design, the pair $(\bX, \by)$ is statistically independent of the iterations of the algorithm thus far.
At iteration $t$, we update our current estimate of the parameter $\bt_t$ to $\bt_{t + 1}$ by solving an optimization problem of the form
\begin{align} \label{eq:opt_gen_intro}
\thetanext \in \argmin_{\bt \in \mathbb{R}^d} \;
\mathcal{L}(\bt; \thetacur, \bd{X}, \bd{y}),
\end{align}
for some loss function $\mathcal{L}$ that depends implicitly on the current point $\bt_t$ and is formed using the data $(\bX, \by)$.  In general terms, what makes the iterative algorithm tractable is that the optimization problem~\eqref{eq:opt_gen_intro} corresponding to each iteration is solvable efficiently. More often than not, this is enabled by the function $\mathcal{L}$ being convex in $\bt$, a property that we will exploit fruitfully in the examples that we study.

It is important to note that owing to our sample splitting heuristic, the total sample size when the iterative algorithm is run for $T$ iterations is given by $n\cdot T$. In the sequel, it is useful to track the per-iteration oversampling ratio, given by
\begin{align*}
\kappa = \frac{n}{d}.
\end{align*}
We will be interested in the near-linear regime of sample size in which $\kappa$ scales at most poly-logarithmically\footnote{In the specific examples that we study, the number of iterations $T$ required to obtain order-optimal parameter estimates will turn out to be at most logarithmic in the dimension, so that the total sample size $nT$ also scales near-linearly in dimension.}  in the dimension~$d$.

Let us conclude by introducing some equivalent operator-theoretic notation that simplifies some of our exposition. It is common to view a step of the algorithm through the lens of an \emph{empirical operator} $\mathcal{T}_n: \real^d \to \real^d$, with 
\begin{align} \label{eq:Tn}
\mathcal{T}_n(\bt) = \argmin_{\bt' \in \mathbb{R}^d} \; \mathcal{L}(\bt'; \bt, \bd{X}, \bd{y}) \;\; \text{ for each } \bt \in \real^d. 
\end{align}
In other words, equation~\eqref{eq:opt_gen_intro} denotes the evaluation of the operator at $\bt_t$, i.e., with $\bt_{t + 1} = \mathcal{T}_n(\bt_t)$. Note that the operator $\mathcal{T}_n$ is random by virtue of randomness in the data, and that since we are interested in the algorithm run with sample-splitting, one may view the random operator $\mathcal{T}_n$ as being generated i.i.d. at each iteration. Adopting this perspective, the parameter estimate obtained at iteration $k$ when starting from an initial point $\bt_0$ is given by applying the random operator $\mathcal{T}_n$ repeatedly,
so that $\bt_k = \mathcal{T}^k_n (\bt_0)$.

We now discuss two specific classes of algorithms from this general perspective.

\subsubsection{Higher-order update methods}

The first class of methods that we consider are those that do not have an interpretation as first-order methods. In particular, they typically involve running least squares in each iteration. As we will see in the examples to follow, each of these algorithms can be written in 
the form~\eqref{eq:opt_gen_intro} with 
\begin{subequations}  \label{eq:sec_order_gen}
\begin{align} \label{eq:sec_order_gen_loss}
\Lc(\bt; \thetacur,\bX,\by):= \sqrt{ \frac{1}{n}\sum_{i=1}^{n}\left( \wfun( \langle \bd{x}_i, \thetacur\rangle, y_i) - \langle \bd{x}_i, \bt \rangle\right)^2},
\end{align}
with $\wfun: \real^2 \to \real$ denoting a \emph{weight} function that is model and algorithm dependent and the square root is taken for convenience.
The minimizer of the loss~\eqref{eq:sec_order_gen_loss} is given by
\begin{align}
\bt_{t+ 1} = \biggl(\sum_{i=1}^{n} \bx_i \bx_i^{\top}\biggr)^{-1} \biggl(\sum_{i=1}^{n} \omega(\langle \bx_i, \thetacur\rangle, y_i) \cdot  \bx_i\biggr); \label{eq:second-order-update-sum}
\end{align}
\end{subequations}
in other words, we apply the empirical operator \mbox{$\mathcal{T}_n: \bt \mapsto \biggl(\sum_{i=1}^{n} \bx_i \bx_i^{\top}\biggr)^{-1} \biggl(\sum_{i=1}^{n} \omega(\langle \bx_i, \bt \rangle, y_i) \cdot \bx_i\biggr)$}.
Let us provide a few examples of such methods in the specific cases~\eqref{eq:PR-model} and~\eqref{eq:MoR-model} for concreteness.

\paragraph{Example: Alternating projections for phase retrieval.} To motivate the first example, consider the phase retrieval model and write the corresponding negative log-likelihood~\eqref{eq:LL-PR} in the equivalent form
$
- \log g(\bt; \bX, \by) = \frac{1}{n} \sum_{i = 1}^n (\sign(\inprod{\bx_i}{\bt}) \cdot y_i -  \inprod{\bx_i}{\bt} )^2 + c_0.
$
This suggests a heuristic that fixes the signs using the current iterate $\bt_{t}$, and obtains $\bt_{t + 1}$ by minimizing the loss
\begin{subequations} \label{eq:overall-AM-PR}
\begin{align} \label{eq:loss-AM-PR}
\mathcal{L}(\bt; \bt_t, \bX, \by) = \sqrt{ \frac{1}{n}\sum_{i=1}^{n}  \left(\sign(\langle \bd{x}_i, \thetacur\rangle)y_i - \langle \bd{x}_i, \bt \rangle\right)^2}.
\end{align}
Concretely, this results in the update
\begin{align} \label{eq:update-explicit-AM-PR}
\bt_{t + 1} = \biggl(\sum_{i=1}^{n} \bx_i \bx_i^{\top}\biggr)^{-1} \biggl(\sum_{i=1}^{n} \sign(\inprod{\bx_i}{\bt_t}) \cdot y_i \bx_i\biggr). 
\end{align}
\end{subequations}
Clearly, this loss function/update pair takes the general form~\eqref{eq:sec_order_gen} with the specific choice \sloppy \mbox{$\wfun( x, y) = \sign(x) \cdot y$}. \hfill $\clubsuit$

\paragraph{Example: Alternating projections for mixtures of two regressions.} This algorithm stems from the observation that while the negative log-likelihood of $\bt$---given by equation~\eqref{eq:LL-MoR}---may be difficult to optimize, the likelihood of the pair $(\bt, \bq)$ is often easier to reason about. In particular, writing
\begin{align*}
- \log h(\bt, \bq; \bX, \by) = \frac{1}{n} \sum_{i = 1}^n \left( y_i - q_i \cdot \langle \bd{x}_i, \bt\rangle \right)^2 + c_0\,,
\end{align*}
for a scalar $c_0$ that is independent of the pair $(\bt, \bq)$, notice that the function $- \log h$ is now individually convex in each of $\bt$ and $\bq$. This suggests an alternating update algorithm: Suppose that the current parameter is $\bt_t$; then for each $i$, the minimizer of $-\log h$ over $q_i \in \{-1, 1\}$ is given by
$
\argmin_{q \in \{-1, 1\}} |y_i - q \cdot \inprod{\bx_i}{\bt_t}| = \sign(y_i\langle \bd{x}_i, \thetacur\rangle).
$
This in turn yields the one-step loss function
\begin{subequations}
\begin{align} \label{eq:loss-AM-MR}
\mathcal{L}(\bt; \bt_t, \bX, \by) = \sqrt{ \frac{1}{n}\sum_{i=1}^{n}  \left(\sign(y_i\langle \bd{x}_i, \thetacur\rangle) \cdot y_i - \langle \bd{x}_i, \bt \rangle\right)^2}
\end{align}
and the corresponding update
\begin{align} \label{eq:update-explicit-AM-MR}
\thetanext = \biggl(\sum_{i=1}^{n} \bx_i \bx_i^{\top}\biggr)^{-1} \biggl(\sum_{i=1}^{n} \sign(y_i \inprod{\bx_i}{\bt_t}) \cdot y_i \bx_i\biggr),
\end{align}
\end{subequations}
which takes the general form~\eqref{eq:sec_order_gen} with $\wfun( x, y) = \sign(yx) \cdot y$. \hfill $\clubsuit$
\medskip

We note in passing that alternating projections for mixtures of linear regressions coincides with the expectation maximization (EM) algorithm~\citep{dempster1977maximum} when $\sigma = 0$, and that the machinery that we develop also applies to the EM algorithm. 
Let us now turn to a second class of (simpler) iterative algorithms.

\subsubsection{First order methods} 
The second class of methods that we analyze are first-order versions of counterparts presented above. As we will see shortly, 
each of these
methods can also be written in the form~\eqref{eq:opt_gen_intro} with
\begin{subequations} \label{eq:GD_gen}
\begin{align}\label{eq:GD_gen_loss}
\Lc(\bt;\thetacur,\bX,\by):=\frac{1}{2}\|\bt\|_2^2 - \langle \bt , \bt_t \rangle + \frac{2\eta}{n}\sum_{i=1}^n \wfun\big(\inp{\bx_i}{\thetacur},y_i\big) \langle \bx_i , \bt \rangle,
\end{align}
where $\eta > 0$ denotes a stepsize and $\wfun$ is some weight function. It is important to note that the function $\wfun$ will be distinct for the higher-order update and its first-order analog.

Minimizing the loss function~\eqref{eq:GD_gen_loss} over $\bt$, the update in this case can be written as
\begin{align}
\thetanext = \thetacur - \eta\cdot\frac{2}{n}\sum_{i=1}^n \wfun\big(\inp{\bx_i}{\thetacur},y_i\big)\cdot \bx_i,
\end{align}
\end{subequations}
which resembles a gradient update and induces the operator\\
 \mbox{$\mathcal{T}_n: \bt \mapsto \thetacur - \eta\cdot\frac{2}{n}\sum_{i=1}^n \wfun\big(\inp{\bx_i}{\thetacur},y_i\big)\cdot \bx_i$}. Examples are collected below.

\paragraph{Example: Subgradient descent for nonsmooth phase retrieval.} Our first example is given by the subgradient descent algorithm on the objective~\eqref{eq:LL-PR}. In particular, straightforward calculation yields that one iteration of this algorithm run with stepsize $\eta$ takes the form
\begin{subequations} 
\begin{align} \label{eq:update-explicit-GD-PR}
\thetanext &= \thetacur - \eta \cdot \left\{ \frac{2}{n} \sum_{i=1}^{n}\sign(\langle \bd{x}_i, \thetacur\rangle) \cdot (\lvert \langle \bd{x}_i, \thetacur \rangle \rvert - y_i) \cdot \bd{x}_i \right\},
\end{align}
which in turn is the minimizer of the loss function
\begin{align} \label{eq:loss-GD-PR}
\mathcal{L}(\bt; \bt_t, \bX, \by) &=  \frac{1}{2}\vecnorm{\bt}{2}^2 - \langle \bt, \thetacur \rangle + \frac{2\eta}{n} \sum_{i=1}^{n} \sign(\langle \bd{x}_i, \thetacur\rangle) \cdot (\lvert \langle \bd{x}_i, \thetacur\rangle \rvert - y_i) \cdot \langle \bd{x}_i, \bt \rangle.
\end{align}
\end{subequations}
This takes the general form~\eqref{eq:GD_gen} with $\wfun( x,y)=\sign(x) \cdot (|x|-y)  = x - \sign(x) \cdot y$. \hfill $\clubsuit$

\paragraph{Example: Subgradient AM for mixtures of regressions.} This update is obtained by running subgradient descent on the loss function in equation~\eqref{eq:loss-AM-MR}. In particular, running this algorithm with stepsize $\eta$ yields
\begin{subequations} 
\begin{align} \label{eq:update-explicit-GD-MR}
\thetanext &= \thetacur - \eta \cdot \left\{ \frac{2}{n} \sum_{i=1}^{n} (\inprod{\bd{x}_i}{\thetacur} - \sign(y_i \langle \bd{x}_i, \thetacur\rangle) \cdot y_i) \cdot \bd{x}_i \right\},
\end{align}
which is clearly the minimizer of the loss
\begin{align} \label{eq:loss-GD-MR}
\mathcal{L}(\bt; \bt_t, \bX, \by) = \frac{1}{2}\vecnorm{\bt}{2}^2 - \langle \bt, \thetacur \rangle + \frac{2\eta}{n} \sum_{i=1}^{n} (\inprod{\bd{x}_i}{\thetacur} - \sign(y_i \langle \bd{x}_i, \thetacur\rangle) \cdot  y_i) \cdot \inprod{\bd{x}_i}{\bt}.
\end{align}
\end{subequations}
These expressions take the general form~\eqref{eq:GD_gen} with $\wfun( x,y)=x - \sign(xy) \cdot y$. 

We note that this algorithm is analogous to a gradient EM update~\citep{neal1998view} in that it is obtained via a first order method applied to the one-step loss function derived with the objective of performing alternating minimization. However, to our knowledge, this algorithm has not been considered before in the literature on mixtures of linear regressions.
\hfill $\clubsuit$
\medskip

\begin{remark} \label{rem:weight-rel}
As noted before, the weight functions of the higher-order and first order updates corresponding to a particular model \emph{do not} coincide. However, note that in the examples presented above, we have
\begin{align*}
\wfun_{\mathsf{FO}}(x, y) = x - \wfun_{\mathsf{HO}}(x, y),
\end{align*}
where $\wfun_{\mathsf{HO}}$ and $\wfun_{\mathsf{FO}}$ denote the higher-order and first-order weight function, respectively. 
\end{remark}

Having introduced illustrative examples, we are now well-placed to introduce our general recipe for establishing convergence
guarantees on iterative algorithms.

%% file: sections/heuristic-derivation/heuristic-derivation_alt.tex

We are now ready to describe the Gordon state evolution update in detail. We begin with a high-level overview, in Section~\ref{sec:high-level-steps}, of the steps involved in the recipe, and then provide a heuristic but illustrative derivation for a specific algorithm in Section~\ref{subsec:heuristic-second-order}. Having conveyed the high-level intuition about how one might derive these updates in concrete problems, we then proceed to a rigorous result, in Section~\ref{sec:main-result-Gordon}, showing that the empirical iteration concentrates around the Gordon state evolution update.

\subsection{High-level sketch of the steps} \label{sec:high-level-steps}

We begin with the ansatz---which will be intuitively justified in the heuristic derivation of Section~\ref{subsec:heuristic-second-order} and proved rigorously when establishing the main results to follow---that it suffices to track the two dimensional state evolution $(\parcomp(\bt), \perpcomp(\bt))$ defined in equation \eqref{eq:def-state}. In particular, when one step of the algorithm is run from the parameter $\bt_t$ to obtain $\bt_{t + 1}$, we are interested in a \emph{deterministic} prediction $(\widebar{\parcomp}_{t + 1}, \widebar{\perpcomp}_{t + 1})$ for the random pair $(\parcomp(\bt_{t + 1}), \perpcomp(\bt_{t + 1}))$ that is (a) a function only of the pair $(\parcomp(\bt_t), \perpcomp(\bt_t))$, and (b) accurate up to a small error. 
We use several steps  to derive such a deterministic \emph{state evolution update}. 
Let us begin by introducing the convex Gaussian minmax theorem, or CGMT, which forms the bedrock of our recipe.  
\begin{proposition}[CGMT~\citep{thrampoulidis2015regularized}]
	\label{thm:gordon}
	Let $\bG$ 
	denote an $n \times d$ standard Gaussian random matrix, and let $\bh \in \mathbb{R}^d$ and $\bg \in \mathbb{R}^n$ denote standard Gaussian random vectors drawn independently of each other and of $\bG$. Let $\bL \in \real^{d \times d}$ and $\bM \in \real^{n \times n}$ denote two fixed matrices. Also, let $\mathcal{U} \subseteq \mathbb{R}^{d}$ and $\mathcal{V} \subseteq \mathbb{R}^n$ denote compact sets, and let $Q: \mathcal{U} \times \mathcal{V} \rightarrow \mathbb{R}$
	denote a continuous function. Define 
	\begin{subequations}
	\begin{align}
		P(\bG) &:= \min_{\bu \in \mathcal{U}} \max_{\bv \in \mathcal{V}} \; \langle \bM \bv, \bG \bL \bu \rangle + Q(\bu, \bv) \quad \text{ and }\label{eq:CGMT-PO}\\
		A(\bg, \bh) &:= \min_{\bu \in \mathcal{U}} \max_{\bv \in \mathcal{V}} \; \| \bM \bv \|_2 \cdot \langle \bh, \bL \bu \rangle + \| \bL \bu \|_2 \cdot \langle \bg, \bM \bv \rangle + Q(\bu, \bv) \label{eq:CGMT-AO}.
	\end{align}
	\end{subequations}
	Then
	\begin{enumerate}[label=(\alph*)]
		\item For all $t \in \mathbb{R}$, we have
		\[\Pro\left\{P(\bG) \leq t\right\} \leq 2\Pro\left\{A (\bg, \bh) \leq
		t\right\}.\]
		\item If, in addition, the sets $\mathcal{U}, \mathcal{V}$ are convex and the function $Q$ is convex-concave, then for all $t \in
		\mathbb{R}$, we have
		\[\Pro\left\{P (\bG) \geq t\right\} \leq 2\Pro\left\{A(\bg, \bh) \geq
		t\right\}.\]
	\end{enumerate}
\end{proposition} 
Strictly speaking, Proposition~\ref{thm:gordon} is a generalization of the result appearing in~\citet{thrampoulidis2015regularized}, which is stated without the matrix pair $(\bL, \bM)$. However, its proof follows identically, and we choose to state the more general result since it is most useful for our development. Following the terminology from~\citet{thrampoulidis2015regularized}, we refer to equation~\eqref{eq:CGMT-PO} as the \emph{primary optimization problem} or PO, and to equation~\eqref{eq:CGMT-AO} as the \emph{auxiliary optimization problem} or AO.
Having stated the CGMT, let us now provide a rough outline of the steps involved in deriving the Gordon state evolution update. These are then concretely instantiated in heuristic derivations carried out in Section~\ref{subsec:heuristic-second-order}. In this section, we will deliberately avoid technical details; Section~\ref{sec:gordon} to follow makes all the steps rigorous in the general case, along the way to proving our main results in Theorems~\ref{thm:one_step_main_HO} and~\ref{thm:one_step_main_FO}.

\paragraph{Step 1: Write one iteration of algorithm as solution to convex optimization problem.}
As alluded to in Section~\ref{sec:it_alg}, each iteration of most algorithms---even on nonconvex log-likelihood functions---can be written as the solution to a convex optimization problem~\eqref{eq:opt_gen_intro}. To recall this more explicitly, suppose that running one step of the algorithm from the parameter $\bt_t$ results in the update $\bt_{t + 1} = \argmin_{\bt \in \real^d} \mathcal{L}(\bt; \bt_t, \bX, \by)$, where $\mathcal{L}$ is convex in $\bt$ for each fixed triple $(\bt_t, \bX, \by)$. This was indeed the case in all the illustrative examples in Section~\ref{sec:setup}, but is true more broadly with many iterative algorithms. 

\paragraph{Step 2: Write equivalent auxiliary optimization problem.} In this step, our goal is to write the minimization of the loss function $\mathcal{L}$---which is a function of the Gaussian design matrix $\bX$---as a \emph{simpler} minimization involving fewer Gaussian random variables. In particular, we would like to show that
\begin{align} \label{eq:aux-step-full}
\min_{\bt \in \real^d} \mathcal{L}(\bt; \bt_t, \bX, \by) \approx \min_{\bt \in \real^d} \mathfrak{L}(\bt; \bt_t, \bh, \bg),
\end{align}
where $\bh$ and $\bg$ denote (either $n$ or $d$-dimensional) standard Gaussian random \emph{vectors} and the $\approx$ symbol denotes some form of approximate equality in distribution motivated by Proposition~\ref{thm:gordon}. The latter optimization problem is typically easier to solve and admits a representation in terms of a small number of decision variables \citep{thrampoulidis2018precise}.

The key workhorse in this step is the CGMT, and the program typically consists of two substeps:
\begin{enumerate}[label=(\roman*)]
\item \underline{Frame optimization problem in the form~\eqref{eq:CGMT-PO}:} First, we show that there exists a standard Gaussian random matrix $\bG$ and a pair of fixed matrices $(\bL, \bM)$ such that the convex optimization problem~\eqref{eq:opt_gen_intro} can be written in the form~\eqref{eq:CGMT-PO}, i.e., 
\[
\min_{\theta \in \mathbb{R}^d} \mathcal{L}(\bt; \bt_t, \bX, \by) \overset{(d)}{=} \underbrace{ \min_{\bu \in \mathcal{U}} \max_{\bv \in \mathcal{V}} \; \langle \bM \bv, \bG \bL \bu \rangle + Q(\bu, \bv) }_{P(\bG)},
\]
where the pair of decision variables $(\bu, \bv)$ is determined by the parameters $(\bt, \bt_t)$.
Here, it is important to note that the function $Q$
may depend on randomness \emph{independent} of $\bG$.
\item \underline{Invoke the CGMT and formulate the auxiliary optimization problem:} 
Next, we use the CGMT to simplify the problem. Proposition~\ref{thm:gordon} shows that $P(\bG)$ is very well approximated (in distribution) by $A(\bg, \bh)$, and moreover, the optimization problem~\eqref{eq:CGMT-AO} involves two Gaussian random vectors $\bh$ and $\bg$ and in many cases is easier to solve. Applying this leads to an equivalence of the form~\eqref{eq:aux-step-full}, as desired.
\end{enumerate}

\paragraph{Step 3: Scalarize to obtain deterministic Gordon state evolution update:} As mentioned before, writing the optimization problem in terms of the objective $\mathfrak{L}$ was motivated by the fact that this objective could be \emph{scalarized} in terms of a low dimensional function. In this step, our goal is to establish the approximate equivalence
\begin{align}
\min_{\bt \in \real^d} \; \mathfrak{L}(\bt; \bt_t, \bh, \bg) \approx  \min_{\bxi} \; \widebar{L}(\bxi ; \bxi_t),
\end{align}
where
$\widebar{L}(\cdot ; \bxi_t): \real^3 \to \real$ is a deterministic function solely of a three-dimensional \emph{state}, and moreover, 
depends on the previous iterate only through its own three-dimensional state $\bxi_t$. 
The minimizers of the optimization problem on the RHS---along with some algebraic simplification---will then yield the deterministic, two-dimensional Gordon state evolution update $(\widebar{\parcomp}_{t + 1}, \widebar{\perpcomp}_{t + 1})$ as alluded to in the ansatz. As before, this step is typically accomplished via two further substeps:
\begin{enumerate}[label=(\roman*)]
\item \underline{Argue equivalence to a random low-dimensional function $\widebar{L}_n$:} This is often easy to do just via a change of variables, expressing the $d$-dimensional parameters $\bt$ and $\bt_t$ in terms of their respective states $\bxi$ and $\bxi_t$. It is important to note however that the objective function $\widebar{L}_n: \real^3 \to \real$ that results from this transformation is still random.
\item \underline{Use the LLN to obtain population loss $\widebar{L} = \lim_{n \to \infty} \; \widebar{L}_n$, and solve:} The key technique enabled by the scalarization above is that since we are now in low (i.e., $3$) dimensions, passing to the population loss still provides an accurate prediction of behavior even when $n$ is moderately large. Solving for the minimizers of $\widebar{L}$ can be done readily; typically, the solutions to this low-dimensional optimization problem will coincide with the solutions to a nonlinear system of equations (in three variables)\footnote{In the examples, we consider in this paper, the solutions turn out to be computable in closed form.}.
\end{enumerate}

\paragraph{Step 4: Argue that the empirical state evolution is tracked by the Gordon update.}
The final step is to use growth properties of the objective functions $\mathfrak{L}$ and $\widebar{L}_n$ around their minima to show that if their optimum values coincide, then so must their optimizers. This is the most technical step of the recipe, and a large portion of the proof is dedicated to establishing these properties.
\medskip

The following subsection clarifies these abstract steps by carrying out a concrete derivation on an example.  We emphasize that the derivations are heuristic and aim to illustrate the recipe.  We defer precise statements and their proofs to Section~\ref{sec:main-result-Gordon}.

\subsection{Implementing the recipe: A heuristic derivation in a special case}
\label{subsec:heuristic-second-order}

To illustrate the steps sketched above, we present a heuristic derivation of the Gordon update in the case of alternating minimization for {noiseless} phase retrieval~\eqref{eq:overall-AM-PR}.

\paragraph{Step 1: One-step update as a convex optimization problem.}
Letting $\odot$ denote the Hadamard product between two vectors of the same dimension, notice that the update when run from $\thetacur$ is given by
\begin{align*}
\thetanext &= 
\argmin_{\bt \in \mathbb{R}^d} \frac{1}{\sqrt{n}} \| \bX\bt - \mathsf{sgn}(\bX \bt_t) \odot \by\|_2,
\end{align*}
which is clearly the minimizer of the convex loss $\mathcal{L}(\bt; \bt_t, \bX, \by) =  \frac{1}{\sqrt{n}} \| \bX\bt - \mathsf{sgn}(\bX \bt_t) \odot \by\|_2$. 

\paragraph{Step 2: Equivalent auxiliary optimization problem.} Let us detail the two substeps individually:
\begin{enumerate}[label=(\roman*)]
\item \underline{Frame optimization problem in the form~\eqref{eq:CGMT-PO}:} First, observe that via the dual norm characterization of the $\ell_2$ norm, we have
\begin{align*}
\mathcal{L}(\bt; \bt_t, \bX, \by) =  \frac{1}{\sqrt{n}} \| \bX\bt - \mathsf{sgn}(\bX \bt_t) \odot \by\|_2 = \max_{\| \bv \|_2 \leq 1} \; \frac{1}{\sqrt{n}}\langle \bv, \bX \bt \rangle - \frac{1}{\sqrt{n}}\langle \bv, \mathsf{sgn}(\bX \bt_t) \odot \by \rangle.
\end{align*}
The first term in the RHS above is bilinear in the Gaussian random matrix $\bX$, but the second term also depends on $\bX$ and so does not immediately take the form required in equation~\eqref{eq:CGMT-PO}.
To remedy this issue, consider the fixed subspace $S_t = \mathsf{span}(\thetastar, \thetacur)$ 
and write 
\begin{align*}
\mathcal{L}(\bt; \bt_t, \bX, \by) =  \max_{\| \bv \|_2 \leq 1} \; \frac{1}{\sqrt{n}}\langle \bv, \bX \proj_{S_t}^{\perp}\bt \rangle + \frac{1}{\sqrt{n}}\langle \bv, \bX \proj_{S_t}\bt \rangle- \frac{1}{\sqrt{n}}\langle \bv, \mathsf{sgn}(\bX \bt_t) \odot \by \rangle,
\end{align*}
where $\proj_{S_t}$ and $\proj_{S_t}^{\perp}$ denote projection matrices onto the subspaces $S_t$ and $S_t^{\perp}$, respectively.
By construction, the first term on the RHS is independent of the rest, and so we may replace the matrix $\bX$ in this term with an independent copy $\bG$ to obtain
\begin{align*}
\mathcal{L}(\bt; \bt_t, \bX, \by) \overset{(d)}{=}
\max_{\| \bv \|_2 \leq 1} \; \frac{1}{\sqrt{n}}\langle \bv, \bG \proj_{S_t}^{\perp}\bt \rangle + Q(\bv, \bt), 
\end{align*}
where $Q(\bv, \bt) = \frac{1}{\sqrt{n}}\langle \bv, \bX \proj_{S_t}\bt \rangle- \frac{1}{\sqrt{n}}\langle \bv, \mathsf{sgn}(\bX \bt_t) \odot \by \rangle$ is now independent of $\bG$. 
This leads to the definition
\begin{align*}
P(\bG) = \min_{\bt \in \real^d} \max_{\| \bv \|_2 \leq 1} \; \frac{1}{\sqrt{n}}\langle \bv, \bG \proj_{S_t}^{\perp}\bt \rangle + Q(\bv, \bt), 
\end{align*}
which takes the form~\eqref{eq:CGMT-PO}.

\item \underline{Invoke the CGMT and approximate the minimum of loss function.}
Given that this is a heuristic derivation, we ignore for the moment that the set $\real^d$ is not compact and use the CGMT to write $P(\bG) \approx A(\bg, \bh)$,
where
\[
A(\bg, \bh) = \min_{\bt \in \real^d} \max_{\| \bv \|_2 \leq 1} \frac{1}{\sqrt{n}} \| \bv \|_2 \langle \bh, \proj_{S_t}^{\perp} \bt \rangle + \frac{1}{\sqrt{n}}\| \proj_{S_t}^{\perp} \bt \|_2 \langle \bg, \bv \rangle + Q(\bv, \bu),
\]
and the approximation $T_1 \approx T_2$ signifies that the CDFs of the two random variables $T_1$ and $T_2$ match up to a factor $2$ (see Proposition~\ref{thm:gordon}).
Note that heuristically speaking, we have shown through the previous steps that \sloppy
\mbox{$\min_{\bt \in \real^d} \mathcal{L}(\bt; \bt_t, \bX, \by) \approx \min_{\bt \in \real^d}\mathfrak{L}(\bt; \bt_t, \bh, \bg)$}, where $\mathfrak{L}(\bt; \bt_t, \bh, \bg)$ has the variational representation
\[
\max_{\| \bv \|_2 \leq 1} \; \frac{1}{\sqrt{n}} \| \bv \|_2 \cdot \langle \bh, \proj^{\perp}_{S_t}\bt \rangle + \frac{1}{\sqrt{n}} \| \proj^{\perp}_{S_t}\bt  \|_2 \cdot \langle \bv, \bg \rangle + \frac{1}{\sqrt{n}}  \langle \bv, \bX \proj_{S_t}\bt \rangle- \frac{1}{\sqrt{n}}\langle \bv, \mathsf{sgn}(\bX \bt_t) \odot \by \rangle.
\]
\end{enumerate}

\paragraph{Step 3: Scalarize and obtain Gordon update.} We now scalarize the problem by introducing the change of variables
\begin{align}
	\label{eq:heuristic-gordon-change-of-var}
\parcomp = \langle \bt, \thetastar \rangle, \qquad \perpone = \frac{\langle \bt, \proj_{\thetastar}^{\perp}\thetacur\rangle}{\| \proj_{\thetastar}^{\perp}\bt_t \|_2}, \qquad \text{ and } \perptwo = \| \proj_{S_t}^{\perp} \bt \|_2, 
\end{align}
where as before
$\parcomp$ denotes the projection of the decision variable $\bt$ onto the ground-truth $\thetastar$ (since by assumption $\| \thetastar \|_2 = 1$), but the perpendicular component $\perpcomp$ (cf.~\eqref{eq:def-state}) has been split into two further components based on the current iterate $\bt_t$. The scalar $\perpone$ is the projection of $\bt$ onto the component of the current iterate $\thetacur$ orthogonal to the ground-truth (i.e., onto the unit vector $\proj_{\thetastar}^{\perp}\bt_t/\|\proj_{\thetastar}^{\perp}\bt_t\|_2$), and the scalar $\perptwo$ is the magnitude of the portion of $\bt$ orthogonal to the subspace spanned by the ground-truth and the current iterate.  Analogously, let
$\parcomp_t = \langle \bt_t, \thetastar \rangle$  and  $\perpcomp_t = \| \proj_{\thetastar}^{\perp} \bt_t \|_2$
and define the independent, Gaussian random vectors
$\bz_1 = \bX \thetastar$  and $\bz_2 = \frac{\bX \proj_{\thetastar}^{\perp} \thetacur}{\| \proj_{\thetastar}^{\perp} \thetacur \|_2}$.
With this notation, we have
\[
\bX \bt_t = \parcomp_t \bz_1 + \perpcomp_t \bz_2, \qquad \by = \lvert \bz_1 \rvert, \qquad \text{ and } \qquad \bX\proj_{S_t} \bt = \parcomp \bz_1 + \perpone \bz_2.
\]
Use these to define, for two scalar Gaussian variates $(Z_1, Z_2)$, the random variable \sloppy
\mbox{$\Omega_t = \mathsf{sgn}(\parcomp_t Z_1 + \perpcomp_t Z_2) \lvert Z_1 \rvert$} as well as the random vector $\bomega_t = \mathsf{sgn}(\parcomp_t \bz_1 + \perpcomp_t \bz_2) \lvert \bz_1 \rvert$. A sequence of steps, detailed in Appendix~\ref{app:step4-heuristic}, implements both substeps referenced above to show that
\begin{subequations}
\begin{align} 
\min_{\bt \in \real^d} \; \mathfrak{L}(\bt; \bt_t, \bX, \by) &\approx \min_{\parcomp \in \mathbb{R}, \perpone \in \mathbb{R}, \perptwo \geq 0} \Bigl( - \frac{\perptwo \| \proj_{S_t}^{\perp} \bh\|_2}{\sqrt{n}}  + \frac{1}{\sqrt{n}}\|\bomega_t - \parcomp \bz_1 - \perpone \bz_2 - \perptwo \bg\|_2 \Bigr)_{+} \label{eq:Lbarn-opt-heuristic}\\
&\approx  \min_{\parcomp \in \mathbb{R}, \perpone \in \mathbb{R}, \perptwo \geq 0} \Bigl( - \frac{\perptwo}{\sqrt{\kappa}}  + \sqrt{ \EE \bigl\{\bigl(\Omega_t   - \perptwo H - \parcomp Z_1 - \perpone Z_2\bigr)^2\bigr\} } \Bigr)_{+},\label{eq:astar-opt}
\end{align}
\end{subequations}
where we have used the shorthand $x_+ = \max\{x, 0\}$. Letting $\bxi = (\parcomp, \perpone, \perptwo)$ and $\bxi_t = (\parcomp_t, \perpone_t, \perptwo_t)$, notice that the RHS of Eq.~\eqref{eq:Lbarn-opt-heuristic} is given by minimizing a random loss $\widebar{L}_n(\bxi; \bxi_t)$ over all $\bxi \in \real^2 \times [0, \infty)$ and the RHS of Eq.~\eqref{eq:astar-opt} is given by minimizing a deterministic loss $\widebar{L}(\bxi; \bxi_t)$ over the same domain.

Since this computation only involves optimizing over a few variables, it can be shown via straightforward calculation---detailed for convenience in Appendix~\ref{app:step5-heuristic}---that the minimizers of the RHS in equation~\eqref{eq:astar-opt} are given by
\begin{align} \label{eq:second-step-calc}
\widebar{\parcomp} = \EE\{ Z_1 \Omega_t\}, \qquad \widebar{\perpone} = \EE\{ Z_2 \Omega_t\}, \quad \text{ and } \; \widebar{\perptwo} = \sqrt{\frac{\EE\{\Omega_t^2\} - (\EE\{Z_1 \Omega_t\})^2 - (\EE\{Z_2 \Omega_t\})^2}{\kappa - 1}}.
\end{align}
Letting $\anglecurr_t = \tan^{-1}(\perpcomp_t/\parcomp_t)$, some calculation shows that
\begin{align*}
	\EE\{Z_1 \Omega_t\} = 1 - \frac{1}{\pi}(2 \anglecurr_t - \sin(2\anglecurr_t)), \qquad \EE\{Z_2 \Omega_t\} = \frac{2}{\pi} \sin^2(\anglecurr_t), \qquad \text{ and } \qquad \EE\{\Omega_t^2\} = 1.
\end{align*}
Finally, recalling the change of variables~\eqref{eq:heuristic-gordon-change-of-var}
 and noting that
$\| \proj_{\thetastar}^{\perp} \bt \|_2 = \sqrt{\perptwo^2 + \perpone^2}$, we have
the Gordon state evolution update
\begin{align} 
\begin{split} \label{eq:Gordon-update-example}
	\widebar{\parcomp}_{t+1} &= 1 - \frac{1}{\pi}(2 \anglecurr_t - \sin(2\anglecurr_t)), \text{ and } \\
	\widebar{\perpcomp}_{t+1} &= \sqrt{\frac{4}{\pi^2} \sin^4(\anglecurr_t) + \frac{1 - (1 - \frac{1}{\pi}(2 \anglecurr_t - \sin(2\anglecurr_t)))^2 - \frac{4}{\pi^2} \sin^4(\anglecurr_t)}{\kappa - 1}}.
	\end{split}
\end{align}

\paragraph{Step 4: Random state evolution is tracked by Gordon update:} The final step is to show that both $|\inprod{\bt_{t + 1}}{\thetastar} - \widebar{\parcomp}_{t + 1}|$ and $| \| \proj^{\perp}_{\thetastar} \bt_{t + 1} \|_2 - \widebar{\perpcomp}_{t + 1} |$ are small, so that the deterministic Gordon update faithfully tracks the random pair $\parcomp_{t + 1} = \inprod{\bt_{t + 1}}{\thetastar}$ and $\perpcomp_{t + 1} = \| \proj^{\perp}_{\thetastar} \bt_{t + 1} \|_2$.  This is achieved by showing \mbox{(a) a} growth condition (typically strong convexity) around the minimum of the scalarized auxiliary loss $\widebar{L}_n$ and (b) that the empirical minimizers 
\[
\bxi_n := (\parcomp_n, \perpone_n, \perptwo_n) = \argmin_{\parcomp \in \mathbb{R}, \perpone \in \mathbb{R}, \perptwo \geq 0} \widebar{L}_n(\parcomp, \perpone, \perptwo)
\]
are close to the deterministic state $\widebar{\bxi} = (\widebar{\parcomp}, \widebar{\perpone}, \widebar{\perptwo})$.  With these two ingredients in hand, we show that for any vector $\bt$ for which $\parcomp(\bt), \perpone(\bt),$ or $\perptwo(\bt)$ is far from $\widebar{\parcomp}, \widebar{\perpone}$, or $\widebar{\perptwo}$, respectively, the value $\mathcal{L}(\bt)$ is far from the deterministic value~\eqref{eq:astar-opt}.  Additionally, we show that the minimum of $\mathcal{L}$ over the entire domain $\mathbb{R}^d$ is close to the deterministic value~\eqref{eq:astar-opt}.  Thus, it must be the case that if $\bt$ is the minimizer of the loss $\mathcal{L}$, then the quantities $\parcomp(\bt), \perpone(\bt)$, and $\perptwo(\bt)$ are close to the respective deterministic quantities $\widebar{\parcomp}, \widebar{\perpone}$, and $\widebar{\perptwo}$. 

\begin{remark} \label{rem:SE-op}
Two key observations to make at this juncture are that (a) the Gordon state evolution update can be run from any  point $\bt \in \real^d$, not just $\bt_t$, and (b) the update equations~\eqref{eq:Gordon-update-example} define a map $(\parcomp_t, \perpcomp_t) \mapsto (\widebar{\parcomp}_{t + 1}, \widebar{\perpcomp}_{t + 1})$. As postulated in the ansatz at the beginning of Section~\ref{sec:high-level-steps}, the Gordon state evolution update takes the form of a \emph{state evolution operator}, mapping $\mathbb{R}^2$ to itself. This will also be true in our other specific examples, and so we use this terminology in the sequel alongside the notation $\Sopgordon = (\parcompgordon, \perpcompgordon)$ to denote this operator.
\end{remark}

With the intuition gained from this heuristic derivation, we are now in a position to state our general result obtained via this recipe.

\subsection{The general result}
\label{sec:main-result-Gordon}
We now formally derive and prove concentration of the one-step Gordon updates for higher-order and first-order methods run on a generic class of problems.
As observed in Remark~\ref{rem:SE-op}, the Gordon state evolution update is well-defined when run from any current iterate $\bt$.
Accordingly, fix an \emph{arbitrary} $d$-dimensional parameter $\bt$ and consider the one-step update~\eqref{eq:Tn}, restated below for convenience
\begin{align}\label{eq:gordon_one_step}
\mathcal{T}_n(\bt) \in \arg\min_{\bt'} \Lc(\bt';\bt,\bX,\by),
\end{align}
where the loss function takes either of the forms in equations~\eqref{eq:sec_order_gen} or~\eqref{eq:GD_gen}. For convenience, use the shorthand 
\begin{align} \label{eq:SE-shorthand}
(\parcomp, \perpcomp) = (\parcomp(\bt), \perpcomp(\bt)) \quad \text{ and } \quad (\parcompplus,\perpcompplus) = (\parcomp(\mathcal{T}_n(\bt)), \perpcomp(\mathcal{T}_n(\bt))).
\end{align}
The main result of this section shows that for algorithms whose one-step updates take the form~\eqref{eq:sec_order_gen} or~\eqref{eq:GD_gen}, the pair $(\parcompplus,\perpcompplus)$ concentrates around the deterministic Gordon state evolution update run from $(\parcomp, \perpcomp)$, i.e., the pair $\Sopgordon(\parcomp, \perpcomp)$. 

This result holds under some mild assumptions on the weight function used to define these algorithms. In particular,
recall that the losses in equations~\eqref{eq:sec_order_gen} and~\eqref{eq:GD_gen} are parameterized by a weight function \sloppy
\mbox{$\omega:\R\times\R\rightarrow\R$}. Also recall the model~\eqref{eq:model}, and let $Q$ denote a random variable drawn from the latent variable distribution $\mathbb{Q}$. Let $(Z_1, Z_2, Z_3)$ denote a triple of i.i.d. standard Gaussian vectors, and 
let
\begin{align}\label{eq:omegat_def_main}
\Omega = \omega\big(\parcomp Z_1 +\perpcomp Z_2\,,\,f(Z_1; Q)+\noisestd Z_3 \big).
\end{align}
The first assumption requires that this random variable is light-tailed. The second assumption is technical, and requires a lower bound on a particular functional of $\Omega$.
\begin{assumption} \label{ass:omega}
The random variable $\Omega$~\eqref{eq:omegat_def_main} is sub-Gaussian with Orlicz norm bounded as $\| \Omega \|_{\psi_2} \leq K_1$, for some parameter $K_1 > 0$.
\end{assumption}

\begin{assumption}\label{ass:omega-lb}
For a parameter $K_2 > 0$, we have
\begin{align*}
\EE[\Omega^2]-\left(\EE[Z_1\Omega]\right)^2-\left(\EE[Z_2\Omega]\right)^2 \geq K_2. 
\end{align*} 
\end{assumption}
We show in Section~\ref{sec:main-results} to follow that several models and algorithms satisfy Assumptions~\ref{ass:omega} and~\ref{ass:omega-lb}. Before stating our main result, it is helpful to first define the deterministic Gordon updates themselves.
\begin{definition}[Gordon state evolution update: Higher-order methods]
\label{def:starnext_main_HO}
Let $(Z_1,Z_2, Z_3)$ denote a triple of independent standard Gaussian random variables and use these to define the random variable $\Omega$ as in equation~\eqref{eq:omegat_def_main}. If the loss function $\Lc$ is as in equation \eqref{eq:sec_order_gen}, then define
\begin{align} \label{eq:theta_starnext_sec_main}
\parcompgordon= \EE[Z_1\Omega]  \quad \text{ and } \quad 
\perpcompgordon= \sqrt{ (\EE[Z_2\Omega])^2 +  \frac{1}{\kappa - 1} \left( \EE[\Omega^2]-\left(\EE[Z_1\Omega] \right)^2- \left( \EE[Z_2\Omega] \right)^2 \right) }.
\end{align}
\end{definition}
Next, we state the update for first-order methods, assuming that\footnote{For larger stepsizes, similar update equations still apply, but some delicacy is required to handle the signs correctly.} $\eta \leq 1/2$.
\begin{definition} [Gordon state evolution update: First-order methods]
\label{def:starnext_main_FO}
Suppose $\eta \leq 1/2$. Let $(Z_1,Z_2, Z_3)$ denote a triple of independent standard Gaussian random variables and use these to define the random variable $\Omega$ as in equation~\eqref{eq:omegat_def_main}. If $\Lc$ is as in equation \eqref{eq:GD_gen}, then define
\begin{align} \label{eq:theta_starnext_gd_main}
\parcompgordon= \parcomp  - 2 \eta \cdot \EE \left[Z_1 \Omega\right] \quad \text{ and } \quad
\perpcompgordon= \sqrt{ ( \perpcomp  - 2\eta \cdot \EE\left[Z_2 \Omega\right] )^2  + \frac{4\eta^2}{\kappa} \cdot \EE\left[\Omega^2\right] }.
\end{align}
\end{definition}

In the sequel, we will evaluate the expressions in equations~\eqref{eq:theta_starnext_sec_main} and~\eqref{eq:theta_starnext_gd_main} for concrete models and algorithms. However, at this level of generality, a salient similarity between higher-order and first-order updates is already apparent, since it can be shown that the population state evolution update can be obtained by taking $\kappa \to \infty$ in its Gordon counterpart.

\begin{remark}[Population updates coincide for stepsize $\eta =1/2$] \label{rem:pop-coincide}
Set $\eta = 1/2$ and send $\kappa \to \infty$, so that the Gordon update now coincides with its population counterpart. Then using Remark~\ref{rem:weight-rel} to relate the weight functions for higher and first-order updates, we obtain that
the two Gordon updates in equations~\eqref{eq:theta_starnext_sec_main} and~\eqref{eq:theta_starnext_gd_main} coincide. On the other hand, for finite $\kappa$, these updates are always distinct.
\end{remark}
We return to explore this phenomenon in Section~\ref{sec:main-results} to follow, deriving convergence guarantees for first-order updates when $\eta = 1/2$ by using the Gordon state evolution update in place of the population update.
But first, we state our main results characterizing the concentration of the random pair $(\parcompplus, \perpcompplus)$ around $(\parcompgordon, \perpcompgordon)$. We state two very similar theorems for convenience since they apply under a slightly different set of assumptions. The first theorem applies to higher-order updates under both Assumptions~\ref{ass:omega} and~\ref{ass:omega-lb}, and the second theorem applies to first-order updates but requires only Assumption~\ref{ass:omega} to hold.

\begin{theorem} [Higher-order deterministic prediction]
\label{thm:one_step_main_HO}
Consider the general model~\eqref{eq:model} for the data, and procedures that obey the general one-step update~\eqref{eq:sec_order_gen}. Recall the shorthand $(\parcomp, \perpcomp, \parcomp^+, \perpcomp^+)$ from equation~\eqref{eq:SE-shorthand}. Suppose that Assumptions~\ref{ass:omega} and~\ref{ass:omega-lb} hold on the associated weight function $\wfun$ with parameters $K_1$ and $K_2$, respectively. Consider the pair of scalars $(\parcompgordon, \perpcompgordon)$ from Definition~\ref{def:starnext_main_HO}.  There exists a universal positive constant $C_1$ as well as a pair of positive constants $(C_K, C_K')$ depending only on the pair $(K_1, K_2)$ such that the following is true.
If $\kappa \geq C_1$, then 
 \begin{subequations} \label{eq:main-conc_HO}
 \begin{enumerate}[label=(\alph*)]
  \item Provided we further have $n \geq C_{K}' \cdot \log(1/\delta)$, the perpendicular component satisfies
 \begin{align}\label{eq:sigma_conc_main_HO}
 \Pro\left\{ |\perpcompplus-\perpcompgordon| \geq C_K \left( \frac{\log (1/\delta)}{n} \right)^{1/4} \right\} \leq \delta, \text{ and }
 \end{align}
 \item The parallel component satisfies
 \begin{align}\label{eq:xi_conc_main_HO}
 \Pro\left\{ \bigl \lvert \parcompplus-\parcompgordon \bigr\rvert  \geq C_K \left( \frac{\log^7 (1 / \delta)}{n} \right)^{1/2} \right\} \leq \delta.
 \end{align}
 \end{enumerate}
 \end{subequations}
\end{theorem}

The main theorem for first-order methods is extremely similar, except that we make the assumption\footnote{This assumption is not required for higher-order methods because the sub-Gaussianity of the $\omega$ function suffices to ensure that the pair $(\parcomp^{\gor}, \perpcomp^{\gor})$ remains bounded (see Definition~\ref{def:starnext_main_HO}). The same is not true for first-order methods; as is evident from Definition~\ref{def:starnext_main_FO}, we also require the pair $(\alpha, \beta)$ to be bounded.} $\parcomp \vee \perpcomp \leq 3/2$ and obtain sharper logarithmic factors. We also state the theorem for stepsize $\eta \leq 1/2$ for convenience.

\begin{theorem} [First-order deterministic prediction]
\label{thm:one_step_main_FO}
Consider the general model~\eqref{eq:model} for the data, and procedures that obey the general one-step update~\eqref{eq:GD_gen} for some $\eta \leq 1/2$. Recall the shorthand $(\parcomp, \perpcomp, \parcomp^+, \perpcomp^+)$ from equation~\eqref{eq:SE-shorthand} and assume that $\parcomp \vee \perpcomp \leq 3/2$. Suppose that Assumption~\ref{ass:omega} holds on the associated weight function $\wfun$ with parameter $K_1$. Consider the pair of scalars $(\parcompgordon, \perpcompgordon)$ from Definition~\ref{def:starnext_main_FO}.  There exists a universal positive constant $C_1$ as well as a pair of positive constants $(C_K, C_K')$, depending only on $K_1$ such that the following is true.
If $\kappa \geq C_1$, then 
 \begin{subequations} \label{eq:main-conc_FO}
 \begin{enumerate}[label=(\alph*)]
  \item Provided we further have $n \geq C_K' \cdot \log(1/\delta)$, the perpendicular component satisfies
 \begin{align}\label{eq:sigma_conc_main_FO}
 \Pro\left\{ |\perpcompplus-\perpcompgordon| \geq C_K \left( \frac{\log (1/\delta)}{n} \right)^{1/4}  \right\} \leq \delta, \text{ and }
 \end{align}
 \item The parallel component satisfies
 \begin{align}\label{eq:xi_conc_main_FO}
 \Pro\left\{ \bigl \lvert \parcompplus-\parcompgordon \bigr\rvert  \geq C_K \left( \frac{\log (1 / \delta)}{n} \right)^{1/2} \right\} \leq \delta.
 \end{align}
 \end{enumerate}
 \end{subequations}
\end{theorem}
We formally derive the one-step Gordon updates $(\parcompgordon,\perpcompgordon)$ 
in a unified fashion for both these theorems in Section \ref{sec:gordon}, with rigorous justifications of the steps outlined in Sections~\ref{sec:high-level-steps} and~\ref{subsec:heuristic-second-order}. In particular, this program is carried out under a weaker set of assumptions on the one-step loss function, which includes equations~\eqref{eq:sec_order_gen} and~\eqref{eq:GD_gen} as special cases (see Assumption~\ref{ass:loss} in Section~\ref{sec:gordon}). 
Section~\ref{sec:gordon} also provides a proof that both $\parcompplus$ and $\perpcompplus$ concentrate at the rate $\ordertil(n^{-1/4})$ around their Gordon counterparts, thereby proving part (a) of both theorems.
In Section \ref{sec:random-init-signal}, we refine the concentration rate for the parallel component $\parcomp$ and establish part (b) of both theorems. 

It should be emphasized that Theorems~\ref{thm:one_step_main_HO} and~\ref{thm:one_step_main_FO} are both non-asymptotic results, in contrast to results typically derived using the CGMT machinery. A non-asymptotic characterization is essential in our case because we intend to apply these results iteratively, once per step of the algorithm. As alluded to in the heuristic derivation, our proof of the $\ordertil(n^{-1/4})$ rate of concentration of the pair $(\alpha^+, \beta^+)$ around the deterministic update follows a generic proof technique reasoning about the growth properties of the scalarized loss function around its minimum, and generalizes results from the linear case~\citep{miolane2018distribution}.  
This technique may prove to be of independent interest in other applications of the CGMT.

While a deviation result of $\ordertil(n^{-1/4})$ can be obtained via this general technique, this rate does not suffice for the parallel update $\parcomp^+$ near a random initialization. In particular, for a random initialization we have $\parcomp \asymp d^{-1/2}$, and it can be shown that 
the deterministic prediction arising from one step of the algorithm also satisfies $\alpha^\gor \asymp d^{-1/2}$. Thus, showing that $\alpha^+$ is within $\ordertil(n^{-1/4})$ of $\parcompgordon$ is only a nontrivial statement---guaranteeing say a nonzero parallel component at the next step---when $n = \widetilde{\Omega}(d^2)$, or equivalently, when $\kappa = \widetilde{\Omega}(d)$. On the other hand, we would like to prove global convergence in the regime $\kappa = \ordertil(1)$, and so dedicate significant effort to improving this concentration result to $\ordertil(n^{-1/2})$, thereby allowing us to obtain part (b) of the theorems. This proof, presented in Section~\ref{sec:random-init-signal},
requires significant subtlety especially for higher-order algorithms since the update~\eqref{eq:second-order-update-sum} involves a matrix inversion. We employ a leave-one-out trick
 to show a sharpened version of a result by~\citet{zhang2020phase}, and believe that this technique will prove more broadly useful in analyzing other higher-order updates from a random initialization. 
Our refined characterization for the parallel component $\alpha^+$ raises the question of whether deviation of $\perpcomp^+$ can also be improved to $\ordertil(n^{-1/2})$. While we conjecture that this is indeed the case, we leave this question open for future investigation, 
turning now to deriving corollaries of the main theorems in two specific models.

%% file: sections/main-results/main-results_alt.tex

In this section, we state consequences of our main results for two specific models and algorithms, although it is important to note that the Gordon recipe itself---as sketched in the previous section---is much more broadly applicable.  In particular, we will consider phase retrieval and a symmetric mixture of linear regressions, as well as the algorithms covered in Section~\ref{sec:setup}.
It is important to note that in both these models, the global sign of the parameter $\thetastar$ is not identifiable from observations, and so parameter estimates should be assessed in terms of their ``distance'' to the set $\{ - \thetastar, \thetastar \}$.

As mentioned before, we track the two-dimensional state
$(\parcomp(\bt), \perpcomp(\bt))$ of each parameter $\bt \in \real^d$, with
$\parcomp(\bt) = \inprod{\bt}{\thetastar}$ and $\perpcomp(\bt) = \| \proj_{\thetastar}^{\perp} \bt \|_2$. The sign ambiguity will be resolved by the initialization, so we assume throughout that $\parcomp(\bt) \geq 0$ for parameters $\bt$ that we consider. For any two-dimensional state evolution element $\bzeta = (\parcomp, \perpcomp)$, define two metrics
\begin{align} \label{eq:metrics}
\DeltaSELtwo(\bzeta) := \sqrt{ (1 - \parcomp)^2 + \perpcomp^2} \qquad \text{ and } \qquad \DeltaSEangle(\bzeta) := \tan^{-1} (\perpcomp/\parcomp).
\end{align}
When $\parcomp = \parcomp(\bt)$ and $\perpcomp = \perpcomp(\bt)$, the quantity $\DeltaSELtwo(\parcomp, \perpcomp)$ measures the $\ell_2$ distance between $\bt$ and the set $\{ - \thetastar, \thetastar \}$, i.e., we have $\DeltaSELtwo(\parcomp, \perpcomp) = \min \{ \| \bt - \thetastar \|_2, \| \bt + \thetastar \|_2 \}$. Similarly, the angular metric satisfies \mbox{$\DeltaSEangle(\parcomp, \perpcomp) = \min \{ \angle \left(\bt, \thetastar \right), \angle \left(\bt, -\thetastar \right) \}$}.

As alluded to in the previous sections (see Remark~\ref{rem:SE-op}), a \emph{state evolution operator} is a function mapping $\mathbb{R}^2$ to itself. We begin with a few useful definitions for such operators. First, 
for any state evolution operator $\mathcal{S}$, recall that $\mathcal{S}^t$ denotes the operator formed by $t$ iterated applications of $\mathcal{S}$. Next, we define an $\mathbb{S}$-faithful state evolution operator.
\begin{definition}[$\mathbb{S}$-faithful operator]
For a set $\mathbb{S} \subseteq \real^2$, a state evolution operator $\mathcal{S}: \mathbb{R}^2 \to \mathbb{R}^2$ is said to be $\mathbb{S}$-faithful if $\mathcal{S}(\bzeta) \in \mathbb{S}$ for all $\bzeta \in \mathbb{S}$.
\end{definition}
Next, we present two formal definitions of convergence rates, measuring linear (geometric) and faster-than-linear convergence. 

\begin{definition}[Linear convergence of state evolution] \label{def:linear}
For parameters $0 < c \leq C < 1$, a state evolution operator $\mathcal{S}: \mathbb{R}^2 \to \mathbb{R}^2$ is said to exhibit $(c, C, t_0)$-linear convergence in the metric $\DeltaSE$ within the set $\mathbb{S}$ to level $\varepsilon$ if $\Sop$ is $\mathbb{S}$-faithful, and for all $\bzeta \in \mathbb{S}$, we have 
\begin{align} \label{eq:rate-L}
c \cdot \DeltaSE(\mathcal{S}^{t}(\bzeta)) + \frac{\varepsilon}{2} \leq \DeltaSE(\mathcal{S}^{t+1}(\bzeta)) \leq C \cdot \DeltaSE(\mathcal{S}^{t}(\bzeta)) + \varepsilon \text{ for all } t \geq t_0. 
\end{align}
\end{definition}

\begin{definition}[Super-linear convergence] \label{def:superlinear}
Set parameters $0 < c \leq C$ and $\lambda > 1$, and suppose that $\mathbb{S} \subseteq \{ \bzeta: \DeltaSE(\bzeta) \leq C^{1 - \lambda} \}$.
A state evolution operator $\mathcal{S}: \mathbb{R}^2 \to \mathbb{R}^2$ is said to exhibit $(c, C, \lambda, t_0)$-super-linear convergence in the metric $\DeltaSE$ within the set $\mathbb{S}$ to level $\varepsilon$ if $\Sop$ is $\mathbb{S}$-faithful, and for all $\bzeta \in \mathbb{S}$, we have 
\begin{align} \label{eq:rate-SL}
c \cdot [\DeltaSE(\mathcal{S}^{t}(\bzeta))]^{\lambda} + \frac{\varepsilon}{2} \leq \DeltaSE(\mathcal{S}^{t+1}(\bzeta)) \leq C \cdot [\DeltaSE(\mathcal{S}^{t}(\bzeta))]^{\lambda} + \varepsilon  \text{ for all } t \geq t_0. 
\end{align}
\end{definition}
A few comments on our definitions are worth making. First, note that both definitions require both upper and lower bounds on the per-step behavior of the algorithm, where the bounds apply after a ``transient'' period of $t_0$ iterations. 
This is a key feature of our framework, in that we are able to exactly characterize the convergence behavior as opposed to solely providing upper bounds. Both upper and lower bounds are characterized both by a rate of decrease of the error (linear in the case of equation~\eqref{eq:rate-L} and super-linear in the case of equation~\eqref{eq:rate-SL}) and the eventual statistical neighborhood $\varepsilon$. Second, our choice of defining the lower bounds in equations~\eqref{eq:rate-L} and~\eqref{eq:rate-SL} with $\varepsilon/2$ is arbitrary;
any absolute constant other than $2$ preserves the qualitative convergence behavior.

As is common in the analysis of nonconvex optimization problems, our convergence guarantee will be established in two stages. In the first stage, we will show that the algorithm converges (typically slowly) to a ``good region'' around the optimal solution; once in the good region, the algorithm converges much faster. For both of the models that we consider, the following definition of the good region suffices. It is important to note that the numerical constants in this definition have not been optimized to be sharp.
\begin{definition}[Good region] \label{def:good-region}
Define the region
\[
\Goodset = \{ (\parcomp, \perpcomp) \mid \; 0.55 \leq \parcomp \leq 1.05, \text{ and } \parcomp/\perpcomp \geq 5 \}.
\]
With slight abuse of terminology, we say that $\bt \in \Goodset$ if $(\parcomp(\bt), \perpcomp(\bt)) \in \Goodset$.
\end{definition}
We are now in a position to present our guarantees for two specific models: phase retrieval and a symmetric mixture of linear regressions.

\subsection{Phase retrieval}

Our first example is the phase retrieval model~\eqref{eq:PR-model}. We characterize the convergence behavior of both the alternating minimization algorithm and the subgradient descent method for this model.

\subsubsection{Alternating minimization}

Recall from equation~\eqref{eq:update-explicit-AM-PR} that the empirical update run from the point $\theta$ is given by
\begin{align} \label{eq:empirics-AM-PR}
\mathcal{T}_n(\bt) = \left( \frac{1}{n} \sum_{i = 1}^n \bx_i \bx_i^\top \right)^{-1} \; \left( \frac{1}{n} \sum_{i = 1}^n \sign(\inprod{\bx_i}{\bt}) \cdot y_i \cdot \bx_i \right).
\end{align}
The following corollary follows from Theorem~\ref{thm:one_step_main_HO}; in it, we state both the explicit Gordon state evolution and the concentration of the empirical iterates assuming that the update is run from some arbitrary ``current'' point $\bt$. Its proof can be found in Appendix~\ref{pf-cor1}.
\begin{corollary} \label{lem:Gordon-SE-AM-PR}
Let $\parcomp = \parcomp(\bt)$ and $\perpcomp = \perpcomp(\bt)$ with $\bzeta = (\parcomp, \perpcomp)$ and $\anglecurr = \tan^{-1}\left(\frac{\perpcomp}{\parcomp}\right)$. Let $(\parcompgordon, \perpcompgordon) = \mathcal{S}_{\gor}(\bzeta)$ denote the Gordon state evolution from Definition~\ref{def:starnext_main_HO}. \\ 
(a) We have
\begin{subequations} \label{eq:Gordon-AM-PR}
\begin{align}
\parcompgordon &= 1 - \frac{1}{\pi}(2\anglecurr - \sin(2\anglecurr))  \quad \text{ and } \label{eq:Gordon-AM-PR-xi} \\
\perpcompgordon &= \sqrt{\frac{4}{\pi^2}\sin^4(\anglecurr) + \frac{1}{\kappa - 1}\left(1 - (1 - \frac{1}{\pi}(2\anglecurr - \sin(2\anglecurr)))^2 - \frac{4}{\pi^2}\sin^4(\anglecurr) + \noisestd^2 \right)}. \label{eq:Gordon-AM-PR-sigma}
\end{align}
(b) Suppose $\noisestd > 0$. There is a constant $C_\noisestd > 0$ depending only on $\sigma$ such that the following holds. With $\mathcal{T}_n$ as defined in equation~\eqref{eq:empirics-AM-PR}, the empirical state evolution $(\parcomp^+, \perpcomp^+) = (\parcomp(\mathcal{T}_n(\bt)), \perpcomp(\mathcal{T}_n(\bt)))$ satisfies
\begin{align*}
\Pro \left\{ | \parcomp^+ - \parcompgordon | \leq  C_\noisestd \left( \frac{\log^7 (1/\delta) }{n} \right)^{1/2} \right\} \leq \delta \quad \text{ and } \quad \Pro \left\{ | \perpcomp^+ - \perpcompgordon | \leq C_\noisestd \left( \frac{\log (1 / \delta) }{n} \right)^{1/4} \right\} \leq \delta.
\end{align*}
\end{subequations}
\end{corollary}

From equation~\eqref{eq:Gordon-AM-PR}, it is possible to recover the following population update by letting $\kappa \to \infty$, which is given by
\begin{align}\label{eq:Pop-AM-PR}
\parcomppop = 1 - \frac{1}{\pi}(2\anglecurr - \sin(2\anglecurr))  \quad \text{ and }  \quad \perpcomppop = \frac{2}{\pi}\sin^2(\anglecurr).
\end{align}
It is easy to show that the population state evolution predicts super-linear convergence with exponent~$2$ (i.e., quadratic convergence) in the good region. The following fact is proved in Appendix~\ref{pf-fact1}.
\begin{fact} \label{prop:Pop-PR}
The population state evolution $\mathcal{S}_{\pop} = (\parcomppop, \perpcomppop)$ is $(\frac{1}{20}, 1, \lambda, t_0)$-super-linearly convergent in the $\ell_2$ metric\footnote{In fact, the population state evolution~\eqref{eq:Pop-AM-PR} enjoys \emph{global} quadratic convergence in the angular metric $\DeltaSEangle$; see Remark~\ref{lem:quad-angle-PR} in the appendix.} $\DeltaSELtwo$ within the region $\Goodset$ to level $\varepsilon = 0$, where $\lambda = 2$ and $t_0 = 1$.
\end{fact}
However, the following theorem shows that the empirics are instead tracked faithfully by the Gordon state evolution, which converges more slowly than the population state evolution.
\begin{theorem} \label{thm:AM-PR}
Consider the alternating minimization update $\mathcal{T}_n$ from equation~\eqref{eq:empirics-AM-PR} and the associated Gordon state evolution update $\mathcal{S}_{\gor}$ from equation~\eqref{eq:Gordon-AM-PR}. There is a universal positive constant $C$ such that the following is true. If $\kappa \geq C (1 + \noisestd^2)$, then:

\noindent (a) The Gordon state evolution update 
\begin{center}
$\mathcal{S}_{\gor}$ is $(c_{\kappa}, C_\kappa, \lambda, t_0)$-super-linearly convergent in the $\ell_2$ metric $\DeltaSELtwo$ within $\Goodset$ to level $\varepsilon_{n, d} = \frac{\noisestd}{\sqrt{\kappa}}$,
\end{center} 
where $0 \leq c_\kappa \leq C_{\kappa} \leq 1$ are constants depending solely on $\kappa$, and we have
\begin{align*}
\lambda = 3/2 \qquad \text{ and } \qquad t_0 = 1.
\end{align*}

\noindent (b) If $\noisestd > 0$, then there exist $C_\noisestd, C'_\noisestd > 0$ depending only on $\noisestd$ such that for all $n \geq C'_\noisestd$ and for any $\bt$ such that $\bzeta = (\parcomp(\bt), \perpcomp(\bt)) \in \Goodset$, we have
\begin{align*}
\max_{1 \leq t \leq T} \; | \DeltaSELtwo(\mathcal{S}^t_{\gor}(\bzeta)) - \| \mathcal{T}^t_n(\bt) - \thetastar \|_2 | \leq C_\noisestd \left( \frac{\log n}{n} \right)^{1/4}
\end{align*}
with probability exceeding $1 - 2T n^{-10}$.

\noindent (c) Suppose $\bt_0$ denotes a point such that $\frac{\parcomp(\bt_0)}{\perpcomp(\bt_0)} \geq \frac{1}{50 \sqrt{d}}$ and further suppose that \sloppy \mbox{$\kappa \geq C''_{\noisestd} \cdot \log^7 \left( \frac{1 + \log d}{\delta} \right)$} for $C''_\sigma$ depending solely on $\noisestd$. 
Then for some $\ttilde \leq C\log d$, we have
\begin{align*}
\mathcal{T}_{n}^{\ttilde}(\bt_0) \in \Goodset
\end{align*} 
with probability exceeding $1 - \delta$.
\end{theorem}

Note that if $\bt_0$ is chosen at random from the $d$-dimensional unit ball $\mathbb{B}_2(1)$ with $d \geq 130$, then we have $\frac{\parcomp(\bt_0)}{\perpcomp(\bt_0)} \geq \frac{1}{50 \sqrt{d}}$ with probability at least $0.95$ (see Lemma~\ref{lem:init-properties}(a) in the appendix). Theorem~\ref{thm:AM-PR} then shows that after $\tau = \mathcal{O}( \log d + \log \log (\kappa/\noisestd^2))$ iterations, the empirics satisfy
\begin{align} \label{eq:final-AM-PR}
\| \mathcal{T}_{n}^{\tau}(\bt) - \thetastar \|_2= \order \left( \noisestd \sqrt{\frac{d}{n}} \right) + \ordertil \left( n^{-1/4} \right)
\end{align}
with high probability. Concretely, after taking $\order(\log d)$ steps to converge to the good region $\Goodset$, 
the AM update converges \emph{very} fast to within statistical error of the optimal parameter. 

Some remarks on specific aspects of Theorem~\ref{thm:AM-PR} are in order. First, note that this theorem predicts super-linear convergence with nonstandard exponent $3/2$ whenever $\kappa$ is bounded above. Comparing with Fact~\ref{prop:Pop-PR}, we see that the population update is overly optimistic, and this corroborates what we saw in Figures~\ref{fig:intro-pop} and~\ref{fig:intro-Gordon} in the introduction. Nonstandard super-linear convergence was recently observed in the noiseless case of this problem~\citep{ghosh2020alternating}, but a larger exponent was conjectured. Theorem~\ref{thm:AM-PR} shows that the exponent $3/2$ is indeed sharp, since we obtain both upper and lower bounds on the error of the algorithm. 
Furthermore, the convergence rate is super-linear with exponent $3/2$ for every value of the noise level. As we will see shortly, this is not the case for the closely related model of a symmetric mixture of regressions, in which the convergence rate of this algorithm is linear for any constant noise level.

Second, note that part (b) of the theorem shows that the (random) empirical state evolution is within $\ell_2$ distance $n^{-1/4}$ of its (deterministic) Gordon counterpart once the iterates enter the good region. Consequently, the final result~\eqref{eq:final-AM-PR} on the empirical error has two terms. Note that this error is dominated by the $\noisestd/\sqrt{\kappa}$ term in modern high dimensional problems.

Third, our convergence result is global, and holds from a random initialization. In particular, part (c) of the theorem guarantees that within $O(\log d)$ iterations, the iterations enter the good region $\Goodset$, at which point parts (a) and (b) of the theorem become active. Convergence from a random initialization is also established by showing that the empirical state evolution tracks its Gordon counterpart closely. But rather than showing two deterministic envelopes around the empirical trajectory, we leverage closeness of the updates iterate-by-iterate. It is worth noting that this is the only step that requires the condition $n \asymp d \log^7 (\log d)$; all other steps only require sample complexity that is linear in dimension.

Finally, we note that our assumption that $\noisestd^2 / \kappa$ be bounded above by a universal constant should not be viewed as restrictive. If this condition does not hold, then one can show using our analysis that running just one step of the algorithm from a random initialization already satisfies $\| \bt_1 - \thetastar \|^2 = \order(1) = \order(\noisestd^2 / \kappa)$, thereby providing an estimate with order-optimal error.

\subsubsection{Subgradient descent}

To contrast with the super-linear convergence shown in the previous section, we now consider subgradient descent with step-size $1/2$. As alluded to in Remark~\ref{rem:pop-coincide} and shown explicitly below, this update shares the same population update as AM, considered before. 
As derived in equation~\eqref{eq:update-explicit-GD-PR}, the general subgradient method for PR is given by the update
\begin{align} \label{eq:GD-PR}
\mathcal{T}_n(\bt) = \bt - \frac{2\eta}{n} \cdot \sum_{i = 1}^n (|\inprod{\bx_i}{\bt}| - y_i) \cdot \sign(\inprod{\bx_i}{\bt}) \cdot \bx_i,
\end{align}
where $\eta > 0$ denotes the step-size. The Gordon state evolution update is given by the following corollary of Theorem~\ref{thm:one_step_main_FO}, proved in Appendix~\ref{pf-cor2}.
\begin{corollary} \label{lem:Gordon-SE-GD-PR}
Let $\parcomp = \parcomp(\bt)$ and $\perpcomp = \perpcomp(\bt)$ with $\anglecurr = \tan^{-1}\left(\frac{\perpcomp}{\parcomp}\right)$. Let $(\parcompgordon, \perpcompgordon) = \Sopgordon(\parcomp, \perpcomp)$ denote the Gordon state evolution update for the subgradient descent operator~\eqref{eq:GD-PR}, given by Definition~\ref{def:starnext_main_FO}. Let $\eta \leq 1/2$. \\
(a) We have
\begin{subequations} \label{eq:Gordon-GD-PR}
\begin{align}
\parcompgordon &= (1 - 2\eta) \parcomp + 2\eta \left(1 - \frac{1}{\pi} (2 \anglecurr - \sin(2 \anglecurr)) \right), \text{ and } \label{eq:Gordon-GD-PR-xi}\\
\perpcompgordon &= \Big(  \Big\{ (1 - 2\eta) \perpcomp + 2 \eta \cdot \frac{2}{\pi} \sin^2 \anglecurr \Big\}^2 \notag \\
&\qquad \qquad \qquad + \frac{4 \eta^2}{\kappa} \Big\{ \parcomp^2 + \perpcomp^2 - 2\parcomp \left(1 - \frac{1}{\pi} (2 \anglecurr - \sin(2 \anglecurr)) \right) - 2\perpcomp \cdot \frac{2}{\pi} \sin^2 \anglecurr + 1 + \noisestd^2 \Big\} \Big)^{1/2}. \label{eq:Gordon-GD-PR-sigma}
\end{align}
\end{subequations}
(b) Suppose $\noisestd > 0$ and $\alpha \vee \beta \leq 3/2$. Then there is a positive constant $C_\noisestd$ depending only on $\noisestd$ such that with $\mathcal{T}_n$ as defined in equation~\eqref{eq:GD-PR}, the empirical state evolution $(\parcomp^+, \perpcomp^+) = (\parcomp(\mathcal{T}_n(\bt)), \perpcomp(\mathcal{T}_n(\bt)))$ satisfies
\begin{align*}
\Pro \left\{ | \parcomp^+ - \parcompgordon | \leq  C_\noisestd \left( \frac{\log (1/\delta) }{n} \right)^{1/2} \right\} \leq \delta \quad \text{ and } \quad \Pro \left\{ | \perpcomp^+ - \perpcompgordon | \leq C_\noisestd \left( \frac{\log (1 / \delta) }{n} \right)^{1/4} \right\} \leq \delta.
\end{align*}
\end{corollary}
Sending $\kappa \to \infty$ in equation~\eqref{eq:Gordon-GD-PR} recovers the infinite-sample population state evolution update
\begin{align}\label{eq:Pop-GD-PR}
\parcomppop &= (1 - 2\eta) \parcomp + 2\eta \left(1 - \frac{1}{\pi} (2 \anglecurr - \sin(2 \anglecurr)) \right) \quad \text{ and } \notag \\ 
\perpcomppop &=  (1 - 2\eta) \perpcomp + 2 \eta \cdot \frac{2}{\pi} \sin^2 \anglecurr.
\end{align}

As previously noted, our interest\footnote{Our techniques are also applicable to analyzing the algorithm with general stepsize $\eta$, but we do not do so in this paper since a variety of other analysis methods tailored to first order updates~\citep[e.g.,][]{zhang2017nonconvex,chen2019gradient,tan2019phase} also work in this case.} will be in analyzing the special case $\eta = 1/2$ so as to compare and contrast with the AM update. In this case, the population updates~\eqref{eq:Pop-AM-PR} and~\eqref{eq:Pop-GD-PR} coincide, and so Fact~\ref{prop:Pop-PR} suggests that subgradient descent ought to converge quadratically fast. This would be quite surprising for a first-order method, and already suggests that the population update may be even more optimistic than before. However, the Gordon state evolution updates~\eqref{eq:Gordon-AM-PR} and~\eqref{eq:Gordon-GD-PR} are distinct even when $\eta = 1/2$, and as we saw before, these provide much more faithful predictions of convergence behavior. 

\begin{theorem} \label{thm:GD-PR}
Consider the subgradient descent update $\mathcal{T}_n$~\eqref{eq:GD-PR} and the associated Gordon state evolution update $\mathcal{S}_{\gor}$ from equation~\eqref{eq:Gordon-GD-PR}, with stepsize $\eta = 1/2$. There is a universal positive constant $C$ such that the following is true. If $\kappa \geq C (1 + \noisestd^2)$, then:

\noindent (a) The Gordon state evolution update
\begin{center}
$\mathcal{S}_{\gor}$  is $(c_{\kappa}, C_\kappa, 0)$-linearly convergent in the $\ell_2$ metric $\DeltaSELtwo$ on $\Goodset$ to level $\varepsilon_{n, d} = \frac{\noisestd}{\sqrt{\kappa}}$.
\end{center}
Here $0 \leq c_\kappa \leq C_{\kappa} < 1$ are constants depending solely on $\kappa$.

\noindent (b) Suppose $\noisestd > 0$. Then there are positive constants $C_\noisestd, C'_\noisestd$ depending only on $\noisestd$ such that for all $n \geq C'_\noisestd$ and for any $\bt$ such that $\bzeta = (\parcomp(\bt), \perpcomp(\bt)) \in \Goodset$, we have
\begin{align*}
\max_{1 \leq t \leq T} \; | \DeltaSELtwo(\mathcal{S}^t_{\gor}(\bzeta)) - \| \mathcal{T}_n^t(\bt) - \thetastar \|_2 | \leq C_\noisestd \left( \frac{\log n}{n} \right)^{1/4}
\end{align*}
with probability exceeding $1 - 2T n^{-10}$.

\noindent (c) Suppose $\bt_0$ denotes a point such that $\frac{\parcomp(\bt_0)}{\perpcomp(\bt_0)} \geq \frac{1}{50 \sqrt{d}}$ and $\parcomp(\bt_0) \lor \perpcomp(\bt_0) \leq 3/2$, and further suppose that \sloppy \mbox{$\kappa \geq C''_{\noisestd} \cdot \log \left( \frac{1 + \log d}{\delta} \right)$} for $C''_\sigma$ depending solely on $\noisestd$. 
Then for some $\ttilde \leq C\log d$, we have
\begin{align*}
\mathcal{T}_{n}^{\ttilde}(\bt_0) \in \Goodset
\end{align*} 
with probability exceeding $1 - \delta$.
\end{theorem}

To be concrete once again, suppose $d \geq 130$. Then using $n \geq d$ observations $(\bx_i, y_i)_{i = 1}^n$ from the model~\eqref{eq:PR-model} and setting $\bt_0 = \sqrt{\frac{1}{n} \sum_{i = 1}^{n} y_i^2 } \cdot \bu$ with the vector $\bu$ chosen uniformly at random from the unit ball, we obtain the required initialization condition with probability greater than $0.95$ (see Lemma~\ref{lem:init-properties}(b) in the appendix). The theorem then guarantees that for some $\tau = \mathcal{O}( \log d + \log(\kappa/\noisestd^2))$, the empirics satisfy
\begin{align}
\| \mathcal{T}_{n}^{\tau}(\bt_0) - \thetastar \| = \order \left( \noisestd\sqrt{\frac{d}{n}} \right) + \ordertil \left( n^{-1/4} \right)
\end{align}
with high probability.
Given our extensive discussion of Theorem~\ref{thm:AM-PR} and that most of these comments also apply here, we make just one
remark in passing that focuses on the difference. Note that as expected, Theorem~\ref{thm:GD-PR} shows that subgradient descent only converges linearly in the good region. This corroborates what we saw in Figures~\ref{fig:intro-pop} and~\ref{fig:intro-Gordon}, and shows once again---and more dramatically than before---that the (quadratically convergent) population update can be significantly optimistic in predicting convergence behavior. 

\subsection{Mixture of regressions}

While the symmetric mixture of linear regressions model is statistically equivalent (for parameter estimation) to the phase retrieval model without additive noise (i.e., $\noisestd = 0$), we show in this section that the models and their associated algorithms have distinct behavior for any nonzero noise level. 

\subsubsection{Alternating minimization}

Recall from equation~\eqref{eq:update-explicit-AM-MR} that the empirical update applied at $\bt$ is given by
\begin{align}
	\label{eq:empirics-AM-MLR}
\mathcal{T}_n(\bt) = \left( \frac{1}{n} \sum_{i = 1}^n \bx_i \bx_i^\top \right)^{-1} \; \left( \frac{1}{n} \sum_{i = 1}^n \sign(y_i \cdot \inprod{\bx_i}{\bt}) \cdot y_i \cdot \bx_i \right).
\end{align}
The Gordon updates are given by the following corollary of Theorem~\ref{thm:one_step_main_HO}, proved in Appendix~\ref{pf-cor3}. Before stating it, we define the convenient shorthand
\begin{align} \label{eq:AB-shorthand}
A_{\noisestd}(\rho) := \frac{2}{\pi}\tan^{-1}\left(\sqrt{\rho^2 + \noisestd^2 + \noisestd^2 \rho^2}\right) \quad \text{ and } \quad
B_{\noisestd}(\rho) := \frac{2}{\pi} \frac{\sqrt{\rho^2 + \noisestd^2 + \noisestd^2\rho^2}}{1 + \rho^2}.
\end{align}
\begin{corollary} \label{lem:Gordon-SE-AM-MLR}
Let $\parcomp = \parcomp(\bt)$ and $\perpcomp = \perpcomp(\bt)$ with $\bzeta = (\parcomp, \perpcomp)$ and $\rho = \frac{\perpcomp}{\parcomp}$. Let $(\parcompgordon, \perpcompgordon) = \mathcal{S}_{\gor}(\bzeta)$ denote the Gordon state evolution update 
in this case, given by Definition~\ref{def:starnext_main_HO}. \\
(a) Using the shorthand~\eqref{eq:AB-shorthand}, we have
\begin{subequations} \label{eq:Gordon-AM-MLR}
\begin{align}
\parcompgordon &= 1 - A_{\noisestd}(\rho) + B_{\noisestd}(\rho), \text{ and } \\
\perpcompgordon &= \sqrt{\rho^2 B_{\noisestd}(\rho)^2 + \frac{1}{\kappa - 1}\left(1 + \noisestd^2 - (1 - A_{\noisestd}(\rho) + B_{\noisestd}(\rho))^2 - \rho^2 B_{\noisestd}(\rho)^2\right)}.
\end{align}
\end{subequations}
(b) Suppose $\noisestd > 0$. Then there is a positive constant $C_\noisestd$ depending only on $\noisestd$ such that with $\mathcal{T}_n$ as defined in equation~\eqref{eq:empirics-AM-MLR}, the empirical state evolution $(\parcomp^+, \perpcomp^+) = (\parcomp(\mathcal{T}_n(\bt)), \perpcomp(\mathcal{T}_n(\bt)))$ satisfies
\begin{align*}
\Pro \left\{ | \parcomp^+ - \parcompgordon | \leq  C_\noisestd \left( \frac{\log^7 (1/\delta) }{n} \right)^{1/2} \right\} \leq \delta \quad \text{ and } \quad \Pro \left\{ | \perpcomp^+ - \perpcompgordon | \leq C_\noisestd \left( \frac{\log (1 / \delta) }{n} \right)^{1/4} \right\} \leq \delta.
\end{align*}
\end{corollary}
By taking $\kappa \to \infty$ in equation~\eqref{eq:Gordon-AM-MLR}, we recover the population update for this case, given by
\begin{align}\label{eq:Pop-AM-MLR}
\parcomppop = 1 - A_{\noisestd}(\rho) + B_{\noisestd}(\rho)  \quad \text{ and }  \quad \perpcomppop = \rho B_{\noisestd}(\rho).
\end{align}
The update~\eqref{eq:Pop-AM-MLR} has no dependence on $\kappa$ and thus cannot recover the noise floor of the problem. On the other hand, and similarly to before, the following theorem shows that the empirics are tracked instead by the Gordon update~\eqref{eq:Gordon-AM-MLR}.
\begin{theorem} \label{thm:AM-MLR}
Consider the alternating minimization update $\mathcal{T}_n$ given in equation~\eqref{eq:empirics-AM-MLR} and the associated Gordon state evolution update $\mathcal{S}_{\gor}$~\eqref{eq:Gordon-AM-MLR}. There are universal positive constants $(c, C)$ such that the following is true. If $\kappa \geq C$ and $0 < \noisestd \leq c$, then:

\noindent (a) The Gordon state evolution update 
\begin{center}
$\mathcal{S}_{\gor}$ is $(c_{\kappa, \noisestd}, C_{\kappa, \noisestd}, 0)$-linearly convergent in the angular metric $\DeltaSEangle$ on $\Goodset$ to level $\varepsilon_{n, d} = \frac{\noisestd}{\sqrt{\kappa}}$,
\end{center} 
where $0 \leq c_{\kappa, \noisestd} \leq C_{\kappa, \noisestd} \leq 1$ are constants depending solely on the pair $(\kappa, \noisestd)$. 

\noindent (b) If $n \geq C'_\noisestd$, then for any $\bt$ such that $\bzeta = (\parcomp(\bt), \perpcomp(\bt)) \in \Goodset$, we have
\begin{align*}
\max_{1 \leq t \leq T} \; | \DeltaSEangle(\mathcal{S}^t_{\gor}(\bzeta)) - \angle(\bt, \thetastar) | \leq C_\noisestd \left( \frac{\log n}{n} \right)^{1/4}
\end{align*}
with probability exceeding $1 - 2T n^{-10}$. Here $C_\noisestd$ and $C'_\sigma$ are positive constants depending solely on $\noisestd$.

\noindent (c) Suppose $\bt_0$ denotes a point such that $\frac{\parcomp(\bt_0)}{\perpcomp(\bt_0)} \geq \frac{1}{50 \sqrt{d}}$ and further suppose that \sloppy \mbox{$\kappa \geq C''_{\noisestd} \cdot \log^7 \left( \frac{1 + \log d}{\delta} \right)$} for $C''_\sigma$ depending solely on $\noisestd$. 
Then for some $\ttilde \leq C\log d$, we have
\begin{align*}
\mathcal{T}_{n}^{\ttilde}(\bt_0) \in \Goodset
\end{align*} 
with probability exceeding $1 - \delta$.
\end{theorem}
Owing to the discussion following Theorem~\ref{thm:GD-PR} (see Lemma~\ref{lem:init-properties} in the appendix), we deduce that with a random initialization $\bt_0$ and after $\tau = \mathcal{O}( \log d + \log \log (\kappa/\noisestd^2))$ iterations, the empirics satisfy
\begin{align}
\angle \left( \mathcal{T}_{n}^{\tau}(\bt_0), \thetastar \right) = \order \left( \noisestd \sqrt{\frac{d}{n}} \right) + \ordertil \left( n^{-1/4} \right)
\end{align}
with high probability.

Let us make a few remarks to compare and contrast Theorem~\ref{thm:AM-MLR} with our previous results.
First, note that the convergence result proved here is in the angular metric $\DeltaSEangle$ and not in the (stronger) $\ell_2$ metric $\DeltaSELtwo$.
This is a crucial difference between the  phase retrieval and mixture of regressions models. Indeed, the parameter estimate for AM in mixtures of regressions can be shown to be inconsistent in the $\ell_2$ distance; to see this, note that when $\bt = \thetastar$, we have $\parcompgordon = 1 + \Theta(\noisestd^3)$. Combining this estimate with part (b) of Corollary~\ref{lem:Gordon-SE-AM-MLR}, we see that $\DeltaSELtwo(\parcomp^+, \perpcomp^+) = \Theta(\noisestd^3) + o(1)$, and so for any constant noise level, the algorithm is not consistent. Inconsistency of parameter estimation is a known phenomenon for alternating minimization algorithms in mixture models with noise (for instance, a similar conclusion follows from the results of~\citet{lu2016statistical} on the label recovery error of Lloyd's algorithm in a Gaussian mixture).

Second, note that when $\noisestd = 0$, the mixture of regressions and phase retrieval models coincide. 
However, when there is noise, the convergence behavior predicted by Theorem~\ref{thm:AM-MLR} changes drastically to a linear rate, while in phase retrieval, super-linear convergence is preserved even when the noise level is nonzero (cf. Theorem~\ref{thm:AM-PR}). The Gordon update---and the ensuing sharpness of our upper and lower bounds of the error of the algorithm---enable us to make this distinction.

Finally, note that our assumption on the noise level in this case is that $\noisestd$ (as opposed to $\noisestd / \sqrt{\kappa}$) be bounded above by a universal constant, resulting in a more stringent condition than what we required in phase retrieval. While we make this assumption for convenience in our proof, we conjecture that it can be weakened to accommodate the optimal condition $\noisestd/\sqrt{\kappa} \leq c$.

\subsubsection{Subgradient AM}

For completeness, we also present corollaries for the subgradient version of the AM update. As mentioned before, we are not aware of this algorithm having been considered in the literature, but it is natural for us to study it since when the stepsize $\eta = 1/2$, it shares the same population update as AM (see Remark~\ref{rem:pop-coincide}). Given that AM converges linearly for a mixture of regressions, it is natural to ask if the first-order method---which has a much lower per-iteration cost---enjoys a similar convergence rate.
As derived in equation~\eqref{eq:update-explicit-GD-MR}, the update with stepsize $\eta$ is given by
\begin{align} \label{eq:empirics-GD-MLR}
\mathcal{T}_n(\bt) = \bt - \frac{2\eta}{n} \cdot \sum_{i = 1}^n (\sign(y_i \inprod{\bx_i}{\bt}) \cdot \inprod{\bx_i}{\bt} - y_i) \cdot \sign(y_i \inprod{\bx_i}{\bt}) \cdot \bx_i.
\end{align}
The Gordon state evolution update in this case is given by the following corollary of Theorem~\ref{thm:one_step_main_FO}, proved in Appendix~\ref{pf-cor4}.
\begin{corollary} \label{lem:Gordon-SE-GD-MLR}
Let $\parcomp = \parcomp(\bt)$ and $\perpcomp = \perpcomp(\bt)$ with $\bzeta = (\parcomp, \perpcomp)$ and $\rho = \frac{\perpcomp}{\parcomp}$. Let $(\parcompgordon, \perpcompgordon) = \mathcal{S}_{\gor}(\bzeta)$ denote the Gordon state evolution corresponding to the update~\eqref{eq:empirics-GD-MLR}, given by Definition~\ref{def:starnext_main_FO}. Let $\eta \leq 1/2$. \\
(a) Using the shorthand~\eqref{eq:AB-shorthand}, we have
\begin{subequations} \label{eq:Gordon-GD-MLR}
\begin{align}
\parcompgordon &= (1 - 2\eta) \parcomp + 2 \eta \cdot \left( 1 - A_{\noisestd}(\rho) + B_{\noisestd}(\rho) \right), \text{ and } \\
\perpcompgordon &= \Big( \big\{ (1 - 2 \eta) \perpcomp + 2\eta \cdot \rho B_{\noisestd}(\rho) \big\}^2 \notag \\
&\qquad \qquad + \frac{4 \eta^2}{\kappa}\left\{ \parcomp^2 + \perpcomp^2 - 2\parcomp(1 - A_{\noisestd}(\rho) + B_{\noisestd}(\rho)) - 2 \perpcomp \rho B_{\noisestd}(\rho) + 1 + \noisestd^2 \right\} \Big)^{1/2}.
\end{align}
\end{subequations}
(b) Suppose $\noisestd > 0$. Then there is a positive constant $C_\noisestd$ depending solely on $\noisestd$ such that with $\mathcal{T}_n$ as defined in equation~\eqref{eq:GD-PR}, the empirical state evolution $(\parcomp^+, \perpcomp^+) = (\parcomp(\mathcal{T}_n(\bt)), \perpcomp(\mathcal{T}_n(\bt)))$ satisfies
\begin{align*}
\Pro \left\{ | \parcomp^+ - \parcompgordon | \leq  C_\noisestd \left( \frac{\log (1/\delta) }{n} \right)^{1/2} \right\} \leq \delta \quad \text{ and } \quad \Pro \left\{ | \perpcomp^+ - \perpcompgordon | \leq C_\noisestd \left( \frac{\log (1 / \delta) }{n} \right)^{1/4} \right\} \leq \delta.
\end{align*}
\end{corollary}
Sending $\kappa \to \infty$ recovers the population update
\begin{align}\label{eq:Pop-GD-MLR}
\parcomppop = (1 - 2 \eta) \parcomp + 2 \eta \cdot (1 - A_{\noisestd}(\rho) + B_{\noisestd}(\rho))  \quad \text{ and }  \quad \perpcomppop = (1 - 2 \eta) \perpcomp + 2\eta \cdot \rho B_{\noisestd}(\rho).
\end{align}
Once again, our interest will be in analyzing the special case $\eta = 1/2$, in which case the population updates~\eqref{eq:Pop-GD-MLR} and~\eqref{eq:Pop-AM-MLR} of both the first-order and higher-order algorithm coincide. The following theorem establishes a sharp characterization of the convergence behavior of the subgradient method. 

\begin{figure*}[!ht]
	\centering
	\begin{subfigure}[b]{0.48\textwidth}
		\centering
		\input{sections/numerical-illustrations/figs/smalldelta-largesigma-MLR.tex}
		\caption{Subgradient descent and alternating minimization for MLR with $\sigma = 0.05$ and $\kappa = 20$.}    
		\label{subfig:linear-MLR-largenoise}
	\end{subfigure}
	\hfill
	\begin{subfigure}[b]{0.48\textwidth}  
		\centering 
		\input{sections/numerical-illustrations/figs/middelta-largesigma-linearillustration-MLR.tex}
		\caption{Subgradient descent and alternating minimization for MLR with $\sigma = 0.25$ and $\kappa = 100$.}
		\label{subfig:linear-MLR-largesample}
	\end{subfigure}
	\caption{The AM and subgradient AM algorithms for two settings of $\kappa$ and $\sigma$ initialized with $\bt_0 = \bt_0 = 0.5 \cdot \btstar + \sqrt{1 - 0.5^2} \proj_{\btstar}^{\perp} \bgam$, for 
	$\bgam$ uniformly distributed on the unit sphere, 
	plotted with the respective Gordon predictions. The shaded region denotes the values taken between the minimum and maximum of the empirics.} 
	\label{fig:linearMLR}
\end{figure*}
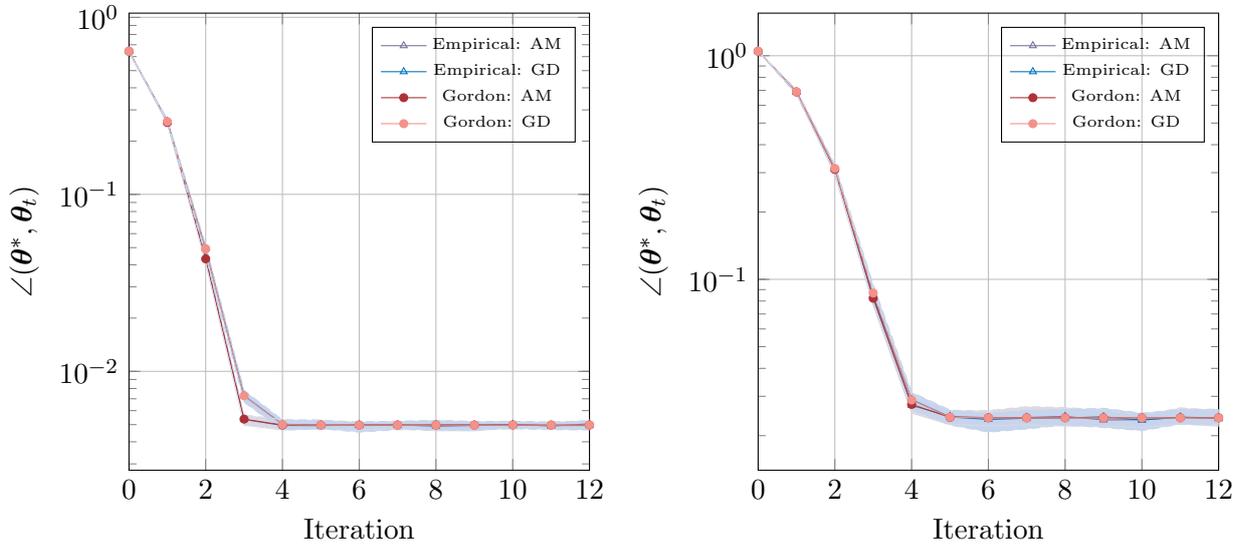

\begin{theorem} \label{thm:GD-MLR}
Let the stepsize $\eta = 1/2$ and consider the subgradient update $\mathcal{T}_n$~\eqref{eq:empirics-AM-MLR} and the associated Gordon state evolution update $\mathcal{S}_{\gor}$~\eqref{eq:Gordon-GD-MLR}. There are universal positive constants $(c, C)$ such that the following is true. If $\kappa \geq C$ and $\noisestd \leq c$, then:

\noindent (a) The Gordon state evolution update 
\begin{center}
$\mathcal{S}_{\gor}$ is $(c_{\kappa, \noisestd}, C_{\kappa, \noisestd}, 1)$-linearly convergent in the angular metric $\DeltaSEangle$ on $\Goodset$ to level $\varepsilon_{n, d} = \frac{\noisestd}{\sqrt{\kappa}}$,
\end{center} 
where $0 \leq c_{\kappa, \noisestd} \leq C_{\kappa, \noisestd} \leq 1$ are constants depending solely on the pair $(\kappa, \sigma)$.

\noindent (b) If $n \geq C'_\noisestd$, then for any $\bt$ such that $\bzeta = (\parcomp(\bt), \perpcomp(\bt)) \in \Goodset$, we have
\begin{align*}
\max_{1 \leq t \leq T} \; | \DeltaSEangle(\mathcal{S}^t_{\gor}(\bzeta)) - \angle(\bt, \thetastar) | \leq C_\noisestd \left( \frac{\log n}{n} \right)^{1/4}
\end{align*}
with probability exceeding $1 - 2T n^{-10}$. Here $C'_\noisestd$ and $C_\noisestd$ are positive constants depending solely on $\noisestd$.

\noindent (c) Suppose $\bt_0$ denotes a point such that $\frac{\parcomp(\bt_0)}{\perpcomp(\bt_0)} \geq \frac{1}{50 \sqrt{d}}$ and $\parcomp(\bt_0) \lor \perpcomp(\bt_0) \leq 3/2$, and further suppose that \sloppy \mbox{$\kappa \geq C''_{\noisestd} \cdot \log \left( \frac{1 + \log d}{\delta} \right)$} for $C''_\sigma$ depending solely on $\noisestd$. 
Then for some $\ttilde \leq C\log d$, we have
\begin{align*}
\mathcal{T}_{n}^{\ttilde}(\bt_0) \in \Goodset
\end{align*} 
with probability exceeding $1 - \delta$.
\end{theorem}
As in the case of subgradient descent for phase retrieval (see also Lemma~\ref{lem:init-properties}(b) in the appendix), we see that if $\bt_0 = \sqrt{\frac{1}{n} \sum_{i = 1}^n y_i^2} \cdot \bu$ for a random vector $\bu$ chosen from the unit sphere, then after $\tau = \mathcal{O}( \log d + \log \log (\kappa/\noisestd^2))$ iterations, the empirics satisfy
\begin{align}
\angle \left( \mathcal{T}_{n}^{\tau}(\bt_0), \thetastar \right) = \order \left( \noisestd \sqrt{\frac{d}{n}} \right) + \ordertil \left( n^{-1/4} \right)
\end{align}
with high probability.

The fact that both subgradient descent and alternating minimization (cf. Theorem~\ref{thm:AM-MLR}) converge linearly in the good region suggest that the first order method, which has smaller per-iteration cost, may be a good choice for a mixture of linear regressions. A closer look at the proof suggests that the corresponding coefficients of contraction $C_{\kappa, \noisestd}$ 
may be comparable for even moderately large $\kappa$.
Indeed, this is illustrated in Figure~\ref{fig:linearMLR}, where we see two settings of the pair $(\kappa, \noisestd)$ in which both algorithms exhibit nearly identical behavior. This observation provides further evidence that the subgradient method is a compelling choice in such scenarios.  

\subsection{A glimpse of the convergence proof mechanism}

To conclude this section, we provide a high level overview of our convergence proof technique, aspects of which may be of independent interest. A schematic of the proof mechanism is presented in Figure~\ref{fig:conv-schematic}. The blue curve in the panel (Top) represents the empirical state evolution $(\parcomp_t, \perpcomp_t)$,
and our proof technique relies on tracking the transitions of this curve across three phases. Points $(\parcomp, \perpcomp)$ in Phase I are such that the ratio $\beta / \alpha$ 
is greater than some threshold. Phase II is characterized by $\beta / \alpha$ being between two distinct thresholds. 
Phase III corresponds to being in the good region $\Goodset$, in which the ratio $\beta / \alpha$ is smaller than some small threshold and the parallel component $\alpha$ is larger than a threshold (see Definition~\ref{def:good-region}). In each phase, depicted in detail in the (Left), (Right), and (Bottom) plots of Figure~\ref{fig:conv-schematic}, we track particular Gordon state evolution updates using red dots. 
The shaded light blue regions schematically depict confidence sets that show how each empirical iterate is ``trapped" around its Gordon counterpart with high probability. In Phases I and II, we track Gordon state evolution updates when run from the ``worst possible'' empirical iterate in the previous confidence set, depicted in the figure using light blue triangles. In Phase III, on the other hand, we track the full Gordon \emph{trajectory}, i.e., the deterministic sequence of points that results from iteratively running the Gordon update from the initial dark blue triangle.
The behavior of the Gordon update itself is model-dependent and governed by specific structural properties of the corresponding state evolution maps. We establish these properties in Section~\ref{sec:prelim-lemmas}, and use them to establish part (a) of all our theorems in this section. For now, let us sketch the key ideas underlying our treatment of the empirical iterates in each phase.

\begin{figure}[!ht]
  \begin{center}
    \includegraphics[width=\textwidth]{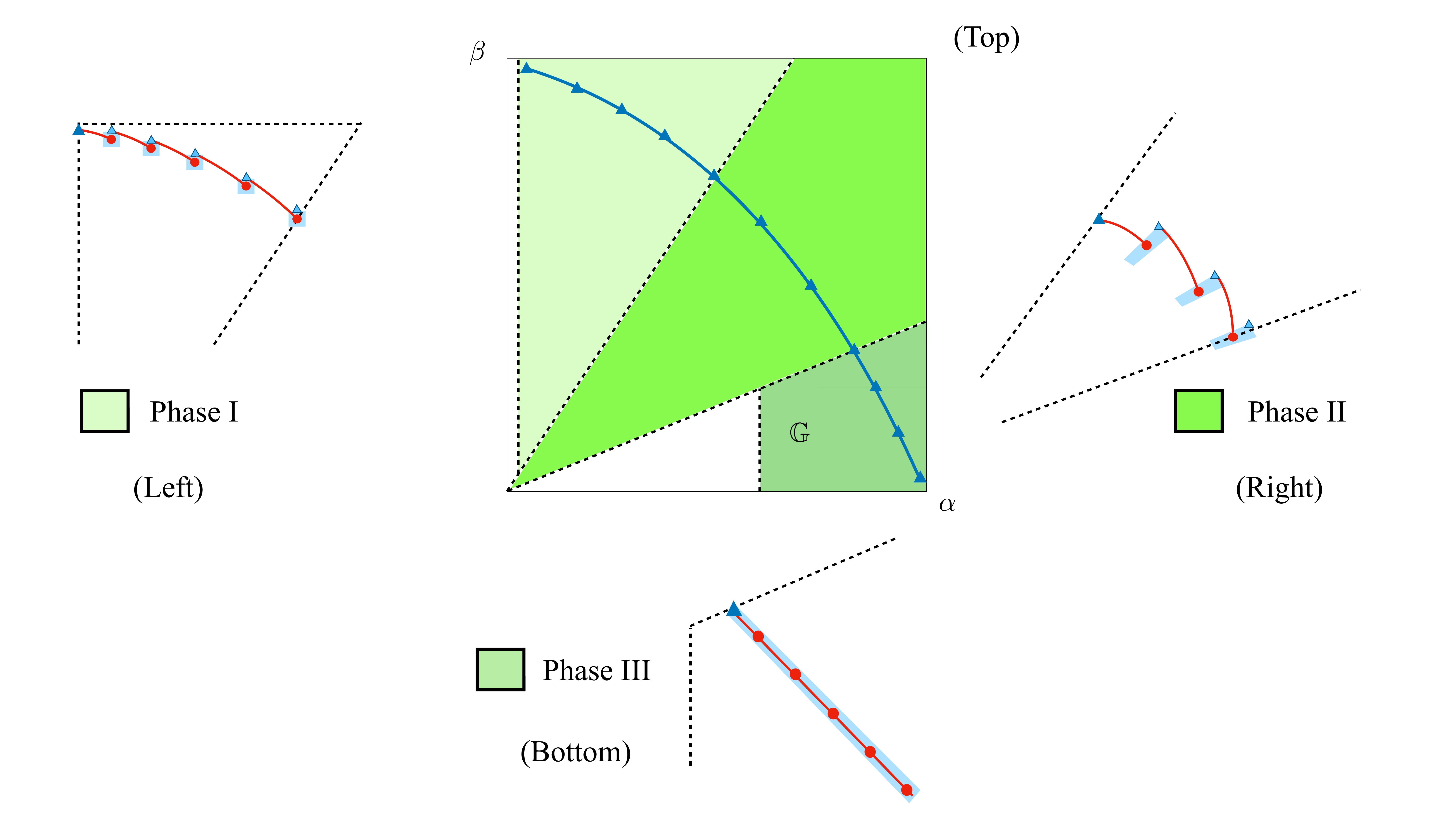}
    \caption{A schematic showing convergence of the algorithm in terms of its state space representation $(\alpha, \beta)$, in three distinct phases. (Top) The triangles in dark blue (and the corresponding curved line) denote the empirical iterates as they proceed through three \emph{phases}. The panels (Left), (Right) and (Bottom) are zoomed-in versions of Phases I, II, and III, respectively, where the dark blue triangle at the start of the phase depicts the point of the trajectory within that phase. Each red circle in these subfigures denotes an iterate of the deterministic Gordon state evolution update when run from the point that it is connected to. 
    The shaded blue regions in all three phases represent high-probability confidence sets for the empirical iterates. 
    In panel (Left), we leverage the fact that the $\beta$-component of each iterate is trapped around that of its Gordon counterpart. In panel (Right), each such region is an angular ``wedge'' around the corresponding Gordon iterate, and in panel (Bottom), the entire region (across iterations) is determined by a small envelope around the full Gordon trajectory. The light blue triangles in Phases I and II denote ``worst-case'' instances of the empirics within the corresponding confidence set. See the accompanying text for a more detailed discussion.}
    \label{fig:conv-schematic}
  \end{center}
\end{figure}

\paragraph{Phase I:} Immediately after initialization, the parallel component $\alpha$ is very small, of the order $d^{-1/2}$. To show that the empirical iterates proceed favorably through Phase I, we use the fact that the Gordon state evolution $\parcompgordon \geq (1 + c) \parcomp$ whenever $\perpcomp/\parcomp$ is large, thereby increasing the parallel component \emph{exponentially} within this phase. The $\order(n^{-1/2})$ concentration of the empirical $\alpha_t$ update around its Gordon prediction 
traps each empirical iterate $\alpha_t$ within a small interval---as depicted in Figure~\ref{fig:conv-schematic}(Left)---and allows us to argue when $n \gtrsim d$ that $\parcomp_t$ also increases exponentially with $t$ in Phase I. At the same time, the $\perpcomp_t$ iterates also remain bounded, so that $\perpcomp_t / \parcomp_t$ decreases below a threshold and enters Phase II.  Phase I takes at most $\order(\log d)$ iterations with high probability.

\paragraph{Phase II:} Next, we show that the ratio $\perpcompgordon/\parcompgordon$ of the Gordon state evolution decreases exponentially, and we translate this convergence to the empirical ratio $\perpcomp_t/ \parcomp_t$ by using the relations~\eqref{eq:main-conc_HO} and~\eqref{eq:main-conc_FO}. This traps each empirical iterate within a small \emph{angular} neighborhood of its Gordon counterpart, and is depicted in Figure~\ref{fig:conv-schematic}(Right). Together with the aforementioned convergence of the Gordon ratio $\perpcompgordon/\parcompgordon$, this ensures that we enter the good region $\Goodset$. We show that with high probability, the iterates stay within Phase II for at most $\order(1)$ iterations.
Along with the previously established convergence in Phase I, this establishes part (c) of all our model-specific theorems, showing that our iterates enter the good region, i.e., Phase III, after at most $\order(\log d)$ steps after random initialization. 

\paragraph{Phase III:} In this final phase, we show a property that, to the best of our knowledge, is absent from local convergence guarantees in prior work. This is collected in part (b) of our individual theorems, and shows that a small \emph{envelope} around the Gordon state evolution trajectory, as depicted in Figure~\ref{fig:conv-schematic}(Bottom), fully traps the random iterates with high probability. The key property that we use to show this is in fact what guides our choice of the good region: The derivatives of the $\parcompgordon$ and $\perpcompgordon$ maps when evaluated for any element in this region are both bounded above by $1-c$ for some universal constant $c > 0$, so that small deviations of the empirics from these maps are not amplified over the course of successive iterations.

%% file: sections/numerical-illustrations/figs/smalldelta-largesigma-MLR.tex
	\begin{tikzpicture}
		\begin{axis}[
			width=\textwidth,
			height=\textwidth,
			xlabel={Iteration},
			ylabel={$\angle(\btstar, \bt_t)$},
			ymode=log,
			xmin=0, xmax=12,
			legend style={font=\tiny, fill=none},
			legend pos=north east,
			ymajorgrids=true,
			grid style=grid,
			]
			
			\addplot[
			color=CadetBlue,
			mark=triangle,
			mark size=1.5pt   
			]
			table[x=iter,y=avgAM] {sections/numerical-illustrations/data/smalldelta-largesigma-MLR.dat};
			\legend{Empirical: AM}
			
			\addplot[
			color=RoyalBlue,
			mark=triangle,
			mark size=1.5pt   
			]
			table[x=iter,y=avgGD] {sections/numerical-illustrations/data/smalldelta-largesigma-MLR.dat};
			\addlegendentry{Empirical: GD}
			
			\addplot[
			color=Maroon,
			mark=*,
			mark size=1.5pt   
			]
			table[x=iter,y=gorAM] {sections/numerical-illustrations/data/smalldelta-largesigma-MLR.dat};
			\addlegendentry{Gordon: AM}
			
			\addplot[
			color=Salmon,
			mark=*,
			mark size=1.5pt   
			]
			table[x=iter,y=gorGD] {sections/numerical-illustrations/data/smalldelta-largesigma-MLR.dat};
			\addlegendentry{Gordon: GD}
			
			\addplot+[name path=A-AM,CadetBlue!20, no markers] table[x=iter,y=lowerAM] {sections/numerical-illustrations/data/smalldelta-largesigma-MLR.dat};
			\addplot+[name path=B-AM,CadetBlue!20, no markers] table[x=iter,y=upperAM] {sections/numerical-illustrations/data/smalldelta-largesigma-MLR.dat};
			
			\addplot[CadetBlue!20] fill between[of=A-AM and B-AM];
			
			\addplot+[name path=A-GD,RoyalBlue!20, no markers] table[x=iter,y=lowerGD] {sections/numerical-illustrations/data/smalldelta-largesigma-MLR.dat};
			\addplot+[name path=B-GD,RoyalBlue!20, no markers] table[x=iter,y=upperGD] {sections/numerical-illustrations/data/smalldelta-largesigma-MLR.dat};
			
			\addplot[RoyalBlue!20] fill between[of=A-GD and B-GD];
		\end{axis}
	\end{tikzpicture}

%% file: sections/numerical-illustrations/figs/middelta-largesigma-linearillustration-MLR.tex
	\begin{tikzpicture}
		\begin{axis}[
			width=\textwidth,
			height=\textwidth,
			xlabel={Iteration},
			ylabel={$\angle(\btstar, \bt_t)$},
			ymode=log,
			xmin=0, xmax=12,
			legend style={font=\tiny,fill=none},
			legend pos=north east,
			ymajorgrids=true,
			grid style=grid,
			]
			
			\addplot[
			color=CadetBlue,
			mark=triangle,
			mark size=1.5pt   
			]
			table[x=iter,y=avgAM] {sections/numerical-illustrations/data/middelta-largesigma-linearillustration-MLR.dat};
			\legend{Empirical: AM}
			
			\addplot[
			color=RoyalBlue,
			mark=triangle,
			mark size=1.5pt   
			]
			table[x=iter,y=avgGD] {sections/numerical-illustrations/data/middelta-largesigma-linearillustration-MLR.dat};
			\addlegendentry{Empirical: GD}
			
			\addplot[
			color=Maroon,
			mark=*,
			mark size=1.5pt   
			]
			table[x=iter,y=gorAM] {sections/numerical-illustrations/data/middelta-largesigma-linearillustration-MLR.dat};
			\addlegendentry{Gordon: AM}
			
			\addplot[
			color=Salmon,
			mark=*,
			mark size=1.5pt   
			]
			table[x=iter,y=gorGD] {sections/numerical-illustrations/data/middelta-largesigma-linearillustration-MLR.dat};
			\addlegendentry{Gordon: GD}
			
			\addplot+[name path=A-AM,CadetBlue!20, no markers] table[x=iter,y=lowerAM] {sections/numerical-illustrations/data/middelta-largesigma-linearillustration-MLR.dat};
			\addplot+[name path=B-AM,CadetBlue!20, no markers] table[x=iter,y=upperAM] {sections/numerical-illustrations/data/middelta-largesigma-linearillustration-MLR.dat};
			
			\addplot[CadetBlue!20] fill between[of=A-AM and B-AM];
			
			\addplot+[name path=A-GD,RoyalBlue!20, no markers] table[x=iter,y=lowerGD] {sections/numerical-illustrations/data/middelta-largesigma-linearillustration-MLR.dat};
			\addplot+[name path=B-GD,RoyalBlue!20, no markers] table[x=iter,y=upperGD] {sections/numerical-illustrations/data/middelta-largesigma-linearillustration-MLR.dat};
			
			\addplot[RoyalBlue!20] fill between[of=A-GD and B-GD];
		\end{axis}
	\end{tikzpicture}

%% file: sections/numerical-illustrations/numerical-illustrations.tex

We provide several numerical simulations to illustrate the sharpness of our results.  For each of the two models and two algorithms we consider, we demonstrate both global convergence 
as well as local convergence.
In particular, for each of the two models, we perform two families of experiments. The first explores convergence from a random initialization for both the higher-order and first-order method. These experiments are performed in dimension $d=800$ with the number of samples $n = 80,000$ (that is, $\kappa = 100$) and noise standard deviation $\sigma = 10^{-6}$.  First, a true parameter vector $\btstar$ is drawn uniformly at random from the unit sphere.  Subsequently, an initialization $\bt_0$ is drawn (independently of $\btstar$) uniformly at random on the unit sphere.  Then, from this vector, we simulate $12$ independent trials of the algorithm for $12$ iterations.  In the second family of experiments, we explore \emph{local} convergence---from an initialization which has constant correlation with the ground truth $\btstar$---for three different settings of noise standard deviation $\sigma$ and oversampling ratio $\kappa$.  Each experiment is performed in dimension $d=500$ with various numbers of samples $n$.  Each simulation is done by first drawing the ground-truth vector $\btstar$ uniformly at random on the unit sphere and subsequently generating an initialization
\begin{align*}
\bt_0 = 0.8 \cdot \btstar + \sqrt{1 - 0.8^2} \proj_{\btstar}^{\perp} \bgam,
\end{align*}
where 
$\bgam$ is uniformly distributed on the unit sphere
 and is independent of all other randomness.  Next, we run $100$ independent trials of both algorithms for $12$ iterations.

\begin{figure*}[!htbp]
	\centering
	\begin{subfigure}[b]{0.48\textwidth}
		\centering
		\input{sections/numerical-illustrations/figs/global-PR-AM.tex}
		\caption{Alternating minimization.}    
		\label{subfig:global-PR-AM}
	\end{subfigure}
	\hfill
	\begin{subfigure}[b]{0.48\textwidth}  
		\centering 
		\input{sections/numerical-illustrations/figs/global-PR-GD.tex}
		\caption{Subgradient descent.}
		\label{subfig:global-PR-GD}
	\end{subfigure}
	\caption{Global convergence in the phase retrieval model.  The hollow triangular marks (barely visible) denote the average over $12$ independent trials and the shaded regions denote the range of values taken by the empirics.} 
	\label{fig:globalPR}
\end{figure*}
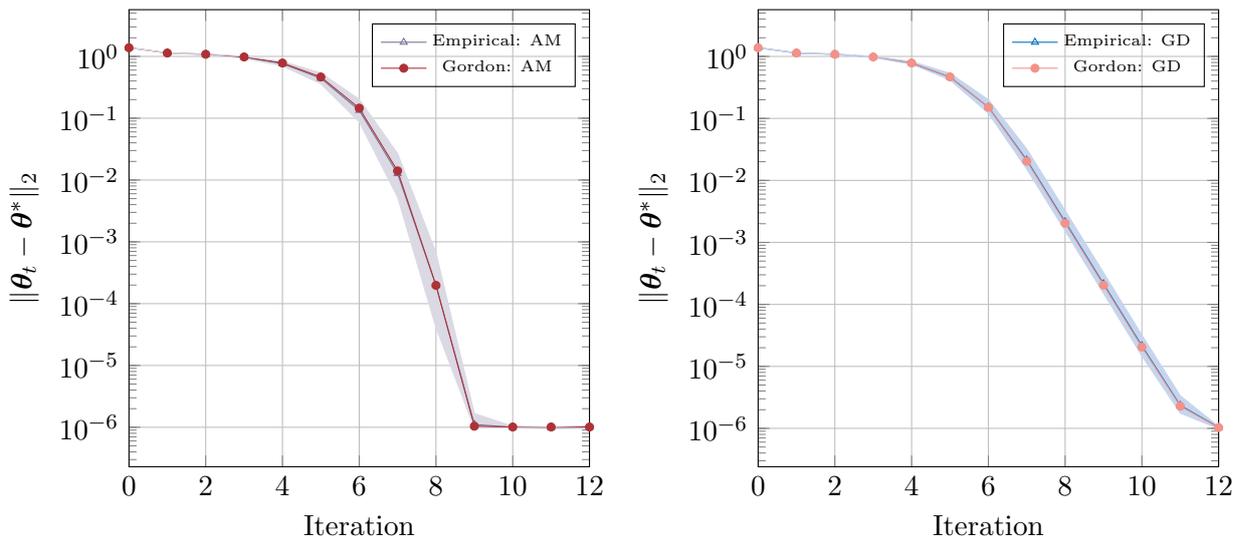

\subsection{Phase retrieval}
\label{subsec:sim-PR}
We first consider phase retrieval.  Figure~\ref{fig:globalPR} illustrates the global convergence of both alternating minimization (in Figure~\ref{subfig:global-PR-AM}) and subgradient descent (in Figure~\ref{subfig:global-PR-GD}).  
  Figure~\ref{subfig:global-PR-AM} plots (i.) filled in circular marks denoting the Gordon state evolution started at the state $(\alpha(\bt_0), \beta(\bt_0))$; (ii.) hollow triangular marks denoting the average of the empirical performance of AM over the $12$ independent trials; and (iii.) a shaded region denoting the region between the minimum and maximum values taken in the empirics.  The same three items are plotted with gradient descent in place of alternating minimization in Figure~\ref{subfig:global-PR-GD}.

\begin{figure*}[!htbp]
	\centering
	\begin{subfigure}[b]{0.48\textwidth}
		\centering
		\input{sections/numerical-illustrations/figs/lower-noise-PR.tex}
		\caption{$\sigma = 10^{-10}, \kappa = 20$}    
		\label{subfig:smalldelta-smallsigma-local-PR}
	\end{subfigure}
	\hfill
	\begin{subfigure}[b]{0.48\textwidth}   
		\centering 
		\input{sections/numerical-illustrations/figs/largedelta-midsigma-PR.tex}
		\caption{$\sigma = 10^{-6}, \kappa=100$}
		\label{subfig:largedelta-largesigma-local-PR}
	\end{subfigure}
	\centering
	\caption{Local convergence for the phase retrieval model.  Each subplot shows: (in purple) the empirics of alternating minimization, (in red) the Gordon updates for alternating minimization, (in blue) the empirics for subgradient descent, and (in orange) the Gordon updates for subgradient descent.  Hollow triangular markers denote the average of the empirics and the shaded regions denote the range of values taken by the empirics over $100$ independent trials.} 
	\label{fig:localPR}
\end{figure*}
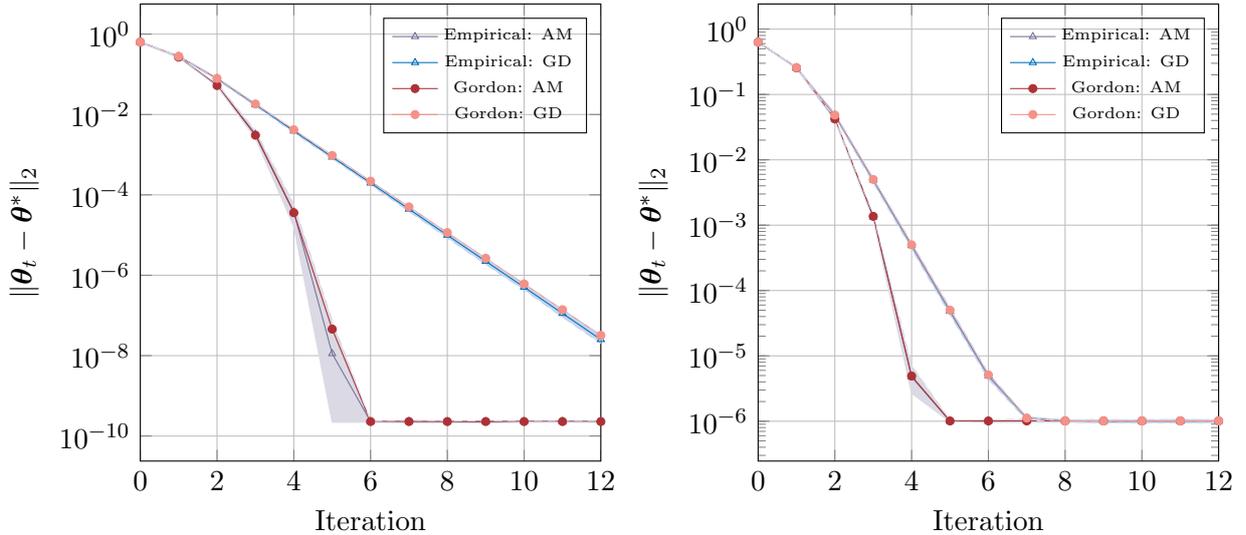

Recall that part (c) of Theorems~\ref{thm:AM-PR} and~\ref{thm:GD-PR} states that each algorithm---when started from a random initialization---first consists of a transient phase which takes $\mathcal{O}(\log{d})$ iterations to reach a ``good" region.  This transient phase is witnessed by the first $5$ iterations of each algorithm, which make very little progress in the $\ell_2$ distance.  Subsequently, parts (a) and (b) of each theorem state that in the ``good" region, the Gordon state evolution converges at a specified rate and the empirics are trapped in a small envelope around this state evolution.  Iterations $5-9$ illustrate the super-linear convergence of alternating minimization (Figure~\ref{subfig:global-PR-AM}) and iterations $5-12$ illustrate the linear convergence of subgradient descent (Figure~\ref{subfig:global-PR-GD}).  We remark that whereas the theorems show the empirics to be trapped in a small envelope surrounding the Gordon state evolution in the ``good" region, the simulations suggest that this may hold even from random initialization---that is, even the transient phase may consist of empirics trapped in an envelope around the Gordon state evolution.

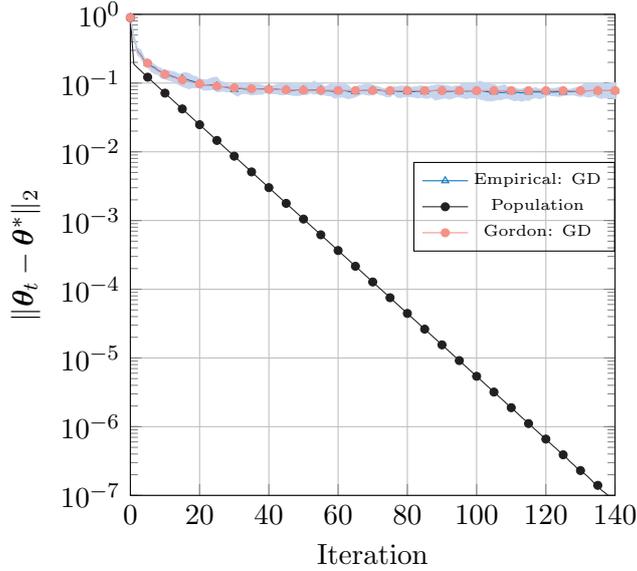
\begin{figure*}[!ht]
	\centering
	\input{sections/numerical-illustrations/figs/pop-converge-empirics-no-converge.tex}
	\caption{Subgradient descent for stepsize $\eta =0.95$. Markers are placed once every $5$ iterations.} 
	\label{fig:stepsize_popvgor}
\end{figure*}

Figure~\ref{fig:localPR} zooms in and demonstrates the local convergence for two different settings of noise standard deviation $\sigma$ and oversampling ratio $\kappa$.  
For each of the parameter values, we make two observations.  First, both subfigures make clear the deterministic qualities of the Gordon updates---the distinction between convergence rates as well as the attainment of the error floor---whereby demonstrating part (a) of Theorem~\ref{thm:AM-PR} and~\ref{thm:GD-PR}.  Second, both simulations demonstrate part (b) of the same two theorems: the empirics are trapped in a small envelope surrounding the Gordon state evolution.

\medskip
We provide one final experiment in noiseless phase retrieval to illustrate the effect of stepsize in subgradient descent.  Here, we take dimension $d=250$, the oversampling ratio $\kappa = 10$ and start from an initial correlation $\alpha_0 = 0.6$.  As opposed to setting the stepsize $\eta = 1/2$, in this experiment, we try using a much larger stepsize; namely, we take $\eta = 0.95$.  We then run 
$140$ 
 iterations of subgradient descent and perform $10$ independent trials.  As is evident from Figure~\ref{fig:stepsize_popvgor}, this is a situation in which the population update predicts convergence, yet the empirics fail to converge.  On the other hand, the Gordon updates continue to sharply characterize the empirical performance and are able to predict the lack of convergence to the ground truth parameter.

\subsection{Mixture of linear regressions}
\label{subsec:sim-MLR}
\begin{figure*}[!htbp]
	\centering
	\begin{subfigure}[b]{0.48\textwidth}
		\centering
		\input{sections/numerical-illustrations/figs/global-MLR-AM.tex}
		\caption{Alternating minimization}    
		\label{subfig:global-MLR-AM}
	\end{subfigure}
	\hfill
	\begin{subfigure}[b]{0.48\textwidth}  
		\centering 
		\input{sections/numerical-illustrations/figs/global-MLR-GD.tex}
		\caption{Subgradient AM}
		\label{subfig:global-MLR-GD}
	\end{subfigure}
	\caption{Global convergence in the mixture of linear regression model.  Hollow triangular marks denote the average over $12$ independent trials and the shaded regions denote the range of values taken by the empirics.} 
	\label{fig:globalMLR}
\end{figure*}
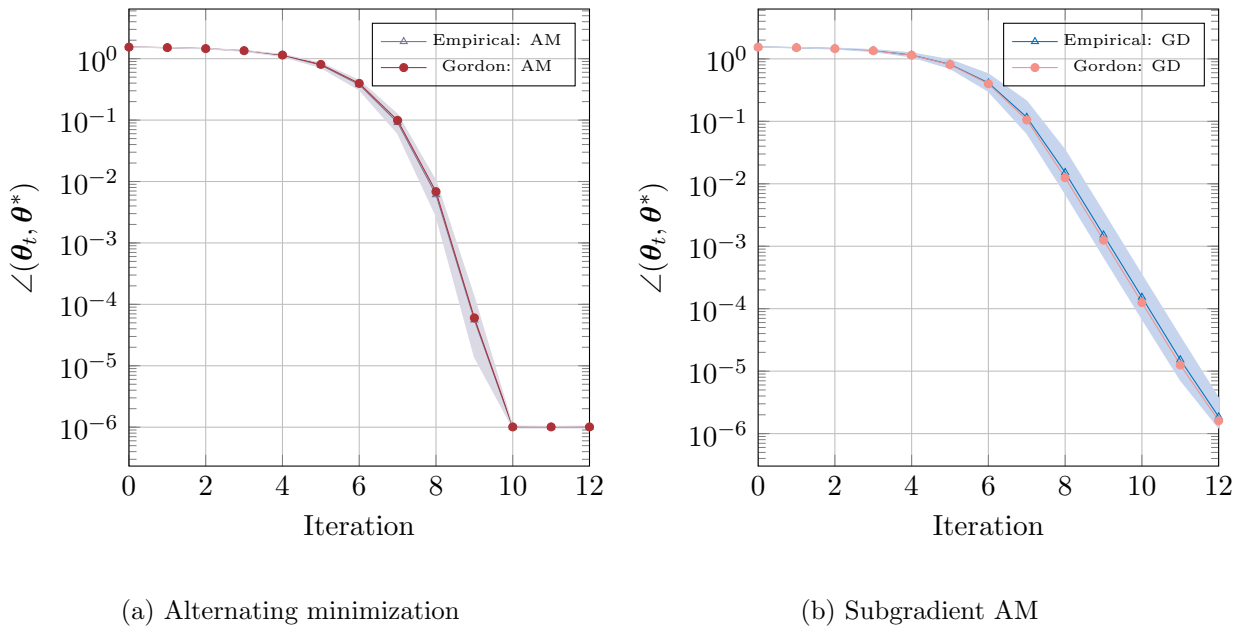
The two sets of simulations performed in this subsection (Figures~\ref{fig:globalMLR} and~\ref{fig:localMLR}) follow the same dichotomy as the two simulations performed in the previous subsection.  An important distinction is that the error metric used is the \emph{angular} metric rather than the $\ell_2$ distance used in the preceding subsection.
Figure~\ref{fig:globalMLR} plots the trajectory of both AM and subgradient AM when started from a random initialization.  As before, the simulations suggest that the empirics are trapped around the Gordon state evolution trajectory even from random initialization.

Next, we turn to the local convergence as illustrated in Figure~\ref{fig:localMLR} under two distinct parameter regimes, with the other details of the setup being identical to local convergence in phase retrieval.  
We pause only to call out the linearly convergent behavior evident in Figure~\ref{subfig:smalldelta-largesigma-local-MLR} as well as the similarity in performance of the two algorithms in the same simulation.  This is an important feature of the mixtures of linear regression model with constant noise.  

\begin{figure*}[!ht]
	\begin{subfigure}[b]{0.48\textwidth}  
		\centering 
		\input{sections/numerical-illustrations/figs/low-noise-MLR.tex}
		\caption{$\sigma = 10^{-6}, \kappa=20$}
		\label{subfig:smalldelta-largesigma-local-MLR}
	\end{subfigure}
	\hfill
	\begin{subfigure}[b]{0.48\textwidth}   
		\centering 
		\input{sections/numerical-illustrations/figs/largedelta-midsigma-MLR.tex}
		\caption{$\sigma = 10^{-2}, \kappa=6$}
		\label{subfig:largedelta-smallsigma-local-MLR}
	\end{subfigure}
	\caption{Local convergence for the mixture of linear regression model.  Each subplot shows: (in purple) the empirics of alternating minimization, (in red) the Gordon updates for alternating minimization, (in blue) the empirics for subgradient AM, and (in orange) the Gordon updates for subgradient AM.  Hollow triangular markers denote the average of the empirics and the shaded regions denote the range of values taken by the empirics over $100$ independent trials.} 
	\label{fig:localMLR}
\end{figure*}
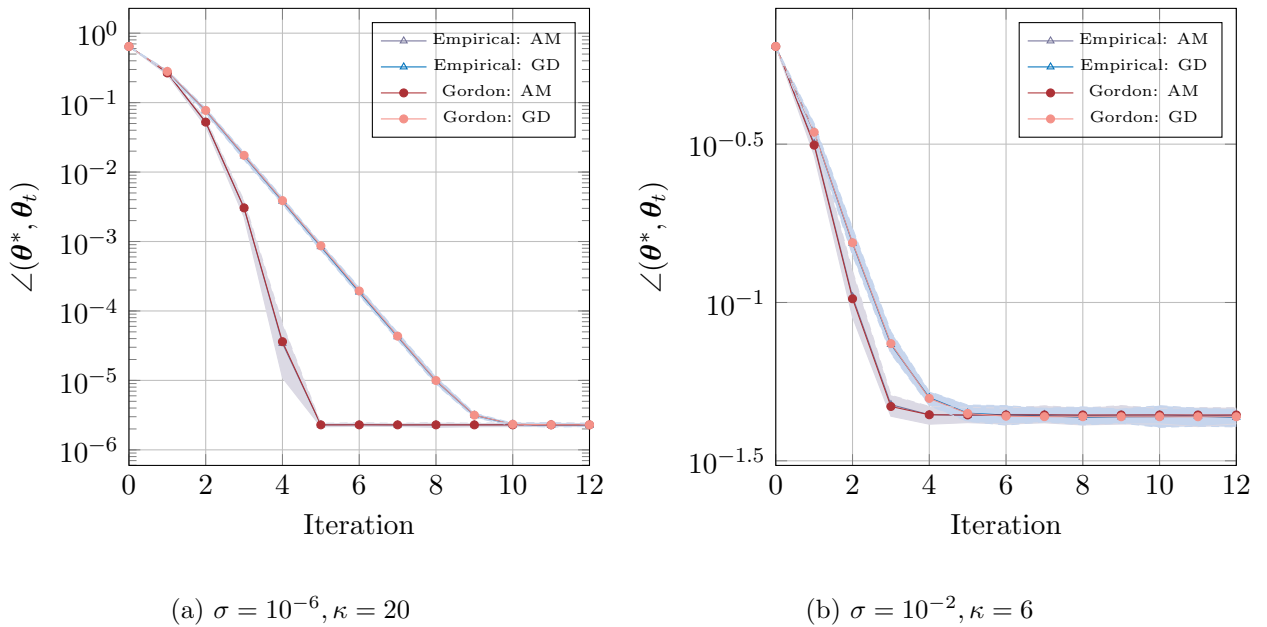

%% file: sections/numerical-illustrations/figs/global-PR-AM.tex
	\begin{tikzpicture}
		\begin{axis}[
			width=\textwidth,
			height=\textwidth,
			xlabel={Iteration},
			ylabel={$\| \bt_t - \bt^* \|_2$},
			ymode=log,
			xmin=0, xmax=12,
			legend style={font=\tiny,fill=none},
			legend pos=north east,
			ymajorgrids=true,
			grid style=grid,
			]
			
			\addplot[
			color=CadetBlue,
			mark=triangle,
			mark size=1.5pt   
			]
			table[x=iter,y=avgAM] {sections/numerical-illustrations/data/global-PR.dat};
			\legend{Empirical: AM}
			
			\addplot[
			color=Maroon,
			mark=*,
			mark size=1.5pt   
			]
			table[x=iter,y=gorAM] {sections/numerical-illustrations/data/global-PR.dat};
			\addlegendentry{Gordon: AM}
			
			\addplot+[name path=A-AM,CadetBlue!20, no markers] table[x=iter,y=lowerAM] {sections/numerical-illustrations/data/global-PR.dat};
			\addplot+[name path=B-AM,CadetBlue!20, no markers] table[x=iter,y=upperAM] {sections/numerical-illustrations/data/global-PR.dat};
			
			\addplot[CadetBlue!20] fill between[of=A-AM and B-AM];
			
		\end{axis}
	\end{tikzpicture}

%% file: sections/numerical-illustrations/figs/global-PR-GD.tex
	\begin{tikzpicture}
		\begin{axis}[
			width=\textwidth,
			height=\textwidth,
			xlabel={Iteration},
			ylabel={$\| \bt_t - \bt^* \|_2$},
			ymode=log,
			xmin=0, xmax=12,
			legend style={font=\tiny,fill=none},
			legend pos=north east,
			ymajorgrids=true,
			grid style=grid,
			]

			\addplot[
			color=RoyalBlue,
			mark=triangle,
			mark size=1.5pt   
			]
			table[x=iter,y=avgGD] {sections/numerical-illustrations/data/global-PR.dat};
			\addlegendentry{Empirical: GD}

			\addplot[
			color=Salmon,
			mark=*,
			mark size=1.5pt   
			]
			table[x=iter,y=gorGD] {sections/numerical-illustrations/data/global-PR.dat};
			\addlegendentry{Gordon: GD}

			\addplot+[name path=A-GD,RoyalBlue!20, no markers] table[x=iter,y=lowerGD] {sections/numerical-illustrations/data/global-PR.dat};
			\addplot+[name path=B-GD,RoyalBlue!20, no markers] table[x=iter,y=upperGD] {sections/numerical-illustrations/data/global-PR.dat};
			
			\addplot[RoyalBlue!20] fill between[of=A-GD and B-GD];
		\end{axis}
	\end{tikzpicture}

%% file: sections/numerical-illustrations/figs/lower-noise-PR.tex
	\begin{tikzpicture}
		\begin{axis}[
			width=\textwidth,
			height=\textwidth,
			xlabel={Iteration},
			ylabel={$\| \bt_t - \bt^* \|_2$},
			ymode=log,
			xmin=0, xmax=12,
			legend style={font=\tiny,fill=none},
			legend pos=north east,
			ymajorgrids=true,
			grid style=grid,
			]
			
			\addplot[
			color=CadetBlue,
			mark=triangle,
			mark size=1.5pt   
			]
			table[x=iter,y=avgAM] {sections/numerical-illustrations/data/lower-noise-PR.dat};
			\legend{Empirical: AM}
			
			\addplot[
			color=RoyalBlue,
			mark=triangle,
			mark size=1.5pt   
			]
			table[x=iter,y=avgGD] {sections/numerical-illustrations/data/lower-noise-PR.dat};
			\addlegendentry{Empirical: GD}
			
			\addplot[
			color=Maroon,
			mark=*,
			mark size=1.5pt   
			]
			table[x=iter,y=gorAM] {sections/numerical-illustrations/data/lower-noise-PR.dat};
			\addlegendentry{Gordon: AM}
			
			\addplot[
			color=Salmon,
			mark=*,
			mark size=1.5pt   
			]
			table[x=iter,y=gorGD] {sections/numerical-illustrations/data/lower-noise-PR.dat};
			\addlegendentry{Gordon: GD}
			
			\addplot+[name path=A-AM,CadetBlue!20, no markers] table[x=iter,y=lowerAM] {sections/numerical-illustrations/data/lower-noise-PR.dat};
			\addplot+[name path=B-AM,CadetBlue!20, no markers] table[x=iter,y=upperAM] {sections/numerical-illustrations/data/lower-noise-PR.dat};
			
			\addplot[CadetBlue!20] fill between[of=A-AM and B-AM];
			
			\addplot+[name path=A-GD,RoyalBlue!20, no markers] table[x=iter,y=lowerGD] {sections/numerical-illustrations/data/lower-noise-PR.dat};
			\addplot+[name path=B-GD,RoyalBlue!20, no markers] table[x=iter,y=upperGD] {sections/numerical-illustrations/data/lower-noise-PR.dat};
			
			\addplot[RoyalBlue!20] fill between[of=A-GD and B-GD];
		\end{axis}
	\end{tikzpicture}

%% file: sections/numerical-illustrations/figs/largedelta-midsigma-PR.tex
	\begin{tikzpicture}
		\begin{axis}[
			width=\textwidth,
			height=\textwidth,
			xlabel={Iteration},
			ylabel={$\| \bt_t - \bt^* \|_2$},
			ymode=log,
			xmin=0, xmax=12,
			legend style={font=\tiny,fill=none},
			legend pos=north east,
			ymajorgrids=true,
			grid style=grid,
			]
			
			\addplot[
			color=CadetBlue,
			mark=triangle,
			mark size=1.5pt   
			]
			table[x=iter,y=avgAM] {sections/numerical-illustrations/data/largedelta-midsigma-PR.dat};
			\legend{Empirical: AM}
			
			\addplot[
			color=RoyalBlue,
			mark=triangle,
			mark size=1.5pt   
			]
			table[x=iter,y=avgGD] {sections/numerical-illustrations/data/largedelta-midsigma-PR.dat};
			\addlegendentry{Empirical: GD}
			
			\addplot[
			color=Maroon,
			mark=*,
			mark size=1.5pt   
			]
			table[x=iter,y=gorAM] {sections/numerical-illustrations/data/largedelta-midsigma-PR.dat};
			\addlegendentry{Gordon: AM}
			
			\addplot[
			color=Salmon,
			mark=*,
			mark size=1.5pt   
			]
			table[x=iter,y=gorGD] {sections/numerical-illustrations/data/largedelta-midsigma-PR.dat};
			\addlegendentry{Gordon: GD}
			
			\addplot+[name path=A-AM,CadetBlue!20, no markers] table[x=iter,y=lowerAM] {sections/numerical-illustrations/data/largedelta-midsigma-PR.dat};
			\addplot+[name path=B-AM,CadetBlue!20, no markers] table[x=iter,y=upperAM] {sections/numerical-illustrations/data/largedelta-midsigma-PR.dat};
			
			\addplot[CadetBlue!20] fill between[of=A-AM and B-AM];
			
			\addplot+[name path=A-GD,RoyalBlue!20, no markers] table[x=iter,y=lowerGD] {sections/numerical-illustrations/data/largedelta-midsigma-PR.dat};
			\addplot+[name path=B-GD,RoyalBlue!20, no markers] table[x=iter,y=upperGD] {sections/numerical-illustrations/data/largedelta-midsigma-PR.dat};
			
			\addplot[RoyalBlue!20] fill between[of=A-GD and B-GD];
		\end{axis}
	\end{tikzpicture}

%% file: sections/numerical-illustrations/figs/pop-converge-empirics-no-converge.tex
	\begin{tikzpicture}
		\begin{axis}[
			width=0.5\textwidth,
			height=0.5\textwidth,
			xlabel={Iteration},
			ylabel={$\| \bt_t - \bt^* \|_2$},
			ymode=log,
xmin=0, xmax=140,
ymin=0.0000001, ymax=1,
legend style={font=\tiny, fill=none, at={(1,0.6)},anchor=east},
ymajorgrids=true,
grid style=grid,
			]

			\addplot[
			color=RoyalBlue,
			mark=triangle,
			mark size=1.5pt,
			mark repeat = 5,
			mark phase = 0
			]
			table[x=iter,y=avgGD]{sections/numerical-illustrations/data/pop_convergence_empirics_no_convergence.dat};
			\addlegendentry{Empirical: GD}
			
			\addplot[
			color=Black,
			mark=*,
			mark size=1.5pt,
			mark repeat = 5,
			mark phase = 0
			]
			table[x=iter,y=pop] {sections/numerical-illustrations/data/pop_convergence_empirics_no_convergence.dat};
			\addlegendentry{Population}
			
%
			\addplot[
			color=Salmon,
			mark=*,
			mark size=1.5pt,
			mark repeat = 5,
			mark phase = 0
			]
			table[x=iter,y=gorGD] {sections/numerical-illustrations/data/pop_convergence_empirics_no_convergence.dat};
			\addlegendentry{Gordon: GD}

			\addplot+[name path=A-GD,RoyalBlue!20, no markers] table[x=iter,y=lowerGD] {sections/numerical-illustrations/data/pop_convergence_empirics_no_convergence.dat};
			\addplot+[name path=B-GD,RoyalBlue!20, no markers] table[x=iter,y=upperGD] {sections/numerical-illustrations/data/pop_convergence_empirics_no_convergence.dat};
			
			\addplot[RoyalBlue!20] fill between[of=A-GD and B-GD];
		\end{axis}
	\end{tikzpicture}

%% file: sections/numerical-illustrations/figs/global-MLR-AM.tex
	\begin{tikzpicture}
		\begin{axis}[
			width=\textwidth,
			height=\textwidth,
			xlabel={Iteration},
			ylabel={$\angle(\bt_t, \btstar)$},
			ymode=log,
			xmin=0, xmax=12,
			legend style={font=\tiny,fill=none},
			legend pos=north east,
			ymajorgrids=true,
			grid style=grid,
			]
			
			\addplot[
			color=CadetBlue,
			mark=triangle,
			mark size=1.5pt   
			]
			table[x=iter,y=avgAM] {sections/numerical-illustrations/data/global-MLR.dat};
			\legend{Empirical: AM}
			
			\addplot[
			color=Maroon,
			mark=*,
			mark size=1.5pt   
			]
			table[x=iter,y=gorAM] {sections/numerical-illustrations/data/global-MLR.dat};
			\addlegendentry{Gordon: AM}
			
			\addplot+[name path=A-AM,CadetBlue!20, no markers] table[x=iter,y=lowerAM] {sections/numerical-illustrations/data/global-MLR.dat};
			\addplot+[name path=B-AM,CadetBlue!20, no markers] table[x=iter,y=upperAM] {sections/numerical-illustrations/data/global-MLR.dat};
			
			\addplot[CadetBlue!20] fill between[of=A-AM and B-AM];
			
		\end{axis}
	\end{tikzpicture}

%% file: sections/numerical-illustrations/figs/global-MLR-GD.tex
	\begin{tikzpicture}
		\begin{axis}[
			width=\textwidth,
			height=\textwidth,
			xlabel={Iteration},
			ylabel={$\angle(\bt_t, \btstar)$},
			ymode=log,
			xmin=0, xmax=12,
			legend style={font=\tiny,fill=none},
			legend pos=north east,
			ymajorgrids=true,
			grid style=grid,
			]

			\addplot[
			color=RoyalBlue,
			mark=triangle,
			mark size=1.5pt   
			]
			table[x=iter,y=avgGD] {sections/numerical-illustrations/data/global-MLR.dat};
			\addlegendentry{Empirical: GD}

			\addplot[
			color=Salmon,
			mark=*,
			mark size=1.5pt   
			]
			table[x=iter,y=gorGD] {sections/numerical-illustrations/data/global-MLR.dat};
			\addlegendentry{Gordon: GD}

			\addplot+[name path=A-GD,RoyalBlue!20, no markers] table[x=iter,y=lowerGD] {sections/numerical-illustrations/data/global-MLR.dat};
			\addplot+[name path=B-GD,RoyalBlue!20, no markers] table[x=iter,y=upperGD] {sections/numerical-illustrations/data/global-MLR.dat};
			
			\addplot[RoyalBlue!20] fill between[of=A-GD and B-GD];
		\end{axis}
	\end{tikzpicture}

%% file: sections/numerical-illustrations/figs/low-noise-MLR.tex
	\begin{tikzpicture}
		\begin{axis}[
			width=\textwidth,
			height=\textwidth,
			xlabel={Iteration},
			ylabel={$\angle(\btstar, \bt_t)$},
			ymode=log,
			xmin=0, xmax=12,
			legend style={font=\tiny,fill=none},
			legend pos=north east,
			ymajorgrids=true,
			grid style=grid,
			]
			
			\addplot[
			color=CadetBlue,
			mark=triangle,
			mark size=1.5pt   
			]
			table[x=iter,y=avgAM] {sections/numerical-illustrations/data/mid-noise-MLR.dat};
			\legend{Empirical: AM}
			
			\addplot[
			color=RoyalBlue,
			mark=triangle,
			mark size=1.5pt   
			]
			table[x=iter,y=avgGD] {sections/numerical-illustrations/data/mid-noise-MLR.dat};
			\addlegendentry{Empirical: GD}
			
			\addplot[
			color=Maroon,
			mark=*,
			mark size=1.5pt   
			]
			table[x=iter,y=gorAM] {sections/numerical-illustrations/data/mid-noise-MLR.dat};
			\addlegendentry{Gordon: AM}
			
			\addplot[
			color=Salmon,
			mark=*,
			mark size=1.5pt   
			]
			table[x=iter,y=gorGD] {sections/numerical-illustrations/data/mid-noise-MLR.dat};
			\addlegendentry{Gordon: GD}
			
			\addplot+[name path=A-AM,CadetBlue!20, no markers] table[x=iter,y=lowerAM] {sections/numerical-illustrations/data/mid-noise-MLR.dat};
			\addplot+[name path=B-AM,CadetBlue!20, no markers] table[x=iter,y=upperAM] {sections/numerical-illustrations/data/mid-noise-MLR.dat};
			
			\addplot[CadetBlue!20] fill between[of=A-AM and B-AM];
			
			\addplot+[name path=A-GD,RoyalBlue!20, no markers] table[x=iter,y=lowerGD] {sections/numerical-illustrations/data/mid-noise-MLR.dat};
			\addplot+[name path=B-GD,RoyalBlue!20, no markers] table[x=iter,y=upperGD] {sections/numerical-illustrations/data/mid-noise-MLR.dat};
			
			\addplot[RoyalBlue!20] fill between[of=A-GD and B-GD];
		\end{axis}
	\end{tikzpicture}

%% file: sections/numerical-illustrations/figs/largedelta-midsigma-MLR.tex
	\begin{tikzpicture}
		\begin{axis}[
			width=\textwidth,
			height=\textwidth,
			xlabel={Iteration},
			ylabel={$\angle(\btstar, \bt_t)$},
			ymode=log,
			xmin=0, xmax=12,
			legend style={font=\tiny, fill=none},
			legend pos=north east,
			ymajorgrids=true,
			grid style=grid,
			]
			
			\addplot[
			color=CadetBlue,
			mark=triangle,
			mark size=1.5pt   
			]
			table[x=iter,y=avgAM] {sections/numerical-illustrations/data/largedelta-midsigma-MLR.dat};
			\legend{Empirical: AM}
			
			\addplot[
			color=RoyalBlue,
			mark=triangle,
			mark size=1.5pt   
			]
			table[x=iter,y=avgGD] {sections/numerical-illustrations/data/largedelta-midsigma-MLR.dat};
			\addlegendentry{Empirical: GD}
			
			\addplot[
			color=Maroon,
			mark=*,
			mark size=1.5pt   
			]
			table[x=iter,y=gorAM] {sections/numerical-illustrations/data/largedelta-midsigma-MLR.dat};
			\addlegendentry{Gordon: AM}
			
			\addplot[
			color=Salmon,
			mark=*,
			mark size=1.5pt   
			]
			table[x=iter,y=gorGD] {sections/numerical-illustrations/data/largedelta-midsigma-MLR.dat};
			\addlegendentry{Gordon: GD}
			
			\addplot+[name path=A-AM,CadetBlue!20, no markers] table[x=iter,y=lowerAM] {sections/numerical-illustrations/data/largedelta-midsigma-MLR.dat};
			\addplot+[name path=B-AM,CadetBlue!20, no markers] table[x=iter,y=upperAM] {sections/numerical-illustrations/data/largedelta-midsigma-MLR.dat};
			
			\addplot[CadetBlue!20] fill between[of=A-AM and B-AM];
			
			\addplot+[name path=A-GD,RoyalBlue!20, no markers] table[x=iter,y=lowerGD] {sections/numerical-illustrations/data/largedelta-midsigma-MLR.dat};
			\addplot+[name path=B-GD,RoyalBlue!20, no markers] table[x=iter,y=upperGD] {sections/numerical-illustrations/data/largedelta-midsigma-MLR.dat};
			
			\addplot[RoyalBlue!20] fill between[of=A-GD and B-GD];
		\end{axis}
	\end{tikzpicture}

%% file: sections/discussion/discussion.tex

\section{Discussion} \label{sec:discussion}

We presented a recipe for deriving accurate deterministic predictions for the behavior of iterative algorithms in nonconvex Gaussian regression models, which applies provided each iteration can be written as a convex optimization problem satisfying mild decomposability conditions. Rather than decouple the deterministic component of these analyses from its random counterpart by passing to the infinite-sample population limit---which is the most prevalent program in the literature---we used duality and Gaussian comparison theorems to obtain our deterministic Gordon state evolution update. We presented several consequences for both higher-order and first-order algorithms applied to the problems of phase retrieval and mixtures of regressions. These results are in themselves novel, but the key takeaway is our sharp characterization of convergence behavior, which we hope will enable a rigorous comparison between algorithms in other related problems. We conclude by listing a few open questions.

We begin with two technical open questions. We showed that our deterministic predictions of the perpendicular component $\beta$ were within $\order(n^{-1/4})$ of their empirical counterparts. We were able to sharpen this rate to $\ordertil(n^{-1/2})$ for the parallel $\alpha$ component, and conjecture that a similar improvement can be carried out for the $\beta$ component. As a second technical question, we highlight the condition $\noisestd \lesssim 1$ present in our results for mixture of regression models. We conjecture that this condition can be weakened to $\noisestd \lesssim \sqrt{\kappa}$ while preserving the same qualitative behavior of the theorem (i.e. linear angular convergence), but establishing this rigorously is an interesting open problem.

The next set of open questions is broader. Note that our analysis---which relied on Gaussianity of the data independent of the current iterate---required fresh observations at each iteration, and to that end, we used a sample-splitting device to partition the data into disjoint batches. While this is a reasonable method to obtain a practical algorithm---indeed, all the algorithms we analyzed converge very fast, so that at most a logarithmic number of batches suffices---it is more common to run these algorithms without sample splitting. The leave-one-out technique~\citep{ma2020implicit,chen2019gradient} has emerged as a powerful analysis framework for the case without sample-splitting, and it is an interesting open question to what extent this can be combined with our Gordon recipe. 
Even more broadly, there is the question of building an analogous theory under weaker distributional assumptions on the data; indeed, some iterative (higher-order) algorithms considered in the literature are known to converge under weaker assumptions~\citep[e.g.,][]{duchi2019solving,ghosh2020max}. While universality theorems (broadly construed) have been proved in related settings~\citep{bayati2015universality,oymak2018universality,panahi2017universal,el2018impact,abbasi2019universality,paquette2020halting}, 
do similar insights apply here? Can we produce an accurate deterministic prediction if the data is no longer i.i.d., akin to the population update in such settings~\citep{yang2017statistical}? These are interesting and important questions for future work.

Finally, there is the question of broadening the scope of problems to which our analysis applies, and we provide two examples along these lines. First, one could consider ``weak" signal-to-noise regimes in the models that we considered. These regimes have been the subject of recent work~\citep{dwivedi2020singularity,wu2019randomly,ho2020instability}, and it is known that the optimal statistical rates of convergence are different from those in the strong signal-to-noise regimes that we consider in this paper. What are sharp rates of convergence of optimization algorithms in these settings?
Second, and more importantly, phase retrieval and mixtures of regressions are just two models to which our framework applies. There are several other models and algorithms that can be analyzed with the Gordon state evolution machinery to sharply characterize (possibly nonstandard) convergence behavior.

%% file: sections/gordon/gordon-alt/gordon-alt2.tex

In this section, we prove part (a) of both Theorem~\ref{thm:one_step_main_HO} and Theorem~\ref{thm:one_step_main_FO}. The structure of the proof follows the recipe sketched in Section~\ref{sec:heuristic}. We proceed by carrying out steps 1--3 of the recipe for a broader class of algorithms (captured by one-step updates satisfying Assumption~\ref{ass:loss} to follow), and derive a general Proposition~\ref{prop:first-three-steps}. With this proposition in hand, we then carry out step 4 of the recipe separately for higher-order methods
to prove Theorem~\ref{thm:one_step_main_HO}(a) 
and for first-order methods 
to prove Theorem~\ref{thm:one_step_main_FO}(a).

Throughout this section, we let $\bt^{\sharp}$ denote the ``current'' iterate of the algorithm, with $(\parcomp^{\sharp}, \perpcomp^{\sharp}) = (\parcomp(\btsharp), \perpcomp(\btsharp))$. This frees up the tuple $(\bt, \parcomp, \perpcomp)$ to denote decision variables that will be used throughout the proof.  In addition, as in the heuristic derivation in Section~\ref{subsec:heuristic-second-order}, it is useful in the proof to track a three dimensional state evolution $(\parcomp, \perpone, \perptwo)$, where
\begin{align}
	\label{eq:gordon-3D-state}
\parcomp = \langle \bt, \thetastar \rangle, \qquad \perpone = \frac{\langle \bt, \proj_{\thetastar}^{\perp}\btsharp\rangle}{\| \proj_{\thetastar}^{\perp}\btsharp \|_2}, \qquad \text{and}\qquad \perptwo = \| \proj_{S_\#}^{\perp} \bt \|_2, 
\end{align}
where $S_{\#} = \mathsf{span}(\thetastar, \btsharp)$ and $\proj^\perp_{S_\#}$ is the projection matrix onto the orthogonal complement of this subspace.
Finally, define the independent random variables
\begin{align} \label{eq:zone-ztwo}
\bz_1  = \bX\btstar \qquad \text{ and } \qquad \bz_2 = \frac{\bX\proj_{\btstar}^{\perp} \btsharp}{\| \bX\proj_{\btstar}^{\perp} \btsharp \|_2},
\end{align}
noting that both $\bz_1, \bz_2 \sim \NORMAL(0, \bI_n)$ since $\| \btstar \|_2 = 1$. 
We are now ready to rigorously implement each step of the recipe.

\subsection{General result from steps 1--3 of recipe}
Our general result is derived by implementing steps 1--3 in a setting involving a general decomposability assumption on the one-step loss function $\mathcal{L}$~\eqref{eq:gordon_one_step}. 

\subsubsection{Implementing step 1: One-step convex optimization} 

Our general assumption takes the following form: 
\begin{assumption}[Decomposability and convexity of loss]\label{ass:loss}
	Consider a fixed vector $\bt^{\sharp} \in \mathbb{R}^d$ and a Gaussian random matrix $\bX \in \mathbb{R}^{n \times d}$ and assume $\by$ is generated, given $\bX$ and $\btstar$, according to the generative model~\eqref{eq:model}.  Then one step of the iterative algorithm run from $\btsharp$ can be written in the form~\eqref{eq:opt_gen_intro}, where the loss
	\[
	\Lc(\bt):={\mathcal{L}(\bt; \bt^{\sharp}, \bd{X}, \bd{y})}
	\]
	satisfies the following properties:
	\begin{enumerate}[label=(\alph*)]
		\item There is a pair of functions $g: \mathbb{R}^n \times \mathbb{R}^n \rightarrow \mathbb{R}$ and $h: \mathbb{R}^d \times \mathbb{R}^d \rightarrow \mathbb{R}$, and a (random) function
		\begin{align}
			\label{eq:Fc}
			F : \mathbb{R}^{n} \times \mathbb{R}^{d} &\rightarrow \mathbb{R} \nonumber\\
			F(\bu, \bt) &\mapsto g(\bu, \bX \bt^{\sharp}; \by) + h(\bt, \bt^{\sharp}),
		\end{align}	
		such that
		\begin{align*}
		\Lc(\bt) = F(\bX \bt, \bt) \qquad \text{ for all } \bt \in \real^d.
		\end{align*}
		
		\item The functions $g$ and $h$ are convex in their first arguments.  Moreover, the function $g$ and thus $F$ are $C_L/\sqrt{n}$-Lipschitz in their first argument.
		
		\item The function $h$ depends on $\bt$ only through its lower dimensional projections.  That is, there exists another function
		\begin{align*}
			\hscal: \mathbb{R}^3 \times \mathbb{R}^d &\rightarrow \mathbb{R}\\
			(\alpha, \mu, \nu), \bt^{\sharp} & \mapsto \hscal((\alpha, \mu, \nu), \bt^{\sharp}),
		\end{align*}
		such that 
		\[
		h(\bt, \btsharp) = \hscal\biggl(\langle \bt, \btstar\rangle, \frac{\langle \bt, \proj_{\btstar}^{\perp} \btsharp\rangle}{\| \proj_{\btstar}^{\perp} \btsharp\|_2}, \| \proj_{\mathsf{span}(\btstar, \btsharp)}^{\perp} \bt \|_2, \btsharp
		\biggr)
		\]
		
		\item $\Lc(\bt)$ is coercive.  That is, $\Lc(\bt) \rightarrow \infty$ whenever $ \| \bt \|_2 \rightarrow \infty$.
	\end{enumerate}
\end{assumption}
Note that specifying
\begin{align*}
h(\bt, \bt^{\sharp}) = 0 \qquad  \text{ and } \qquad g(\bu, \bX \bt^{\sharp}; \by) = \frac{1}{\sqrt{n}}\| \omega(\bX \bt^{\sharp}, \by) - \bu  \|_2
\end{align*} 
recovers the higher-order loss functions~\eqref{eq:sec_order_gen}, whereas specifying 
\begin{align*}
h(\bt, \btsharp) = 1/2 \cdot \| \bt \|_2^2 - \langle \bt, \btsharp \rangle \qquad  \text{ and } \qquad g(\bu, \bX\btsharp; \by) = 2 \eta/n \cdot \langle \bu, \omega(\bX \bt^{\sharp}, \by) \rangle
\end{align*}
recovers the first-order loss functions~\eqref{eq:GD_gen}.  The remaining properties (b)-(d) of the assumption can be straightforwardly verified for these two choices. Thus, Assumption~\ref{ass:loss} captures both the special cases corresponding to Theorems~\ref{thm:one_step_main_HO} and~\ref{thm:one_step_main_FO}. Having written one step of the iterative algorithm of interest as a minimization of a convex loss, we are now ready to proceed to step 2 of the recipe.

\subsubsection{Implementing step 2: The auxiliary optimization problem} 
Next, we state a formal definition of the auxiliary loss function $\mathfrak{L}_n$.
\begin{definition}[Auxiliary loss]
\label{def:aux-gordon-loss}
Let $\bg \sim \NORMAL(0, \bI_n)$ and $\bh \sim \NORMAL(0, \bI_d)$ denote independent random vectors drawn independently of the pair $(\bX, \by)$, let $\btsharp \in \mathbb{R}^d$, and define the subspace $S_{\sharp} = \mathsf{span}(\btstar, \btsharp)$.  Further, let $\mathcal{L}$ denote a loss function which satisfies Assumption~\ref{ass:loss} for functions $g$ and $h$.  Then, given a positive scalar $r$ 
define the auxiliary loss function
\begin{align*}
	\mathfrak{L}_n(\bt, \bu; r) := \max_{\bv \in \mathbb{B}_2(r)}\;\frac{1}{\sqrt{n}} \| \bv \|_2 \langle \proj_{S_{\sharp}}^{\perp} \bt, \bh \rangle + h(\bt, \bt^{\sharp}) + g(\bu, \bX \bt^{\sharp}; \by)  + \frac{1}{\sqrt{n}}\langle \bv, \bX \proj_{S_{\sharp}} \bt  + \| \proj_{S_{\sharp}}^{\perp} \bt \|_2 \bg - \bu\rangle.
\end{align*}
\end{definition}
The following lemma shows that our original optimization problem over the loss function $\mathcal{L}$ is essentially equivalent to an auxiliary optimization problem involving the loss $\mathfrak{L}_n$. 
\begin{lemma}
	\label{lem:bilinear}
	Let $\btsharp \in \mathbb{R}^d$ and suppose that the loss function $\Lc(\bt) = \Lc(\bt; \btsharp, \bX, \by)$ satisfies Assumption~\ref{ass:loss} (and recall the Lipschitz constant $C_L$ therein) and associate with it the auxiliary loss $\mathfrak{L}_n$.  Let $D \subseteq \mathbb{B}_2(R)$ denote a closed subset for some positive constant $R$.  Then there exists a positive constant $C_1 \geq 6R$, depending only on $R$, 
	such that for any scalar $r \geq C_L$ and scalar $t \in \mathbb{R}$,
	\[
	\Pro\Bigl\{ \min_{\bt \in D} \Lc(\bt) \leq t\Bigr\} \leq 2\Pro\Bigl\{\min_{\bt \in D, \bu \in \mathbb{B}_2(C_1\sqrt{n}) } \mathfrak{L}_n(\bt, \bu; r) \leq t\Bigr\} + 2e^{-2n}.
	\]
	If, in addition, $D$ is  convex, then
	\[
	\Pro\Bigl\{\min_{\bt \in D} \mathcal{L}(\bt) \geq t\Bigr\} \leq 2\Pro\Bigl\{\min_{\bt \in D, \bu \in \mathbb{B}_2(C_1\sqrt{n}) } \mathfrak{L}_n(\bt, \bu; r) \geq t\Bigr\} + 2e^{-2n}.
	\]
\end{lemma}
Given that the minimization over $\mathcal{L}$ can be approximately written as a minimization over an auxiliary loss, we are now ready to proceed to step 3.

\subsubsection{Implementing step 3: Scalarization} 
Next, we define the scalarized auxiliary loss. Recall the definition of the convex conjugate of a function $g: \real^d \to \real$, given by $g^*(\bx) = \sup_{\bx' \in \real^d}\; \inprod{\bx'}{\bx} - g(\bx')$.
\begin{definition}[Scalarized auxiliary loss]
	Let $\bg \sim \NORMAL(0, \bI_n)$, $\bh \sim \NORMAL(0, \bI_d)$, $\bz_1 \sim \NORMAL(0, \bI_n)$, and $\bz_2 \sim \NORMAL(0, \bI_n)$ denote mutually independent random vectors, with the pair $(\bz_1, \bz_2)$ chosen according to equation~\eqref{eq:zone-ztwo}. Let $\btsharp \in \mathbb{R}^d$ and define the subspace $S_{\sharp} = \mathsf{span}(\btstar, \btsharp)$.  Further, let $\mathcal{L}$ denote a loss function which satisfies Assumption~\ref{ass:loss} for functions $g$ and $h$. 
	Then associate with it the \emph{scalarized auxiliary loss}
	\label{def:scalarized-ao}
	\[
	\widebar{L}_n(\alpha, \mu, \nu; \btsharp) := \max_{\bv \in \R^n} \hscal(\alpha, \mu, \nu, \btsharp) - g^{*}(\bv, \bX \btsharp; \by) - \nu \| \proj_{S_{\sharp}}^{\perp}\bh\|_2 \| \bv \|_2 + \langle \nu \bg + \alpha \bz_1 + \mu \bz_2, \bv \rangle,
	\]
	where $g^*$ denotes the convex conjugate of the function $g$ and $\hscal$ is as in part (c) of Assumption~\ref{ass:loss}.
\end{definition}
Our next lemma implements step 3, scalarizing the auxiliary loss $\mathfrak{L}_n$.  Before stating the lemma, we require the definition of a scalarized set and an amenable set.
\begin{definition}[Scalarized set]
	\label{def:scalarized-set}
Let $\btstar \in \mathbb{R}^d$ denote the ground truth, $\btsharp \in \mathbb{R}^d$, and $S_{\sharp}$ the subspace $S_{\sharp} = \mathsf{span}(\btstar, \btsharp)$.  For any subset $D \subseteq \mathbb{R}^d$, define the scalarized set
	\[
\Pc(D) := \biggl\{(\alpha, \mu, \nu) \in \mathbb{R}^3: \bt \in D \text{ and } \alpha = \langle \bt, \btstar \rangle, \mu = \frac{\langle \bt, \proj_{\btstar}^{\perp}\bt^{\sharp}\rangle}{\| \proj_{\btstar}^{\perp} \bt^{\sharp}\|_2}, \nu = \| \proj_{S_{\sharp}}^{\perp} \bt \|_2\biggr\}.
\]
\end{definition}

\begin{definition}[Amenable set]
	\label{def:amenable-set}
Let $\btstar \in \mathbb{R}^d$ denote the ground truth, $\btsharp \in \mathbb{R}^d$, and $S_{\sharp}$ the subspace $S_{\sharp} = \mathsf{span}(\btstar, \btsharp)$.
A subset $D \subseteq \mathbb{R}^d$ is \emph{amenable} with respect to the subspace $S_{\sharp}$ if the set $D\cap {S_{\sharp}}^\perp$ is rotationally invariant, i.e. for all unit vectors $\|\bv\|_2=1$ such that $\bv\in{S_{\sharp}}^\perp$, there exists $\bt\in D$ such that ${\proj_{S_{\sharp}}^{\perp} \bt}/\|{\proj_{S_{\sharp}}^{\perp} \bt}\|=\bv$.
\end{definition}

With these definitions in hand, we have the following lemma, whose proof we provide in Subsection~\ref{subsec:proof-scalarize-ao}.
\begin{lemma}
	\label{lem:scalarize-ao}
	Let $\btsharp \in \mathbb{R}^d$ and $S_{\sharp} = \mathsf{span}(\btstar, \btsharp)$.  Suppose that the loss function $\Lc$ satisfies Assumption~\ref{ass:loss} and associate with it the auxiliary loss $\mathfrak{L}_n$ as well as the scalarized auxiliary loss $\widebar{L}_n$.  Further, let $R$ be a positive constant and suppose that the subset $D \subseteq \mathbb{B}_2(R)$ is amenable with respect to the subspace $S_{\sharp}$.  Then, with $C_1$ as in Lemma~\ref{lem:bilinear}, for all $r \geq C_L$, we have the sandwich relation
	\[
	\min_{ \substack{\bt \in D \\ \bu \in \mathbb{B}_2(C_1\sqrt{n})}}\; \mathfrak{L}_n(\bt, \bu;r) \leq \min_{(\alpha, \mu, \nu) \in \Pc(D)} \widebar{L}_n(\alpha, \mu, \nu; \btsharp)\leq \min_{ \substack{\bt \in D \\ \bu \in \mathbb{B}_2(C_1\sqrt{n})}}\; \mathfrak{L}_n(\bt, \bu;r) 
	+ \frac{3 C_L^2 C_1}{r},
	\]
	with probability at least $1 - 8e^{-n/2}$.
\end{lemma}

\subsubsection{Putting together steps 1--3}

We are now in a position to put the pieces together and prove a formal equivalence between the original minimization problem over the loss $\mathcal{L}$ and a low-dimensional minimization problem over the loss $\widebar{L}_n$.
\begin{proposition} \label{prop:first-three-steps}
Let $\btsharp \in \mathbb{R}^d$ and $S_{\sharp} = \mathsf{span}(\btstar, \btsharp)$.  Suppose that the loss function $\Lc$ satisfies Assumption~\ref{ass:loss} and associate with it the scalarized auxiliary loss $\widebar{L}_n$.  Let $R$ be a positive constant and suppose that the subset $D \subseteq \mathbb{B}_2(R)$ is amenable with respect to the subspace $S_{\sharp}$.  Then, there exists a positive constant $C_1$, depending only on $R$, such that for each triple of scalars $r \geq C_L, \mathsf{L} \in \mathbb{R}$, and $\epsilon' > 0$, we have
\begin{align*}
	\Pro\Bigl\{\argmin_{\bt \in \mathbb{B}_2(R)} \mathcal{L}(\bt) \in D \Bigr\} &\leq 2 \Pro\Bigl\{\min_{(\alpha, \mu, \nu) \in \Pc( {D} ) }
	\widebar{L}_n(\alpha, \mu, \nu) \leq \mathsf{L} - \frac{3 C_L^2 C_1}{r} + 2\epsilon' \Bigr\} \\
	&\quad  + 2 \Pro\Bigl\{\min_{(\alpha, \mu, \nu) \in \Pc(\mathbb{B}_2(R))} \widebar{L}_n(\alpha, \mu, \nu) > \mathsf{L} + \epsilon' \Bigr\} + 20 e^{-n/2}.
\end{align*}
\end{proposition}

\begin{proof}
Applying the law of total probability, we obtain for any $\mathsf{L}$ and any $\epsilon'>0$, the chain of inequalities
\begin{align}
	\Pro\Bigl\{\argmin_{\bt \in \mathbb{B}_2(R)} \mathcal{L}(\bt) \in D \Bigr\} &\leq \Pro\Bigl\{\min_{\bt \in D \cap \mathbb{B}_2(R)} \mathcal{L}(\bt) \leq \min_{\bt \in \mathbb{B}_2(R)} \mathcal{L}(\bt) + \epsilon' \Bigr\}\nonumber\\
	&\leq \Pro\Bigl\{\min_{\bt \in D} \mathcal{L}(\bt) \leq \mathsf{L} + 2\epsilon'\Bigr\} + \Pro\Bigl\{\min_{\bt \in \mathbb{B}_2(R)} \mathcal{L}(\bt) > \mathsf{L} + \epsilon' \Bigr\}, \label{ineq:step1-prop-gor}
\end{align}
where we note that in the second inequality we have used the fact that $D \subseteq \mathbb{B}_2(R)$.  Now, we apply Lemma~\ref{lem:bilinear} to obtain the pair of inequalities
\begin{subequations}\label{ineq:step2-prop-gor}
\begin{align}
	\Pro\Bigl\{\min_{\bt \in D} \mathcal{L}(\bt) \leq \mathsf{L} + 2\epsilon' \Bigr\} &\leq 2\Pro\Biggl\{\min_{\substack{\bt \in D,\\ \bu \in \mathbb{B}_2(C_1\sqrt{n})}}\mathfrak{L}_n(\bt, \bu;r) \leq \mathsf{L} + 2\epsilon' \Biggr\} + 2e^{-2n},\\
	\Pro\Bigl\{\min_{\bt \in \mathbb{B}_2(R)} \mathcal{L}(\bt) > \mathsf{L} + \epsilon' \Bigr\} &\leq 2\Pro\Biggl\{\min_{\substack{\bt \in  \mathbb{B}_2(R),\\ \bu \in \mathbb{B}_2(C_1\sqrt{n})}}\mathfrak{L}_n(\bt, \bu;r) > \mathsf{L} + \epsilon' \Biggr\} + 2e^{-2n}.
\end{align}
\end{subequations}
Next, we note that $D$ is an amenable set with respect to $S_{\sharp}$ (as in Definition~\ref{def:amenable-set}).  Thus, we apply Lemma~\ref{lem:scalarize-ao} to further obtain the pair of inequalities
\begin{subequations}\label{ineq:step3-prop-gor}
	\begin{align}
		\Pro\Biggl\{\min_{\substack{\bt \in D ,\\ \bu \in \mathbb{B}_2(C_1\sqrt{n})}}\mathfrak{L}_n(\bt, \bu;r) \leq \mathsf{L} + 2\epsilon' \Biggr\} &\leq \Pro\Bigl\{\min_{(\alpha, \mu, \nu) \in 
		\Pc(D)
		}\widebar{L}_n(\alpha, \mu, \nu) \leq \mathsf{L} - \frac{3 C_L^2 C_1}{r} + 2\epsilon' \Bigr\} + 8e^{-n/2},\\
		\Pro\Biggl\{\min_{\substack{\bt \in  \mathbb{B}_2(R),\\ \bu \in \mathbb{B}_2(C_1\sqrt{n})}}\mathfrak{L}_n(\bt, \bu;r) > \mathsf{L} + \epsilon' \Biggr\} &\leq \Pro\Bigl\{\min_{(\alpha, \mu, \nu) \in \Pc(\mathbb{B}_2(R))} \widebar{L}_n(\alpha, \mu, \nu) > \mathsf{L}  + \epsilon' \Bigr\}  + 8e^{-n/2}. 
	\end{align}
\end{subequations}
Combining the inequalities~\eqref{ineq:step1-prop-gor}--\eqref{ineq:step3-prop-gor} yields the desired conclusion. 
\end{proof}
Having established steps 1--3 of the recipe under the general Assumption~\ref{ass:loss} on the one-step loss function, we now carry out step 4 of the recipe individually for each theorem. For clarity, we include a schematic diagram of the various ingredients in Figure~\ref{fig:gordon-schematic}.  Recall from the recipe described earlier that the Gordon state evolution update is obtained as the minimizer of the deterministic loss $\widebar{L}$, which is in turn obtained from $\widebar{L}_n$ in the limit $n \to \infty$. We prove each of the two theorems below without making this equivalence explicit, but the connection is evident from the proofs of Lemmas~\ref{lem:ao-analysis} and~\ref{lem:ao-analysis-FO} in the appendix.

\begin{figure}[!ht]
	\centering
	\input{figs/gordon-proof-fig.tex}
	\caption{Schematic diagram of the proof of Theorems~\ref{thm:one_step_main_HO} and~\ref{thm:one_step_main_FO}, drawn for $\bt \in \mathbb{R}^3$.  For shorthand, we consider the state $\bxi = (\alpha, \mu, \nu)$.  The quantity $\Tc_n(\btsharp)$ denotes the minimizer of the loss $\Lc$ and the state $\bxi_n = (\alpha_n, \mu_n, \nu_n)$ denotes the minimizer of the scalarized auxiliary loss $\widebar{L}_n$. With $\widebar{L} = \lim_{n \to \infty} \widebar{L}_n$ denoting the deterministic scalarized auxiliary loss in the limit, the state $\bxi^{\gor} = (\alpha^\gor, \mu^\gor, \nu^\gor)$ denotes the minimizer of $\widebar{L}$.  Proposition~\ref{prop:first-three-steps} utilizes the CGMT to connect the minima of the original loss $\Lc$ over any amenable set (which includes the entire set $\mathbb{R}^3$ as well as the set $\mathbb{B}_{\infty}(\bxi^{\gor}; \epsilon)^{c}$) to minima over the simpler empirical loss $\widebar{L}_n$.  The growth conditions and concentration properties given by Lemmas~\ref{lem:ao-analysis} and~\ref{lem:ao-analysis-FO} imply that $\widebar{L}(\bxi^{\gor})$ and $\widebar{L}_n(\bxi_n)$ lie in the purple shaded region as well as the inclusion $\bxi_n \in \mathbb{B}_{\infty}(\bxi^{\gor}; \epsilon)$.  Used in conjunction with Proposition~\ref{prop:first-three-steps}, this shows that the minimizer of the original loss $\Lc$ also satisfies the inclusion $\Tc_n(\btsharp) \in \mathbb{B}_{\infty}(\bxi^{\gor}; \epsilon)$.}
	\label{fig:gordon-schematic}
\end{figure}
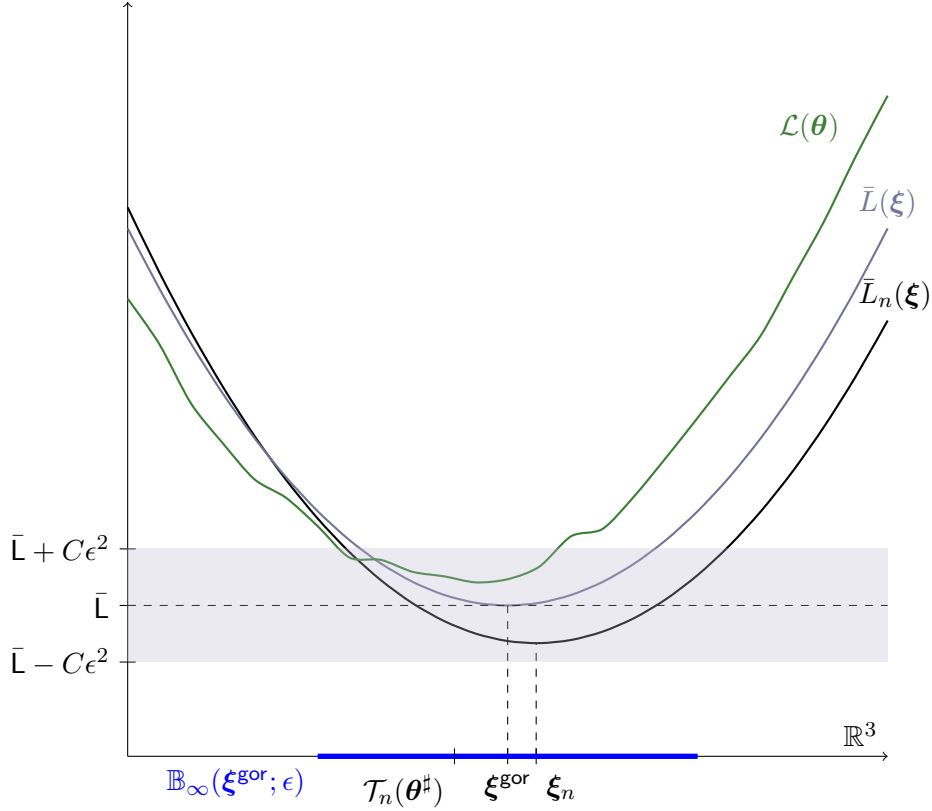

\subsection{Proof of Theorem~\ref{thm:one_step_main_HO}(a)}
First, we define the expanded (i.e., three dimensional) Gordon state evolution update for higher-order methods.

\begin{definition}[Expanded Gordon state evolution update: Higher-order methods] \label{def:expanded-gordon-HO}
Recall the model~\eqref{eq:model}, and let $Q$ denote a random variable drawn from the latent variable distribution $\mathbb{Q}$. Suppose that the loss function $\Lc$ takes the form~\eqref{eq:sec_order_gen}. Let $(Z_1, Z_2, Z_3)$ denote a triplet of independent standard Gaussian random variables and let $\alpha^{\sharp} \in \mathbb{R}$ and $\beta^{\sharp} \in \mathbb{R}_{\geq 0}$ denote arbitrary scalars.  Let
\[
\Omega = \omega\big(\parcomp^{\sharp} Z_1 +\perpcomp^{\sharp} Z_2\,,\,f(Z_1; Q)+\noisestd Z_3 \big).
\]
Then define the scalars
\[
\alpha^{\gor} = \EE [ Z_1 \Omega], \qquad \mu^{\gor} = \EE[Z_2 \Omega], \qquad \text{ and } \qquad \nu^{\gor} = \sqrt{\frac{\EE[\Omega^2] - (\EE[Z_1 \Omega])^2 - (\EE[Z_2 \Omega])^2}{\kappa - 1}}.
\]
\end{definition}

With this definition in hand, note that in order to prove Theorem~\ref{thm:one_step_main_HO}(a) it suffices to show that for a parameter $c_K$ depending only on the pair $(K_1, K_2)$ in Assumptions~\ref{ass:omega} and~\ref{ass:omega-lb}, we have
\begin{align} \label{eq:equiv-thm1a}
\Pro \left\{ \max \left( \left| \frac{\inprod{\mathcal{T}_n(\btsharp)}{\proj^{\perp}_{\thetastar} \btsharp}}{ \| \proj^{\perp}_{\thetastar} \btsharp \|_2} - \perpone^{\gor} \right|, \left| \| \proj^{\perp}_{S_{\#}} \mathcal{T}_n(\btsharp) \|_2 - \perptwo^{\gor} \right| \right) > \epsilon \right\} \leq C \exp \left( -c_K n \epsilon^4 \right).
\end{align}
Indeed, setting $\epsilon = \epsilon_0 := \left( \frac{\log (C/\delta)}{c_K \cdot n} \right)^{1/4}$, we have
\begin{align*}
\Pro \left\{ \max \left( \left| \frac{\inprod{\mathcal{T}_n(\btsharp)}{\proj^{\perp}_{\thetastar} \btsharp}}{ \| \proj^{\perp}_{\thetastar} \btsharp \|_2} - \perpone^{\gor} \right|, \left| \| \proj^{\perp}_{S_{\#}} \mathcal{T}_n(\btsharp) \|_2 - \perptwo^{\gor} \right| \right) > \epsilon_0 \right\} \leq \delta.
\end{align*}
Note that $(\perpcomp^+)^2 = \frac{\inprod{\mathcal{T}_n(\btsharp)}{\proj^{\perp}_{\thetastar} \btsharp}^2}{ \| \proj^{\perp}_{\thetastar} \btsharp \|^2_2} + \| \proj^{\perp}_{S_{\#}} \mathcal{T}_n(\btsharp) \|_2^2$ and $(\perpcompgordon)^2 = (\perpone^{\gor})^2 + (\perptwo^\gor)^2$. Applying the triangle inequality and adjusting constant factors, we have $\Pro( |\beta^+ - \beta^\gor| \geq c \epsilon_0) \leq \delta$, as desired. Consequently, we dedicate our effort toward establishing inequality~\eqref{eq:equiv-thm1a} by proving growth properties for $\widebar{L}_n$.

\subsubsection{Implementing step 4: Growth properties of $\bar{L}_n$} 
The following lemma guarantees growth conditions on the function $\widebar{L}_n$ when the one-step loss function takes the form~\eqref{eq:sec_order_gen}.  Its proof is deferred to Appendix~\ref{sec:growth-HO}.

\begin{lemma}
	\label{lem:ao-analysis}
	Suppose that the loss function $\mathcal{L}$ can be written in the form
	\eqref{eq:sec_order_gen} and let the scalarized auxiliary loss $\widebar{L}_n$ be as in Definition~\ref{def:scalarized-ao}.  Define the constant 
	\[
	\bar{\mathsf{L}} =  \sqrt{\Bigl(1 - \frac{1}{\kappa}\Bigr)\Bigl(\EE[\Omega^2] - (\EE[Z_1 \Omega])^2 - (\EE[Z_2 \Omega])^2\Bigr)},
	\]
	and let the tuple $(\parcompgordon, \perpone^{\gor}, \perptwo^{\gor})$ be as in Definition~\ref{def:expanded-gordon-HO}.
	Suppose that Assumptions~\ref{ass:omega} and~\ref{ass:omega-lb} hold with parameters $K_1$ and~$K_2$, respectively.
	Then, there exist positive constants $C_{K_1},  C'_{K_1}, C_{K_2}$ and $c_{K_1,K_2}$
	each depending on a subset of $\{K_1, K_2\}$ 
	 and universal positive constants $c, c', C, C'$ such that for all $\kappa \geq C'$ and all $\epsilon \in (0, c')$, the following hold:
		\begin{enumerate}[label=(\alph*)]
		\item The minimizer
		\[
		(\alpha_n, \mu_n, \nu_n) = \argmin_{(\alpha, \mu, \nu) \in \Pc(\mathbb{B}_2(C_{K_1}))} \widebar{L}_n(\alpha, \mu, \nu) 
		\]
		is unique and satisfies both
		\[
		\max\bigl\{\lvert \alpha_n - \alpha^{\gor} \rvert, \lvert \mu_n - \mu^{\gor} \rvert, \lvert \nu_n - \nu^{\gor}\rvert\bigr\} \leq C'_{K_1} \epsilon,
		\]
		and 
		\[
		\biggl \lvert \widebar{L}_n(\alpha_n, \mu_n, \nu_n) - \bar{\mathsf{L}} \biggr \rvert \leq   C'_{K_1} \epsilon,
		\]
		with probability at least $1-C\exp \{ -c_{K_1,K_2} \cdot n \}- Ce^{-cn\eps^2}$.
		\item The scalarized auxiliary loss $\widebar{L}_n$ is 
		$C_{K_2}$-strongly convex on the domain $\mathbb{B}_2(C_{K_1})$ with probability at least $1 - C\exp\{-c_{K_1,K_2} \cdot n\}$.
	\end{enumerate}
\end{lemma}
With each of the individual steps of the recipe completed, we can now put everything together to prove equation~\eqref{eq:equiv-thm1a}.

\subsubsection{Combining the pieces} 
Let the tuple $(\parcompgordon, \perpone^{\gor}, \perptwo^{\gor})$ be as in Definition~\ref{def:expanded-gordon-HO}. For each nonnegative scalar $\epsilon$, define the deviation set
\begin{align}
	\label{eq:deviation-set}
	D_{\epsilon} := \{\bt \in \mathbb{R}^d: \max\Bigl(\lvert \langle \bt, \btstar \rangle - \alpha^{\gor} \rvert, \Bigl\lvert \frac{\langle \bt, \proj_{\btstar}^{\perp}\bt^{\sharp}\rangle}{\| \proj_{\btstar}^{\perp} \bt^{\sharp}\|_2} - \mu^{\gor} \Bigr\rvert, \lvert \| \proj_{S_{\sharp}}^{\perp} \bt \|_2 - \nu^{\gor} \rvert \Bigr) \geq \epsilon\},
\end{align}
and note that the deviation set $D_{\epsilon}$ is amenable with respect to the subspace $S_{\sharp} = \mathsf{span}(\btstar, \btsharp)$.
Note that it suffices to bound $\Pro\{\mathcal{T}_n (\btsharp) \in D_\epsilon \}$.
To this end, note that
\begin{align}\label{eq:R3-to-B2}
	\Pro\{\mathcal{T}_n(\btsharp) \in D_{\epsilon} \} = \Pro\Bigl\{\argmin_{\bt \in \mathbb{R}^d} \mathcal{L}(\bt) \in D_{\epsilon}\Bigr\} \overset{\1}{=} \Pro\Bigl\{\argmin_{\bt \in \mathbb{B}_2(C_{K_1})} \mathcal{L}(\bt) \in D_{\epsilon}\Bigr\} + 2e^{-cn},
\end{align}
where step $\1$ follows upon applying Lemma~\ref{lem:bound-operator-norm} in the appendix.  

First, note that the loss $\Lc$ satisfies Assumption~\ref{ass:loss} as we can take the functions
\[
h(\bt, \btsharp) = 0 \qquad \text{ and } \qquad g(\bu, \bX \btsharp; \by) = \frac{1}{\sqrt{n}} \| \omega(\bX\btsharp, \by) - \bu \|_2.
\]  
Each of the properties (a)-(d) is evident as the norm $\| \cdot \|_2$ is convex, Lipschitz continuous, and coercive.  Moreover, the Lipschitz constant $C_L = 1$.

Next, recall the constant $\bar{\mathsf{L}}$ as defined in Lemma~\ref{lem:ao-analysis} and subsequently invoke Proposition~\ref{prop:first-three-steps} to obtain, for  constants $r \geq 1$ and $\epsilon' > 0$ to be specified later, the inequality 
\begin{align}\label{ineq:T1-T2-decomp-thm1a-HO}
	\Pro\Bigl\{\argmin_{\bt \in \mathbb{B}_2(C_{K_1})} \mathcal{L}(\bt) \in D_{\epsilon} \Bigr\} &\leq 
	T_1 + T_2 + 20e^{-n/2},
\end{align}
where 
\begin{align*}
T_1 &= 2 \Pro\Bigl\{\min_{(\alpha, \mu, \nu) \in \Pc( D \cap \mathbb{B}_2(C_{K_1}))}\widebar{L}_n(\alpha, \mu, \nu) \leq \bar{\mathsf{L}} - \frac{3 C_L^2 C_1}{r} + 2\epsilon' \Bigr\},\\
T_2 &= 2 \Pro\Bigl\{\min_{(\alpha, \mu, \nu) \in \Pc(\mathbb{B}_2(C_{K_1}))} \widebar{L}_n(\alpha, \mu, \nu) > \bar{\mathsf{L}} + \epsilon' \Bigr\}.
\end{align*}
Now, let $0 < \epsilon_1 \leq \epsilon$ be a constant to be specified later and define the two events
\begin{align*}
	\Ec_1 &= \Bigl\{ \| (\parcomp_n, \perpone_n, \perptwo_n) - (\parcomp^{\gor}, \perpone^{\gor}, \perptwo^{\gor}) \|_2 \leq  \epsilon_1, \bigl \lvert \widebar{L}_n(\alpha_n, \mu_n, \nu_n) - \bar{\mathsf{L}} \bigr \rvert \leq \epsilon_1\Bigr\},\\
	\Ec_2 &= \biggl\{	\min_{\|(\alpha, \mu, \nu) - (\alpha_n, \mu_n, \nu_n)\|_2 \geq \epsilon - \epsilon_1} \widebar{L}_n(\alpha, \mu, \nu) \geq \widebar{L}_n(\alpha_n, \mu_n, \nu_n) + c_{K_2} (\epsilon - \epsilon_1)^2\biggr\}.
\end{align*}
Assume for the moment that $\epsilon \leq c'$ (a fact that will be true for the eventual setting of $\epsilon$) and subsequently invoke Lemma~\ref{lem:ao-analysis}(a) to obtain
\[
\Pro\{\Ec_1\} \geq 1 - C\exp\{-c_{K_1, K_2} \cdot n\} - Ce^{-\widetilde{c}_{K_1} n\epsilon_1^2},
\]
for some constant $\widetilde{c}_{K_1}$ depending only on $K_1$.
Also apply Lemma~\ref{lem:ao-analysis}(b) to obtain the inequality
\[
\Pro\{\Ec_2\} \geq 1 - C\exp\{-c_{K_1, K_2} \cdot n\}.
\]
For the remainder of the proof, we work on the event $\Ec_1 \cap \Ec_2$.  We have
\begin{align} \label{ineq:bound-T2-HO}
	\min_{(\alpha, \mu, \nu) \in \Pc(D_{\epsilon} \cap \mathbb{B}_2(C_{K_1}))} \bar{L}_n(\alpha, \mu, \nu) &\geq \min_{\| (\alpha, \mu, \nu) - (\alpha^{\gor}, \mu^{\gor}, \nu^{\gor})\|_2 \geq \epsilon} \bar{L}_n(\alpha, \mu, \nu) \nonumber\\
	&\overset{\1}{\geq} \min_{\|(\alpha, \mu, \nu) - (\alpha_n, \mu_n, \nu_n)\|_2 \geq \epsilon - \epsilon_1} \widebar{L}_n(\alpha, \mu, \nu) \nonumber\\
	&\geq \widebar{L}_n(\alpha_n, \mu_n, \nu_n) + c_{K_2} (\epsilon - \epsilon_1)^2 \geq \bar{\mathsf{L}} - \epsilon_1 + c_{K_2} (\epsilon - \epsilon_1)^2,
\end{align}
where step $\1$ follows since on $\Ec_1$, if \mbox{$\|(\alpha, \mu, \nu) - (\alpha^{\gor}, \mu^{\gor}, \nu^{\gor})\|_2 \geq \epsilon$} then \mbox{$\|(\alpha, \mu, \nu) - (\alpha_n, \mu_n, \nu_n)\|_2 \geq \epsilon - \epsilon_1$}.
Now, set $\epsilon_1 = \frac{c_{K_2}}{4}\epsilon^2$, which is a valid choice provided $\epsilon \leq 4/c_{K_2}$.
For each such value of $\epsilon$, continuing from the inequality~\eqref{ineq:bound-T2-HO}, we obtain
\begin{align}\label{ineq:bound-T2-HO-2}
		\min_{(\alpha, \mu, \nu) \in \Pc(D_{\epsilon} \cap \mathbb{B}_2(C_{K_1}))} \bar{L}_n(\alpha, \mu, \nu) \geq \bar{\mathsf{L}} + c_{K_2}' \epsilon^2,
\end{align}
where $c_{K_2}' \leq c_{K_2}/2$ is a small enough constant depending only on $K_2$.  Continuing, set \\$r = 12 \cdot C_1 \cdot C_L^2 / (c_{K_2}' \cdot \epsilon^2)$ and $\epsilon' = c_{K_2}' \cdot \epsilon^2/4$.  Thus, in view of the inequality~\eqref{ineq:bound-T2-HO-2}, we obtain the bounds
\begin{align*}
	\max\{T_1, T_2\} \leq \Pro\{\Ec_1^c \cup \Ec_2^c \} \leq C\exp\{-c_{K_1, K_2} \cdot n\} + Ce^{-\widetilde{c}_{K_1,K_2}n\epsilon^4},
\end{align*}
where $\widetilde{c}_{K_1,K_2}$ is once again a constant depending solely on $K_1$ and $K_2$.
Substituting the bound in the display above into the inequality~\eqref{ineq:T1-T2-decomp-thm1a-HO}, we obtain
\[
	\Pro\Bigl\{\argmin_{\bt \in \mathbb{B}_2(C_{K_1})} \mathcal{L}(\bt) \in D_{\epsilon} \Bigr\} \leq C\exp\{-c_{K_1, K_2} \cdot n\} + Ce^{-\widetilde{c}_{K_1, K_2}n\epsilon^4} + 20e^{-n/2}.
\]
Combining the display above with the inequality~\eqref{eq:R3-to-B2}, we obtain
\begin{align} \label{eq:final-pb}
\Pro\{\mathcal{T}_n(\btsharp) \in D_{\epsilon}\} \leq C\exp\{-c_{K_1, K_2} \cdot n\} + Ce^{-\widetilde{c}_{K_1, K_2}n\epsilon^4} + 22e^{-n/2}.
\end{align}
To complete the proof, set $\epsilon = \order \left( \widetilde{c}^{-1}_{K_1, K_2} \cdot \left( \frac{\log (1/\delta)}{n} \right)^{1/4} \right)$, which we can ensure is a valid choice owing to the condition $n \geq C_{K_1,K_2} \log (1/\delta)$. Finally, we use this lower bound on $n$ to also bound the remaining two terms of the RHS in equation~\eqref{eq:final-pb} by $\order(\delta)$, thereby obtaining the claimed result.
\qed

\subsection{Proof of Theorem~\ref{thm:one_step_main_FO}(a)}

We begin by defining the first-order analog of the expanded Gordon state evolution update.

\begin{definition}[Expanded Gordon state evolution update: First-order methods]\label{def:expanded-gordon-FO}
Recall the model~\eqref{eq:model}, and let $Q$ denote a random variable drawn from the latent variable distribution $\mathbb{Q}$. Suppose the loss function $\Lc$ takes the form~\eqref{eq:GD_gen}. Let $(Z_1, Z_2, Z_3)$ denote a triplet of independent standard Gaussian random variables and let $\alpha^{\sharp} \in \mathbb{R}$ and $\beta^{\sharp} \in \mathbb{R}_{\geq 0}$ denote arbitrary scalars.  Let
\[
\Omega = \omega\big(\parcomp^{\sharp} Z_1 +\perpcomp^{\sharp} Z_2\,,\,f(Z_1; Q)+\noisestd Z_3 \big).
\]
Then define the scalars
\[
\alpha^{\gor} = \alpha^{\sharp} - 2\eta \cdot \EE[Z_1 \Omega], \qquad \perpone^{\gor} = \perpcomp^{\sharp} - 2\eta \cdot \EE[Z_2 \Omega], \qquad \text{ and } \qquad \perptwo^{\gor} = \frac{2 \eta}{\sqrt{\kappa}} \cdot \sqrt{\EE[\Omega^2]}.
\]
\end{definition}
As before, it suffices to show that for a parameter $c_K$ depending only on $K_1$ from Assumption~\ref{ass:omega}, we have
\begin{align} \label{eq:equiv-thm2a}
\Pro \left\{ \max \left( \left| \frac{\inprod{\mathcal{T}_n(\btsharp)}{\proj^{\perp}_{\thetastar} \btsharp}}{ \| \proj^{\perp}_{\thetastar} \btsharp \|_2} - \perpone^{\gor} \right|, \left| \| \proj^{\perp}_{S_{\#}} \mathcal{T}_n(\btsharp) \|_2 - \perptwo^{\gor} \right| \right) > \epsilon \right\} \leq C \exp \left( -c_K n \epsilon^4 \right).
\end{align}

\subsubsection{Implementing step 4: Growth properties of $\bar{L}_n$}

The following lemma guarantees growth conditions on the function $\widebar{L}_n$ when the one-step loss function takes the form~\eqref{eq:GD_gen}.  We provide its proof in Subsection~\ref{sec:growth-FO}.

\begin{lemma}
	\label{lem:ao-analysis-FO}
	Suppose that the loss function $\mathcal{L}$ takes the form
	\eqref{eq:GD_gen} and let the scalarized auxiliary loss $\widebar{L}_n$ be as in Definition~\ref{def:scalarized-ao}.  Define the constant 
	\[
	\bar{\mathsf{L}} = -\frac{1}{2}((\alpha^{\gor})^2 + (\mu^{\gor})^2 + (\nu^{\gor})^2),
	\]
	and let the tuple $(\parcompgordon, \perpone^{\gor}, \perptwo^{\gor})$ be as in Definition~\ref{def:expanded-gordon-FO}.
	Suppose that Assumption~\ref{ass:omega} holds with parameter $K_1$.
	Then, there exists $c_{K_1}, C_{K_1} > 0$ depending only on $K_1$ and universal positive constants $c, C, c', C_1, C_2$  such that for all $\kappa \geq C_1$  and all $\epsilon \in (0, c')$, the following hold:
		\begin{enumerate}[label=(\alph*)]
		\item The minimizer
		\[
		(\alpha_n, \mu_n, \nu_n) = \argmin_{(\alpha, \mu, \nu) \in \Pc(\mathbb{B}_2(C_2))} \widebar{L}_n(\alpha, \mu, \nu) 
		\]
		is unique and satisfies both
		\[
		\max\bigl\{\lvert \alpha_n - \alpha^{\gor} \rvert, \lvert \mu_n - \mu^{\gor} \rvert, \lvert \nu_n - \nu^{\gor}\rvert\bigr\} \leq C_{K_1}  \epsilon,
		\]
		and 
		\[
		\biggl \lvert \widebar{L}_n(\alpha_n, \mu_n, \nu_n) - \bar{\mathsf{L}} \biggr \rvert \leq C_{K_1} \epsilon,
		\]
		with probability at least 
		$1 - C e^{-cn\eps^2}$.
		\item The scalarized auxiliary loss $\widebar{L}_n$ is $1$-strongly convex on the domain $\mathbb{B}_2(C_2)$ with probability at least $1 - Ce^{-cn}$.
	\end{enumerate}
\end{lemma}
Note that the strong convexity constant here is absolute (equal to $1$) instead of dependent on Assumption~\ref{ass:omega-lb} like in higher-order methods. We can now put everything together exactly like before.

\subsubsection{Combining the pieces}
The calculations here are very similar to before, so we only sketch the differences. First, consider the first order loss as in equation~\eqref{eq:GD_gen}
\[
\Lc(\bt; \btsharp, \bX, \by) = \frac{1}{2} \| \bt \|_2^2 - \langle \bt, \btsharp \rangle - \frac{2\eta}{n} \sum_{i=1}^{n} \omega(\langle \bx_i, \btsharp \rangle, y_i) \cdot \langle \bx_i, \bt \rangle,
\]
which corresponds to setting
\[
h(\bt, \btsharp) = \frac{1}{2} \| \bt\|_2^2 - \langle \bt, \btsharp \rangle, \qquad \hscal(\parcomp, \perpone, \perptwo) = \frac{1}{2}(\parcomp^2 + \perpone^2 + \perptwo^2) - (\parcomp \parcomp_{\sharp} + \perpone \perpcomp_{\sharp}),
\]
and 
\[
g(\bu, \bX \btsharp; \by) = -\frac{2\eta}{n} \langle \bu, \omega(\bX \btsharp, \by) \rangle.
\]
Note that 
\[
g^*(\bv, \bX\btsharp; \by) = \begin{cases} 0 & \text{ if } \bv = \frac{2\eta}{n}\cdot \omega(\bX\btsharp, \by), \\
	+\infty & \text{ otherwise}.\end{cases}
\]
Consequently, we obtain
\begin{align} \label{eq:Lbar-first-order}
	\widebar{L}_n(\parcomp, \perpone, \perptwo) = \;&\frac{\parcomp^2 + \perpone^2 + \perptwo^2}{2} - (\parcomp \parcomp_{\sharp} + \perpone \perpcomp_{\sharp}) - \frac{2\eta}{n} \cdot \langle \omega(\bX\btsharp, \by), \perptwo \bg + \parcomp \bz_1 + \perpone \bz_2 \rangle \nonumber\\
	&- \frac{2 \eta \perptwo}{n} \cdot \| \proj_{S_{\sharp}}^{\perp} \bh\|_2 \cdot \| \omega(\bX\btsharp, \by) \|_2.
\end{align}
Note that $g$ is a linear function of $\bu$, and that an application of Hoeffding's inequality yields $\| \omega(\bX \btsharp, \by) \|_2 \leq C_{K_1}$ with probability at least $1 - e^{-cn}$.  Thus, as is evident from this inequality and the displays above, $\Lc$ satisfies Assumption~\ref{ass:loss}, with Lipschitz constant $C_L = C_{K_1}$.

Now, with the tuple $(\parcompgordon, \perpone^{\gor}, \perptwo^{\gor})$ as in Definition~\ref{def:expanded-gordon-FO}, define the events
\begin{align*}
\mathcal{A}_1 &= \Bigl\{\max\bigl\{\lvert \alpha_n - \alpha^{\gor} \rvert, \lvert \mu_n - \mu^{\gor} \rvert, \lvert \nu_n - \nu^{\gor}\rvert\bigr\} \leq C_{K_1} \epsilon_1,  \bigl \lvert \widebar{L}_n(\alpha_n, \mu_n, \nu_n) - \bar{\mathsf{L}} \bigr \rvert \leq  \epsilon_1\Bigr\} \text{ and } \\
	\mathcal{A}_2 &= \Bigl\{	\min_{\|(\alpha, \mu, \nu) - (\alpha_n, \mu_n, \nu_n)\|_2 \geq \epsilon - \epsilon_1} \widebar{L}_n(\alpha, \mu, \nu) \geq \widebar{L}_n(\alpha_n, \mu_n, \nu_n) + (\epsilon - \epsilon_1)^2\Bigr\}.
\end{align*}
Carrying out the proof exactly as for higher-order methods but now with $\epsilon_1 = c_{K_1} \epsilon^2$ leads to the desired result.
\qed

%% file: figs/gordon-proof-fig.tex
\usetikzlibrary{spy,backgrounds}
	\begin{tikzpicture}[spy using outlines={circle, magnification=2, size=1cm, connect spies}]
	
	\draw[thin,->] (0,0) -- (10,0) node[anchor=south east] {$\mathbb{R}^3$};
	\draw[thin,->] (0,0) -- (0,10);

	\draw[thick, scale=1, domain=0:10, smooth, variable=\x, black] plot ({\x},
	{0.2*(\x - 5.375) * (\x - 5.375)  + 1.5});
	\node[above, black] at (10.1,{0.2*(10 - 5.375) * (10 - 5.375)  + 1.5}) {$\bar{L}_n(\bxi)$};
	
	\draw[thick, scale=1, domain=0:10, smooth, variable=\x, CadetBlue] plot ({\x},
	{0.2*(\x - 5)*(\x - 5) +2});
	\node[above, CadetBlue] at (10, 7) {$\bar{L}(\bxi)$};

	\draw (5.375cm,3pt) -- (5.375cm,-3pt) node[anchor=north west] {$\bxi_n$};
	\draw[thin, dashed] (5.375,0) -- (5.375, 1.5);
	
	\draw (5cm,3pt) -- (5cm,-3pt) node[anchor=north] {$\bxi^{\mathsf{gor}}$};
	\draw[thin, dashed] (5, 0) -- (5, 0.2 * 25 - 2.4 * 5 + 9);
	\draw[thin, dashed] (0,2) -- (10, 2);
	\draw (3pt, 2) -- (-3pt, 2)
	    node[anchor=east] {$\bar{\mathsf{L}}$};
	\draw (3pt, 2.75) -- (-3pt, 2.75)
	    node[anchor=east] {$\bar{\mathsf{L}} + C\epsilon^2$};
	\draw (3pt, 1.25) -- (-3pt, 1.25)
	node[anchor=east] {$\bar{\mathsf{L}} - C\epsilon^2$};
	
	\draw[thick, blue, line width=2] (7.5, 0) -- (2.5, 0) node[anchor=north east]
	{$\mathbb{B}_{\infty}(\bxi^{\gor}; \epsilon)$};

	\draw[thick, OliveGreen, scale=1, domain=0:10, smooth, variable=\x] plot ({\x},
	{(0.2)*(\x - 4.3)*(\x - 4.3) + 2.35 + 0.1*rand});
	\draw (4.3cm,3pt) -- (4.3cm,-3pt) node[anchor=north east] {$\mathcal{T}_n(\btsharp)$};
	\node[above left, OliveGreen] at (9.5, 8) {$\mathcal{L}(\bt)$};
	
	\fill [CadetBlue!40,opacity=0.3] (0, 1.25) rectangle (10,2.75);
\end{tikzpicture}

%% file: sections/random-init/random-init.tex

The main result of this section is the following proposition, which---in words---shows that the one-dimensional projection of the empirical operator $\mathcal{T}_n$ onto the ground truth $\btstar$ concentrates around $\parcompgordon$ for both higher-order and first order methods at rate $\ordertil(n^{-1/2})$. 
\begin{proposition}
	\label{prop:concentration-signal}
	Consider the one-step empirical updates $\mathcal{T}_n(\bt)$ taking either of the forms~\eqref{eq:sec_order_gen} or~\eqref{eq:GD_gen} and the associated state evolution update $\parcomp^{\gor}$.  There are universal positive constants $C$ and $C'$ such that if $\kappa \geq C$ the following hold.
	\begin{itemize}
		\item [(a)] Suppose that the empirical update $\mathcal{T}_n(\bt)$ takes the first-order form~\eqref{eq:GD_gen} and consider the associated state evolution update $\parcomp^{\gor}$ as in Definition~\ref{def:starnext_main_FO}.  Then,
		\[
		\bigl \lvert \langle\btstar, \mathcal{T}_n(\bt) \rangle - \alpha^{\gor} \bigr\rvert  \leq C K_1  \cdot  \biggl(\frac{\log{(1/\delta)}}{n}\biggr)^{1/2},
		\]
		with probability at least $1 - \delta$.
		\item[(b)] Suppose that the empirical update $\mathcal{T}_n(\bt)$ takes the second-order form~\eqref{eq:sec_order_gen} and consider the associated state evolution update $\parcomp^{\gor}$ as in Definition~\ref{def:starnext_main_HO}.  Then,
		\[
		\bigl \lvert \langle\btstar, \mathcal{T}_n(\bt) \rangle - \alpha^{\gor} \bigr\rvert  \leq C K_1 \cdot  \biggl(\frac{\log^7{(1/\delta)}}{n}\biggr)^{1/2},
		\]
		with probability at least $1 - \delta$.
	\end{itemize}
\end{proposition}
Clearly, the proposition directly implies part (b) of both Theorems~\ref{thm:one_step_main_HO} and~\ref{thm:one_step_main_FO}. Before providing the proof, we pause to make a few comments.  First, we note that
in the previous section, we showed concentration of $\langle\btstar, \mathcal{T}_n(\bt) \rangle$ with fluctuations of order $n^{-1/4}$, but additionally provided control over the random variable $\| \proj^{\perp}_{\btstar} \bt \|_2$, which has $d$ degrees of freedom.  On the other hand, by focusing directly on the quantity $\langle\btstar, \mathcal{T}_n(\bt) \rangle$, which has one degree of freedom, we are able to tighten the fluctuations to the order $n^{-1/2}$. 
Second, we comment briefly on the proof, especially for higher-order methods.  Our proof improves upon the strategy utilized in~\cite{zhang2020phase} by (i) proving a concentration inequality with exponential tails and (ii) extending the methodology beyond alternating minimization for phase retrieval to more general updates of the form~\eqref{eq:second-order-update-sum}.  This extension relies on a delicate combination of the leave-one-out technique of~\cite{zhang2020phase} with the moment inequalities in~\cite{boucheron2005moment}.  We turn now to the proof of the main proposition, proving the result for first-order and higher-order methods separately.

\begin{proof}[Proof of Proposition~\ref{prop:concentration-signal}(a): first-order methods]
First, recall the generic first-order update~\eqref{eq:GD_gen} and specify the operator
\[
\mathcal{T}_n(\bt) = \bt - \eta \cdot \frac{2}{n} \sum_{i=1}^{n} \omega(\langle \bx_i, \bt \rangle, y_i) \cdot \bx_i.
\]
Now, evaluating the quantity $\EE \{\langle \btstar, \mathcal{T}_n(\bt) \rangle \}$ and recalling Definition~\ref{def:starnext_main_FO}, we obtain the characterization
\begin{align}
	\label{eq:first-order-random-init-char}
\EE \{\langle \btstar, \mathcal{T}_n(\bt) \rangle \} = \langle \btstar, \bt \rangle - \eta \cdot \frac{2}{n} \sum_{i=1}^{n}\omega(\langle \bx_i, \bt \rangle, y_i) \cdot \langle \btstar, \bx_i \rangle  = \alpha(\bt) - 2\eta\cdot \EE\{Z_1 \Omega\} = \alpha^{\gor},
\end{align}
where we have drawn $\Omega$ according to equation~\eqref{eq:omegat_def_main} using the random variables\\ $Z_1 = \langle \bx_1, \btstar \rangle, Z_2 = \langle \bx_1, \proj_{\btstar}^{\perp} \bt \rangle, $ and 
$Z_3 \sim \NORMAL(0,\sigma^2)$.
Now, note that $\langle \btstar, \bx_i \rangle \sim \NORMAL(0, 1)$ and that by Assumption~\ref{ass:omega}, $\| \omega(\langle \bx_i, \bt \rangle, y_i) \|_{\psi_2} \leq K_1$.  Thus, we apply~\citet[Lemma 2.7.7]{vershynin2018high} to obtain the bound
\[
\| \omega(\langle \bx_i, \bt \rangle, y_i) \cdot \langle \btstar, \bx_i \rangle \|_{\psi_1} \leq K_1.
\]
Subsequently applying Bernstein's inequality in conjunction with the characterization of the expectation~\eqref{eq:first-order-random-init-char}, we obtain the inequality
\[
\Pro\Bigl\{ \bigl \lvert \langle \btstar, \mathcal{T}_n(\bt)\rangle - \alpha^{\gor} \bigr\rvert \geq t\Bigr\} \leq 2\exp\Bigl\{-c \min\Bigl(\frac{nt^2}{K_1^2}, \frac{nt}{K_1}\Bigr)\Bigr\}.
\]
The conclusion follows immediately from the above inequality. 
\end{proof}

\begin{proof}[Proof of Proposition~\ref{prop:concentration-signal}(b): higher-order methods] 
Recall the weight functions $\omega: \mathbb{R}^2 \rightarrow \mathbb{R}$ as in the equations~\eqref{eq:sec_order_gen} and~\eqref{eq:second-order-update-sum} and consider its separable extension to vector valued functions:
\begin{align*}
	\omega: \ \mathbb{R}^n \times \mathbb{R}^n &\rightarrow \mathbb{R}^n\\
				 (\bx, \by) &\mapsto (\omega(x_1, y_1), \omega(x_2, y_2), \dots \omega(x_n, y_n)).
\end{align*}
We can then re-write the updates~\eqref{eq:second-order-update-sum} using matrix notation as
\begin{align}
	\label{eq:empirics-AM}
	\mathcal{T}_n(\bt) = (\bX^{\top} \bX)^{-1} \bX^{\top} \omega(\bX \bt, \by),
\end{align}
where specifying $\omega(x, y) = \mathsf{sgn}(x) \cdot y$ recovers the alternating minimization update for phase retrieval and specifying $\omega(x, y) = \mathsf{sgn}(yx) \cdot y$ recovers the alternating minimization update for mixtures of linear regressions.  We also specify, for convenience, the following population operator
\begin{align}
	\label{eq:population-AM}
	\mathcal{T}(\bt) = \frac{1}{n} \EE \bigl\{\bX^{\top} \omega(\bX\bt, \by)\bigr\}.
\end{align}
Recalling the definition of $\parcomp^{\gor}$ as in Definition~\ref{def:starnext_main_HO}, we note that
\[
\langle \btstar, \mathcal{T}(\bt) \rangle = \frac{1}{n} \EE\{\langle  \omega(\bX\bt, \by), \bX \btstar \rangle\} = \EE\{ Z_1 \Omega \} = \alpha^{\gor},
\]
where we have let $Z_1$ denote the first component of the vector $\bX\btstar$ and $Z_2$ denote the first component of the vector $\bX\proj_{\btstar}^{\perp} \bt/\| \proj_{\btstar}^{\perp} \bt \|_2$.  We now state two lemmas whose proofs are postponed to Subsections~\ref{sec:proof-concentration-loo} and~\ref{sec:proof-expectation-leave-one-out}, respectively. 
\begin{lemma}
	\label{lem:concentration-leave-one-out}
	Under the setting of Proposition~\ref{prop:concentration-signal}, there exist universal, positive constants $c$ and $C$ such that for all $t > 0 $,
	\begin{align}
		\label{ineq:probability-leave-one-out}
		\Pr\bigl\{ \bigl \lvert \inprod{\btstar}{\mathcal{T}_n(\bt)} - \inprod{\btstar}{\EE \mathcal{T}_n(\bt)} \bigr \rvert \geq t \bigr\}& \leq C \exp\Bigl\{-c\Bigl(\frac{t \sqrt{n}}{K_1}\Bigr)^{2/7}\Bigr\} + e^{-cn}.
	\end{align}
\end{lemma}
\begin{lemma}
	\label{lem:expectation-leave-one-out}
	Under the setting of Proposition~\ref{prop:concentration-signal}, there exists a  universal, positive constant $C$ such that, 
	\[
	\lvert \inprod{\btstar}{\EE\mathcal{T}_n(\bt)} - \inprod{\btstar}{\mathcal{T}(\bt)} \rvert \leq \frac{C K_1}{\sqrt{n}}.
	\]
\end{lemma}

To prove the proposition from these two lemmas, apply the triangle inequality to obtain the following decomposition
\begin{align}
	\label{ineq:decomp-proof-concentration-signal}
\bigl \lvert \inprod{\btstar}{\mathcal{T}_n(\bt)} - \inprod{\btstar}{\mathcal{T}(\bt)} \bigr\rvert \leq \bigl \lvert \inprod{\btstar}{\mathcal{T}_n(\bt)} - \inprod{\btstar}{\EE \mathcal{T}_n(\bt)} \bigr \rvert + \bigl \lvert \inprod{\btstar}{\EE\mathcal{T}_n(\bt)} - \inprod{\btstar}{\mathcal{T}(\bt)} \bigr \rvert.
\end{align}
The conclusion follows immediately upon applying Lemmas~\ref{lem:concentration-leave-one-out} and~\ref{lem:expectation-leave-one-out} to upper bound the two terms above.
\end{proof}

\subsection{Proof of Lemma~\ref{lem:concentration-leave-one-out}: Concentration of centered term}
\label{sec:proof-concentration-loo}
Using the shorthand 
\begin{align}
	\label{eq:shorthand-sigma-z}
\bz_i = \bx_i \omega(\langle \bx_i, \bt \rangle, y_i) \qquad \text{ and } \qquad \bSig = \sum_{i=1}^{n} \bx_i \bx_i^{\top}, 
\end{align}
and recalling the empirical updates $\mathcal{T}_n(\bt)$~\eqref{eq:empirics-AM}, we denote the random variable of interest
\begin{align}
	\label{eq:def-Z}
Z := \inprod{\btstar}{\mathcal{T}_n(\bt)}  = \inprod{\btstar}{\bSig^{-1}\sum_{i=1}^{n} \bz_i },
\end{align}
where in the last equality we have simply simply written $\mathcal{T}_n(\bt)$ using the shorthand~\eqref{eq:shorthand-sigma-z}.  Noting that $Z$ is a non-linear function of the $n$ independent samples $\{\bx_i, y_i\}_{i=1}^{n}$, we introduce some notation in order to isolate the contribution of the $i$-th sample.  For all $1 \leq j \leq n$,  let $\{\bx_j', y_j'\}$ denote an independent copy of the pair $\{\bx_j, y_j\}$, and define (cf. Eq.~\eqref{eq:shorthand-sigma-z})
\begin{align}
	\label{eq:shorthand-sigma-z-loo}
	\bz_j' = \bx_j'  \omega(\langle \bx_j', \bt \rangle, y_j'), \;\;\;\; \bSig_j = \sum_{i \neq j} \bx_i \bx_i^{\top}.
\end{align}
We then define
\begin{align*}
	Z_j' = (\btstar)^{\top} (\bx_j' (\bx_j')^{\top} + \bSig_j)^{-1}\Bigl(\bz_j' + \sum_{i\neq j} \bz_i\Bigr).
\end{align*}
Applying the Sherman-Morrison formula, we obtain the pair of identities
\begin{subequations}
	\begin{align}
		Z &= (\btstar)^{\top}\Bigl(\bSig_j^{-1} - \frac{\bSig_j^{-1} \bx_j \bx_j^{\top} \bSig_j^{-1}}{1 + \bx_j^{\top} \bSig_{j}^{-1}\bx_j}\Bigr)\Bigl(\bz_j + \sum_{i \neq j} \bz_i\Bigr),\label{eq:Z-expandj}\\
		Z_j^{'} &= (\btstar)^{\top}\Bigl(\bSig_j^{-1} - \frac{\bSig_j^{-1} \bx_j' \bx_j'^{\top} \bSig_j^{-1}}{1 + \bx_j'^{\top} \bSig_{j}^{-1}\bx_j'}\Bigr)\Bigl(\bz_j' + \sum_{i \neq j} \bz_i\Bigr).\label{eq:Zj-expandj}
	\end{align}
\end{subequations}
Note that in Eq.~\eqref{eq:Z-expandj}, the index $j$ is arbitrary and the same equation can be written for all $1 \leq j \leq n$. 

With this notation defined, we now state two lemmas. 
 Their proofs are deferred to Subsections~\ref{sec:proof-bounding-centered-moments} and~\ref{sec:proof-lemma-moment-sum-zj}, respectively. For the first lemma, recall that we use $\| X \|_q = ( \EE \lvert X\rvert^q )^{1/q}$ to denote the $L^q$ norm of a random variable $X$. 
\begin{lemma}
	\label{lem:bounding-centered-moments}
	Consider the random variable $Z$~\eqref{eq:def-Z}.  There exists a universal, positive constant $C$ such that for all integers $q$ satisfying $1 \leq q \leq \frac{n - d - 1}{16}$, it holds that
	\[
	\| Z - \EE Z \|_{q} \leq C K_1\frac{q^{7/2}}{\sqrt{n}}.
	\]
\end{lemma}
\begin{lemma}
	\label{lem:moment-sum-zj}
	There exists a constant $C > 0$ such that for all $q \geq 1$, 
	\[
	\Bigl(\E\Bigl\{\Bigl\| \sum_{i \neq j} \bz_i\Bigr\|_2^{2q}\Bigr\}\Bigr)^{1/2q} \leq C K_1 q^2 n.
	\]
\end{lemma}
We now use the moment bound from Lemma~\ref{lem:bounding-centered-moments} to obtain a tail bound.  Note that by assumption, $n \geq 2d$, so that $(n - d - 1)/16 \geq n/64$.  Additionally, since $Z = \langle \btstar, \mathcal{T}_n(\bt) \rangle$ by definition (recall Eq.~\eqref{eq:def-Z}),
we apply Lemma~\ref{lem:bound-operator-norm} to obtain the inequality $\Pro\{\lvert Z \rvert \geq CK_1\} \leq e^{-cn}$.  Thus, taking $\zeta = 7/2$ and invoking Lemma~\ref{lem:truncation-tail-bound}, we obtain the desired result.
\qed

\noindent It remains to prove the technical lemmas.

\subsubsection{Proof of Lemma~\ref{lem:bounding-centered-moments}}
\label{sec:proof-bounding-centered-moments}
For all $1 \leq j \leq n$, let $X_j = \{\bx_j, y_j\}$.  Following~\cite{boucheron2005moment}, define
\begin{align}
	V^{+} &= \E\biggl\{\sum_{j=1}^{n}\bigl(Z- Z_{j}'\bigr)_{+}^2 \Big \lvert \{\bx_i, y_i\}_{i=1}^{n}\biggr\},
\end{align}
and 
\begin{align}
	V^{-} &= \E\biggl\{\sum_{j=1}^{n}\bigl(Z- Z_{j}'\bigr)_{-}^2 \Big \lvert \{\bx_i, y_i\}_{i=1}^{n}\biggr\}.
\end{align}
With this notation in hand, we have that for all $q \geq 2$, 
\begin{align}
	\label{eq:decomposition-Z}
	\| Z - \EE Z \|_{q} \overset{\1}{\leq} \|(Z - \EE Z)_{+} \|_q + \| (Z - \EE Z)_{-} \|_{q} \overset{\2} \leq C\sqrt{q} \sqrt{\| V^{+} \|_{q/2}} + C\sqrt{q} \sqrt{\| V^{-} \|_{q/2}},
\end{align}
where step $\1$ follows from the triangle inequality and step $\2$ follows from~\citet[Theorem 2]{boucheron2005moment}.  Additionally, applying the triangle inequality in combination with the simple numeric inequalities $a_+^2 \leq a^2$, $a_-^2 \leq a^2$ yields the pair of inequalities
\begin{align*}
	\| V^{+} \|_{q} \leq \sum_{j=1}^{n} \bigl\| \E \bigl[(Z - Z_j')^2 \big \lvert \{\bx_i, y_i\}_{i=1}^{n}\bigr] \bigr\|_{q},
\end{align*}
and 
\begin{align*}
	\| V^{-} \|_{q} \leq \sum_{j=1}^{n} \bigl\| \E \bigl\{(Z - Z_j')^2 \big \lvert \{\bx_i, y_i\}_{i=1}^{n}\bigr\} \bigr\|_{q}.
\end{align*}
Combining the inequality~\eqref{eq:decomposition-Z} with the above two displays yields the inequality
\begin{align}
	\label{eq:Z-moment-ub-V}
	\| Z - \EE Z \|_{q}  \leq C\sqrt{q} \biggl(\sum_{j=1}^{n} \bigl\| \E \bigl[ (Z - Z_j')^2 \big \lvert \{\bx_i, y_i\}_{i=1}^{n}\bigr] \bigr\|_{q} \biggr)^{1/2}.
\end{align}
Continuing, we see that
\begin{align}
	\label{ineq:ri-cond-bound-z-diff-moment}
	\bigl\| \EE \bigl[(Z - Z_j')^2 \big \lvert \{\bx_i, y_i\}_{i=1}^{n}\bigr] \bigr\|_{q}  &\overset{\1}{\leq} \Bigl(\EE \Bigl\{
	\EE \bigl[(Z - Z_j')^{2q} \big \lvert \{\bx_i, y_i\}_{i=1}^{n}\bigr]
	\Bigr\}\Bigr)^{1/q}\nonumber\\
	& \overset{\2}{=} \Bigl(\EE \Bigl\{
	\EE \bigl[(Z - Z_j')^{2q} \big \lvert \{\bx_i, y_i\}_{i\neq j}\bigr]
	\Bigr\}\Bigr)^{1/q}
\end{align}
where step $\1$ follows by noting that the map $x \mapsto x^{a}$ is convex on $\mathbb{R}_{\geq 0}$ for $a \geq 1$ and applying Jensen's inequality; and step $\2$ follows by applying the tower property of conditional expectation to remove the conditioning on the sample $\{\bx_j, y_j\}$ .
Recalling the representations~\eqref{eq:Z-expandj} and~\eqref{eq:Zj-expandj}, we let
\begin{align}
	\label{eq:leave-one-out-Z-decomp}
	Z - Z_j' = \bigl(T_1 - T_1'\bigr) + \bigl(T_2 - T_2'\bigr) + \bigl(T_3- T_3'\bigr),
\end{align}
where
\begin{subequations}
\begin{align}
	T_1 &= (\btstar)^{\top}\bSig_j^{-1} \bz_j,\label{eq:T1-Z-decomp}\\
	T_2 &= (\btstar)^{\top} \Bigl(\frac{\bSig_j^{-1} \bx_j \bx_j^{\top} \bSig_j^{-1}}{1 + \bx_j^{\top} \bSig_{j}^{-1}\bx_j}\Bigr) \sum_{i \neq j} \bz_i, \label{eq:T2-Z-decomp}\\
	T_3&= (\btstar)^{\top} \Bigl(\frac{\bSig_j^{-1} \bx_j \bx_j^{\top} \bSig_j^{-1}}{1 + \bx_j^{\top} \bSig_{j}^{-1}\bx_j}\Bigr)\bz_j,\label{eq:T3-Z-decomp}
\end{align}
\end{subequations}
and $T_1', T_2', T_3'$ are equivalently defined, with $\bx_j', \bz_j'$ in place of $\bx_j, \bz_j$.  Now, applying the numeric inequality $(\sum_{i=1}^{k} a_i)^{2q} \leq k^{2q - 1}\sum_{i=1}^{k} a_i^{2q}$ to the term $(Z - Z_j')^{2q}$, using the decomposition~\eqref{eq:leave-one-out-Z-decomp} and noting that for $a = \{1, 2, 3\}$, the terms $T_a$ and $T_a'$ are identically distributed conditioned on $\{\bx_i, y_i\}_{i \neq j}$, we obtain the inequality
\begin{align}
	\label{ineq:bound-diff-z-t1-t3}
	\EE \bigl[(Z - Z_j')^{2q} \big \lvert \{\bx_i, y_i\}_{i\neq j}\bigr] \leq 6^{2q}\E\bigl[T_1^{2q} + T_2^{2q} + T_3^{2q} \lvert \{\bx_i, y_i\}_{i\neq j}\bigr].
\end{align}
Recalling the shorthand~\eqref{eq:shorthand-sigma-z}, plugging into the definition of $T_3$~\eqref{eq:T3-Z-decomp}, and noting that since $\bSig_j^{-1}$ is positive semidefinite, $(\bx_j^{\top} \bSig_j^{-1}\bx_j)/(1 + \bx_j^{\top} \bSig_j^{-1}\bx_j) \leq 1$, we have
\begin{align*}
T_3^{2q} =  \biggl((\btstar)^{\top} \frac{\bSig_j^{-1} \bx_j \bx_j^{\top} \bSig_j^{-1}\bx_j \omega(\langle \bx_j, \bt \rangle, y_j)}{1 + \bx_j^{\top} \bSig_{j}^{-1}\bx_j}\biggr)^{2q} \leq T_1^{2q}.
\end{align*}
Consequently, we obtain the bound
\begin{align}
	\label{ineq:ri-bound-t3}
		\E\bigl\{T_3^{2q} \lvert \{\bx_i, y_i\}_{i\neq j}\bigr\} \leq 	\E\bigl\{T_1^{2q} \lvert \{\bx_i, y_i\}_{i\neq j}\bigr\}.
\end{align}
We proceed to bound the conditional moments $\E\bigl\{T_1^{2q} \lvert \{\bx_i, y_i\}_{i\neq j}\bigr\}$ and $\E\bigl\{T_2^{2q} \lvert \{\bx_i, y_i\}_{i\neq j}\bigr\}$ in turn.  Recall the sub-exponential norm $\psi_1$ (see, for instance~\citet[Definition 2.7.5]{vershynin2018high}) and note that for any vector $\bu \in \mathbb{R}^{d}$:
\begin{align}
	\label{eq:subexponential-z}
	\| \bz_j^{\top} \bu \|_{\psi_1} = \| \omega(\langle \bx_j, \bt \rangle, y_j) \bx_j^{\top} \bu \|_{\psi_1} \overset{\1}{\leq} K_1 \| \bu \|_2,
\end{align}
where step $\1$ follows since $\omega(\langle \bx_j, \bt \rangle, y_j)$ is sub-Gaussian, $\bx_j^{\top} \bu$ is $\|\bu\|_2$-sub-Gaussian,  and  the product of sub-Gaussian random variables is sub-exponential (see for instance~\citet[Lemma 2.7.7]{vershynin2018high}).  Thus, 
\begin{align}
	\label{ineq:ri-bound-t1}
	\E\bigl[T_1^{2q} \lvert \{\bx_i, y_i\}_{i\neq j}\bigr] \overset{\1}{\leq} (C K_1 q)^{2q} \| \bSig_j^{-1} \btstar \|_{2}^{2q} \leq (C K_1 q)^{2q} \| \bSig_j^{-1} \|_{\mathsf{op}}^{2q},
\end{align}
where step $\1$ follows by recalling the definition of $T_1$~\eqref{eq:T1-Z-decomp}, applying the inequality~\eqref{eq:subexponential-z} for $\bu = \bSig_j^{-1} \btstar$, and using the $L^p$ norm characterization of sub-exponential random variables (see for instance~\citet[Proposition 2.7.1, part b.]{vershynin2018high}).  We now consider $T_2$.  We have
\begin{align}
	\label{ineq:ri-bound-t2}
\E\bigl\{T_2^{2q} \lvert \{\bx_i, y_i\}_{i\neq j}\bigr\} &\overset{\1}{\leq} \Bigl(\E\bigl\{\bigl(\bx_j^{\top}\bSig_j^{-1}\btstar\bigr)^{4q} \lvert \{\bx_i, y_i\}_{i\neq j}\bigr\} \Bigr)^{1/2} \Bigl(\E\bigl\{\bigl(\bx_j^{\top}\bSig_j^{-1}\sum_{i \neq j}\bz_i\bigr)^{4q} \lvert \{\bx_i, y_i\}_{i\neq j}\bigr\} \Bigr)^{1/2}\nonumber\\
&\overset{\2}{\leq} (Cq)^{2q} \| \bSig_j^{-1}\btstar \|_2^{2q} \Bigl\| \bSig_{j}^{-1} \sum_{i \neq j} \bz_i \Bigr\|_2^{2q} \nonumber\\
&\leq (Cq)^{2q} \| \bSig_{j}^{-1}\|_{\mathsf{op}}^{4q} \Bigl\| \sum_{i \neq j} \bz_i\Bigr\|_2^{2q},
\end{align}
where step $\1$ follows from substituting the definition of $T_2$~\eqref{eq:T2-Z-decomp}, using the fact that \\$\bx_j^{\top} \bSig_j^{-1}\bx_j \geq 0$
 and applying the Cauchy-Schwarz inequality; and step $\2$ follows from the $L^p$ norm characterization of sub-exponential random variables.  Now, plugging the bounds~\eqref{ineq:ri-bound-t3},~\eqref{ineq:ri-bound-t1}, and~\eqref{ineq:ri-bound-t2} into the inequality~\eqref{ineq:bound-diff-z-t1-t3}, we obtain
\begin{align*}
	\EE \bigl[(Z - Z_j')^{2q} \big \lvert \{\bx_i, y_i\}_{i=1}^{n}\bigr] \leq (Cq)^{2q} \| \bSig_{j}^{-1}\|_{\mathsf{op}}^{4q} \Bigl\| \sum_{i \neq j} \bz_i\Bigr\|_2^{2q} + (CK_1q)^{2q} \| \bSig_{j}^{-1}\|_{\mathsf{op}}^{2q}.
\end{align*}
Consequently, plugging the inequality above into the RHS of the inequality~\eqref{ineq:ri-cond-bound-z-diff-moment} and subsequently using the Cauchy-Schwarz inequality yields
\begin{align}
	\label{ineq:penultimate-martingale-ineq}
	\bigl\| \EE \bigl\{(Z - Z_j')^2 \big \lvert \{\bx_i, y_i\}_{i=1}^{n}\bigr\} \bigr\|_{q} \leq Cq^2\Bigl(\Bigl(\EE \bigl\{\| \bSig_{j}^{-1}\|_{\mathsf{op}}^{8q}\bigr\}  \EE \Bigl\{\Bigl\| \sum_{i \neq j} \bz_i\Bigr\|_2^{4q}\Bigr\}\Bigr)^{1/2q} +  K_1 \Bigl(\EE \bigl\{\| \bSig_{j}^{-1}\|_{\mathsf{op}}^{2q}\bigr\}\Bigr)^{1/q}\Bigr).
\end{align}
Now note that Lemma~\ref{lem:wishart_mineig_moments} from the appendix yields
\[
\E\left\{\|\bSig^{-1}\|_{\mathsf{op}}^p \right\} \leq \left(\frac{C}{n}\right)^p, \qquad \text{ for } 1 \leq p < \frac{n - d -1}{2}.
\]
Applying Lemmas~\ref{lem:moment-sum-zj} and~\ref{lem:wishart_mineig_moments} to the RHS of the inequality~\eqref{ineq:penultimate-martingale-ineq}, we obtain
\begin{align*}
	\bigl\| \EE \bigl\{(Z - Z_j')^2 \big \lvert \{\bx_i, y_i\}_{i=1}^{n}\bigr\} \bigr\|_{q} \leq \frac{C K_1 q^4}{n^2}, \qquad \text{ for } 1 \leq q < \frac{n - d - 1}{16}.
\end{align*}
Substituting the above bound into the RHS of the inequality~\eqref{eq:Z-moment-ub-V}, we have
\begin{align}
	\| Z - \EE Z \|_{q} \leq \frac{CK_1 q^{7/2}}{\sqrt{n}},  \qquad \text{ for } 1 \leq q < \frac{n - d - 1}{16},
\end{align}
which completes the proof.
\qed

\subsubsection{Proof of Lemma~\ref{lem:moment-sum-zj}}
\label{sec:proof-lemma-moment-sum-zj}
We begin by centering.  We have
\begin{align}
	\label{ineq:initial-lemma-moment-sum-zj}
\E\Bigl\{\Bigl\| \sum_{i \neq j} \bz_i\Bigr\|_2^{2q}\Bigr\} = \E\Bigl\{\Bigl\| \sum_{i \neq j} \big(\bz_i - \E \bz_i\big) + \sum_{i \neq j} \E \bz_i\Bigr\|_2^{2q}\Bigr\} \overset{\1}{\leq} 2^{2q-1}\Bigl( \E\Bigl\{\Bigl\| \sum_{i \neq j} \bz_i - \E \bz_i \Bigr\|_2^{2q} \Bigr\} + \Bigl\|\sum_{i \neq j}\E \bz_i \Bigr\|_2^{2q}\Bigr),
\end{align}
where step $\1$ used the numeric inequality $(a + b)^q \leq 2^{2q-1}(a^{2q} + b^{2q})$.  We tackle each of the two terms on the RHS of the above display in turn.

\paragraph{Bounding $\bigl\|\sum_{i \neq j}\E \bz_i \bigr\|_2^{2q}$.}
First, we apply the triangle inequality to obtain
\begin{align*}
	\Bigl\| \sum_{i \neq j} \E \bz_i \Bigr\|_2 \leq n \| \E\bz_i \|_2.
\end{align*}
Towards bounding the RHS of the inequality in the above display, we consider the subspace \mbox{$S = \mathsf{span}(\bt, \btstar)$} to obtain
\begin{align*}
	\E \bz_i = \E \{\bx_i \omega(\langle \bx_i, \bt \rangle, y_i)\} &= \E \bigl\{\proj_{S} \bx_i \omega(\langle \bx_i, \bt \rangle, y_i)\bigr\} + \E \bigl\{\proj_{S}^{\perp}\bx_i \omega(\langle \bx_i, \bt \rangle, y_i)\bigr\}\\
	&\overset{\1}{=}  \E \bigl\{\proj_{S} \bx_i \omega(\langle \bx_i, \bt \rangle, y_i)\bigr\} + \E \bigl\{\proj_{S}^{\perp}\bx_i \bigr\} \E\{\omega(\langle \bx_i, \bt \rangle, y_i)\}\\
	&\overset{\2}{=} \E \bigl\{\proj_{S} \bx_i \omega(\langle \bx_i, \bt \rangle, y_i)\bigr\} - \E \bigl\{\proj_{S}\bx_i \bigr\} \E\{\omega(\langle \bx_i, \bt \rangle, y_i)\}\,,
\end{align*}
where step $\1$ follows since $\bx_i \sim \NORMAL(0, \bI)$ so that by definition of the subspace $S$, the random vector $\proj_{S}^{\perp}\bx_i$ and the random variable $\omega(\langle \bx_i, \bt \rangle, y_i)$ are independent and step $\2$ follows since $\proj_{S}^{\perp} \bx_i = \bx_i - \proj_{S} \bx_i$ and $\E \bx_i = 0$.  Continuing from the display above, we apply the triangle inequality, Jensen's inequality, and the Cauchy-Schwarz inequality in succession to obtain the bound
\begin{align}
	\label{ineq:bound-norm-expec-z-1}
\| \E\bz_i \|_2 \leq \bigl(\E\{\| \proj_{S}\bx_i \|_2^2\} \E\{\omega(\langle \bx_i, \bt \rangle, y_i)^2\}\bigr)^{1/2} + \E \{\| \proj_{S}\bx_i \|_2\} \E \{ \lvert \omega(\langle \bx_i, \bt \rangle, y_i)\rvert\}.
\end{align}
Next, by Gram-Schmidt, 
\[
\proj_{S} \bx_i  = \frac{\langle \bx_i, \btstar \rangle}{\| \btstar \|_2^2}\btstar + \frac{\langle \bx_i, \proj_{\btstar}^{\perp} \bt \rangle}{\| \proj_{\btstar}^{\perp} \bt \|_2^2}\proj_{\btstar}^{\perp} \bt.
\]
Hence, 
\begin{align}
	\label{ineq:bound-expec-norm-projected-x}
\E\{ \| \proj_{S}\bx_i \|_2^2\} = \E\Bigl\{\frac{\langle \bx_i, \btstar \rangle^2}{\|\btstar \|_2^2}\Bigr\} + \E\Bigl\{\frac{\langle \bx_i, \proj_{\btstar}^{\perp} \bt \rangle^2}{\| \proj_{\btstar}^{\perp} \bt \|_2^2}\Bigr\} \overset{\1}{=} 2,
\end{align}
where step $\1$ follows since for any vector $\vb$, $\langle \bx_i, \vb \rangle \sim \NORMAL(0, \| \vb \|_2^2)$.  Note additionally that by Jensen's inequality $\E\| \proj_{S} \bx_i \|_2 \leq \sqrt{\E\| \proj_{S} \bx_i \|_2^2}$.  Thus, substituting the bound~\eqref{ineq:bound-expec-norm-projected-x} into the RHS of the inequality~\eqref{ineq:bound-norm-expec-z-1}, we obtain
\[
\| \E \bz_i \|_2 \leq \sqrt{2} \bigl(\E\{\omega(\langle \bx_i, \bt \rangle, y_i)^2\}\bigr)^{1/2} + \sqrt{2} \E \{ \lvert \omega(\langle \bx_i, \bt \rangle, y_i)\rvert\} \overset{\1}{\leq} CK_1,
\]
where step $\1$ follows by noting that by Assumption~\ref{ass:omega}, $\omega(\langle \bx_i, \bt \rangle, y_i)$ is a sub-Gaussian random variable, and further bounding its moments using the $L^{p}$ characterization of sub-Gaussian random variables~\citet[Proposition 2.5.2 (ii)]{vershynin2018high}.  Taking stock, we have shown that
\begin{align}
	\label{ineq:bound-expected-moments-sum-z}
	\bigl\|\sum_{i \neq j}\E \bz_i \bigr\|_2^{2q} \leq (CK_1 n)^{2q}.
\end{align}

\paragraph{Bounding $\E\bigl\{\bigl\| \sum_{i \neq j} \bz_i - \E \bz_i \bigr\|_2^{2q}\bigr\}$.}  To reduce notation, let
\[
\widebar{\bz}_i := \bz_i - \E \bz_i.
\]
Then, applying~\citet[Theorem 6.20]{ledoux2013probability}, we obtain
\begin{align}
	\label{decomp:moments-sum-z}
	\bigl(\E\bigl\{\bigl\| \sum_{i \neq j} \widebar{\bz}_i \bigr\|_2^{2q}\bigr\}\bigr)^{1/2q} \leq C \frac{2q}{\log{2q}} \bigl(\E\bigl\{\bigl\| \sum_{i \neq j} \widebar{\bz}_i \bigr\|_2\bigr\} + \bigl(\E\bigl\{\max_{i \neq j} \| \widebar{\bz}_i \|_2^{2q}\bigr\} \bigr)^{1/2q}\bigr).
\end{align}
To bound the first term, note that we have
\begin{align*}
	\E\bigl\{\bigl\| \sum_{i \neq j} \widebar{\bz}_i \bigr\|_2\bigr\} \overset{\1}{\leq} \bigl(\E\bigl\{\bigl\| \sum_{i \neq j} \widebar{\bz}_i \bigr\|_2^2\bigr\}\bigr)^{1/2} &\overset{\2}{=} \Bigl( \sum_{i \neq j} \E \| \widebar{\bz}_i \|_2^2 \Bigr)^{1/2} = \Bigl(\sum_{i \neq j, k \in [d]} \E \{X_{ik}^2 \omega(\langle \bx_i, \bt \rangle, y_i)^2 \}\Bigr)^{1/2}.
\end{align*}
Here, we have used $X_{ik}$ to denote the $ik$-th entry of the matrix $\bX$.  Step $\1$ follows by Jensen's inequality and step $\2$ follows since the random vectors $(\widebar{\bz}_i)_{1 \leq i \leq n}$ are zero-mean and independent.  Now, we apply the Cauchy-Schwarz inequality to obtain
\begin{align*}
	\E \{X_{ik}^2 \omega(\langle \bx_i, \bt \rangle, y_i)^2 \} \leq \EE\{X_{ik}^4\}^{1/2} \EE\{ \omega(\langle \bx_i, \bt \rangle, y_i)^4\}^{1/2} \overset{\1}{\leq} CK_1^2,
\end{align*}
where step $\1$ follows by noting that the random variables $X_{ik}$ and $ \omega(\langle \bx_i, \bt \rangle, y_i)$ are sub-Gaussian and subsequently using the $L^{p}$ characterization of sub-Gaussian random variables.  Combining the bounds in the above two displays, we obtain the inequality
\begin{align}
	\label{ineq:sum-centered}
	\E\bigl\{\bigl\| \sum_{i \neq j} \widebar{\bz}_i \bigr\|_2\bigr\} \leq CK_1\sqrt{nd}.
\end{align}
Turning to the next term, let $\widebar{Z}_{ik}$ denote the $ik$-th entry of the matrix $\widebar{\bZ} = [\widebar{\bz_1} \mid \widebar{\bz_2} \mid \dots \mid \widebar{\bz_n}]^{\top}$, and note
\begin{align}
	\label{ineq:max-moment}
	\E\bigl\{\max_{i \neq j} \| \widebar{\bz}_i \|_2^{2q}\bigr\} \overset{\1}{\leq} \sum_{i \neq j} \E \{\| \widebar{\bz}_i \|_2^{2q}\} \overset{\2}{\leq} nd^{q-1}\E\Bigl\{\sum_{k=1}^{d} \widebar{Z}_{ik}^{2q}\Bigr \} \overset{\3}{\leq} (C K_1 qn)^{2q}.
\end{align}
Step $\1$ follows since $\| \widebar{\bz}_i \|_2$ are non-negative random variables, step $\2$ follows by applying Jensen's inequality to the term $(1/d \sum_{i=1}^{d} \widebar{Z}_{ik}^{2})^{q}$, and step $\3$ follows from the $L^{p}$ characterization of sub-exponential random variables.  Finally, using the facts that $d \leq n$ and $q \geq 1$, we substitute inequalities~\eqref{ineq:sum-centered} and~\eqref{ineq:max-moment} into the inequality~\eqref{decomp:moments-sum-z} to obtain the inequality
\begin{align*}
	\bigl(\E\bigl\{\bigl\| \sum_{i \neq j} \widebar{\bz}_i \bigr\|_2^{2q}\bigr\}\bigr)^{1/2q} \leq \frac{C K_1}{\log{2q}}q^2 n.
\end{align*}
Consequently, we have
\begin{align}
	\label{ineq:conclusion-bounding-sum-centered}
	\E\bigl\{\bigl\| \sum_{i \neq j} \bz_i - \E \bz_i \bigr\|_2^{2q}\bigr\} \leq (CK_1q^2 n)^{2q}.
\end{align}

\paragraph{Putting it all together.}

Substituting the inequalities~\eqref{ineq:bound-expected-moments-sum-z} and~\eqref{ineq:conclusion-bounding-sum-centered} into the inequality~\eqref{ineq:initial-lemma-moment-sum-zj} yields
\begin{align*}
	\E\Bigl\{\Bigl\| \sum_{i \neq j} \bz_i\Bigr\|_2^{2q}\Bigr\} \leq (C K_1 q^2n)^{2q},
\end{align*}
and the lemma is proved upon raising each side of the inequality in the display above to the power $1/(2q)$.
\qed

\subsection{Proof of Lemma~\ref{lem:expectation-leave-one-out}: Controlling the bias}
\label{sec:proof-expectation-leave-one-out}
We recall the leave-one-out notation $\bSig_j$ and $\bz_j$~\eqref{eq:shorthand-sigma-z-loo} as well as the empirical update $\mathcal{T}_n(\bt)$~\eqref{eq:empirics-AM} and apply the Sherman-Morrison formula to obtain
\begin{align*}
\mathcal{T}_n(\bt) = (\bX^{\top}\bX)^{-1}\bX^{\top}\omega(\bX \bt, \by) &= \sum_{i=1}^{n} \Bigl(\bSig_{i} + \bx_i \bx_i^{\top}\Bigr)^{-1}\bx_i \omega(\langle \bx_i, \bt \rangle, y_i) \\
&= \sum_{i=1}^{n}\frac{\bSig_i^{-1}\bx_i}{1 + \bx_i^{\top} \bSig_i^{-1} \bx_i}\omega(\langle \bx_i, \bt \rangle, y_i).
\end{align*}
Before proceeding, we state the following lemma, whose proof we provide in Subsection~\ref{sec:proof-quadratic-form-sigma}.
\begin{lemma}
	\label{lem:quadratic-form-sigma}
	Under the setting of Proposition~\ref{prop:concentration-signal}, there exist universal, positive constants $c_0, c$, and $C$ such that for all $c_0/\sqrt{n} < t \leq 1$,
\begin{align*}
\Pr\Bigl\{\Bigl \lvert \bx_i^{\top}\bSig_{i}^{-1} \bx_i - \frac{d}{n - d -2} \Bigr \rvert \geq t\Bigr\} \leq C \exp\{-c (t\sqrt{n})^{2/3}\}.
\end{align*}
\end{lemma}
Taking this lemma as given, we proceed to prove Lemma~\ref{lem:expectation-leave-one-out}.
Introduce the shorthand 
\[
U_i = \bx_i^{\top} \bSig_i^{-1} \bx_i - \frac{d}{n- d- 2}, 
\]
so that 
\begin{align}
	\label{decomp:quad-form}
\inprod{\btstar}{\mathcal{T}_n(\bt)} = \sum_{i=1}^{n} \frac{(\btstar)^{\top}\bSig_i^{-1} \bx_i}{1 + \frac{d}{n - d- 2} + U_i} \cdot \omega(\langle \bx_i, \bt \rangle, y_i) = T_1 + T_2,
\end{align}
where we let
\begin{align*}
	T_1 &= \sum_{i=1}^{n} \frac{(\btstar)^{\top}\bSig_i^{-1} \bx_i}{1 + \frac{d}{n - d- 2}} \cdot \omega(\langle \bx_i, \bt \rangle, y_i),\\
	T_2 &= \sum_{i=1}^{n} \frac{(\btstar)^{\top}\bSig_i^{-1} \bx_i U_i}{\bigl(1 + \frac{d}{n - d- 2} + U_i\bigr)\bigl(1 + \frac{d}{n - d- 2}\bigr)} \cdot \omega(\langle \bx_i, \bt \rangle, y_i).
\end{align*}
Note that
\begin{align*}
\EE \{T_1\} &\overset{\1}{=} \frac{n - d - 2}{n - 2} \cdot \inprod{\btstar}{\sum_{i=1}^{n} \EE \{\bSig_{i}^{-1}\} \EE \{\bx_i \omega(\langle \bx_i, \bt \rangle, y_i)\}} \\
&\overset{\2}{=} \frac{n(n - d - 2)}{(n-2)(n - d - 2)} \cdot \EE \{\inprod{\btstar}{\bx_i} \cdot \omega(\langle \bx_i, \bt \rangle, y_i)\},
\end{align*}
where step $\1$ follows since $\bSig_i$ and $\bx_i \omega(\langle \bx_i, \bt \rangle, y_i)$ are independent and step $\2$ follows since $\bSig_{i}^{-1}$ follows the inverse Wishart distribution 
with $n-1$ degrees
of freedom and scale matrix $\bI_{d}$.  Consequently, using the fact that $2d \leq n$ in conjunction with sub-Gaussianity of $\omega(\langle \bx_i, \bt \rangle, y_i)$ by Assumption~\ref{ass:omega}, we deduce
\begin{align}
	\label{ineq:quad-form-T1}
	\bigl \lvert \EE \{ T_ 1\} - \EE \{ \inprod{\btstar}{\bx_i} \cdot \omega(\langle \bx_i, \bt \rangle, y_i)\} \bigr \rvert  \leq \frac{C K_1}{n}.
\end{align}
We turn now to bounding the term $T_2$.  Note that the denominator is each summand of $T_2$ is lower bounded by $1$, since $\bSig^{-1}$ is a PSD matrix. This, in conjunction with the triangle inequality, yields
\begin{align*}
	\lvert \EE\{T_2\} \rvert \leq  \sum_{i=1}^{n} \EE\bigl\{\lvert U_i \omega(\langle \bx_i, \bt \rangle, y_i) (\btstar)^{\top} \bSig_{i}^{-1} \bx_i \rvert \bigr\}
	\overset{\1}{\leq} n \sqrt{\EE \{\langle \btstar, \bSig_{i}^{-1} \bx_i\rangle^2\}}\sqrt{\EE \{U_i^2 \omega(\langle \bx_i, \bt \rangle, y_i)^2 \}},
\end{align*}
where
step $\1$ follows from the Cauchy-Schwarz inequality.  Next, we have 
\[
\EE \{ \inprod{\btstar}{\bSig_{i}^{-1} \bx_i}^2\} = \EE\bigl\{\EE\{ \inprod{\btstar}{\bSig_{i}^{-1} \bx_i}^2 \mid \{\bx_j\}_{j \neq i}\}\} \overset{\1}{=} \EE \{\| \bSig_{i}^{-1} \btstar \|_{2}^{2}\} \overset{\2}{\leq} \frac{C}{n^2},
\]
where step $\1$ follows since $\bx_i$ and $\bSig_i^{-1}$ are independent, whence $\bx_i^{\top} \bSig_{i}^{-1}\btstar \sim \NORMAL(0, \| \bSig_{i}^{-1} \btstar \|_2^2)$ and step $\2$ makes use of Lemma~\ref{lem:wishart_mineig_moments} from the appendix.  Once more applying the Cauchy-Schwarz inequality, we obtain
\begin{align*}
\sqrt{\EE \{U_i^2 \omega(\langle \bx_i, \bt \rangle, y_i)^2 \}} \leq (\EE \{U_i^4\})^{1/4} (\EE\{\omega(\langle \bx_i, \bt \rangle, y_i)^4 \})^{1/4} \leq C K_1 (\EE \{U_i^4\})^{1/4},
\end{align*}
where in the last inequality, we have noted that by Assumption~\ref{ass:omega}, $\omega(\langle \bx_i, \bt \rangle, y_i)$ is a sub-Gaussian random variable and subsequently applied the $L^{p}$ characterization of sub-Gaussian random variables.  Now, the integration by parts formula for non-negative random variables implies
\begin{align*}
\EE \{U_i^4\} = \EE \{\lvert U_i \rvert^4\} = \int_{0}^{\infty} 4t^3 \Pr\{\lvert U_i \rvert \geq t\} \mathrm{d}t &\overset{\1}{\leq} \int_{0}^{c_0/\sqrt{n}} 4t^3 \mathrm{d}t + \int_{c_0/\sqrt{n}}^{\infty} 4Ct^3 e^{-c(t\sqrt{n})^{2/3}} \mathrm{d}t \leq \frac{C}{n^2},
\end{align*}
where step $\1$ follows by applying Lemma~\ref{lem:quadratic-form-sigma}.  Putting the pieces together, we obtain
\begin{align}
	\label{ineq:quad-form-T2}
	\lvert \EE \{T_2\} \rvert \leq \frac{CK_1}{\sqrt{n}}.
\end{align}
Finally, substituting the bound on $T_1$~\eqref{ineq:quad-form-T1} and the upper bound on $T_2$~\eqref{ineq:quad-form-T2} into the decomposition~\eqref{decomp:quad-form}, we obtain
\begin{align*}
	\lvert \EE\{\inprod{\btstar}{\mathcal{T}_n(\bt)} \} - \inprod{\btstar}{\mathcal{T}(\bt)} \rvert \leq \frac{C  K_1}{\sqrt{n}},
\end{align*}
as desired.
\qed

\subsubsection{Proof of Lemma~\ref{lem:quadratic-form-sigma}}
\label{sec:proof-quadratic-form-sigma}
We begin with the decomposition
\begin{align*}
	\Bigl \lvert \bx_i^{\top}\bSig_{i}^{-1} \bx_i - \frac{d}{n - d -2} \Bigr \rvert \leq \Bigl \lvert \bx_i^{\top}\bSig_{i}^{-1} \bx_i - \mathsf{Tr}(\bSig_i^{-1}) \Bigr \rvert + \Bigl \lvert \mathsf{Tr}(\bSig_i^{-1}) - \frac{d}{n - d- 2} \Bigr \rvert,
\end{align*}
so that
\begin{align}
	\label{ineq:main-quadratic-form-signal}
\Pr\Bigl\{\Bigl \lvert \bx_i^{\top}\bSig_{i}^{-1} \bx_i - \frac{d}{n - d -2} \Bigr \rvert \geq t\Bigr\} \leq \Pr\Bigl\{ \Bigl \lvert \bx_i^{\top}\bSig_{i}^{-1} \bx_i - \mathsf{Tr}(\bSig_i^{-1}) \Bigr \rvert \geq \frac{t}{2} \Bigr\} + \Pr\Bigl\{\Bigl \lvert \mathsf{Tr}(\bSig_i^{-1}) - \frac{d}{n - d- 2} \Bigr \rvert \geq \frac{t}{2} \Bigr\}.
\end{align}
The result is a consequence of the following claim:
\begin{subequations} \label{eq:suff-claim-lastlemma}
	\begin{align}
		\Pr\Big\{\Bigl \lvert \bx_i^{\top}\bSig_{i}^{-1} \bx_i - \mathsf{Tr}(\bSig_i^{-1}) \Bigr \rvert \geq \frac{t}{2}\Bigr\} &\leq 2\exp\Bigl\{-c\min\Bigl(\frac{(nt)^2}{d},  nt\Bigr)\Bigr\} + 2e^{-cn},\label{ineq:bound-deviation-T1}\\
		\Pr\Bigl\{\Bigl \lvert \mathsf{Tr}(\bSig_i^{-1}) - \frac{d}{n - d- 2} \Bigr \rvert \geq \frac{t}{2} \Bigr\} &\leq C\exp\{-c(t\sqrt{n})^{2/3}\} + e^{-cn}.	\label{ineq:bound-deviation-T2b}
\end{align}
\end{subequations}
Indeed, Lemma~\ref{lem:quadratic-form-sigma} follows immediately from substituting claim~\ref{eq:suff-claim-lastlemma} into inequality~\eqref{ineq:main-quadratic-form-signal} and using the condition $n \geq Cd$ to simplify.

\noindent It remains to prove claim~\eqref{eq:suff-claim-lastlemma}.

\paragraph{Proof of the inequality~\eqref{ineq:bound-deviation-T1}.}
To begin, define the event
\[
\mathcal{A}_i := \{\lambda_{\min} := \lambda_{\min}(\bSig_i) \geq c_2 d\},
\]
and apply~\citet[Theorem 4.6.1]{vershynin2018high} to obtain
\[
\Pr\{\mathcal{A}_i^{c}\} \leq 2e^{-cn}.
\]
We then see that
\begin{align}
	\label{ineq:bound-deviation-T1-initial}
\Pr\Big\{\Bigl \lvert \bx_i^{\top}\bSig_{i}^{-1} \bx_i - \mathsf{Tr}(\bSig_i^{-1}) \Bigr \rvert \geq t\Bigr\} \leq \Pr\Big\{\Bigl \lvert \bx_i^{\top}\bSig_{i}^{-1} \bx_i - \mathsf{Tr}(\bSig_i^{-1}) \Bigr \rvert \geq t, \mathcal{A}_i\Bigr\} + 2e^{-cn}.
\end{align}
Then, for any fixed $\bSig_i$, we note that $\EE_{\bx_i} \{\bx_i^{\top} \bSig_{i}^{-1} \bx_i\} = \mathsf{Tr}(\bSig_{i}^{-1})$ and subsequently apply the Hanson-Wright inequality~\citet[Theorem 6.2.1]{vershynin2018high}, to obtain
\begin{align}
	\label{ineq:bound-deviation-T1-fixed}
	\Pr\Big\{\Bigl \lvert \bx_i^{\top}\bSig_{i}^{-1} \bx_i - \mathsf{Tr}(\bSig_i^{-1}) \Bigr \rvert \geq t\Bigr\} &\leq 2\exp\Bigl\{-c\min\Bigl(\frac{t^2}{\| \bSig_{i}^{-1} \|_{F}^2}, \frac{t}{\| \bSig_{i}^{-1}\|_{\mathsf{op}}}\Bigr)\Bigr\}\nonumber\\
	&\overset{\1}{\leq} 2\exp\Bigl\{-c\min\Bigl(\frac{t^2}{d\| \bSig_{i}^{-1} \|_{\mathsf{op}}^2}, \frac{t}{\| \bSig_{i}^{-1}\|_{\mathsf{op}}}\Bigr)\Bigr\}\nonumber\\
	&= 2\exp\Bigl\{-c\min\Bigl(\frac{ \lambda^2_{\min}  t^2}{d}, \lambda^2_{\min} t\Bigr)\Bigr\},
\end{align}
where step $\1$ follows from the chain of inequalities (which hold for any matrix $\bA$) 
\[
\|\bA \|_{F}^{2} = \mathsf{Tr}(\bA^{\top}\bA) \leq d \| \bA \|_{\mathsf{op}}^2.
\]
Subsequently, 
substituting the inequality~\eqref{ineq:bound-deviation-T1-fixed} into the inequality~\eqref{ineq:bound-deviation-T1-initial}, we obtain
\begin{align*}
	\Pr\Big\{\Bigl \lvert \bx_i^{\top}\bSig_{i}^{-1} \bx_i - \mathsf{Tr}(\bSig_i^{-1}) \Bigr \rvert \geq t\Bigr\} \leq 2\exp\Bigl\{-c\min\Bigl(\frac{(nt)^2}{d},  nt\Bigr)\Bigr\} + 2e^{-cn}.
\end{align*}

\paragraph{Proof of the inequality~\eqref{ineq:bound-deviation-T2b}.}  The proof of this statement follows a similar strategy to that of Lemma~\ref{lem:concentration-leave-one-out}.  In lighten notation, we prove the inequality for the full matrix $\bSig$ rather than the leave-one-out matrix $\bSig_i$. Since $\bSig^{-1}$ follows an inverse Wishart distribution with $n$ degrees of freedom (instead of $n - 1$ for the matrix $\bSig_i^{-1}$), it suffices to show that
\begin{align} \label{eq:equiv-claim-last}
\Pr\Bigl\{\Bigl \lvert \mathsf{Tr}(\bSig^{-1}) - \frac{d}{n - d- 1} \Bigr \rvert \geq \frac{t}{2} \Bigr\} &\leq C\exp\{-c(t\sqrt{n})^{2/3}\} + e^{-cn}.
\end{align}

Define the random variable 
$Z := \mathsf{Tr}(\bSig^{-1})$,
and notice that the Sherman-Morrison formula implies 
\[
Z = \mathsf{Tr}(\bSig_{j}^{-1}) - \frac{1}{1 + \bx_j^{\top} \bSig_{j}^{-1} \bx_j}\mathsf{Tr}(\bSig_{j}^{-1} \bx_j \bx_j^{\top} \bSig_j^{-1}).
\]
Then, for each $j \in [n]$, define the random variable
\[
Z_{j}' := \mathsf{Tr}(\bSig_{j}^{-1}) - \frac{1}{1 + \bx_j'^{\top} \bSig_{j}^{-1} \bx_j'}\mathsf{Tr}(\bSig_{j}^{-1} \bx_j' \bx_j'^{\top} \bSig_j^{-1}),
\]
and note the following inequality, whose proof is analogous to the proof of the inequality~\eqref{eq:Z-moment-ub-V} used in Lemma~\ref{lem:concentration-leave-one-out}:
\begin{align}
	\label{ineq:Z-moment-ub-T2}
	\| Z - \EE Z \|_{q}  \leq C\sqrt{q} \biggl(\sum_{j=1}^{n} \bigl\| \E \bigl\{(Z - Z_j')^2 \big \lvert \{\bx_i\}_{i=1}^{n}\bigr\} \bigr\|_{q} \biggr)^{1/2}, \qquad \text{ for all  integers } q \geq 1. 
\end{align}
Subsequently applying Jensen's inequality to the function $x \mapsto x^{q}$, which is convex on $\mathbb{R}_{+}$ for $q \geq 1$ followed by the tower property of conditional expectation then yields
\[
\bigl\| \E \bigl\{(Z - Z_j')^2 \big \lvert \{\bx_i\}_{i=1}^{n}\bigr\} \bigr\|_{q} \leq \bigl(\E \bigl\{(Z - Z_j')^{2q} \bigr\}\bigr)^{1/q}.
\]
Continuing, we have
\begin{align*}
	\E \bigl\{(Z - Z_j')^{2q} \bigr\} &= \EE\biggl\{\biggl(\frac{\mathsf{Tr}(\bSig_{j}^{-1} \bx_j \bx_j^{\top} \bSig_j^{-1})}{1 + \bx_j^{\top} \bSig_{j}^{-1} \bx_j} - \frac{\mathsf{Tr}(\bSig_{j}^{-1} \bx_j' \bx_j'^{\top} \bSig_j^{-1})}{1 + \bx_j'^{\top} \bSig_{j}^{-1} \bx_j'}\biggr)^{2q} \biggr\}\\
	&\overset{\1}{\leq} 2^{2q} \cdot \EE\biggl\{\biggl(\frac{\mathsf{Tr}(\bSig_{j}^{-1} \bx_j \bx_j^{\top} \bSig_j^{-1})}{1 + \bx_j^{\top} \bSig_{j}^{-1} \bx_j}\biggr)^{2q} \biggr\}\\
	&\overset{\2}{\leq} 2^{2q} \cdot \EE\bigl\{\bigl(\mathsf{Tr}(\bSig_{j}^{-1} \bx_j \bx_j^{\top} \bSig_j^{-1})\bigr)^{2q} \bigr\}.
\end{align*}
Here step $\1$ follows by using the numeric inequality $(a + b)^{2q} \leq 2^{2q - 1}(a^{2q} + b^{2q})$ as well as the fact that $\bx_j, \bx_j'$ are identically distributed. On the other hand, step $\2$ follows since $\bSig_{j}^{-1}$ is positive semidefinite.  We then obtain the following chain of inequalities:
\begin{align}
	\label{ineq:trace-product-q-moment}
	\EE\bigl\{\bigl(\mathsf{Tr}(\bSig_{j}^{-1} \bx_j \bx_j^{\top} \bSig_j^{-1})\bigr)^{2q} \bigr\} = \EE\bigl\{\bigl(\mathsf{Tr}(\bSig_{j}^{-2} \bx_j \bx_j^{\top} )\bigr)^{2q} \bigr\} \overset{\1}{\leq} \EE\bigl\{\| \bSig_j^{-2} \|_{\mathsf{op}}^{2q}\bigr\}\EE\bigl\{\bigl(\mathsf{Tr}(\bx_j \bx_j^{\top})\bigr)^{2q} \bigr\},
\end{align}
where in step $\1$, we used the bound $\Tr(\bA_1 \bA_2) \leq \| \bA_1 \|_{\mathsf{op}} \mathsf{Tr}(\bA_2)$ (which holds as long as $\bA_1$ is positive semidefinite) as well as the fact that $\bSig_j$ and $\bx_j$ are independent.  Now, recall that $X_{jk}$ refers to the $jk$-th entry of the data matrix $\bX$ and note that
\begin{align}
	\label{ineq:trace-q-moment}
	\EE\bigl\{\bigl(\mathsf{Tr}(\bx_j \bx_j^{\top})\bigr)^{2q} \bigr\} = \EE\Bigl\{ \Bigl(\sum_{k=1}^{d} X_{jk}\Bigr)^{2q}\Bigr\} \overset{\1}{\leq} d^{2q} \EE X_{jk}^{4q} \overset{\2}{\leq} (Cdq)^{2q},
\end{align}
where step $\1$ follows by applying Jensen's inequality to the term $\left( (1/d) \cdot \sum_{k=1}^{d}X_{jk} \right)^{2q}$ and step $\2$ follows since $X_{jk} \sim \NORMAL(0, 1)$.  Additionally, we note that Lemma~\ref{lem:wishart_mineig_moments} implies
\begin{align}
	\label{ineq:Sigma-operator-norm-moments}
\EE\bigl\{\| \bSig_j^{-2} \|_{\mathsf{op}}^{2q}\bigr\} = \EE\bigl\{\| \bSig_j^{-1} \|_{\mathsf{op}}^{4q}\bigr\} \leq \Bigl(\frac{C}{n}\Bigr)^{4q}, \qquad \text{ for } 1 \leq q < \frac{n - d - 1}{16}.
\end{align}
Thus, substituting the inequalities~\eqref{ineq:trace-q-moment} and~\eqref{ineq:Sigma-operator-norm-moments} into the RHS of the inequality~\eqref{ineq:trace-product-q-moment} yields
\[
\EE\bigl\{\bigl(\mathsf{Tr}(\bSig_{j}^{-1} \bx_j \bx_j^{\top} \bSig_j^{-1})\bigr)^{2q} \bigr\} \leq \Bigl(\frac{Cq}{n}\Bigr)^{2q}, \qquad \text{ for } 1 \leq q < \frac{n - d - 1}{16}.
\]
Next, substituting the above display into the inequality~\eqref{ineq:Z-moment-ub-T2}, we obtain
\[
\| Z - \EE Z \|_{q} \leq \frac{C q^{3/2}}{\sqrt{n}}, \qquad \text{ for } 1 \leq q < \frac{n - d - 1}{16}.
\]
Now, invoking Lemma~\ref{lem:truncation-tail-bound}, we obtain the tail bound
\begin{align}
	\label{ineq:bound-deviation-T2}
	\Pr\bigl\{\bigl \lvert Z - \EE Z \bigr \rvert \geq t\bigr\} \leq Ce^{-c(t\sqrt{n})^{2/3}} + e^{-cn}.
\end{align}
Finally, recalling that
$\bSig^{-1}$ follows the inverse Wishart distribution with $n$ degrees of freedom and scale matrix $\bI_{d}$, we see that
$\EE Z = \mathsf{Tr}(\bSig^{-1}) - \frac{d}{n - d- 1}$, and this proves claim~\eqref{eq:suff-claim-lastlemma} as desired. 
\qed

%% file: sections/deterministic-convergence/deterministic-convergence_alt.tex

\section{Proofs of results for specific models} \label{sec:proofs-specific}

Recall the shorthand $\rho = \perpcomp / \parcomp$ and $\anglecurr = \tan^{-1} (\rho)$. Also recall our definition of the good region:
\begin{align*}
\Goodset = \{ \bzeta = (\parcomp, \perpcomp): \; 0.55 \leq \parcomp \leq 1.05, \quad \text{ and } \quad \rho \leq 1/5 \}.
\end{align*}
Note that by definition, this ensures that $\perpcomp \leq 0.21$ for all $(\parcomp, \perpcomp) \in \Goodset$.
The shorthand $A_{\noisestd}(\rho) = \frac{2}{\pi}\tan^{-1}\left(\sqrt{\rho^2 + \noisestd^2 + \noisestd^2 \rho^2}\right)$ and $B_{\noisestd}(\rho) = \frac{2}{\pi} \frac{\sqrt{\rho^2 + \noisestd^2 + \noisestd^2\rho^2}}{1 + \rho^2}$ was defined in equation~\eqref{eq:AB-shorthand}. Using these, define the functions
\begin{align*} 
F(\parcomp, \perpcomp) &= 1 - \frac{1}{\pi} (2 \anglecurr - \sin(2\anglecurr)) \\
G(\parcomp, \perpcomp) &= \sqrt{\frac{4}{\pi^2}\sin^4(\anglecurr) + \frac{1}{\kappa - 1}\left(1 - (1 - \frac{1}{\pi}(2\anglecurr - \sin(2\anglecurr)))^2 - \frac{4}{\pi^2}\sin^4(\anglecurr) + \noisestd^2 \right)}, \\
g(\parcomp, \perpcomp) &= \sqrt{\frac{4}{\pi^2} \sin^4 \anglecurr + \frac{1}{\kappa} \Big\{ \parcomp^2 + \perpcomp^2 - 2\parcomp \left(1 - \frac{1}{\pi} (2 \anglecurr - \sin(2 \anglecurr)) \right) - 2\perpcomp \cdot \frac{2}{\pi} \sin^2 \anglecurr + 1 + \noisestd^2 \Big\} }, \\
F_{\noisestd}(\parcomp, \perpcomp) &= 1 - A_{\noisestd}(\rho) + B_{\noisestd}(\rho), \\
G_{\noisestd}(\parcomp, \perpcomp) &=  \sqrt{\rho^2 B_{\noisestd}(\rho)^2 + \frac{1}{\kappa - 1}\left(1 + \noisestd^2 - (1 - A_{\noisestd}(\rho) + B_{\noisestd}(\rho))^2 - \rho^2 B_{\noisestd}(\rho)^2\right)}, \text{ and } \\
g_{\noisestd}(\parcomp, \perpcomp) &= \sqrt{ \rho^2 B_{\noisestd}(\rho)^2
 + \frac{1}{\kappa}\left\{ \parcomp^2 + \perpcomp^2 - 2\parcomp(1 - A_{\noisestd}(\rho) + B_{\noisestd}(\rho)) - 2 \perpcomp \rho B_{\noisestd}(\rho) + 1 + \noisestd^2 \right\} }.
\end{align*}
The pair $(F, G)$ denotes the $(\parcompgordon, \perpcompgordon)$ map for the alternating minimization update for phase retrieval, $(F, g)$ the map for subgradient descent in phase retrieval with stepsize $\eta = 1/2$, $(F_{\noisestd}, G_{\noisestd})$ the map for alternating minimization for mixtures of regressions, and $(F_{\noisestd}, g_{\noisestd})$ the map for subgradient method in mixtures of regressions with stepsize $\eta = 1/2$.
With this notation defined, we collect some preliminary lemmas.

\subsection{Preliminary lemmas} \label{sec:prelim-lemmas}

The first two lemmas collect properties of the maps defined above, and are proved in Sections~\ref{app:pf-maps} and~\ref{app:pf-gradmaps} of the appendix, respectively.
A key consequence of these lemmas is that we obtain bounds on the the derivatives of the $\parcompgordon$ and $\perpcompgordon$ maps when evaluated for any element in this region.

\begin{lemma} \label{lem:map-ineqs}
Suppose $(\parcomp, \perpcomp, \noisestd, \kappa)$ are all nonnegative scalars. There is a universal positive constant $C$ such that the maps above satisfy the following relations.
\begin{enumerate}[label=(\alph*)]
\item For all $(\parcomp, \perpcomp)$ pairs, we have $\left( 1 - \frac{4\anglecurr^3}{3\pi} \right) \lor 0 \leq F_0 (\parcomp, \perpcomp)\leq 1$. Additionally, if $\rho \leq 1/5$, then $1 - F_0 \geq \frac{2}{5} \cdot \anglecurr^3$.
\item For all $(\parcomp, \perpcomp)$ pairs satisfying $\rho \geq 2$ and $\perpcomp \leq 1$, we have $F_0(\parcomp, \perpcomp) \geq 1.06 \parcomp$.
\item For all $(\parcomp, \perpcomp)$ pairs, $F_{\noisestd}(\parcomp, \perpcomp)$ is non-decreasing in $\noisestd$ and $F_{\noisestd}(\parcomp, \perpcomp) \leq 1 + \frac{2\noisestd^3}{3\pi}$ for all $\noisestd \geq 0$.
\item If $\noisestd \leq 0.5$ and $\kappa \geq C$, then for all $(\parcomp, \perpcomp)$ pairs, we have
\begin{align*}
\frac{1}{\sqrt{\kappa - 1}} (1 - [F_{\noisestd}(\parcomp, \perpcomp)]^2)^{1/2} \leq G_{\noisestd}(\parcomp, \perpcomp) \leq 0.8.
\end{align*}
\item If $\kappa \geq C$, then for all $(\parcomp, \perpcomp) \in \Goodset$, we have $[G_0(\parcomp, \perpcomp)]^2 \leq \frac{\anglecurr^3}{10}$.
\item For all $\noisestd \geq 0$, we have
\begin{align*}
\frac{\kappa - 1}{\kappa} \cdot [G_{\noisestd} (\parcomp, \perpcomp)]^2 \leq [g_{\noisestd}(\parcomp, \perpcomp)]^2 \leq [G_{\noisestd} (\parcomp, \perpcomp)]^2 + \frac{2}{\kappa} \left( (1 - \parcomp)^2 + \perpcomp^2 \right) + \frac{2}{\kappa} [\rho B_{\noisestd}(\rho)]^2 + \frac{2}{\kappa} [1 - F_{\noisestd}(\parcomp, \perpcomp)]^2.
\end{align*}
\item If $\noisestd \leq 0.5$ and $\kappa \geq C$, then for all $\rho \leq 2$, we have
\begin{align*}
\left( 1 + \frac{2\noisestd^3}{3\pi} \right)^{-1} \cdot \sqrt{ [\rho B_{\noisestd}(\rho)]^2 \cdot \frac{\kappa - 2}{\kappa - 1} + \frac{\noisestd^2}{2 (\kappa - 1)}} \leq \frac{G_{\noisestd}(\parcomp, \perpcomp)}{F_{\noisestd}(\parcomp, \perpcomp)} \leq \frac{4}{5} \cdot \rho + \frac{2\noisestd}{\sqrt{\kappa - 1}}.
\end{align*}
\end{enumerate}
\end{lemma}

\begin{lemma} \label{lem:grad-map-ineqs}
Suppose $(\parcomp, \perpcomp, \noisestd, \kappa)$ are all nonnegative scalars. There is a universal positive constant $C$ such that the gradients of the maps above satisfy the following relations.
\begin{enumerate}[label=(\alph*)]
\item If $\noisestd \leq 1/2$, we have $\| \nabla F_{\noisestd}(\parcomp, \perpcomp) \|_1 \leq 0.5$ for all $\rho \leq 1/4$ and $\parcomp \geq 1/2$.
\item If $\noisestd \leq 1/2$ and $\kappa \geq C$, we have $\| \nabla G_{\noisestd}(\parcomp, \perpcomp) \|_1 \leq 0.98$ for all $\rho \leq 1/4$ and $\parcomp \geq 1/2$.
\item For all $(\parcomp, \perpcomp)$ pairs, we have $\| \nabla G(\parcomp, \perpcomp) \|_1 \leq \| \nabla G_0 (\parcomp, \perpcomp) \|_1$ and $\| \nabla g(\parcomp, \perpcomp) \|_1 \leq \| \nabla g_0 (\parcomp, \perpcomp) \|_1$.
\item If $\noisestd \leq 1/2$ and $\rho \leq 1/4$, we have $\| \nabla g_{\noisestd} (\parcomp, \perpcomp) \|_1 \leq \| \nabla G_{\noisestd} (\parcomp, \perpcomp) \|_1 + \frac{1}{\sqrt{\kappa}} \left( 3 + \| \nabla F_{\noisestd} (\parcomp, \perpcomp) \|_1 \right)$. 
\end{enumerate}
\end{lemma}

Next, we present two technical lemmas that allow us to argue part (b) and part (c) in our theorems, respectively. These lemmas are proved in Sections~\ref{sec:bproof} and~\ref{sec:cproof} of the appendix, respectively. 
In the first lemma, we show that provided the gradients of the Gordon state evolution updates are bounded above by $1- \tau$ in $\ell_1$, small deviations of the empirics from these maps are not amplified over the course of successive iterations.

\begin{lemma} \label{lem:generalb}
Suppose $\Sopbar = (\Fbar, \Gbar)$ denotes a state evolution operator that is $\Goodset$-faithful. Let $\{ \bzeta_t = (\parcomp_t, \perpcomp_t\}\}_{t = 0}^T$ denote a sequence of state evolution elements satisfying
\begin{align*}
\| \bzeta_{t + 1} - \overline{\Sop}(\parcomp_t, \perpcomp_t) \|_\infty \leq \Delta \text{ for each } 0 \leq t \leq T - 1.
\end{align*}
Also suppose that $\| \nabla \Fbar(\parcomp, \perpcomp) \|_1 \lor \| \nabla \Gbar(\parcomp, \perpcomp) \|_1 \leq 1 - \tau$ for some $\tau > 0$ and all $(\parcomp, \perpcomp) \in \mathbb{B}_{\infty}(\Goodset; \Delta/\tau )$.
Then provided $(\parcomp_0, \perpcomp_0) \in \Goodset$, we have
\begin{align*}
\max_{0 \leq t \leq T} \; \| \bzeta_t - \Sopbar^t(\parcomp_0, \perpcomp_0) \|_\infty \leq \frac{\Delta}{\tau}. 
\end{align*}
\end{lemma}
To state the last lemma, let us state some generic conditions on a state evolution operator $\Sopbar = (\Fbar, \Gbar)$. A subset of these will be used in the lemma.
\begin{enumerate}
\item[C1.] $\Fbar(\parcomp, \perpcomp) \geq (50 \sqrt{d})^{-1}$ for all $(\parcomp, \perpcomp)$ such that $\anglecurr \leq \pi/2 - (50 \sqrt{d})^{-1}$.
\item[C2.] $\Fbar(\parcomp, \perpcomp) \geq 1.06 \cdot \parcomp$ for all $\rho \geq 2$, and $\Fbar(\parcomp, \perpcomp) \geq 0.56$ if $\rho \leq 2$.
\item[C3.] $\frac{\Gbar(\parcomp, \perpcomp)}{\Fbar(\parcomp, \perpcomp)} \leq \frac{7}{8} \cdot \rho$ for all $\frac{1}{5} \leq \rho \leq 2$ and $1/2 \leq \parcomp \leq 3/2$.
\item[C4.] $\Fbar(\parcomp, \perpcomp) \leq 1.04$ for all $(\parcomp, \perpcomp)$ and $\frac{\Gbar(\parcomp, \perpcomp)}{\Fbar(\parcomp, \perpcomp)} \leq 1/6$ for all $\rho \leq 1/5$.
\item[C5a.] $\Gbar(\parcomp, \perpcomp) \leq 0.99$ for all $(\parcomp, \perpcomp)$.
\item[C5b.] $\Gbar(\parcomp, \perpcomp) \leq 0.99$ if $\parcomp \lor \perpcomp \leq 3/2$.
\end{enumerate}
As will be shown in the proof of Lemma~\ref{lem:generalc}, conditions C1 and C2 are useful to ensure that the iterates are boosted from a random initialization to a region in which $\rho \leq 2$. Post that point, we use condition C3 to show that the ratio is boosted further to $\rho \geq 5$. Finally, conditions C4 and---depending on context---one of C5a/b are used to show that one more step of the operator pushes the iterates into the good region.

We also state two possible assumptions on the initialization $(\parcomp_0, \perpcomp_0)$, where the second assumption is strictly stronger than the first. These will be used in conjunction with conditions C1 and C2 to handle the first few iterates of the algorithm from a random initialization.
\begin{enumerate}
	\item[Ia.] $\parcomp_0/\perpcomp_0 \geq (50 \sqrt{d})^{-1}$.
	\item[Ib.] $\parcomp_0/\perpcomp_0 \geq (50 \sqrt{d})^{-1}$ and $\parcomp_0 \lor \perpcomp_0 \leq 3/2$.
\end{enumerate}
Having stated the various assumptions, we are now in a position to state Lemma~\ref{lem:generalc}.

\begin{lemma} \label{lem:generalc}
There is a universal constant $c > 0$ such that the following is true.
Let 
\[
t_0 := \log_{1.05} (50 \sqrt{d}) + \log_{55/54} (10) + 2. 
\]
Suppose $\Sopbar = (\Fbar, \Gbar)$ denotes a state evolution operator and that there there exists a sequence of elements $\{ \bzeta_t = (\parcomp_t, \perpcomp_t)\}_{t \geq 0}$ satisfying 
\begin{subequations}
\begin{align}
\max_{0 \leq t \leq t_0} \; | \parcomp_{t+1} - \Fbar(\parcomp_t, \perpcomp_t) | &\leq \frac{c}{\sqrt{d}}, \text{ and } \label{eq:par-condition} \\
\max_{0 \leq t \leq t_0} \; | \perpcomp_{t + 1} - \Gbar(\parcomp_t, \perpcomp_t) | &\leq c. \label{eq:perp-condition}
\end{align}
\end{subequations}
\noindent (a) If $\Sopbar$ satisfies conditions C1-C4 and C5a and the initialization $(\parcomp_0, \perpcomp_0)$ satisfies condition Ia, then
\begin{align*}
\bzeta_t \in \Goodset \text{ for some } t \leq t_0.
\end{align*}

\noindent (b) If $\Sopbar$ satisfies conditions C1-C4 and C5b and the initialization $(\parcomp_0, \perpcomp_0)$ satisfies condition Ib, then
\begin{align*}
\bzeta_t \in \Goodset \text{ for some } t \leq t_0.
\end{align*}
\end{lemma}
With all of these lemmas stated, we are now in a position to prove the various theorems. Before proceeding to this, we make one remark about the proofs of part (a) of these theorems, in particular the transient period.

\begin{remark} \label{rem:transient}
Some of our bounds---especially the lower bounds on convergence rates---rely on an explicit relation between the quantities $|1 - \parcomp|$ and $\perpcomp$ that comes from the Gordon state evolution. This is the reason why these bounds require a transient period of $1$ iteration: Once the Gordon update is run for just one iteration, the requisite relationship can be ensured. 
\end{remark}

\subsection{Proof of Theorem~\ref{thm:AM-PR}: Alternating minimization for phase retrieval}

We prove each step of the theorem in turn. It is useful to note that 
\begin{align} \label{eq:PR-MLR-equiv}
F(\parcomp, \perpcomp) = F_0(\parcomp, \perpcomp) \quad \text{ and } \quad G(\parcomp, \perpcomp)^2 = G_0(\parcomp, \perpcomp)^2 + \frac{\noisestd^2}{\kappa - 1}.
\end{align}

\subsubsection{Part (a): Convergence of Gordon state evolution in good region}

In order to establish this part of the theorem, it suffices to show that the Gordon state evolution is $\Goodset$-faithful, and to prove the upper and lower bounds on its one-step convergence behavior.

\paragraph{Verifying that $\Sopgordon$ is $\Goodset$-faithful:} We must show that if the pair $(\parcomp, \perpcomp)$ satisfies $0.55 \leq \parcomp \leq 1.05$ and $\rho \leq 1/5$, then $0.55 \leq F(\parcomp, \perpcomp) \leq 1.05$ and $G(\parcomp, \perpcomp)/F(\parcomp, \perpcomp) \leq 1/5$. We show each of these bounds separately.

\noindent \underline{Bounding $F(\parcomp, \perpcomp)$:} Applying Lemma~\ref{lem:map-ineqs}(a), we have $0.95 \leq F_0(\parcomp, \perpcomp) \leq 1$, where the lower bound follows since $\anglecurr \leq \tan^{-1}(1/5)$.  Using equation~\eqref{eq:PR-MLR-equiv} finishes the claim.

\noindent \underline{Bounding $\frac{G(\parcomp, \perpcomp)}{F(\parcomp, \perpcomp)}$:} From equation~\eqref{eq:PR-MLR-equiv} and Lemma~\ref{lem:map-ineqs}(g), we have
\begin{align} \label{eq:ratiobd-AM-PR}
\frac{G(\parcomp, \perpcomp)}{F(\parcomp, \perpcomp)} \leq \frac{G_0(\parcomp, \perpcomp) + \noisestd/\sqrt{\kappa - 1}}{F_0(\parcomp/\perpcomp)} \leq \frac{4}{5} \cdot \rho + \frac{\noisestd}{F_0(\parcomp, \perpcomp) \cdot \sqrt{\kappa - 1}}.
\end{align}
Now from Lemma~\ref{lem:map-ineqs}(a), we have $F_0 \geq 0.95$ if $\rho \leq 1/5$. Furthermore, $\noisestd^2/\kappa \leq c$ for a small enough constant $c$. Putting together the pieces completes the proof.

\paragraph{Establishing upper bound on one-step distance:} Equation~\eqref{eq:PR-MLR-equiv} and Lemma~\ref{lem:map-ineqs}(e) directly yield that if $(\parcomp, \perpcomp) \in \Goodset$ and $\kappa \geq C$ and $\noisestd^2 / \kappa \leq c$, then
\begin{align}
[G(\parcomp, \perpcomp)]^2 &\leq \frac{\anglecurr^3}{10} + \frac{\noisestd^2}{\kappa - 1} \leq \frac{6\perpcomp^3}{7} + \frac{\noisestd^2}{\kappa - 1},
\end{align}
where the last inequality follows since $\parcomp \geq 0.55$.
On the other hand, equation~\eqref{eq:PR-MLR-equiv} and Lemma~\ref{lem:map-ineqs}(a) together yield the bound
\begin{align}
(1 - F(\parcomp, \perpcomp))^2 \leq \frac{16}{9 \pi^2} \anglecurr^6 \leq \perpcomp^3 / 100, \label{eq:xi-ineq}
\end{align}
where the final inequality follows since $\anglecurr \leq 1/5$ and $\parcomp \geq 0.55$. 
Putting together the pieces, we have
\begin{align*}
[\DeltaSE(\mathcal{S}_{\gor}(\bzeta))]^2 = [G(\parcomp, \perpcomp)]^2 + (1 - F(\parcomp, \perpcomp))^2 \leq \perpcomp^3 + \frac{\noisestd^2}{\kappa - 1} \leq \left\{ \perpcomp^2 + (1 - \parcomp)^2 \right\}^{3/2} + \frac{\noisestd^2}{\kappa - 1},
\end{align*}
and the desired upper bound follows from the elementary inequality $\sqrt{a + b} \leq \sqrt{a} + \sqrt{b}$.

\paragraph{Establishing lower bound on two-step distance:} Given that we are interested in a transient period of $t_0 = 1$ (see Remark~\ref{rem:transient}), let us now compute two steps of the Gordon update, letting $F_{+} = F^2(\parcomp, \perpcomp)$ and $G_+ = G^2(\parcomp, \perpcomp)$. Analogously, we let $F = F(\parcomp, \perpcomp)$ and $G = G(\parcomp, \perpcomp)$, and use $\anglecurr_+ = \tan^{-1} (G/F)$ to denote the angle after one step of the Gordon update. Recall that $\rho = \tan \anglecurr = \perpcomp/\parcomp$. Combining equation~\eqref{eq:PR-MLR-equiv} and Lemma~\ref{lem:map-ineqs}(d), we have
\begin{align*}
G_{+}^2 &\geq \frac{1}{\kappa - 1} \cdot (1 - F_+^2 ) + \frac{\noisestd^2}{\kappa - 1} \geq \frac{2}{5(\kappa - 1)} \left( \frac{G}{\sqrt{1 + G^2}} \right)^3 + \frac{\noisestd^2}{\kappa - 1} \geq \frac{1}{4(\kappa - 1)\pi} G^3 + \frac{\noisestd^2}{\kappa - 1}.
\end{align*}
Here, the penultimate inequality uses Lemma~\ref{lem:map-ineqs}(a) and the facts that $\anglecurr_+ \geq \sin \anglecurr_+ = \frac{G}{\sqrt{F^2 + G^2}}$ and $F \leq 1$. The last inequality makes use of $G \leq 1$. Turning now to the $F_+$ component, we use $\rho \leq 1/5$ and Lemma~\ref{lem:map-ineqs}(a) to obtain
\begin{align} \label{eq:Fmap-stuff}
\pi^2 (1 - F_+)^2  \geq \left( \frac{G}{\sqrt{1 + G^2}} \right)^6 \geq \frac{G^6}{8}.
\end{align}
Putting together the pieces yields
\begin{align*}
G_{+}^2 + (1 - F_+)^2 &\geq c_{\kappa} \cdot G^3 + \frac{\noisestd^2}{\kappa - 1} + \frac{G^6}{8\pi^2}  \\
&\geq c_{\kappa} \cdot G^3 + \frac{\noisestd^2}{\kappa - 1} + \frac{1}{8\pi^2} \left( \frac{1}{\kappa - 1} \cdot (1 - F^2 ) \right)^3 + \frac{1}{8\pi^2} \left( \frac{\noisestd^2}{\kappa - 1} \right)^3 \\
&\geq c_{\kappa} \cdot G^3 + \frac{1}{8\pi^2 (\kappa - 1)^3} \cdot (1 - F)^3 +  \frac{\noisestd^2}{\kappa - 1} \\
&\geq \left( c'_{\kappa} \cdot \left\{ G^2 + (1 - F)^2 \right\}^{3/4} + \frac{\noisestd}{2\sqrt{\kappa - 1}} \right)^2
\end{align*}
where the second inequality uses Lemma~\ref{lem:map-ineqs}(d) and equation~\eqref{eq:PR-MLR-equiv}, and the last step follows because $(A + B)^{\kappa} \leq 2^{\kappa} (A^\kappa + B^{\kappa})$ for any positive scalars $(A, B)$ and $\kappa \geq 1$. Taking square roots completes the proof.
\qed

\subsubsection{Part (b): Empirical error is sharply tracked by Gordon state evolution}

As mentioned before, the proof of this result relies on Lemma~\ref{lem:generalb}, and so we dedicate our effort towards verifying the assumptions required to apply it. We set $\Delta_0 = 1/2000$ and $\tau_0 = 1/50$ for convenience in computation, so that $\Delta_0 / \tau_0 = 0.025$.

\paragraph{Verifying gradient conditions:} The first step is to verify that the gradients of the Gordon state evolution maps are bounded as desired. It is easy to verify that for all $\bzeta = (\parcomp, \perpcomp) \in \mathbb{B}_{\infty}(\Goodset; \Delta_0/ \tau_0)$, we have $\parcomp \geq 1/2$ and $\perpcomp/\parcomp \leq 1/4$. Consequently, parts (a) and (b) of Lemma~\ref{lem:grad-map-ineqs} yield that
\begin{align*}
\| \nabla F(\parcomp, \perpcomp) \|_1 \lor \| \nabla G(\parcomp, \perpcomp) \|_1 \leq 0.98 = 1 - \tau_0
\end{align*}
for all $(\parcomp, \perpcomp) \in \mathbb{B}_{\infty} (\Goodset; \Delta_0/ \tau_0)$.

\paragraph{Defining the iterates:}
Put $\bt_t = \mathcal{T}_n^t(\bt)$ for each $t \geq 1$ with the convention that $\bt_0 = \bt$, and let $(\parcomp_t, \perpcomp_t) = (\parcomp(\bt_t), \perpcomp(\bt_t))$ for each $t \geq 0$. By Corollary~\ref{lem:Gordon-SE-AM-PR}(b), we have that with probability exceeding $1 - n^{-10}$,
\begin{align*}
| \parcomp_{t + 1} - F(\parcomp, \perpcomp) | \lor | \perpcomp_{t + 1} - G(\parcomp, \perpcomp) | \leq C_\noisestd \left( \frac{\log^7 n}{n} \right)^{1/4} =: \Delta_{n} \leq \Delta_0,
\end{align*}
where the final inequality follows for $n \geq C'_\noisestd$. Taking a union bound over $t = 0, \ldots, T-1$, we see that
\begin{align*}
\max_{0 \leq t \leq T - 1} \; \| \bzeta_{t + 1} - \Sopgordon(\parcomp_t, \perpcomp_t) \|_\infty \leq \Delta_n.
\end{align*}
with probability greater than $1 - T n^{-10}$.

\paragraph{Putting together the pieces:} Applying Lemma~\ref{lem:generalb} along with the conditions verified above, we have that provided $(\parcomp_0, \perpcomp_0) \in \Goodset$,
\begin{align*}
\max_{1 \leq t \leq T} \; | \parcomp_t - F^t(\parcomp_0, \perpcomp_0) | \lor |\perpcomp_t - G^t(\parcomp_0, \perpcomp_0)| \leq 50 \Delta_n
\end{align*}
with probability exceeding $1 - Tn^{-10}$.
Note that $\| \mathcal{T}_n^t(\bt_0) - \thetastar \|^2 = (1 - \parcomp_{t})^2 + \perpcomp_t^2$, and $[\DeltaSE(\mathcal{S}^t_{\gor}(\bzeta_0))]^2 = [1 - F^t(\parcomp_0, \perpcomp_0)]^2 + [G^t(\parcomp_0, \perpcomp_0)]^2$. Consequently, for each $1 \leq t \leq T$, we have
\begin{align*}
\Big| \| \mathcal{T}_n^t (\theta) - \thetastar \| - \DeltaSE(\mathcal{S}^t_{\gor}(\bzeta) \Big| \leq  \sqrt{2} \left\{ |\parcomp_t - F^t(\parcomp_0, \perpcomp_0) | \lor |\perpcomp_t - G^t(\parcomp_0, \perpcomp_0)| \right\} \lesssim \Delta_{n},
\end{align*}
as desired.
\qed

\subsubsection{Part (c): Iterates converge to good region from random initialization}

The proof of this result relies on Lemma~\ref{lem:generalc}(a), and we will apply it for the Gordon map $(F, G)$ playing the role of $(\Fbar, \Gbar)$.

\noindent \underline{Verifying condition C1:} Letting $\zeta = \pi/2 - \anglecurr$, note that $F(\parcomp, \perpcomp) = \frac{2\zeta + \sin(2 \zeta)}{\pi}$. This is clearly an increasing function of $\zeta$ in the range $0 \leq \zeta \leq \pi/2$, and greater than $(50 \sqrt{d})^{-1}$ when $\zeta = (50 \sqrt{d})^{-1}$.

\noindent \underline{Verifying condition C2:} From Lemma~\ref{lem:map-ineqs}(b) and (d), we have $F(\parcomp, \perpcomp) \geq 1.06 \parcomp$. We also have $F(\parcomp, \perpcomp) \geq 0.56$ for $\rho = 2$, and $F$ is a non-decreasing function of $\rho$.

\noindent \underline{Verifying condition C3:} The bound~\eqref{eq:ratiobd-AM-PR} yields
\begin{align*}
\frac{G(\parcomp, \perpcomp)}{F(\parcomp, \perpcomp)} \leq \frac{4}{5} \rho + \frac{\noisestd}{F(\parcomp, \perpcomp) \cdot \sqrt{\kappa - 1}}.
\end{align*}
Using $F(\parcomp, \perpcomp) \geq 0.56$ for all $\rho \leq 2$ in conjunction with the fact that $\kappa \geq C$, $\noisestd^2 /\kappa \leq c$, and $\rho \geq 1/5$, we obtain
\[
\frac{4}{5} \rho + \frac{\noisestd}{F(\parcomp, \perpcomp) \cdot \sqrt{\kappa - 1}} \leq \frac{4}{5} \rho + \frac{3}{40} \rho.
\]

\noindent \underline{Verifying condition C4:} We have $F(\parcomp, \perpcomp) \leq 1$ for all $(\parcomp, \perpcomp)$. The inequality $\frac{G(\parcomp, \perpcomp)}{F(\parcomp, \perpcomp)} \leq 1/6$ for $\rho \leq 1/5$ follows from the bound~\eqref{eq:ratiobd-AM-PR} and the inequalities $F(\parcomp, \perpcomp) \geq 0.95$ and  $\noisestd^2 / \kappa \leq c$.

\noindent \underline{Verifying condition C5a:} Clearly, we have $G(\parcomp, \perpcomp) \leq \sqrt{0.8 + \frac{\noisestd^2}{\kappa - 1}} \leq 0.95$, where the last inequality holds for $\noisestd^2 / \kappa \leq c$.

\paragraph{Putting together the pieces:} Note that Corollary~\ref{lem:Gordon-SE-AM-PR}(b) in conjunction with the union bound yields that the empirical updates satisfy 
\begin{align*}
\max_{1 \leq t \leq t_0} \; |\parcomp_t - F^t(\parcomp_0, \perpcomp_0) | \leq C_\noisestd \left( \frac{\log^7 (t_0/\delta)}{n} \right)^{1/2} \; \text{ and } \; \max_{1 \leq t \leq t_0} \; |\perpcomp_t - G^t(\parcomp_0, \perpcomp_0) | \leq C \left( \frac{\log (t_0/\delta)}{n} \right)^{1/4}
\end{align*}
with probability exceeding $1 - \delta$. Furthermore, by assumption, we have $\parcomp_0/\perpcomp_0 \geq (50\sqrt{d})^{-1}$. Thus, applying Lemma~\ref{lem:generalc}(a) yields that if $n \geq C'_\noisestd$ 
for a sufficiently large constant $C'_\sigma$, we have $\mathcal{T}_n^t (\bt_0) \in \Goodset$ on this event, completing the proof.
\qed.

\subsection{Proof of Theorem~\ref{thm:GD-PR}: Subgradient descent for phase retrieval}

Recall that the Gordon update in this case is given by the pair $(F,g)$. It is also useful to note that 
\begin{align} \label{eq:PR-MLR-equiv-g}
g(\parcomp, \perpcomp)^2 = g_0(\parcomp, \perpcomp)^2 + \frac{\noisestd^2}{\kappa}.
\end{align}

\subsubsection{Part (a): Convergence of Gordon state evolution in good region}

As before, it suffices to show that the Gordon state evolution is $\Goodset$-faithful, and to prove the upper and lower bounds on its one-step convergence behavior.

\paragraph{Verifying that $\Sopgordon$ is $\Goodset$-faithful:} The bounds on $F$ were shown already in the proof of Theorem~\ref{thm:AM-PR}. It remains to handle the ratio $g/F$.

\noindent \underline{Bounding $\frac{g(\parcomp, \perpcomp)}{F(\parcomp,\perpcomp)}$:} 
We begin by bounding $g(\parcomp, \perpcomp)$ alone. Using Lemma~\ref{lem:map-ineqs}(f) and equation~\eqref{eq:PR-MLR-equiv-g} together yields
\begin{align} \label{eq:g-start}
g(\parcomp, \perpcomp) \leq \sqrt{G_0(\parcomp, \perpcomp)^2 + \frac{1}{\kappa} \left( 2 [\rho B_0(\rho)]^2 + 2 [1 - F_0(\parcomp, \perpcomp)]^2  + 2(1 - \parcomp)^2 + 2\perpcomp^2 + \noisestd^2 \right) }.
\end{align}
Now note that if $\rho \leq 1/5$, then $\rho B_0(\rho) = \frac{2}{\pi}\sin^2 \anglecurr \leq \frac{1}{13\pi}$ and $0.95 \leq F_0(\parcomp, \perpcomp) \leq 1$, so that
\begin{align} \label{eq:bineq}
[\rho B_0(\rho)]^2 + [1 - F_0(\parcomp, \perpcomp)]^2 \leq 1/300.
\end{align}
Since $(\parcomp, \perpcomp) \in \Goodset$, we also have 
\begin{align} \label{eq:dineq}
(1 - \parcomp)^2 + \perpcomp^2 \leq 3/10. 
\end{align}
Putting together equations~\eqref{eq:bineq} and~\eqref{eq:dineq} with Lemma~\ref{lem:map-ineqs}(g) and the inequality $\sqrt{a + b} \leq \sqrt{a} + \sqrt{b}$ for two positive scalars $(a, b)$, yields
\begin{align}
\frac{g(\parcomp, \perpcomp)}{F(\parcomp, \perpcomp)} &\leq \frac{4}{5} \cdot \rho + \frac{\sqrt{\frac{1}{3(\kappa - 1)}} + \frac{\noisestd}{\sqrt{\kappa - 1}}}{F_0(\parcomp, \perpcomp)} \label{eq:ratiobd-GD-PR},
\end{align}
But from Lemma~\ref{lem:map-ineqs}(a), we have $F_0 \geq 0.95$ if $\rho \leq 1/5$, and furthermore, $\kappa \geq C$ and $\noisestd^2/\kappa \leq c$. Putting together the pieces completes the proof.

\paragraph{Establishing upper bound on one-step distance:} Equation~\eqref{eq:PR-MLR-equiv-g} and Lemma~\ref{lem:map-ineqs} parts (e) and (f) yield that if $(\parcomp, \perpcomp) \in \Goodset$ and $\kappa \geq C$ and $\noisestd^2 / \kappa \leq c$, then
\begin{align}
[g(\parcomp, \perpcomp)]^2 &\leq \frac{\anglecurr^3}{10} + \frac{1}{\kappa} \left( 2 [\rho B_0(\rho)]^2 + 2 [1 - F_0(\parcomp, \perpcomp)]^2  + 2(1 - \parcomp)^2 + 2\perpcomp^2 + \noisestd^2 \right).
\end{align}
When $\anglecurr \leq 1/5$ and $\parcomp \geq 0.55$, Lemma~\ref{lem:map-ineqs}(a) also yields
\begin{align}
(1 - F_0(\parcomp, \perpcomp))^2 \leq \frac{16}{9 \pi^2} \anglecurr^6.
\end{align}
Finally, note that $\rho B_0(\rho) = \frac{2}{\pi} \sin^2 \anglecurr \leq \frac{2\anglecurr^2}{\pi}$.
Putting together the pieces and noting that $\anglecurr \leq 1/5$ and $\parcomp \geq 0.55$, we have
\begin{align*}
[\DeltaSELtwo(\mathcal{S}_{\gor}(\bzeta))]^2 = [g(\parcomp, \perpcomp)]^2 + (1 - F_0(\parcomp, \perpcomp))^2 &\leq \frac{\anglecurr^2}{50} + \frac{1}{10\kappa} \anglecurr^2 + \frac{2}{\kappa} [\DeltaSELtwo(\bzeta)]^2 + \sigma^2/\kappa \\
&\leq \frac{[\DeltaSELtwo(\bzeta)]^2}{10} + \frac{3}{\kappa} [\DeltaSELtwo(\bzeta)]^2 + \sigma^2/\kappa,
\end{align*}
and the desired upper bound follows from choosing $\kappa \geq C$.

\paragraph{Establishing lower bound on one-step distance:} We begin with the following convenient characterization of the $g$ map:
\begin{align}
[g(\parcomp, \perpcomp)]^2 &= \frac{\kappa - 2}{\kappa - 1} \cdot [\rho B_0(\rho)]^2 + \frac{1}{\kappa} \Big\{ (\parcomp - F(\parcomp, \perpcomp))^2 + (\perpcomp - [\rho B_0(\rho)]^2 )^2 + (1 - [F(\parcomp, \perpcomp)]^2) + \noisestd^2 \Big\}, \label{eq:conv-rearrange}
\end{align}
Applying Young's inequality, we have
\begin{align*}
[g(\parcomp, \perpcomp)]^2 &\geq \frac{\kappa - 2}{\kappa - 1} \cdot [\rho B_0(\rho)]^2 \\
&\qquad + \frac{1}{\kappa} \left( \frac{1}{2} (1 - \parcomp)^2 - (1 - F(\parcomp, \perpcomp))^2 + \frac{1}{2} \perpcomp^2 - [\rho B_0(\rho)]^2 + (1 - F(\parcomp, \perpcomp))(1 + F(\parcomp, \perpcomp)) + \noisestd^2 \right) \\
&\geq \frac{\kappa - 3}{\kappa - 1} \cdot [\rho B_0(\rho)]^2 + \frac{1}{\kappa} \left( \frac{1}{2} (1 - \parcomp)^2 + \frac{1}{2} \perpcomp^2 + 2F(\parcomp, \perpcomp) \cdot (1 - F(\parcomp, \perpcomp)) + \noisestd^2 \right) \\
&\geq \frac{1}{2 (\kappa - 1)} [\DeltaSELtwo(\bzeta)]^2 + \frac{\noisestd^2}{\kappa},
\end{align*}
where the final inequality follows since $\kappa \geq C$. The proof follows by noting that
\mbox{$\DeltaSELtwo(\mathcal{S}_{\gor}(\bzeta)) \geq g(\parcomp, \perpcomp)$}.
\qed

\subsubsection{Part (b): Empirical error is sharply tracked by Gordon state evolution}

This proof is almost identical to that of Theorem~\ref{thm:AM-PR}(b), so we only sketch the major difference: verifying the gradient conditions. We set $\Delta_0 = 1/4000$ and $\tau_0 = 1/100$ for convenience in computation, so that $\Delta_0 / \tau_0 = 0.025$.

\paragraph{Verifying gradient conditions:} It is easy to verify that for all $\bzeta = (\parcomp, \perpcomp) \in \mathbb{B}_{\infty} ( \Goodset ; \Delta_0/ \tau_0)$, we have $\parcomp \geq 1/2$ and $\perpcomp/\parcomp \leq 1/4$. Consequently, parts (ii-iv) of Lemma~\ref{lem:grad-map-ineqs} yield that $\| \nabla g(\parcomp, \perpcomp) \|_1 \leq 0.99$, where the final step follows for $\kappa \geq C$. Combining this with Lemma~\ref{lem:grad-map-ineqs}(a), we have
\begin{align*}
\| \nabla F(\parcomp, \perpcomp) \|_1 \lor \| \nabla g(\parcomp, \perpcomp) \|_1 \leq 0.99 = 1 - \tau_0
\end{align*}
for all $(\parcomp, \perpcomp) \in \mathbb{B}_{\infty}(\Goodset ; \Delta_0/ \tau_0)$.

The rest of the proof proceeds identically.
\qed

\subsubsection{Part (c): Iterates converge to good region from random initialization}

This proof is almost identical to that of Theorem~\ref{thm:AM-PR}(c), so we only sketch the differences. In this case, we will apply Lemma~\ref{lem:generalc}(b). Conditions C1 and C2 are verified exactly as before. It remains to verify conditions C3, C4, and C5b.

\noindent \underline{Verifying condition C3:} If $\rho \leq 2$, then $\rho B_0(\rho) = \frac{2}{\pi}\sin^2 \anglecurr \leq 0.85$ and $0.56 \leq 1 - F_0(\parcomp, \perpcomp) \leq 1$, so that
$[\rho B_0(\rho)]^2 + [1 - F_0(\parcomp, \perpcomp)]^2 \leq 1.1$.
Since $\rho \leq 2$ and $\parcomp \leq 3/2$, we also have 
$(1 - \parcomp)^2 + \perpcomp^2 \leq 10$.
Together with the bound~\eqref{eq:g-start} and Lemma~\ref{lem:map-ineqs}(g), this yields
\begin{align*}
\frac{g(\parcomp, \perpcomp)}{F(\parcomp, \perpcomp)} \leq \frac{4}{5} \rho + \frac{2\noisestd + \sqrt{22.2}}{F_0(\parcomp, \perpcomp) \cdot \sqrt{\kappa - 1}}.
\end{align*}
Since $F(\parcomp, \perpcomp) \geq 0.56$ for all $\rho \leq 2$, $\kappa \geq C$, $\noisestd^2 /\kappa \leq c$, and $\rho \geq 1/5$, we obtain
\[
\frac{2\noisestd + \sqrt{22.2}}{F_0(\parcomp, \perpcomp) \cdot \sqrt{\kappa - 1}} \leq \frac{3}{40} \rho.
\]

\noindent \underline{Verifying condition C4:} We have $F(\parcomp, \perpcomp) \leq 1$ for all $(\parcomp, \perpcomp)$, as before.
The inequality $\frac{g(\parcomp, \perpcomp)}{F(\parcomp, \perpcomp)} \leq 1/6$ for $\rho \leq 1/5$ follows from the bound~\eqref{eq:ratiobd-GD-PR} and the inequalities $F(\parcomp, \perpcomp) \geq 0.95$, $\kappa \geq C$, and  $\noisestd^2 / \kappa \leq c$.
\qed.

\noindent \underline{Verifying condition C5b:} Note also that $|\rho B_0(\rho)| \lor |1 - F_0(\parcomp, \perpcomp)| \leq 1$, and that $(1 - \parcomp)^2 + \perpcomp^2 \leq 5/2$ if $\parcomp \lor \perpcomp \leq 3/2$. Combining with equation~\eqref{eq:g-start} and Lemma~\ref{lem:map-ineqs}(d), we have
$g(\parcomp, \perpcomp) \leq \sqrt{0.8 + \frac{9 + \noisestd^2}{\kappa}} \leq 0.95$, where the last inequality holds for $\kappa \geq C$ and $\noisestd^2 / \kappa \leq c$.

\subsection{Proof of Theorem~\ref{thm:AM-MLR}: Alternating minimization for mixtures of regressions}

Recall that the Gordon update in this case is given by the pair $(F_{\noisestd}, G_{\noisestd})$. 

\subsubsection{Part (a): Convergence of Gordon state evolution in good region}

As before, it suffices to show that the Gordon state evolution is $\Goodset$-faithful, and to prove the upper and lower bounds on its one-step convergence behavior.

\paragraph{Verifying that $\Sopgordon$ is $\Goodset$-faithful:} We verify the two inequalities separately.

\noindent \underline{Bounding $F_{\noisestd}(\parcomp, \perpcomp)$:} From Lemma~\ref{lem:map-ineqs}(c), we have $F_0(\parcomp, \perpcomp) \leq F_{\noisestd}(\parcomp, \perpcomp) \leq 1 + \frac{2\noisestd^3}{3\pi} \leq 1.05$, where the final inequality holds since $\noisestd \leq c$. The lower bound on $F_0$ established in the previous proof completes the claim.

\noindent \underline{Bounding $\frac{G_{\noisestd}(\parcomp, \perpcomp)}{F_{\noisestd}(\parcomp,\perpcomp)}$:} From Lemma~\ref{lem:map-ineqs}(g), we directly have
\begin{align}
\frac{G_{\noisestd}(\parcomp, \perpcomp)}{F_{\noisestd}(\parcomp, \perpcomp)} &\leq \frac{4}{5} \cdot \rho + \frac{2\noisestd}{\sqrt{\kappa - 1}}  \label{eq:ratiobd-AM-MLR},
\end{align}
Noting that $\rho \leq 1/5$ and $\noisestd^2/\kappa \leq c$ completes the proof.

\paragraph{Establishing upper bound on one-step distance:} 
From Lemma~\ref{lem:map-ineqs}(g), we have
\begin{align*}
\frac{G_{\noisestd}(\parcomp, \perpcomp)}{F_{\noisestd}(\parcomp, \perpcomp)} &\leq \frac{4}{5} \cdot \frac{\perpcomp}{\parcomp} + \frac{2\noisestd}{\sqrt{\kappa - 1}}.
\end{align*}
 Now applying Lemma~\ref{lem:taninv}(b) yields
\begin{align*}
\tan^{-1} \left( \frac{G_{\noisestd}(\parcomp, \perpcomp)}{F_{\noisestd}(\parcomp, \perpcomp)} \right) \leq \frac{51}{50} \cdot \frac{4}{5} \cdot \tan^{-1} \left( \frac{\perpcomp}{\parcomp} \right) + \frac{2\noisestd}{\sqrt{\kappa - 1}};
\end{align*}
to complete the proof, note that $\DeltaSEangle(\bzeta) = \tan^{-1} (\perpcomp/\parcomp)$ for an element $\bzeta = (\parcomp, \perpcomp)$ of the state-evolution.  

\paragraph{Establishing lower bound on one-step distance:} Using the lower bound in Lemma~\ref{lem:map-ineqs}(g) in conjunction with the assumptions $\kappa \geq C$ and $\noisestd \leq c$, we obtain
\begin{align*}
\frac{G_{\noisestd}(\parcomp, \perpcomp)}{F_{\noisestd}(\parcomp, \perpcomp)} &\geq \frac{\rho B_{\noisestd}(\rho)}{2} + \frac{\noisestd}{1.8\sqrt{\kappa - 1}} \geq c \noisestd\rho + \frac{\noisestd}{1.8\sqrt{\kappa - 1}}.
\end{align*}
The second inequality follows by noting that $B_{\noisestd}(\rho) \geq 2c\noisestd$ for all $\rho$ for some absolute constant $c \in (0, 1]$. Note also that since $\noisestd$ is small, we have $1 - c\noisestd \geq 0.95$, and also that $\sigma^2 / (\kappa - 1) \leq c'$ for a small enough constant $c'$. Putting these together with Lemma~\ref{lem:taninv}(a) yields
\[
\tan^{-1} \left( \frac{G_{\noisestd}(\parcomp, \perpcomp)}{F_{\noisestd}(\parcomp, \perpcomp)} \right) \geq c \noisestd \tan^{-1} \rho + \frac{\noisestd}{2\sqrt{\kappa - 1}},
\]
as desired. 
\qed

\subsubsection{Part (b): Empirical error is sharply tracked by Gordon state evolution}

This proof is almost identical to that of Theorem~\ref{thm:AM-PR}(b), so we only sketch the major difference: verifying the gradient conditions. We set $\Delta_0 = 1/2000$ and $\tau_0 = 1/50$ for convenience in computation, so that $\Delta_0 / \tau_0 = 0.025$.

\paragraph{Verifying gradient conditions:} It is easy to verify that for all $\bzeta = (\parcomp, \perpcomp) \in \mathbb{B}_{\infty} (\Goodset; \Delta_0/ \tau_0)$, we have $\parcomp \geq 1/2$ and $\perpcomp/\parcomp \leq 1/4$. Consequently, parts (a) and (b) of Lemma~\ref{lem:grad-map-ineqs} yield that
\begin{align*}
\| \nabla F_{\noisestd}(\parcomp, \perpcomp) \|_1 \lor \| \nabla G_{\noisestd}(\parcomp, \perpcomp) \|_1 \leq 0.98 = 1 - \tau_0
\end{align*}
for all $(\parcomp, \perpcomp) \in \mathbb{B}_{\infty} (\Goodset; \Delta_0/ \tau_0)$.

The rest of the proof proceeds identically.
\qed

\subsubsection{Part (c): Iterates converge to good region from random initialization}

This proof is almost identical to that of Theorem~\ref{thm:AM-PR}(c), so we only sketch the differences. Conditions C1 and C2 are follow directly from the fact that $F_{\noisestd}(\parcomp, \perpcomp) \geq F_0(\parcomp, \perpcomp)$. It remains to verify conditions C3, C4, and C5a.

\noindent \underline{Verifying condition C3:} Lemma~\ref{lem:map-ineqs}(g) yields
\begin{align*}
\frac{G_{\noisestd}(\parcomp, \perpcomp)}{F_{\noisestd}(\parcomp, \perpcomp)} \leq \frac{4}{5} \rho + \frac{2\noisestd}{\sqrt{\kappa - 1}}.
\end{align*}
Since $1/5 \leq \rho \leq 2$ and $\noisestd^2 /\kappa \leq c$, we obtain
\[
\frac{2\noisestd}{\sqrt{\kappa - 1}} \leq \frac{3}{40} \rho.
\]

\noindent \underline{Verifying condition C4:} As argued before, when $\noisestd \leq c$, we have $F_{\noisestd}(\parcomp, \perpcomp) \leq 1.05$ for all $(\parcomp, \perpcomp)$. The inequality $\frac{G_{\noisestd}(\parcomp, \perpcomp)}{F_{\noisestd}(\parcomp, \perpcomp)} \leq 1/6$ for $\rho \leq 1/5$ follows from Lemma~\ref{lem:map-ineqs}(g) and the inequality $\noisestd \leq c$.

\noindent \underline{Verifying condition C5a.} We have $G_{\noisestd}(\parcomp, \perpcomp) \leq 0.8$ by Lemma~\ref{lem:map-ineqs}(d)
\qed.

\subsection{Proof of Theorem~\ref{thm:GD-MLR}: Subgradient descent for mixtures of regressions}

Recall that the Gordon update in this case is given by the pair $(F_{\noisestd}, g_{\noisestd})$. It is also useful to note that 
\begin{align} \label{eq:PR-MLR-equiv-final}
[g_{\noisestd}(\parcomp, \perpcomp)]^2 = G_{\noisestd}(\parcomp, \perpcomp)^2 + \frac{1}{\kappa} \left((\parcomp - F_{\noisestd}(\parcomp, \perpcomp))^2 + (\perpcomp - \rho B_{\noisestd}(\rho))^2 \right).
\end{align}

\subsubsection{Part (a): Convergence of Gordon state evolution in good region}

As before, it suffices to show that the Gordon state evolution is $\Goodset$-faithful, and to prove the upper and lower bounds on its one-step convergence behavior. Unlike before, we show the lower bound on one-step convergence first, since some steps here are used in the proof of the upper bound.

\paragraph{Verifying that $\Sopgordon$ is $\Goodset$-faithful:} The bounds on $F_{\noisestd}$ were already shown in the previous proof for AM.

\noindent \underline{Bounding $\frac{g_{\noisestd}(\parcomp, \perpcomp)}{F_{\noisestd}(\parcomp,\perpcomp)}$:} Putting together parts (a) and (c) of Lemma~\ref{lem:map-ineqs}, we have $|1 - F_{\noisestd}(\parcomp, \perpcomp)| \lesssim \anglecurr^3 \lor \noisestd^3 \lesssim \rho^3 + \noisestd^3$. 
We also have $\rho B_{\noisestd}(\rho) \lesssim (\rho \land 1) \cdot (\noisestd \lor (\rho \land 1)) \lesssim \noisestd$ and $F_{\noisestd}(\parcomp, \perpcomp) \geq 0.95$ for all $\rho \leq 1/5$.
Combining this with Lemma~\ref{lem:map-ineqs} parts (f) and (g) yields
\begin{align}
\frac{g_{\noisestd}(\parcomp, \perpcomp)}{F_{\noisestd}(\parcomp, \perpcomp)} &\leq \frac{4}{5} \cdot \rho + \frac{2\noisestd}{\sqrt{\kappa}} + \frac{2}{\sqrt{\kappa}} \left( (1 - \parcomp)^2 + \perpcomp^2 \right)^{1/2} + \frac{C}{\sqrt{\kappa}} \left( \noisestd + \rho^3 + \noisestd^3 \right) \notag \\
&\leq \frac{5}{6} \cdot \rho + \frac{C\noisestd }{\sqrt{\kappa}} + \frac{2}{\sqrt{\kappa}} \cdot \DeltaSELtwo(\parcomp, \perpcomp). \label{eq:ratiobd-GD-MLR}
\end{align}
Noting that $\rho \leq 1/5$ and $\noisestd^2/\kappa \leq c$ and setting $\kappa \geq C$ completes the proof.

\paragraph{Establishing lower bound on one-step distance:} Note that
\begin{align}
[g_{\noisestd}(\parcomp, \perpcomp)]^2 &= \frac{\kappa - 2}{\kappa - 1} \cdot [\rho B_{\noisestd}(\rho)]^2 + \frac{1}{\kappa} \Big\{ (\parcomp - F_{\noisestd}(\parcomp, \perpcomp))^2 + (\perpcomp - [\rho B_{\noisestd}(\rho)]^2 )^2 + (1 - [F_{\noisestd}(\parcomp, \perpcomp)]^2) + \noisestd^2 \Big\}, \label{eq:conv-rearrange-s}
\end{align}
Applying Young's inequality, we have
\begin{align}
[g_{\noisestd}(\parcomp, \perpcomp)]^2 &\geq \frac{\kappa - 2}{\kappa - 1} \cdot [\rho B_{\noisestd}(\rho)]^2  \notag \\
&\qquad + \frac{1}{\kappa} \left( \frac{1}{2} (1 - \parcomp)^2 - (1 - F_{\noisestd}(\parcomp, \perpcomp))^2 + \frac{1}{2} \perpcomp^2 - [\rho B_{\noisestd}(\rho)]^2 + (1 - F_{\noisestd}(\parcomp, \perpcomp))(1 + F_{\noisestd}(\parcomp, \perpcomp)) + \noisestd^2 \right) \notag \\
&= \frac{\kappa - 3}{\kappa - 1} \cdot [\rho B_{\noisestd}(\rho)]^2 + \frac{1}{\kappa} \left( \frac{1}{2} (1 - \parcomp)^2 + \frac{1}{2} \perpcomp^2 + 2F_{\noisestd}(\parcomp, \perpcomp) \cdot (1 - F_{\noisestd}(\parcomp, \perpcomp)) + \noisestd^2 \right) \notag \\
&\geq \frac{1}{2 (\kappa - 1)} [\DeltaSELtwo(\bzeta)]^2 + \frac{\noisestd^2}{1.5\kappa}, \notag
\end{align}
where the final inequality follows since $1 - F_{\noisestd} \gtrsim - \noisestd^3$, $\kappa \geq C$ and $\noisestd \leq c$. Dividing both sides of the above inequality by $[F_{\noisestd}(\parcomp, \perpcomp)]^2$ and noting that $[F_{\noisestd}(\parcomp, \perpcomp)]^2 \geq 0.95$ for all $\rho \leq 1/5$, we have
\begin{align} \label{eq:ratio-last}
\frac{[g_{\noisestd}(\parcomp, \perpcomp)]^2}{[F_{\noisestd}(\parcomp, \perpcomp)]^2} \geq \frac{2}{5 (\kappa - 1)} \rho^2 + \frac{\noisestd^2}{1.6\kappa}
\end{align}
where we have also used Lemma~\ref{lem:angle-l2}(a) to conclude that $\DeltaSELtwo(\bzeta) \geq 0.9 \rho$ for all $\rho \leq 1/5$. Now using the inequality $\sqrt{a + b} \geq \sqrt{\frac{a}{1 + c}} + \sqrt{\frac{bc}{1 + c}}$ (valid for any three non-negative scalars $(a, b, c)$) we have
\begin{align*}
\frac{[g_{\noisestd}(\parcomp, \perpcomp)]}{[F_{\noisestd}(\parcomp, \perpcomp)]} \geq c_{\kappa} \rho + \frac{\noisestd^2}{1.8\kappa},
\end{align*}
where $1 - c_{\kappa} \geq 0.9$. Using the fact that $\frac{\noisestd^2}{\kappa} \leq c$ and applying Lemma~\ref{lem:taninv}(a) completes the proof.
\qed

\paragraph{Establishing upper bound on two-step distance:} 
We require an explicit relationship between the parallel and perpendicular components in this proof (see Remark~\ref{rem:transient}), so we use a transient period $t_0 = 1$. For notational convenience, let $(F_{\noisestd}, g_{\noisestd})$ denote the pair $(F_{\noisestd}(\parcomp, \perpcomp), g_{\noisestd}(\parcomp, \perpcomp))$, and let $(F_+, g_+) = (F_{\noisestd}(F_{\noisestd}, g_{\noisestd}), g_{\noisestd}(F_{\noisestd}, g_{\noisestd}))$ denote the element of the state evolution obtained after two steps of the Gordon update. Let $\rho_+ = g_\noisestd/F_{\noisestd}$.
Equation~\eqref{eq:ratiobd-GD-MLR} yields
\begin{align} 
\frac{g_+}{F_+} \leq \frac{5}{6} \cdot \rho_+ + \frac{C\noisestd }{\sqrt{\kappa}} + \frac{2}{\sqrt{\kappa}} \cdot \DeltaSELtwo(F_{\noisestd}, g_{\noisestd}) &\leq  \frac{5}{6} \cdot \rho_+ + \frac{C\noisestd }{\sqrt{\kappa}} + \frac{1}{\sqrt{\kappa}} \cdot (\rho^3 + 2g_{\noisestd})  \label{eq:twostep} \\
&= \frac{5}{6} \cdot \rho_+ + \frac{C\noisestd }{\sqrt{\kappa}} + \rho^2 \cdot \frac{\rho}{\sqrt{\kappa - 1}} +  \frac{2\rho_+ \cdot F_{\noisestd}}{\sqrt{\kappa}} \notag \\
&\leq \frac{5}{6} \cdot \rho_+ + \frac{C\noisestd }{\sqrt{\kappa}} + \frac{\rho_+}{25} +  \frac{2.2\rho_+}{\sqrt{\kappa}} \notag \\
&\leq \frac{7}{8} \cdot \rho_+ + \frac{C\noisestd }{\sqrt{\kappa}}. \notag
\end{align}
Here, the second inequality follows since $\DeltaSELtwo(F_{\noisestd}, g_{\noisestd}) \leq |1 - F_{\noisestd}| + g_{\noisestd} \leq \frac{1}{2} (\rho^3 + \noisestd^3) + g_{\noisestd}$, the third inequality is a consequence of the relation~\eqref{eq:ratio-last} and the fact that $\rho \leq 1/5$ and $F_{\noisestd} \leq 1.1$ (which in turn follows from Lemma~\ref{lem:map-ineqs}(c) and $\noisestd \leq c$). 
Applying Lemma~\ref{lem:taninv}(b) completes the proof.
\qed

\subsubsection{Part (b): Empirical error is sharply tracked by Gordon state evolution}

This proof is almost identical to that of Theorem~\ref{thm:AM-PR}(b), so we only sketch the major difference: verifying the gradient conditions. We set $\Delta_0 = 1/4000$ and $\tau_0 = 1/100$ for convenience in computation, so that $\Delta_0 / \tau_0 = 0.025$.

\paragraph{Verifying gradient conditions:} It is easy to verify that for all $\bzeta = (\parcomp, \perpcomp) \in \mathbb{B}_{\infty}( \Goodset; \Delta_0/ \tau_0)$, we have $\parcomp \geq 1/2$ and $\perpcomp/\parcomp \leq 1/4$. Consequently, parts (a) and (b) of Lemma~\ref{lem:grad-map-ineqs} yield that
\begin{align*}
\| \nabla F_{\noisestd}(\parcomp, \perpcomp) \|_1 \lor \| \nabla g_{\noisestd}(\parcomp, \perpcomp) \|_1 \leq 0.99 = 1 - \tau_0
\end{align*}
for all $(\parcomp, \perpcomp) \in \mathbb{B}_{\infty}(\Goodset; \Delta_0/ \tau_0)$.
The rest of the proof proceeds identically.
\qed

\subsubsection{Part (c): Iterates converge to good region from random initialization}

This proof is almost identical to that of Theorem~\ref{thm:AM-MLR}(c), except that we use Lemma~\ref{lem:generalc}(b). Consequently, we only verify conditions C3, C4, C5b.

\noindent \underline{Verifying condition C3:} For all $1/5 \leq \rho \leq 2$ and $1/2 \leq \parcomp \leq 3/2$, equation~\eqref{eq:ratiobd-GD-MLR} and Lemma~\ref{lem:angle-l2}(b) in the appendix together yield
\begin{align*}
\frac{g_{\noisestd}(\parcomp, \perpcomp)}{F_{\noisestd}(\parcomp, \perpcomp)} &\leq \frac{5}{6} \cdot \frac{\perpcomp}{\parcomp} + \frac{C\noisestd}{\sqrt{\kappa}} + \frac{C}{\sqrt{\kappa}} \cdot \frac{\perpcomp}{\parcomp}.
\end{align*}
Since $1/5 \leq \rho \leq 2$ and $\noisestd^2 /\kappa \leq c$, we obtain
$\frac{C\noisestd}{\sqrt{\kappa}} \leq \frac{3}{40} \rho$ and choosing large enough $\kappa$ completes the proof.

\noindent \underline{Verifying condition C4:} As argued before, when $\noisestd \leq c$, we have $F_{\noisestd}(\parcomp, \perpcomp) \leq 1.05$ for all $(\parcomp, \perpcomp)$.
The inequality $\frac{g_{\noisestd}(\parcomp, \perpcomp)}{F_{\noisestd}(\parcomp, \perpcomp)} \leq 1/6$ for $\rho \leq 1/5$ follows from the bound~\eqref{eq:ratiobd-GD-MLR} and the inequalities $F_{\noisestd}(\parcomp, \perpcomp) \geq 0.95$, $\kappa \geq C$, and  $\noisestd^2 / \kappa \leq c$.

\noindent \underline{Verifying condition C5b:} Note also that the quantities $\rho B_{\noisestd}(\rho)$, $|1 - F_{\noisestd}(\parcomp, \perpcomp)|$, and $(1 - \parcomp)^2 + \perpcomp^2$ are all bounded by an absolute constant $C$ if $\noisestd \leq c$ and $\parcomp \lor \perpcomp \leq 3/2$. Combining with Lemma~\ref{lem:map-ineqs} parts (d) and (f), we have
$g_{\noisestd}(\parcomp, \perpcomp) \leq \sqrt{0.8 + \frac{C + \noisestd^2}{\kappa}} \leq 0.95$, where the last inequality holds for large enough $\kappa$ and small enough $\noisestd$.
\qed.

%% file: sections/appendix/aux_heuristic.tex

\section{Heuristic derivations deferred from Section~\ref{subsec:heuristic-second-order}}

In this section, we collect two calculations that were deferred from Section~\ref{subsec:heuristic-second-order}.

\subsection*{Calculations to obtain equation~\eqref{eq:astar-opt}} \label{app:step4-heuristic}

We begin by applying the Cauchy--Schwarz inequality to maximize over $\bv$, obtaining
\[
\min_{\bt \in \real^d} \mathfrak{L}(\bt; \bt_t, \bh, \bg) = \min_{\bt \in \mathbb{R}^d} \max_{\| \bv \|_2 \leq 1} \frac{\| \bv \|_2}{\sqrt{n}}\bigl(\langle \bh, \proj_{S_t}^{\perp} \bt \rangle + \bigl\|\mathsf{sgn}(\bX \bt_t) \odot \by - \|\proj_{S_t}^{\perp}\bt \|_2 \bg - \bX \proj_{S_t} \bt \bigr\|_2 \bigr).
\]
Note that the objective is linear in the magnitude $\| \bv \|_2$; thus, maximizing over it is straightforward giving
\begin{align*}
\min_{\bt \in \real^d} \mathfrak{L}(\bt; \bt_t, \bh, \bg) &= \min_{\bt \in \mathbb{R}^d} \left( \langle \bh, \proj_{S_t}^{\perp} \bt \rangle + \bigl\|\mathsf{sgn}(\bX \bt_t) \odot \by - \|\proj_{S_t}^{\perp}\bt \|_2 \bg - \bX \proj_{S_t} \bt \bigr\|_2 \bigr)  \right)_{+}.
\end{align*}
Now, using the fact that $\min_{\bt \in \real^d}\left( f(\bt)\right)_+=\left(\min_{\bt \in \real^d} f(\bt)\right)_+$ for any function $f(\cdot)$  and noting that the function of interest in the display above is linear in the direction $\proj_{S_t}^{\perp} \bt/ \| \proj_{S_t}^{\perp} \bt\|_2$, we may minimize over the latter to obtain
\begin{align*}
\min_{\bt \in \real^d} \mathfrak{L}(\bt; \bt_t, \bh, \bg) &=  \left(\min_{\bt \in \mathbb{R}^d} \; -\| \proj_{S_t}^{\perp} \bt \|_2 \frac{\| \proj_{S_t}^{\perp} \bh\|_2}{\sqrt{n}} + \frac{1}{\sqrt{n}}\bigl\|\mathsf{sgn}(\bX \bt_t) \odot \by - \|\proj_{S_t}^{\perp}\bt \|_2 \bg - \bX \proj_{S_t} \bt \bigr\|_2 \right)_{+}.
\end{align*}

When written in this form, the scalarization is apparent; recall our scalars
\begin{align}
\parcomp = \langle \bt, \thetastar \rangle, \qquad \perpone = \frac{\langle \bt, \proj_{\thetastar}^{\perp}\bt_t\rangle}{\| \proj_{\thetastar}^{\perp}\bt_t \|_2}, \qquad \text{ and } \perptwo = \| \proj_{S_t}^{\perp} \bt \|_2,
\end{align}
and the analogous quantities for the current iterate
$\parcomp_t = \langle \bt_t, \thetastar \rangle$ and  
$\perpcomp_t = \| \proj_{\thetastar}^{\perp} \bt_t \|_2$.
Also recall the independent, $n$-dimensional Gaussian random vectors
$\bz_1 = \bX \thetastar$  and  $\bz_2 = \frac{\bX \proj_{\thetastar}^{\perp} \bt_t}{\| \proj_{\thetastar}^{\perp} \bt_t \|_2}$, using which we obtain
\[
\bX\bt_t = \parcomp_t \bz_1 + \perpcomp_t \bz_2, \qquad \by = \lvert \bz_1 \rvert, \qquad \text{ and } \qquad \bX\proj_{S_t}^{\perp} \bt = \parcomp \bz_1 + \perpone \bz_2.
\]
Thus,
\begingroup
\allowdisplaybreaks
\begin{align}
\min_{\bt \in \real^d} \mathfrak{L}(\bt; \bt_t, \bh, \bg) &= \min_{\parcomp \in \mathbb{R}, \perpone \in \mathbb{R}, \perptwo \geq 0} \Bigl( - \perptwo \frac{\| \proj_{S_t}^{\perp} \bh\|_2}{\sqrt{n}} + \frac{1}{\sqrt{n}}\bigl\| \underbrace{\mathsf{sgn}(\parcomp_t \bz_1 + \perpcomp_t \bz_2) \odot \lvert \bz_1 \rvert}_{\bomega_t}  - \perptwo \bg - \parcomp \bz_1 - \perpone \bz_2 \bigr\|_2 \Bigr)_{+}\nonumber\\
&= \min_{\parcomp \in \mathbb{R}, \perpone \in \mathbb{R}, \perptwo \geq 0} \Bigl(- \perptwo \frac{\| \proj_{S_t}^{\perp} \bh\|_2}{\sqrt{n}} + \frac{1}{\sqrt{n}}\bigl\| \bomega_t - \perptwo \bg - \parcomp \bz_1 - \perpone \bz_2 \bigr\|_2 \Bigr)_{+}\nonumber\\
&\overset{\1}{\approx}  \min_{\parcomp \in \mathbb{R}, \perpone \in \mathbb{R}, \perptwo \geq 0}\Bigl( - \frac{\perptwo}{\sqrt{\kappa}}  + \sqrt{ \EE \bigl\{\bigl(\Omega_t   - \perptwo H - \parcomp Z_1 - \perpone Z_2\bigr)^2\bigr\} } \Bigr)_{+},
\end{align}
where step $\1$ follows by concentration of the norms of sub-Gaussian random variables.  

Summarizing, we have 
\begin{align} \label{eq:astar-app}
A_{\star} = \Bigl( \min_{\parcomp \in \mathbb{R}, \perpone \in \mathbb{R}, \perptwo \geq 0} - \frac{\perptwo}{\sqrt{\kappa}}  + \sqrt{ \EE \bigl\{\bigl(\Omega_t   - \perptwo H - \parcomp Z_1 - \perpone Z_2\bigr)^2\bigr\} } \Bigr)_{+},
\end{align}
so that 
\[
A(\bg, \bh) \approx A_{\star}.
\]

\subsection*{Calculations to obtain equation~\eqref{eq:second-step-calc}} \label{app:step5-heuristic}
 Note that the random variable $H$ is zero-mean and independent of $Z_1$ and $Z_2$, whence we obtain
\[
-\frac{\nu}{\sqrt{\kappa}} + \sqrt{\EE\{(\Omega_t - \nu H - \alpha Z_1 - \mu Z_2)^2\}} = -\frac{\nu}{\sqrt{\kappa}} + \sqrt{\EE\{(\Omega_t - \alpha Z_1 - \mu Z_2)^2\} + \nu^2}.
\]
It is evident from the RHS of the display above that the minimizers $\widebar{\parcomp}$ and $\widebar{\perpone}$ are given by
\begin{align*}
	\widebar{\parcomp} = \EE\{ Z_1 \Omega_t\}, \qquad \text{ and } \widebar{\perpone} = \EE\{ Z_2 \Omega_t\}.
\end{align*}
Substituting these back into $A_{\star}$, we obtain
\begin{align*}
	A_{\star} &= \Bigl( \min_{\perptwo \geq 0} -\frac{\perptwo}{\sqrt{\kappa}} + \sqrt{\perptwo^2 + \bigl(\EE\{\Omega_t^2\} - (\EE\{Z_1 \Omega_t\})^2 - (\EE\{Z_2 \Omega_t\})^2 \bigr)}\Bigr)_{+}.
\end{align*}
Minimizing the above in $\perptwo$, we obtain
\[
\widebar{\perptwo} = \sqrt{\frac{\EE\{\Omega_t^2\} - (\EE\{Z_1 \Omega_t\})^2 - (\EE\{Z_2 \Omega_t\})^2}{\kappa - 1}},
\]
and this establishes the claimed scalarization.

%% file: sections/appendix/aux-gordon.tex

\section{Auxiliary proofs for general results, part (a)}
\label{sec:aux-gordon}

In this appendix, we prove the technical lemmas stated in Section~\ref{sec:gordon}.

\subsection{Proofs of technical lemmas in steps 1--3}
In this subsection, we prove each of our technical lemmas used in the proof of Proposition~\ref{prop:first-three-steps}.

\subsubsection{Proof of Lemma~\ref{lem:bilinear}}
\label{subsec:proof-of-bilinear}
We require two additional lemmas that are proved at the end of this subsection.  The first lemma shows that the optimization can be done over a compact set. 
\begin{lemma}
	\label{lem:bounded-optimization}
 	Suppose that the loss $\mathcal{L}$ satisfies Assumption~\ref{ass:loss} and let $\kappa > 1$.  Let $D \subseteq \mathbb{B}_2(R)$ be a closed subset for some constant $R> 0$.  Then, there exists a positive constant $C_1$ depending only on $R$ such that for any scalar $r \geq C_L$ (where $C_L$ denotes the Lipschitz constant in Assumption~\ref{ass:loss}), we have
 		\[
 	\min_{\bt \in D} \Lc(\bt; \btsharp, \bX, \by) = \min_{\substack{\bt \in D \\ \bu \in \mathbb{B}_2(C_1\sqrt{n}) }}\max_{\bv \in \mathbb{B}_2(r)} \frac{1}{\sqrt{n}} \langle \bv, \bX \bt - \bu \rangle + F(\bu, \bt),
 	\]
 	with probability at least $1 - 2e^{-2n}$.  
\end{lemma}

The following lemma uses the bilinear characterization of Lemma~\ref{lem:bounded-optimization} and subsequently invokes the CGMT to connect to the auxiliary loss $\mathfrak{L}_n$. 
\begin{lemma}
	\label{lem:invoke-CGMT}
	Let Assumption \ref{ass:loss} hold and recall the definition of the auxiliary loss $\mathfrak{L}_n$ in Definition \ref{def:aux-gordon-loss}. 
	For any compact set $D$ and any positive scalars $C_1$ and $r$, it holds that
	\[
	\Pro\Biggl\{\min_{\substack{\bt \in D \\ \bu \in \mathbb{B}_2(C_1\sqrt{n}) }}\max_{\bv \in \mathbb{B}_2(r)} \frac{1}{\sqrt{n}} \langle \bv, \bX \bt - \bu \rangle + F(\bu, \bt) \geq t\Biggr\} \leq 2\Pro\Biggl\{\min_{\substack{\bt \in D \\ \bu \in \mathbb{B}_2(C_1\sqrt{n}) }} \mathfrak{L}_n(\bt, \bu;r) \geq t\Biggr\} 
	\]
	If, in addition, $D$ is convex, then
		\[
	\Pro\Biggl\{\min_{\substack{\bt \in D \\ \bu \in \mathbb{B}_2(C_1\sqrt{n}) }}\max_{\bv \in \mathbb{B}_2(r)} \frac{1}{\sqrt{n}} \langle \bv, \bX \bt - \bu \rangle + F(\bu, \bt) \leq t\Biggr\} \leq 2\Pro\Biggl\{\min_{\substack{\bt \in D \\ \bu \in \mathbb{B}_2(C_1\sqrt{n}) }} \mathfrak{L}_n(\bt, \bu;r) \leq t\Biggr\} 
	\]
\end{lemma}

Lemma~\ref{lem:bilinear} follows immediately upon combining Lemmas~\ref{lem:bounded-optimization} and~\ref{lem:invoke-CGMT}. \qed

\paragraph{Proof of Lemma~\ref{lem:bounded-optimization}}
First, recall that Assumption~\ref{ass:loss} implies the existence of a function $F$ such that $\Lc(\bt) = F(\bX \bt, \bt)$.  Consequently, we obtain
\[
\min_{\bt \in D} \Lc(\bt) = \min_{\bt \in D, \bu \in \mathbb{R}^n} F(\bu, \bt) \quad \text{ s.t. } \quad \bu = \bX\bt.
\]
Applying~\citet[Theorem 6.1]{wainwright2019high} in conjunction with the assumptions $\kappa > 1$ and $D \subseteq \mathbb{B}_2(R)$ yields that the event
\[
\mathcal{A} = \{ \sup_{\bt \in D} \; \| \bX\bt \|_2 \leq 4R \sqrt{n} \}
\]
occurs with probability at least $1 - 2e^{-2n}$.  We carry out the rest of the proof on this event. For 
$C_1 > 4R$, we obtain
\begin{align}
	\label{eq:bounded-optimization-intermediate}
\min_{\bt \in D} \Lc(\bt) &= \min_{\substack{\bt \in D \\ \bu \in \mathbb{B}_2(C_1\sqrt{n}) }} F(\bu, \bt) \quad \text{ s.t. }\quad \bu = \bX\bt\nonumber\\
&= \min_{\substack{\bt \in D \\ \bu \in \mathbb{B}_2(C_1 \sqrt{n}) }} \max_{\bv \in \mathbb{R}^n} \; F(\bu, \bt) - \langle \bv, \bu - \bX \bt \rangle.
\end{align}
It remains only to prove that $\bv$ can further be constrained to a large enough ball.
 To this end, recall that by Assumption~\ref{ass:loss}(b), the function $F(\bu, \bt)$ is $C_L/\sqrt{n}$-Lipschitz in its first argument.  We thus obtain the inequality
 \begin{align}
 \min_{\substack{\bt \in D \\ \bu \in \mathbb{B}_2(C_1\sqrt{n}) }} \max_{\bv \in \mathbb{B}_2(r)} F(\bu, \bt) - \frac{1}{\sqrt{n}}\langle \bv, \bu - \bX \bt \rangle &\overset{\1}{\geq} \min_{\substack{\bt \in D \\ \bu \in \mathbb{B}_2(C_1\sqrt{n}) }} F(\bX \bt, \bt) + \max_{\bv \in \mathbb{B}_2(r)}  \biggl(\frac{\| \bv \|_2}{\sqrt{n}} - \frac{C_L}{\sqrt{n}}\biggr)\| \bu - \bX\bt \|_2\nn\\
 &\overset{\2}{\geq} \min_{\substack{\bt \in D \\ \bu \in \mathbb{B}_2(C_1\sqrt{n}) }} F(\bX \bt, \bt) = \min_{\substack{\bt \in D }} \Lc(\bt).
  \label{eq:bounded-optimization-low}
 \end{align} 
Step $\1$ follows by utilizing the Lipschitz continuity of $F$ in its first argument in conjunction with 
the Cauchy--Schwarz inequality
and step $\2$ follows from the assumption $r \geq C_L$.
On the other hand, it also holds that
 \begin{align}
 \label{eq:bounded-optimization-up}
 \min_{\substack{\bt \in D \\ \bu \in \mathbb{B}_2(C_1\sqrt{n}) }} \max_{\bv \in \mathbb{B}_2(r)} F(\bu, \bt) - \frac{1}{\sqrt{n}}\langle \bv, \bu - \bX \bt \rangle &\leq  \min_{\substack{\bt \in D \\ \bu \in \mathbb{B}_2(C_1\sqrt{n}) }} \max_{\bv \in \R^n} F(\bu, \bt) - \frac{1}{\sqrt{n}}\langle \bv, \bu - \bX \bt \rangle \nonumber\\
 &= \;\;\;\;\;\min_{\bt \in D} \Lc(\bt),
  \end{align} 
  where the equality holds due to equation \eqref{eq:bounded-optimization-intermediate}.
The desired result follows immediately by combining equations~\eqref{eq:bounded-optimization-low} and~\eqref{eq:bounded-optimization-up}.
\qed

\paragraph{Proof of Lemma~\ref{lem:invoke-CGMT}}
First, recall that by Assumption~\ref{ass:loss}, $F(\bu, \bt)$ can be decomposed as 
\[
F(\bu, \bt) = g(\bu,  \bX \btsharp; \by) + h(\bt, \btsharp).
\]
Next, recall the subspace $S_{\sharp} = \mathsf{span}(\btstar, \btsharp)$ and consider the orthogonal decomposition \\$\bt = \proj_{S_{\sharp}}\bt + \proj_{S_{\sharp}}^{\perp} \bt$.  Combining these two pieces, we obtain the representation
\begin{align*}
\frac{1}{\sqrt{n}} \langle \bv, \bX \bt - \bu \rangle + F(\bu, \bt) &= h(\bt, \btsharp) + g(\bu, \bX\btsharp; \by) - \frac{1}{\sqrt{n}} \langle \bv, \bu \rangle + \frac{1}{\sqrt{n}} \langle \bv, \bX \proj_{S_{\sharp}}\bt \rangle + \frac{1}{\sqrt{n}} \langle \bv, \bX\proj_{S_{\sharp}}^{\perp} \bt \rangle.
\end{align*}

Note that the Gaussian random variable $\langle \bv, \bX\proj_{S_{\sharp}}^{\perp} \bt \rangle$ is independent of all other randomness in the expression.  Thus, in that term, we replace the random matrix $\bX$ with an independent copy $\bG \in \mathbb{R}^{n \times d}$.  Turning to the variational problem of interest, we have
\begin{align*}
	&\min_{\substack{\bt \in D \\ \bu \in \mathbb{B}_2(C_1\sqrt{n}) }}\max_{\bv \in \mathbb{B}_2(r)} h(\bt, \btsharp) + g(\bu, \bX\btsharp; \by) - \frac{1}{\sqrt{n}} \langle \bv, \bu \rangle + \frac{1}{\sqrt{n}} \langle \bv, \bX \proj_{S_{\sharp}}\bt \rangle + \frac{1}{\sqrt{n}} \langle \bv, \bX\proj_{S_{\sharp}}^{\perp} \bt \rangle \\
	\overset{(d)}{=}&\min_{\substack{\bt \in D \\ \bu \in \mathbb{B}_2(C_1\sqrt{n}) }}\max_{\bv \in \mathbb{B}_2(r)} h(\bt, \btsharp) + g(\bu, \bX\btsharp; \by) - \frac{1}{\sqrt{n}} \langle \bv, \bu \rangle + \frac{1}{\sqrt{n}} \langle \bv, \bX \proj_{S_{\sharp}}\bt \rangle + \frac{1}{\sqrt{n}} \langle \bv, \bG\proj_{S_{\sharp}}^{\perp} \bt \rangle.
\end{align*}
At this juncture, we invoke the CGMT (Proposition~\ref{thm:gordon}) with
\[
P(\bG) = \min_{(\bt, \bu) \in D \times \mathbb{B}_2(C_1\sqrt{n})}\max_{\bv \in \mathbb{B}_2(r)} h(\bt, \btsharp) + g(\bu, \bX\btsharp; \by) - \frac{1}{\sqrt{n}} \langle \bv, \bu \rangle + \frac{1}{\sqrt{n}} \langle \bv, \bX \proj_{S_{\sharp}}\bt \rangle + \frac{1}{\sqrt{n}} \langle \bv, \bG\proj_{S_{\sharp}}^{\perp} \bt \rangle
\]
and 
\[
A(\bg, \bh) = \min_{(\bt, \bu) \in D \times \mathbb{B}_2(C_1\sqrt{n})}
\; \mathfrak{L}_n(\bt,\ub;r).
\]
 \qed

\subsubsection{Proof of Lemma~\ref{lem:scalarize-ao}}
\label{subsec:proof-scalarize-ao}
We begin by defining a few additional optimization problems. Under the setting of Lemma~\ref{lem:scalarize-ao}, recall the map $\Pc$ as defined in Definition~\ref{def:scalarized-set}, and the random vectors $\bz_1$ and $\bz_2$~\eqref{eq:zone-ztwo}. Define the minimum of the variational problem
	\begin{align}\label{eq:def-phi-var}
		\phi_{\mathsf{var},C_1,r} \;= &\min_{\substack{(\parcomp, \perpone, \perptwo) \in \Pc(D), \\ \bu \in \mathbb{B}_2(C_1\sqrt{n}) }} \; h_{\mathsf{scal}}(\parcomp, \perpone, \perptwo, \btsharp) + g(\bu, \bX \btsharp; \by) \notag \\
		& + \max_{0 \leq \rho \leq r}\; \frac{\rho}{\sqrt{n}} \cdot \Bigl( \| \perptwo \cdot  \bg + \parcomp \cdot \bz_1 + \perpone \cdot \bz_2 - \bu \|_2  - \perptwo \cdot  \| \proj_{S_{\sharp}}^{\perp} \bh \|_2 \Bigr),
	\end{align}
where we explicitly track the dependence on the pair $(C_1, r)$ as their scaling will be important in the proof.
Next, define the minimum of the constrained problem
	\begin{align}\label{eq:def-phi-con}
		\phi_{\mathsf{con},C_1} \;= &\min_{\substack{(\parcomp, \perpone, \perptwo) \in \Pc(D), \\ \bu \in \mathbb{B}_2(C_1\sqrt{n}) }}\; h_{\mathsf{scal}}(\parcomp, \perpone, \perptwo, \btsharp) + g(\bu, \bX \btsharp; \by) \notag \\
		&\;\;\;\;\;\;\; \mathrm{ s.t. }\;\;\;  \| \perptwo \cdot  \bg + \parcomp \cdot \bz_1 + \perpone \cdot \bz_2  - \bu \|_2  \leq  \perptwo \cdot  \| \proj_{S_{\sharp}}^{\perp} \bh \|_2.
	\end{align}
	Finally, define the minimum of the following closely related constrained problem, which---as we will show in Lemma \ref{lem:solve-constrained} below---is equal with high probability to the minimum of the scalarized auxiliary loss $\widebar{L}_n$:
	\begin{align} \label{eq:def-phi-scal}
		\phi_{\mathsf{scal},C_1} \;= &\min_{(\parcomp, \perpone, \perptwo) \in \Pc(D)} \max_{\bv \in \mathbb{R}^n} \;\; h_{\mathsf{scal}}(\parcomp, \perpone, \perptwo, \btsharp) - g^*(\bv, \bX \btsharp; \by) + \min_{\bu \in \mathbb{B}_2(C_1\sqrt{n}) } \langle \bu, \bv \rangle \nonumber\\
		&\;\;\;\;\;\;\; \mathrm{ s.t. }\;\;\;  \| \perptwo \cdot  \bg + \parcomp \cdot \bz_1 + \perpone \cdot \bz_2  - \bu \|_2  \leq  \perptwo \cdot  \| \proj_{S_{\sharp}}^{\perp} \bh \|_2.
	\end{align}

We now state two lemmas that establish the relation between the above optimization problems and will prove useful in the proof.
\begin{lemma}
	\label{lem:constrained-equivalence}
	Under the setting of Lemma~\ref{lem:scalarize-ao}, we have
		\[
	\Pro\Bigl\{\phi_{\mathsf{var},C_1,r}\leq \phi_{\mathsf{con},C_1} \leq \phi_{\mathsf{var},C_1,r} + \frac{3 C_L^2 C_1}{r}\Bigr\} \geq 1 - 6e^{-n/2}.
	\]
\end{lemma}

In interpreting Lemma \ref{lem:constrained-equivalence}, note that if the inner maximization over $\rho \geq 0$ in the definition of $\phi_{\mathsf{var}, C_1, r}$~\eqref{eq:def-phi-var} were unbounded, then it would be equivalent to the constrained minimum $\phi_{\mathsf{con}, C_1}$~\eqref{eq:def-phi-con}.  Lemma~\ref{lem:constrained-equivalence} uses Lipschitz continuity of the objective function, which holds thanks to Assumption \ref{ass:loss}, and demonstrates the impact of finite values of the scalar $r$ on the gap between the values of the two optimization problems.

\begin{lemma}\label{lem:solve-constrained}
	Under the setting of Lemma~\ref{lem:scalarize-ao}, we have the inequality
		\[
	\Pro\biggl\{\phi_{\mathsf{scal},C_1} = \min_{(\parcomp,\perpone,\perptwo)\in \Pc(D)} \; \widebar{L}_n(\parcomp, \perpone, \perptwo; \btsharp)\biggr\} \geq 1 - 8e^{-n/2}.
	\]
\end{lemma}

Taking these lemmas as given, the proof of Lemma~\ref{lem:scalarize-ao} consists of three major steps, which we perform in sequence: we (i) scalarize the maximization over $\bv$; (ii) scalarize the minimization over $\bt$; and (iii) optimize over $\bu$.  We proceed now to the execution of these steps.
\medskip

\noindent \underline{Scalarize the maximization over $\bv$.}
Recall the auxiliary loss $\mathfrak{L}_n$ from Definition~\ref{def:aux-gordon-loss} and note that the Cauchy--Schwarz inequality implies that
\[
\min_{\substack{\bt \in D \\ \bu \in \mathbb{B}_2(C_1 \sqrt{n}) }}  \mathfrak{L}_n(\bt, \bu;r) = \min_{\substack{\bt \in D \\ \bu \in \mathbb{B}_2(C_1 \sqrt{n}) }} \max_{0 \leq \rho \leq r} F(\bu, \bt) + \frac{\rho}{\sqrt{n}} \cdot \Bigl(\langle \bh, \proj_{S_{\sharp}}^{\perp} \bt \rangle + \Bigl\| \| \proj_{S_{\sharp}}^{\perp} \bt \|_2 \cdot \bg + \bX\proj_{S_{\sharp}} \bt - \bu \Bigl\|_2 \Bigr).
\]
\medskip

\noindent \underline{Scalarize the minimization in $\bt$.}
We now show how to perform the  minimization  over the \emph{direction} of the projection vector $\proj_{S_{\sharp}}^{\perp} \bt$. We do this in two steps. 
\\
\indent 
First, we argue that we can decouple the minimization over its projection and its norm. For this, note by assumption that $D \subseteq \mathbb{B}_2(R)$ is amenable (recall Definition~\ref{def:amenable-set}), whence for any feasible value of the norm $\| \proj_{S_{\sharp}}^{\perp} \bt \|_2$, the set of feasible directions $\proj_{S_{\sharp}}^{\perp} \bt/\| \proj_{S_{\sharp}}^{\perp} \bt \|_2$ remains the same.  Thus, decoupling is indeed allowed. 
\\
\indent Second, we argue that we can minimize over the \emph{direction} of the projection vector $\proj_{S_{\sharp}}^{\perp} \bt$ despite the inner maximization over the variable $\rho$.  To this end, we note that because the optimal direction is the same irrespective of the choice of $\rho$, we may invoke \citet[Lemma 8]{kammoun2021precise}. In particular, recall that the function $F(\bu, \bt)$ only depends on the vector $\proj_{S_{\sharp}}^{\perp} \bt$ through its norm. Thus, the objective in the preceding display depends on the direction of the vector $\proj_{S_{\sharp}}^{\perp} \bt$ only through the linear term $\rho\langle \bh, \proj_{S_{\sharp}}^{\perp} \bt \rangle$, which,  for any $\rho\geq 0$, is minimized by   setting $\nicefrac{\proj_{S_{\sharp}}^{\perp} \bt}{\|\proj_{S_{\sharp}}^{\perp} \bt \|_2 }= -\nicefrac{\proj_{S_{\sharp}}^{\perp}\bh}{\|\proj_{S_{\sharp}}^{\perp} \bh \|_2 }$. We thus apply \citet[Lemma 8]{kammoun2021precise} to obtain the characterization
\begin{align}
	\label{eq:scalarize-bt-characterization}
\min_{\substack{\bt \in D \\ \bu \in \mathbb{B}_2(C_1 \sqrt{n}) }}  \mathfrak{L}_n(\bt, \bu;r) = &\min_{\substack{\bt \in D \\ \bu \in \mathbb{B}_2(C_1\sqrt{n}) }} \max_{0 \leq \rho \leq r} F(\bu, \bt)\nonumber\\ 
&+ \frac{\rho}{\sqrt{n}} \cdot \Bigl( \Bigl\| \| \proj_{S_{\sharp}}^{\perp} \bt \|_2 \cdot  \bg + \bX\proj_{S_{\sharp}} \bt - \bu \Bigl\|_2  - \| \proj_{S_{\sharp}}^{\perp} \bt \|_2\cdot  \| \bh \|_2 \Bigr).
\end{align}
Recalling the independent random variables $\bz_1$ and $\bz_2$~\eqref{eq:zone-ztwo},
write
\[
\by = f(\bz_1; \bq) + \beps \qquad \text{ and } \qquad \bX\btsharp = \alpha^{\sharp}\bz_1 + \beta^{\sharp} \bz_2.
\]
For each $\bt \in \real^d$, we have the orthogonal decomposition
\[
\bt = \proj_{S_{\sharp}} \bt + \proj_{S_{\sharp}}^{\perp} \bt = \parcomp(\bt) \cdot \btstar + \perpone(\bt) \cdot \frac{\proj_{\btstar}^{\perp} \btsharp}{\| \proj_{\btstar}^{\perp} \btsharp \|_2} + \nu(\bt) \cdot \frac{\proj_{S_{\sharp}}^{\perp} \bt}{\| \proj_{S_{\sharp}}^{\perp} \bt \|_2},
\]
where the second equality follows from the Gram--Schmidt orthogonalization.  Combining this with the characterization~\eqref{eq:scalarize-bt-characterization}, we obtain
\begin{align}
\min_{\substack{\bt \in D \\ \bu \in \mathbb{B}_2(C_1\sqrt{n}) }} \mathfrak{L}_n(\bt, \bu;r) = &\min_{\substack{(\parcomp, \perpone, \perptwo) \in \Pc(D), \\ \bu \in \mathbb{B}_2(C_1\sqrt{n}) }} \; h_{\mathsf{scal}}(\parcomp, \perpone, \perptwo, \btsharp) + g(\bu, \bX \btsharp; \by)\nonumber\\
& + \max_{0 \leq \rho \leq r}\; \frac{\rho}{\sqrt{n}} \cdot \Bigl( \| \perptwo \cdot  \bg + \parcomp \cdot \bz_1 + \perpone \cdot \bz_2 - \bu \|_2  - \perptwo \cdot  \| \proj_{S_{\sharp}}^{\perp} \bh \|_2 \Bigr) \nn
\\&=	\phi_{\mathsf{var},C_1,r}, 	\label{eq:scalar-in-bt}
\end{align}
where the last line follows by definition~\eqref{eq:def-phi-var}.
\medskip

\noindent \underline{Optimize over $\bu$.}
To be able to optimize over $\bu$, we recall the constant $\phi_{\mathsf{con},C_1}$~\eqref{eq:def-phi-con} and invoke Lemma \ref{lem:constrained-equivalence}, which yields the sandwich relation
\begin{align}\label{eq:Lemma7}
\phi_{\mathsf{var},C_1,r} \leq \phi_{\mathsf{con},C_1} \leq \phi_{\mathsf{var},C_1,r}+ \frac{3C_L^2C_1}{r}\,,
\end{align}
with probability at least $1-6e^{-n/2}$.
Next, we linearize the constrained objective of $\phi_{\mathsf{con},C_1}$~\eqref{eq:def-phi-con} in the optimization variable $\bu$.  To this end, recall from Assumption \ref{ass:loss} that $g$ is convex in its first argument and let $g^*$ denote its convex conjugate so that
\begin{align*} 
		\phi_{\mathsf{con},C_1}	\;= &\min_{\substack{(\parcomp, \perpone, \perptwo) \in \Pc(D), \\ \bu \in \mathbb{B}_2(C_1\sqrt{n}) }}\max_{\bv \in \mathbb{R}^n}\; h_{\mathsf{scal}}(\parcomp, \perpone, \perptwo, \btsharp) - g^*(\bv, \bX \btsharp; \by) + \langle \bv, \bu \rangle \\
	&\;\;\;\;\;\;\; \mathrm{ s.t. }\;\;\;  \| \perptwo \cdot  \bg + \parcomp \cdot \bz_1 + \perpone \cdot \bz_2  - \bu \|_2  \leq  \perptwo \cdot  \| \proj_{S_{\sharp}}^{\perp} \bh \|_2.
\end{align*}
Next, we perform three steps in sequence: (i) write the equivalent Lagrangian to the problem above; (ii) note that the minimization over $\bu$ is of a convex function over a compact constraint and the maximization over $\bv$ is of a concave function and invoke Sion's minimax theorem to swap the minimization over $\bu$ with the maximization over $\bv$; and (iii) re-write as a constrained optimization problem to obtain
	\begin{align*}
	\phi_{\mathsf{con},C_1}	= &\min_{(\parcomp, \perpone, \perptwo) \in \Pc(D)} \max_{\bv \in \mathbb{R}^n} \;\; h_{\mathsf{scal}}(\parcomp, \perpone, \perptwo, \btsharp) - g^*(\bv, \bX \btsharp; \by) + \min_{\bu \in \mathbb{B}_2(C_1\sqrt{n}) } \langle \bu, \bv \rangle\\
	&\;\;\;\;\;\;\; \mathrm{ s.t. }\;\;\;  \| \perptwo \cdot  \bg + \parcomp \cdot \bz_1 + \perpone \cdot \bz_2  - \bu \|_2  \leq  \perptwo \cdot  \| \proj_{S_{\sharp}}^{\perp} \bh \|_2 \nn\\
	&=\phi_{\mathsf{scal},C_1}.
\end{align*}
The last line follows by definition of $\phi_{\mathsf{con},C_1}$~\eqref{eq:def-phi-scal}.
We complete the proof by invoking Lemma~\ref{lem:solve-constrained}.\qed

\paragraph{Proof of Lemma~\ref{lem:constrained-equivalence}:}
Note that $\phi_{\mathsf{con},C_1}$ admits the variational representation
	\begin{align*}
	\phi_{\mathsf{con}, C_1} \;= &\min_{\substack{(\parcomp, \perpone, \perptwo) \in \Pc(D), \\ \bu \in \mathbb{B}_2(C_1\sqrt{n}) }}\; h_{\mathsf{scal}}(\parcomp, \perpone, \perptwo, \btsharp) + g(\bu, \bX \btsharp; \by) \\
	& + \max_{\rho \geq 0}\; \frac{\rho}{\sqrt{n}} \cdot \Bigl( \| \perptwo \cdot  \bg + \parcomp \cdot \bz_1 + \perpone \cdot \bz_2 - \bu \|_2  - \perptwo \cdot  \| \proj_{S_{\sharp}}^{\perp} \bh \|_2 \Bigr),
\end{align*}
whence we obtain the inequality $\phi_{\mathsf{con}, C_1} \geq \phi_{\mathsf{var}, C_1, r}$.  The rest of the section is devoted to the proof of the reverse inequality $\phi_{\mathsf{con}, C_1} \leq \phi_{\mathsf{var}, C_1, r} + 3 C_L^2 C_1 / r$.  To this end, fix an arbitrary triple $(\parcomp, \perpone, \perptwo) \in \Pc(D)$.  Given this triple, define the constraint function
\[
\psi(\bu) = \| \perptwo \cdot  \bg + \parcomp \cdot \bz_1 + \perpone \cdot \bz_2 - \bu \|_2  - \perptwo \cdot  \| \proj_{S_{\sharp}}^{\perp} \bh \|_2,
\]
as well as the loss functions
\[
L_n^{\mathsf{var}}(\bu) = \max_{0 \leq \rho \leq r}h_{\mathsf{scal}}(\parcomp, \perpone, \perptwo, \btsharp) + g(\bu, \bX \btsharp; \by) + \frac{\rho}{\sqrt{n}} \cdot \psi(\bu),
\]
and 
\[
L_n^{\mathsf{con}}(\bu) = \max_{\rho \geq 0}\;h_{\mathsf{scal}}(\parcomp, \perpone, \perptwo, \btsharp) + g(\bu, \bX \btsharp; \by) + \frac{\rho}{\sqrt{n}} \cdot \psi(\bu).
\]
Let $\bu^{\var}$ denote an arbitrary minimizer of the loss function $L_n^{\mathsf{var}}$ and suppose that $\bu^{\var}$ does not satisfy the constraint $\psi(\bu^{\var}) \leq 0$---if any such minimizer does satisfy this constraint, then $\phi_{\mathsf{var}, C_1, r} \geq \phi_{\mathsf{con}, C_1}$ and there is nothing to prove.  

The remainder of the proof is thus dedicated to showing the inequality \sloppy \mbox{$\phi_{\var, C_1, r} + 3 C_L^2 C_1 /r \geq \phi_{\con, C_1}$} under the proviso that the minimizer $\bu^{\var}$ satisfies the inequality $\psi(\bu^{\var}) \leq 0$.  To this end, consider the sublevel sets (constrained to a ball)
\[
S_{r'} := \{\bu: \psi(\bu) \leq {r'}\} \cap \mathbb{B}_2(C_1) = \mathbb{B}_2\Bigl(\perptwo \cdot  \bg + \parcomp \cdot \bz_1 + \perpone \cdot \bz_2; \;\;\perptwo \cdot  \| \proj_{S_{\sharp}}^{\perp} \bh \|_2 + {r'}\Bigr).
\]
Note that the set $S_0$ contains all of the feasible points we are interested in.  Next, define the event
\begin{align}
	\mathcal{A}_0 = \{\; \| \bg \|_2 \leq 2\sqrt{n},\; \| \bz_1 \|_2 \leq 2\sqrt{n},\; \| \bz_2 \|_2 \leq 2\sqrt{n}\; \} \label{eq:event_A1},
\end{align}
which, by~\citet[Theorem 3.1.1]{vershynin2018high}, occurs with probability at least $1 - 6e^{-n/2}$.  Working on this event, we note that since $C_1 \geq 6R$ by assumption, the set $S_0$ is non-empty and thus since ${S}_0$ is a non-empty, closed and compact set, projections onto it are well-defined.  

Now, note that ${S}_0$ and $S_{r}'$ are concentric balls, whence if $\bu \in S_{r'}$ and $\widetilde{\bu}$ denotes its projection onto the set ${S}_0$, the distance between the two points satisfies the inequality 
\begin{align}
	\label{ineq:nested-balls}
	\| \bu - \widetilde{\bu} \|_2 \leq {r'}.
\end{align}
We note the following inequality, which we take for granted now and prove at the end of the section,
\begin{align}
	\label{ineq:bound-psi-uvar}
	\psi(\bu^{\var}) \leq \frac{3C_LC_1\sqrt{n}}{r}. 
\end{align}
Consequently, we note the inclusion $\bu^{\var} \in S_{3C_L C_1\sqrt{n}/r}$.  
Letting $\widetilde{\bu}^{\var}$ denote the projection of $\bu^{\var}$ onto the set $S_0$, we obtain the chain of inequalities
\begin{align*}
	L_n^{\var} (\bu^{\var}) \overset{\1}{\geq} L_n^{\con}(\widetilde{\bu}^{\var}) - \frac{C_L \| \bu^{\var} - \widetilde{\bu}^{\var}\|_2}{\sqrt{n}} &\overset{\2}{\geq} L_n^{\con}(\bu^{\con}) - \frac{C_L \| \bu^{\var} - \widetilde{\bu}^{\var}\|_2}{\sqrt{n}}\\
	&\geq L_n^{\con}(\bu^{\con})  - \frac{3C_L^2 C_1}{r}.
\end{align*}
Above, step $\1$ follows since from Assumption \ref{ass:loss}, the function $g$ is $C_L/\sqrt{n}$-Lipschitz in its first argument, step $\2$ follows since $\bu^{\con}$ minimizes $L_n^{\con}$, and the final inequality follows from the inequality bounding distances~\eqref{ineq:nested-balls}, taking $r' = 3C_L C_1 \sqrt{n}/r$.  Taking stock, since the above inequality holds for all values $(\parcomp, \perpone, \perptwo) \in \Pc(D)$, we have shown that on the event $\Ac_0$
\[
\phi_{\con,C_1} - \frac{3C_L^2 C_1}{r} \leq \phi_{\var,C_1, r} \leq \phi_{\con,C_1},
\]
Rearranging this relation completes the proof. It remains to prove the claim~\eqref{ineq:bound-psi-uvar}.
\medskip

\noindent \underline{Proof of the inequality~\eqref{ineq:bound-psi-uvar}.}  Assume for the sake of contradiction that the inequality does not hold and let $\widetilde{\bu}^{\var}$ denote the projection of $\bu^{\var}$ onto the set $S_0$.  Then, since from Assumption \ref{ass:loss} the function $g$ is $C_L/\sqrt{n}$-Lipschitz in its first argument, we obtain the chain of inequalities
\[
L_n^{\var}(\bu^{\var}) \geq L_n^{\con}(\widetilde{\bu}^{\var}) - \frac{C_L \| \bu^{\var} - \widetilde{\bu}^{\var}\|_2}{\sqrt{n}} + \frac{r \psi(\bu^{\var})}{\sqrt{n}} \geq L_n^{\con}(\widetilde{\bu}^{\var})  - 2C_L C_1 + 3C_L C_1 > L_n^{\con}(\widetilde{\bu}^{\var}),
\]
where the penultimate inequality follows since $\bu \in \mathbb{B}_2(C_1\sqrt{n})$.
But the above display contradicts the fact that $\bu^{\var}$ minimizes $L_n^{\var}$, whence we obtain the desired result. \qed

\paragraph{Proof of Lemma~\ref{lem:solve-constrained}:}
Recall the value $\phi_{\mathsf{scal}, C_1}$~\eqref{eq:def-phi-scal}.  Now, we consider a fixed triplet $(\parcomp, \perpone, \perptwo) \in \Pc(D)$ and perform the minimization over $\bu$.  To this end, let
	\begin{align*}
	L_u = \; \min_{\bu \in \mathbb{B}_2(C_1\sqrt{n}) } \langle \bu, \bv \rangle \quad \mathrm{ s.t. } \quad \| \perptwo \cdot  \bg + \parcomp \cdot \bz_1 + \perpone \cdot \bz_2  - \bu \|_2  \leq  \perptwo \cdot  \| \proj_{S_{\sharp}}^{\perp} \bh \|_2.
\end{align*}
Then, introducing a Lagrange multiplier  and invoking Sion's minimax theorem to interchange minimization and maximization, we obtain
\begin{align*}
L_u &= \max_{\lambda \geq 0} \min_{\bu \in \mathbb{B}_2(C_1\sqrt{n})} \;\; \langle \bu, \bv \rangle + \frac{\lambda}{2} \| \perptwo \cdot  \bg + \parcomp \cdot \bz_1 + \perpone \cdot \bz_2  - \bu \|_2^2 - \frac{\lambda \nu^2}{2} \| \proj_{S_{\sharp}}^{\perp} \bh \|_2^2 \\
&\geq  \max_{\lambda \geq 0} \min_{\bu \in \mathbb{R}^n} \;\; \langle \bu, \bv \rangle + \frac{\lambda}{2} \| \perptwo \cdot  \bg + \parcomp \cdot \bz_1 + \perpone \cdot \bz_2  - \bu \|_2^2 - \frac{\lambda \nu^2}{2} \| \proj_{S_{\sharp}}^{\perp} \bh \|_2^2\\
&= \max_{\lambda \geq 0} \;\;\langle \perptwo \cdot  \bg + \parcomp \cdot \bz_1 + \perpone \cdot \bz_2, \bv \rangle - \frac{1}{2\lambda} \| \bv\|_2^2 - \frac{\lambda \nu^2}{2} \| \proj_{S_{\sharp}}^{\perp} \bh \|_2^2 \\
&= \langle \perptwo \cdot  \bg + \parcomp \cdot \bz_1 + \perpone \cdot \bz_2, \bv \rangle  - \perptwo \| \proj_{S_{\sharp}}^{\perp} \bh \|_2 \| \bv \|_2.
\end{align*}  
Conversely, let 
\[
\bar{\bu} = \perptwo \cdot  \bg + \parcomp \cdot \bz_1 + \perpone \cdot \bz_2 - \perptwo \| \proj_{S_{\sharp}}^{\perp} \bh \|_2 \cdot \frac{\bv}{\| \bv \|_2}.
\]
We claim that this choice is feasible with high probability for a large enough constant $C_1$ (which may depend on $R$), which means
we obtain
\[
L_u \leq \langle \bar{\bu}, \bv \rangle =  \langle \perptwo \cdot  \bg + \parcomp \cdot \bz_1 + \perpone \cdot \bz_2, \bv \rangle  - \perptwo \| \proj_{S_{\sharp}}^{\perp} \bh \|_2 \| \bv \|_2.
\]
To see that this is true, condition on the following event (recall the event $\Ac_0$~\eqref{eq:event_A1})
\[
\Ac'_0=\Ac_0\cap\{\;\|\proj_{S_{\sharp}}^{\perp} \bh\|_2\leq 2\sqrt{n}\;\},
\] 
which after applying~\citet[Theorem 3.1.1]{vershynin2018high} holds with $\Pro\{\mathcal{A}_0'\} \geq 1 - {8e^{-n/2}}$.
On this event, apply triangle inequality to obtain 
\[
\bar{\bu}/\sqrt{n}\leq4({\nu+|\alpha|+|\mu|})\leq 8R, 
\]
where in the last inequality, we recalled from the definition of the scalarized set $\Pc(D)$ that $\alpha^2+\mu^2+\nu^2=\|\bt\|_2^2$ and used the assumption that $D\subset \mathbb{B}_2(R)$. 
Evidently, on the event $\Ac'_0$,
\[
L_u =  \langle \perptwo \cdot  \bg + \parcomp \cdot \bz_1 + \perpone \cdot \bz_2, \bv \rangle  - \perptwo \| \proj_{S_{\sharp}}^{\perp} \bh \|_2 \| \bv \|_2.
\]
Thus, we obtain
\begin{align*}
\phi_{\mathsf{scal}, C_1} &= \min_{(\parcomp, \perpone, \perptwo) \in \Pc(D)} \max_{\bv \in \mathbb{R}^n} \;\; h_{\mathsf{scal}}(\parcomp, \perpone, \perptwo, \btsharp) - g^*(\bu, \bX \btsharp; \by) + \langle \perptwo \cdot  \bg + \parcomp \cdot \bz_1 + \perpone \cdot \bz_2, \bv \rangle  - \perptwo \| \proj_{S_{\sharp}}^{\perp} \bh \|_2 \| \bv \|_2\\
&=\min_{(\parcomp,\perpone,\perptwo)\in \Pc(D)} \; \widebar{L}_n(\parcomp, \perpone, \perptwo; \btsharp),
\end{align*}
which concludes the proof. \qed

\subsection{Establishing growth conditions for higher-order methods} \label{sec:growth-HO}
In this subsection, we prove Lemma~\ref{lem:ao-analysis}, specialized to second order methods where the loss function~\eqref{eq:sec_order_gen} takes the form 
\[
\Lc(\bt; \btsharp, \bX, \by) = \frac{1}{\sqrt{n}} \| \omega(\bX\btsharp, \by) - \bX\bt \|_2,
\]
which corresponds to setting 
\[
F(\bu, \bt) = \frac{1}{\sqrt{n}}\| \omega(\bX\btsharp, \by) - \bu \|_2, \quad h(\bt, \btsharp) = 0, \qquad \text{ and } \quad g(\bu, \bX \btsharp; \by) = \frac{1}{\sqrt{n}}\| \omega(\bX\btsharp, \by) - \bu \|_2.
\]
We now state several lemmas, which we will invoke in sequence.  The first specializes the function $\widebar{L}_n$ when the loss corresponds to a higher-order method.
\begin{lemma}
	\label{lem:remove-positive-restriction}
	Let the loss $\Lc$ correspond to a higher-order method as in equation~\eqref{eq:sec_order_gen}, and let $\widebar{L}_n$ denote the corresponding scalarized loss given by Definition~\ref{def:scalarized-ao}. Also recall the pair of random vectors $(\bz_1, \bz_2)$ from equation~\eqref{eq:zone-ztwo}.
	  There exists a universal positive constant $c$ such that with probability at least $1 - 6e^{-cn}$, it holds simultaneously for all scalars $\alpha,\mu\in\R$ and $\nu\geq0$ that
	\[
	\widebar{L}_n(\parcomp, \perpone, \perptwo; \btsharp) = \frac{1}{\sqrt{n}} \| \perptwo \cdot \bg + \parcomp \cdot \bz_1 + \perpone \cdot \bz_2 - \omega(\bX \btsharp, \by) \|_2 - \perptwo \frac{\| \proj_{S_{\sharp}}^{\perp} \bh \|_2}{\sqrt{n}}.
	\]
\end{lemma}
This lemma is proved in Subsection~\ref{subsec:proof-lemma-remove-positive-restriction}.
Before stating the next lemma, we introduce the shorthand 
\[
\qquad \bA:= [\bz_1, \bz_2, \bg] \in \mathbb{R}^{n \times 3}, \qquad \text{ and } \qquad \bv_n := [0, 0, \nicefrac{\| \proj^{\perp}_{S_{\sharp}} \bh\|_2}{\sqrt{n}}]^{\top} \in \mathbb{R}^3,
\]
and
\[
\bomega_{\sharp} = \omega(\bX \btsharp, \by),
\]
so that letting $\bxi = [\parcomp, \perpone, \perptwo]$ and invoking Lemma~\ref{lem:remove-positive-restriction}, we obtain
\[
\widebar{L}_n(\bxi; \btsharp) = \frac{1}{\sqrt{n}} \| \bA \bxi - \bomega_{\sharp}\|_2 - \langle \bv_n, \bxi \rangle.
\]
Now, introduce the variable $\tau$ and write the above display in the form
\[
\widebar{L}_n(\bxi; \btsharp) = \inf_{\tau > 0} \; \underbrace{\frac{\tau}{2} + \frac{1}{2\tau n} \| \bA \bxi - \bomega_{\sharp}\|_2^2 - \langle \bv_n, \bxi \rangle}_{\widetilde{L}_n(\bxi, \tau)}.
\]
Next, recall the parameters $(K_1, K_2)$ from Assumptions~\ref{ass:omega} and~\ref{ass:omega-lb} and fix $\epsilon > 0$. For a universal constant $C > 0$ and  a constant $C_{K_1} > 0$ depending only on $K_1$, define the events
\begin{subequations} \label{eq:three-events}
\begin{align} \label{eq:event1}
\Ac_1 = \Bigl\{(1 - C\epsilon) \cdot \bI_3 \preceq \frac{1}{n} \bA^{\top} \bA \preceq (1 +C\epsilon) \cdot \bI_3\Bigr\},
\end{align}
\begin{align} \label{eq:event2}
	\Ac_2 = \Bigl\{\frac{1}{n^2}\lvert \langle \bg, \bomega_{\sharp} \rangle \rvert^2 \vee \Bigl\lvert \frac{1}{n^2}\langle \bz_1, \bomega_{\sharp}\rangle^2 - (\EE \{ Z_1 \Omega \})^2\Bigr\rvert \vee \Bigl\lvert \frac{1}{n^2}\langle \bz_2, \bomega_{\sharp} \rangle^2 - (\EE\{Z_2 \Omega \})^2 \Bigr \rvert \leq C_{K_1} \cdot \epsilon\Bigr\},
\end{align}
and 
\begin{align} \label{eq:event3}
	\Ac_3 &= \Bigl\{\Bigl\lvert \frac{1}{n} \| \bomega_{\sharp} \|_2^2 - \EE\{\Omega^2\} \Bigr\rvert \leq C_{K_1}\epsilon \Bigr\}.
\end{align}
\end{subequations}
Finally, note that on the event $\Ac_1$, $\min_{\bxi \in \real^3} \| \bA\bxi - \bomega_{\sharp}\|_2 = \| \left( \bI - \bA (\bA^\top \bA)^{-1} \bA^\top \right) \cdot \bomega_{\sharp}\|_2$.
\begin{lemma}
	\label{lem:row-space-A}
	Let Assumptions \ref{ass:omega} and \ref{ass:omega-lb} hold.
	 There exists a positive constant $C_{K_1,K_2}$ depending only on $K_1$ and $K_2$ and universal positive constants $c, C,$ and $C'$ such that for all $C' > \epsilon > 0$, the following hold.\\ 
	 \begin{itemize}
	\item[(a)] Recall events $\Ac_1, \Ac_2, $ and $\Ac_3$ as in equations~\eqref{eq:event1}--\eqref{eq:event3}.  On the event $\Ac_1 \cap \Ac_2 \cap \Ac_3$, we have
	\[
	\Bigl\lvert \frac{1}{\sqrt{n}} \| \left( \bI - \bA (\bA^\top \bA)^{-1} \bA^\top \right) \cdot \bomega_{\sharp}\|_2 - \sqrt{(\EE\{\Omega^2\} - (\EE\{Z_1 \Omega\})^2 - (\EE\{Z_2 \Omega\})^2} \Bigr\rvert \leq C_{K_1,K_2} \cdot \epsilon.
	\]
	\item[(b)] We have $\Pro \{ \Ac_1 \cap \Ac_2 \cap \Ac_3 \} \geq 1 - C e^{-cn \epsilon^2}$. 
	\end{itemize}
\end{lemma}
This lemma is proved in Subsection~\ref{subsec:proof-lemma-row-space-A}.

\begin{lemma}
	\label{lem:minimize-lbar-higher-order}
	Let $\kappa>1$. Let Assumptions \ref{ass:omega} and \ref{ass:omega-lb} hold. There exist positive constant $c_{K_1,K_2}$ depending only on $K_1,K_2$, a positive constant $C_{K_1}$ depending only on $K_1$ and a universal, positive constant $C$ such that the following statements hold with probability at least $ 1 - Ce^{-c_{K_1,K_2}n}$. 
	\begin{itemize}
	\item[(a)] 
	For any $R \geq C_{K_1}$, we have the equivalence
	\begin{subequations}
	\begin{align}\label{eq:interchange-inf-min}
	\min_{\bxi \in \Pc(\mathbb{B}_2(R))} \inf_{\tau > 0}\; \widetilde{L}_n(\bxi, \tau) = \inf_{\tau > 0}	\min_{\bxi \in \Pc(\mathbb{B}_2(R))}\; \widetilde{L}_n(\bxi, \tau).
	\end{align}
	\item[(b)] The optimization problem on the RHS of equation~\eqref{eq:interchange-inf-min} admits unique minimizers
	\begin{align} \label{eq:unique-empirical-xi}
	\bxi_n(\tau) = (\bA^{\top} \bA)^{-1} \bA^{\top} \omega(\bX\btsharp, \by) + \tau \cdot n \cdot (\bA^{\top} \bA)^{-1} \bv_n,
	\end{align}
	and 
	\begin{align} \label{eq:unique-empirical-tau}
	\tau_n = \frac{\frac{1}{\sqrt{n}}\|(\bI_n - \bA(\bA^{\top} \bA)^{-1}\bA^{\top}) \cdot \omega(\bX \btsharp, \by) \|_2}{\sqrt{1 - n\cdot \bv_n^{\top} (\bA^{\top} \bA)^{-1} \bv_n}}.
	\end{align}
	\end{subequations}
\end{itemize}
\end{lemma}
We prove this lemma in Subsection~\ref{subsec:proof-lemma-minimize-lbar}.

\subsubsection{Proof of Lemma~\ref{lem:ao-analysis}(a)} 
The first part of the statement---uniqueness of the minimizer---is an immediate consequence of Lemma~\ref{lem:minimize-lbar-higher-order}; moreover, the unique minimizer is given by $\bxi_n(\tau_n)$.  It remains to prove the concentration properties.  As a preliminary step, we show that the dual variable $\tau_n$ concentrates around a deterministic quantity 
\[
\tau^{\gor} = \sqrt{\Bigl(\frac{\kappa}{\kappa - 1} \Bigr) (\EE\{ \Omega^2\} - (\EE\{ Z_1 \Omega\})^2 - (\EE\{ Z_2 \Omega\})^2)}.
\]
To this end, consider the events $\Ac_1, \Ac_2, \Ac_3$ from equation~\eqref{eq:three-events}, and note that Lemma~\ref{lem:row-space-A}(b) implies that $\Ac_1 \cap \Ac_2 \cap \Ac_3$ occurs with probability greater than $1 - C e^{-cn \epsilon^2}$. We carry out the proof on this event.

\paragraph{Bounding $\tau_n$~\eqref{eq:unique-empirical-tau}.} We proceed in three steps.  First, we bound the numerator; second, we bound the denominator; and finally we combine the two bounds.\\

\noindent\underline{Bounding the numerator.}
This step is immediate, since Lemma~\ref{lem:row-space-A}(a) yields
\[
\Bigl\lvert \frac{1}{\sqrt{n}} \| \left( \bI - \bA (\bA^\top \bA)^{-1} \bA^\top \right) \cdot \bomega_{\sharp}  \|_2 - \sqrt{(\EE\{\Omega^2\} - (\EE\{Z_1 \Omega\})^2 - (\EE\{Z_2 \Omega\})^2} \Bigr\rvert \leq C_{K_1,K_2} \cdot \epsilon.
\]

\noindent\underline{Bounding the denominator.} On the event $\Ac_1$, we note the sandwich relation
\[
D_L \leq \sqrt{1 - n\cdot \bv_n^{\top} (\bA^{\top} \bA)^{-1} \bv_n} \leq D_U,
\]
where we have let
\begin{align*}
	D_L &= \sqrt{1 - (1 + C\epsilon) \|\bv_n\|_2^2}, \qquad \text{ and } \qquad D_U = \sqrt{1 - (1 - C\epsilon) \|\bv_n\|_2^2}.
\end{align*}
Now, consider the event
\[
\Ac_4 = \bigl\{\bigl\lvert \| \bv_n \|_2^2 - \kappa^{-1} \bigr \rvert \leq C\epsilon \bigr\},
\]
noting that an application of Bernstein's inequality implies $\Pro\{\Ac_4\} \geq 1 - 2e^{-cn\epsilon^2}$.  Now, on the event $\Ac_4$, we further obtain the bounds (recalling also the assumption $\kappa>C$), 
\[
D_L \geq \sqrt{1 - \kappa^{-1}} - C'\epsilon, \qquad \text{ and } \qquad D_U \leq \sqrt{1 - \kappa^{-1}} + C'\epsilon.
\]

\noindent\underline{Putting the pieces together to control $\tau_n$.}  Now, we combine the two-sided bounds on both the numerator and denominator to obtain the inequality
\begin{align}\label{ineq:probabalistic-control-taun}
\Pro\bigl\{\bigl \lvert \tau_n - \tau^{\gor} \bigr \rvert \leq C_{K_1,K_2} \cdot \epsilon \bigr\} \geq 1 - Ce^{-cn\epsilon^2}.
\end{align}

\paragraph{Bounding the minimizer $\bxi_n(\tau_n)$~\eqref{eq:unique-empirical-xi}.} First, let $\be_1$ denote the first standard basis vector in $\real^3$ and consider the quantity
\[
\lvert \langle \be_1, \bxi_n(\tau_n) \rangle - \parcomp^{\gor} \rvert = \lvert \langle \be_1, (\bA^{\top} \bA)^{-1} \bA^{\top} \bomega_{\sharp}  + \tau \cdot n \cdot (\bA^{\top} \bA)^{-1} \bv_n\rangle - \parcomp^{\gor} \rvert.
\]
First, note that on the event $\Ac_1$, we can write $n\cdot (\bA^{\top} \bA)^{-1} = \bI_3 + \bB$, where $\| \bB \|_{\mathsf{op}} \leq C\epsilon$.  We thus obtain the decomposition
\[
\lvert \langle \be_1, (\bA^{\top} \bA)^{-1} \bA^{\top} \bomega_{\sharp}  + \tau_n \cdot n \cdot (\bA^{\top} \bA)^{-1} \bv_n\rangle - \parcomp^{\gor} \rvert \leq T_1 + T_2,
\]
where
\begin{align*}
	T_1 &= \Bigl \lvert \frac{1}{n} \langle \be_1, \bA^{\top} \bomega_{\sharp}  \rangle + \tau_n \langle \be_1, \bv_n \rangle  - \alpha^{\gor}\Bigr\rvert\\
	T_2 &= \Bigl \lvert  \frac{1}{n} \langle \be_1, \bB  \bA^{\top} \bomega_{\sharp}  \rangle + \tau_n \cdot \langle \be_1, \bB \bv_n \rangle \Bigr\rvert.
\end{align*}
Now, consider the event
\begin{align*}
	\Ac_5 = \Bigl\{\frac{1}{n}\lvert \langle \bg, \bomega_{\sharp} \rangle \rvert \vee \Bigl\lvert \frac{1}{n}\langle \bz_1, \bomega_{\sharp}\rangle - \EE \{ Z_1 \Omega \}\Bigr\rvert \vee \Bigl\lvert \frac{1}{n}\langle \bz_2, \bomega_{\sharp} \rangle - \EE\{Z_2 \Omega \} \Bigr \rvert \leq C'_{K_1} \epsilon\Bigr\},
\end{align*}
and note that on $\Ac_1 \cap \Ac_2 \cap \Ac_3$, we have
\begin{align}\label{eq:tau_n_again}
	\lvert \tau_n - \tau^{\gor}\rvert \leq C_{K_1,K_2} \cdot \epsilon
\end{align}
for a pair $(C'_{K_1}, C_{K_1,K_2})$ depending only on $K_1$ and only on $K_1,K_2$ respectively. In the following lines, we note that $C_{K_1,K_2}$ may change from line to line, but always depends only on $K_1,K_2$.
Note that
by applying Bernstein's inequality,  we obtain $\Pro\{\Ac_5\} \geq 1 - 6e^{-cn\epsilon^2}$. Onward, we work on the event $\bigcap_{k=1}^{5} \Ac_k$. First, applying the triangle inequality, we obtain the upper bound
\[
T_1 \leq \lvert \EE\{Z_1 \Omega\} - \parcomp^{\gor} \rvert + C_{K_1} \cdot \epsilon.
\]
Next, we have 
\begin{align}
T_2 &\leq \frac{1}{n}\| \bB\|_{\mathsf{op}}  \| \bA^{\top} \bomega_{\sharp}  \|_2 + \frac{1}{n} |\tau_n| \cdot \| \bB\|_{\mathsf{op}}  \| \bv_n  \|_2 \notag \\
&\leq C_{K_1} \cdot \epsilon. \label{eq:Cepsilon}
\end{align}
To establish inequality\eqref{eq:Cepsilon}, we employ the following steps. 
First, we bound $\| \bB \|_{\mathsf{op}}$ on event $\Ac_1$ and bound $\| \bA^{\top} \bomega_{\sharp} \|_2$ on event $\Ac_2$. Next, applying the Cauchy--Schwarz inequality yields $\max\{\EE\{Z_1 \Omega\}, \EE\{Z_2\Omega\}\} \leq  C_{K_1}$.  Third, we bound $\| \bv_n \|_2$ on event $\Ac_4$ and note that $\kappa > C$. Finally, we $|\tau_n|$ by invoking the bound~\eqref{eq:tau_n_again} and note that
$
\tau^{\gor} \leq C_{K_1} 
$
since $\kappa>C$ and $\E\{\Omega^2\}\leq C_{K_1}$.

Summarizing, we have shown that
\begin{subequations}\label{ineq:bound-minimizers-scalarized-ao}
\begin{align}
\Pro\{\lvert \langle \be_1, \bxi_n(\tau_n) \rangle - \parcomp^{\gor} \rvert \leq C_{K_1} \cdot \epsilon\} \geq 1 - Ce^{-cn\epsilon^2}.
\end{align}
Proceeding in a parallel manner, we obtain the two bounds
\begin{align}
\Pro\{\lvert \langle \be_2, \bxi_n(\tau_n) \rangle - \perpone^{\gor} \rvert \leq C_{K_1} \cdot \epsilon\} \geq 1 - Ce^{-cn\epsilon^2},
\end{align} 
\begin{align}
	 \Pro\{\lvert \langle \be_1, \bxi_n(\tau_n) \rangle - \perptwo^{\gor} \rvert \leq  C_{K_1} \cdot \epsilon\} \geq 1 - Ce^{-cn\epsilon^2}.
\end{align}
\end{subequations}

\paragraph{Bounding the minimum $\widebar{L}_n(\bxi_n(\tau_n), \btsharp)$.}  Note that 
\[
\widebar{L}_n(\bxi_n(\tau_n); \btsharp) = \frac{1}{\sqrt{n}} \| \bA\bxi_n(\tau_n) - \bomega_{\sharp} \|_2 - \langle \bv_n, \bxi_n(\tau_n) \rangle.
\]
Now, let $\bxi^{\gor}=[\alpha^\gor,\mu^\gor,\nu^\gor]^T\in\R^3$ and consider the event
\begin{align}\label{eq:xi-concentration}
\Ac_6 = \{\|\bxi_n(\tau_n) - \bxi^{\gor}\|_{\infty} \leq C_{K_1} \cdot \epsilon\},
\end{align}
noting that the inequalities~\eqref{ineq:bound-minimizers-scalarized-ao} imply $\Pro\{\Ac_6\} \geq 1 - Ce^{-cn\epsilon^2}$.  For the remainder of the proof, we will work on the event $\mathcal{A} = \bigcap_{k=1}^{6} \Ac_k$.  Adding and subtracting the quantity $\bxi^{\gor}$ yields
\begin{align*}
	\widebar{L}_n(\bxi_n(\tau_n); \btsharp) = \frac{1}{\sqrt{n}} \| \bA\bxi^{\gor} + \bA (\bxi_n(\tau_n) - \bxi^{\gor}) - \bomega_{\sharp} \|_2 - \langle \bv_n, \bxi^{\gor} \rangle - \langle \bv_n, \bxi_n(\tau_n) - \bxi^{\gor} \rangle.
\end{align*}
Thus, on the event $\Ac$, we obtain
\begin{align}
\Bigl \lvert \widebar{L}_n(\bxi_n(\tau_n); \btsharp)  - \frac{1}{\sqrt{n}} \| \bA\bxi^{\gor} - \bomega_{\sharp} \|_2 + \langle \bv_n, \bxi^{\gor} \rangle \Bigr\rvert \leq C_{K_1} \cdot \epsilon.\label{eq:sec-cost}
\end{align}
Next, we claim that
\begin{align}
\biggl \lvert\frac{1}{\sqrt{n}} \| \bA\bxi^{\gor} - \bomega_{\sharp} \|_2 - \langle \bv_n, \bxi^{\gor} \rangle - \sqrt{\Bigl(1 - \frac{1}{\kappa}\Bigr)\Bigl(\EE\{\Omega^2\} - (\EE\{Z_1 \Omega\})^2 - (\EE\{Z_2 \Omega\})^2\Bigr)} \biggr \rvert \leq C_{K_1} \cdot \epsilon.\label{eq:needs-explanation}
\end{align}
The proof of the lemma follows upon combining inequalities~\eqref{eq:sec-cost} and~\eqref{eq:needs-explanation}, so the only remaining piece is to establish inequality~\eqref{eq:needs-explanation}.

\paragraph{Proof of claim~\eqref{eq:needs-explanation}:} On event $\Ac_1 \cap \Ac_2 \cap \Ac_3$, we have
\begin{align*}
\left|\frac{1}{n}\|\bA\bxi^\gor\|_2^2 - \left(\left(\alpha^\gor\right)^2+\left(\mu^\gor\right)^2+\left(\nu^\gor\right)^2\right)\right|\leq C_{K_1}\cdot \eps\,,
\end{align*}
\begin{align*}
\left|\frac{2}{\sqrt{n}}\bomega_{\sharp}^T\bA\bxi^\gor + 2\left(\left(\alpha^\gor\right)^2+\left(\mu^\gor\right)^2\right)\right|\leq C_{K_1}\cdot \eps\,,
\end{align*}
and
\begin{align*}
\left|\frac{1}{n}\|\bomega_{\sharp}\|_2^2 - \EE\{\Omega^2\}\right|\leq C_{K_1}\cdot \eps\,,
\end{align*}
respectively.
Combining the above displays and noting that $\EE\{\Omega^2\}-\left(\alpha^\gor\right)^2-\left(\mu^\gor\right)^2=(\kappa-1)\left(\nu^\gor\right)^2$ yields
\[
\left|\frac{1}{\sqrt{n}}\| \| \bA\bxi^{\gor} - \bomega_{\sharp}  \|_2-\nu^\gor\sqrt{\kappa}\right|\leq C_{K_1}\cdot\eps.
\]
But, under event $\Ac_4$,
\[
\left|\langle \bv_n, \bxi^{\gor} \rangle - \frac{\nu^\gor}{\sqrt{\kappa}}\right|\leq C\cdot \eps
\]
Hence, the inequality~\eqref{eq:needs-explanation} follows by combining the above two displays and recalling the definition of $\nu^\gor$.
\qed

\subsubsection{Proof of Lemma~\ref{lem:ao-analysis}(b)}
We begin by defining the events
\begin{align} \label{eq:prime-event1}
\Ac'_1 = \Bigl\{\min_{\bxi \in \mathbb{R}^3} \frac{1}{\sqrt{n}}\| \bA \bxi - \bomega_{\sharp} \|_2 \geq c_{K_1,K_1} \Bigr\},
\end{align}
\[
 \Ac'_2 = \Bigl\{ \frac{1}{2} \bI_3 \preceq \frac{1}{n} \bA^{\top} \bA \preceq 2 \bI_3\Bigr\}, \qquad \text{ and } \qquad \Ac'_3 = \Bigl\{\frac{1}{\sqrt{n}}\| \bomega_{\sharp}\|_2 \leq C_{K_1} \Bigr\}.
\]
Note that on the event $\Ac'_2 \cap \Ac'_3$, we obtain the inequality
\begin{align}\label{ineq:upper-bound-quadratic}
\max_{\bxi \in \mathbb{B}_2(C')} \frac{1}{\sqrt{n}}\| \bA \bxi - \bomega_{\sharp}\|_2 \leq C_{K_1}.
\end{align}
Moreover, recall that  $\min_{\bxi \in \real^3} \| \bA\bxi - \bomega_{\sharp}\|_2 = \| \left( \bI - \bA (\bA^\top \bA)^{-1} \bA^\top \right) \cdot \bomega_{\sharp}\|_2$. Thus, 
applying Lemma~\ref{lem:row-space-A} for small enough $\epsilon$ in conjunction with Assumption~\ref{ass:omega-lb} yields the inequality $\Pro\{\Ac'_1\} \geq 1 - Ce^{-c'_{K_1,K_2} \cdot n}$. Finally, applying~\citet[Theorem 6.1]{wainwright2019high} yields $\Pro\{\Ac'_2\} \geq 1 - 2e^{-cn}$; and applying Bernstein's inequality (as each component of $\bomega_{\sharp}$ is $K$-sub-Gaussian by Assumption~\ref{ass:omega}) implies $\Pro\{\Ac'_3\} \geq 2e^{-cn}$. For the rest of the proof, we work on the event $\Ac'_1 \cap \Ac'_2 \cap \Ac'_3$. 

Note that on the event $\Ac'_1$, the function $\widebar{L}_n$ is twice continuously differentiable.  Thus, our strategy is to bound the minimum eigenvalue of the Hessian $\nabla^2 \widebar{L}_n$. For $\bxi$ in a bounded domain $\mathbb{B}_2(C)$,  we compute
\begin{align*}
\nabla^2 \widebar{L}_n(\bxi; \btsharp) &= \frac{1}{\| \bA\bxi - \bomega_{\sharp} \|_2} \cdot \frac{1}{\sqrt{n}} \cdot \Bigl( \bA^{\top} \bA - \frac{\bA^{\top} (\bA \bxi - \bomega_{\sharp})(\bA \bxi - \bomega_{\sharp} )^{\top}\bA}{\| \bA\bxi - \bomega_{\sharp} \|_2^2}\Bigr)\\
&\overset{\1}{\succeq} \frac{1}{C_{K_1}n} \cdot \Bigl( \bA^{\top} \bA - \frac{\bA^{\top} (\bA \bxi - \bomega_{\sharp}(\bA \bxi - \bomega_{\sharp} )^{\top}\bA}{\| \bA\bxi - \bomega_{\sharp} \|_2^2}\Bigr),
\end{align*}
where step $\1$ follows from inequality \eqref{ineq:upper-bound-quadratic}.
 Subsequently, we utilize the variational characterization of eigenvalues to obtain
\begin{align}\label{ineq:lower-bound-mineig}
	\lambda_{\min}(\nabla^2 \widebar{L}_n(\bxi; \btsharp) ) &\geq \frac{1}{{C_{K_1}}n} \min_{\| \bv \|_2 = 1} \langle \bv, \bA^{\top} \bA\bv \rangle - \frac{\langle \bv, \bA^{\top} (\bA \bxi - \bomega_{\sharp})(\bA \bxi - \bomega_{\sharp} )^{\top}\bA\bv \rangle}{\| \bA\bxi - \bomega_{\sharp} \|_2^2}\nonumber\\
	&= \frac{1}{{C_{K_1}}n} \min_{\| \bv \|_2 = 1} \| \bA \bv \|_2^2 - \frac{\langle \bA \bv, \bA \bxi - \bomega_{\sharp} \rangle^2}{\| \bA\bxi - \bomega_{\sharp} \|_2^2}.
\end{align}
Next, consider the orthogonal decomposition
\[
\bomega_{\sharp} = \bA \bxi_0 + \bomega^{\perp},
\]
where $\bA \bxi_0$ denotes an element in the 
column 
 space of the random matrix $\bA$ and $\bomega^{\perp}$ is orthogonal to the 
 column space of $\bA$.  Thus, 
\begin{align}\label{eq:variational-characterization}
	\min_{\| \bv \|_2 = 1} \| \bA \bv \|_2^2 - \frac{\langle \bA \bv, \bA \bxi - \bomega_{\sharp} \rangle^2}{\| \bA\bxi - \bomega_{\sharp} \|_2^2} &= \min_{\| \bv \|_2 = 1} \| \bA \bv \|_2^2 - \frac{\langle \bA \bv, \bA (\bxi - \bxi_0) - \bomega^{\perp}\rangle^2}{\| \bA(\bxi - \bxi_0)\|_2^2 + \| \bomega^{\perp} \|_2^2}\nonumber\\
	&= \min_{\| \bv \|_2 = 1} \| \bA \bv \|_2^2 - \frac{\langle \bA \bv, \bA (\bxi - \bxi_0)\rangle^2}{\| \bA(\bxi - \bxi_0)\|_2^2 + \| \bomega^{\perp} \|_2^2}\nonumber\\
	&\overset{\1}{=}  \min_{\| \bv \|_2 = 1} \| \bA \bv \|_2^2 - \frac{\| \bA \bv\|_2^2 \cdot \| \bA (\bxi - \bxi_0)\|_2^2}{\| \bA(\bxi - \bxi_0)\|_2^2 + \| \bomega^{\perp} \|_2^2},
\end{align}
where step $\1$ follows by applying the Cauchy--Schwarz inequality.  Now, 
\begin{align} \label{ineq:lower-bound-variational}
	\min_{\| \bv \|_2 = 1} \| \bA \bv \|_2^2 - \frac{\| \bA \bv\|_2^2 \cdot \| \bA (\bxi - \bxi_0)\|_2^2}{\| \bA(\bxi - \bxi_0)\|_2^2 + \| \bomega^{\perp} \|_2^2} \overset{\1}{\geq} \frac{n}{2} \cdot \frac{\| \bomega^{\perp} \|_2^2}{\| \bA(\bxi - \bxi_0)\|_2^2 + \| \bomega^{\perp} \|_2^2} \overset{\2}{\geq} c_K \cdot n
\end{align}
where step $\1$ follows by re-arranging and subsequently utilizing event $\Ac'_2$ to lower bound the term $\| \bA\bv \|_2$.  Step $\2$ follows by upper bounding the denominator using events $\Ac_2'$ and $\Ac_3'$ and lower bounding the numerator by using the event $\Ac_1'$.
 Finally, combining the lower bound~\eqref{ineq:lower-bound-mineig}, the equation~\eqref{eq:variational-characterization}, and the lower bound~\eqref{ineq:lower-bound-variational}, we obtain
\[
\lambda_{\min}(\nabla^2 \widebar{L}_n(\bxi; \btsharp) ) \geq c_K.
\]
The final step holds on the event $\Ac'_1 \cap \Ac'_2 \cap \Ac'_3$, which occurs with probability at least $1 - Ce^{-c'_{K_1,K_2} \cdot n}$.\qed

\subsubsection{Proof of Lemma~\ref{lem:remove-positive-restriction}}
\label{subsec:proof-lemma-remove-positive-restriction}
Recall the loss function $\Lc$~\eqref{eq:sec_order_gen}:
\[
\Lc(\bt; \btsharp, \bX, \by) = \frac{1}{\sqrt{n}} \| \omega(\bX\btsharp, \by) - \bX\bt \|_2,
\]
which corresponds to setting 
\[
F(\bu, \bt) = \frac{1}{\sqrt{n}}\| \omega(\bX\btsharp, \by) - \bu \|_2, \quad h(\bt, \btsharp) = 0, \quad \text{ and } \qquad g(\bu, \bX \btsharp; \by) = \frac{1}{\sqrt{n}}\| \omega(\bX\btsharp, \by) - \bu \|_2.
\]
Additionally, note that the convex conjugate is given by
\[
g^*(\bv, \bX\btsharp; \by) = \begin{cases}\langle \bv, \omega(\bX\btsharp, \by) \rangle & \text{ if } \| \bv \|_2 \leq \frac{1}{\sqrt{n}},\\
	+\infty & \text{ otherwise.} \end{cases}
\]
Substituting into the definition of the scalarized auxiliary loss $\widebar{L}_n$ (see Definition~\ref{def:scalarized-ao}), we obtain
\[
\widebar{L}_n(\alpha, \mu, \nu; \btsharp) = \biggl(\frac{1}{\sqrt{n}} \| \perptwo \cdot \bg + \parcomp \cdot \bz_1 + \perpone \cdot \bz_2 - \bomega_{\sharp} \|_2 - \perptwo \frac{\| \proj_{S_{\sharp}}^{\perp} \bh \|_2}{\sqrt{n}}\biggr)_{+}.
\]
Now, consider the shorthand
\[
\bw = \parcomp \cdot \bz_1 + \perpone \cdot \bz_2 - \bomega_{\sharp},
\]
and define the three events
\[
\Ec_1 = \Bigl\{\frac{1}{n}\lvert \langle \perptwo \cdot \bg, \bw\rangle \rvert \leq \frac{4\perptwo}{5\sqrt{n}}\|\bw\|_2\Bigr\}, \quad \Ec_2 = \Bigl\{\| \bg \|_2^2 \geq 4n/5\}, \quad \text{ and } \quad \Ec_3 = \{\| \bh \|_2^2 \leq 36d/25\}.
\]
Next, apply Bernstein's inequality to bound $\Pro\{\Ec_1^c\}$, Hoeffding's inequality to bound $\Pro\{\Ec_2^c\}$ and $\Pro\{\Ec_3^c\}$, and the union bound to obtain the inequality 
\[\Pro\{\Ec_1 \cap \Ec_2 \cap \Ec_3\} \geq 1 - 6e^{-cn}.
\]
Working on the intersection of these three events, we obtain
\begin{align*}
	\frac{1}{\sqrt{n}} \| \perptwo \cdot \bg + \bw \|_2  &= \sqrt{ \frac{4 \perptwo^2}{5n}\| \bg \|_2^2 + \frac{\perptwo^2}{5n} \| \bg \|_2^2 + \frac{1}{n}\langle \perptwo \cdot \bg, \bw \rangle + \frac{1}{n} \| \bw \|_2^2 }\\
	&\geq \sqrt{ \frac{16 \perptwo^2}{25} + \frac{4\perptwo^2}{25} - \frac{4\perptwo}{5\sqrt{n}}\| \bw \|_2 + \frac{1}{n} \| \bw \|_2^2 } = \sqrt{\frac{16 \perptwo^2}{25}  + \Bigl(\frac{2\perptwo}{5} - \frac{\| \bw \|_2}{\sqrt{n}}\Bigr)^2} \geq \frac{4\perptwo}{5}.
\end{align*}
Thus, we obtain
\[
\frac{1}{\sqrt{n}} \| \perptwo \cdot \bg + \bw \|_2 - \perptwo \frac{\| \proj_{S_{\sharp}}^{\perp} \bh \|_2}{\sqrt{n}} \geq \frac{4\perptwo}{5} - \frac{6\perptwo}{5\sqrt{\kappa}} \geq 0, 
\]
where the last inequality holds for $\kappa \geq C$ with $C$ a large enough constant.  Summarizing, we see that 
\[
\frac{1}{\sqrt{n}} \| \perptwo \cdot \bg + \parcomp \cdot \bz_1 + \perpone \cdot \bz_2 - \bomega_{\sharp} \|_2 - \perptwo \frac{\| \proj_{S_{\sharp}}^{\perp} \bh \|_2}{\sqrt{n}} \geq 0,
\]
whence 
\[
\widebar{L}_n(\alpha, \mu, \nu; \btsharp)  = \frac{1}{\sqrt{n}} \| \perptwo \cdot \bg + \parcomp \cdot \bz_1 + \perpone \cdot \bz_2 - \bomega_{\sharp} \|_2 - \perptwo \frac{\| \proj_{S_{\sharp}}^{\perp} \bh \|_2}{\sqrt{n}}
\]
as desired.
\qed

\subsubsection{Proof of Lemma~\ref{lem:row-space-A}}
\label{subsec:proof-lemma-row-space-A}

Recall the three events~\eqref{eq:three-events}. We prove each part of the lemma in turn.

\paragraph{Proof of part (a):} First, 
expand the norm to obtain
\[
\frac{1}{\sqrt{n}}\|(\bI_n - \bA(\bA^{\top} \bA)^{-1}\bA^{\top})\omega(\bX \btsharp, \by) \|_2 = \sqrt{ \frac{1}{n}\| \bomega_{\sharp}\|_2^2 - \frac{1}{n}(\bA^{\top}\bomega_{\sharp})^{\top}(\bA^{\top} \bA)^{-1} (\bA^{\top} \bomega_{\sharp})}.
\]
Now, on the event $\Ac_1$, we obtain the sandwich relation
\[
A_L \leq \sqrt{ \frac{1}{n}\| \bomega_{\sharp}\|_2^2 - \frac{1}{n}(\bA^{\top}\bomega_{\sharp})^{\top}(\bA^{\top} \bA)^{-1} (\bA^{\top} \bomega_{\sharp})} \leq A_U,
\]
where we have let
\begin{align*}
	A_L &= \sqrt{ \frac{1}{n}\| \bomega_{\sharp}\|_2^2 - \frac{1 + C\epsilon/2}{n^2}\|\bA\bomega_{\sharp}\|_2^2} =  \sqrt{ \frac{1}{n}\| \bomega_{\sharp}\|_2^2 - \frac{1 + C\epsilon/2}{n^2}(\langle \bg, \bomega_{\sharp}\rangle^2 + \langle \bz_1, \bomega_{\sharp}\rangle^2 + \langle \bz_2, \bomega_{\sharp} \rangle^2)}\\
	A_U &= \sqrt{ \frac{1}{n}\| \bomega_{\sharp}\|_2^2 - \frac{1 - C\epsilon/2}{n^2}\|\bA\bomega_{\sharp}\|_2^2} =\sqrt{ \frac{1}{n}\| \bomega_{\sharp}\|_2^2 - \frac{1 - C\epsilon/2}{n^2}(\langle \bg, \bomega_{\sharp}\rangle^2 + \langle \bz_1, \bomega_{\sharp}\rangle^2 + \langle \bz_2, \bomega_{\sharp} \rangle^2)}, 
\end{align*}
where the last equality in both lines follows by recalling that $\bA^{\top}\bomega_{\sharp} = [\langle \bz_1, \bomega_{\sharp} \rangle, \langle \bz_2, \bomega_{\sharp}\rangle, \langle \bg, \bomega_{\sharp} \rangle ]$. 
On the event $\Ac_2 \cap \Ac_3$, we obtain 
\[
A_L \overset{\1}{\geq} \sqrt{(\EE\{\Omega^2\} - (\EE\{Z_1 \Omega\})^2 - (\EE\{Z_2 \Omega\})^2 - C_{K_1} \epsilon} \overset{\2}{\geq}  \sqrt{(\EE\{\Omega^2\} - (\EE\{Z_1 \Omega\})^2 - (\EE\{Z_2 \Omega\})^2} - C_{K_1,K_2} \cdot \epsilon.
\]
Specifically, step $\1$ holds since by Assumption~\ref{ass:omega}, $\max\left\{\E\{Z_1\Omega\},\E\{Z_1\Omega\}\right\}\leq \sqrt{\E\{\Omega^2\}}\leq C{K_1}$.  Step $\2$ follows since by Assumption~\ref{ass:omega-lb}, the first term on the RHS is at least $K_2$.  
Proceeding similarly, we obtain the upper bound
\[
A_U \leq \sqrt{(\EE\{\Omega^2\} - (\EE\{Z_1 \Omega\})^2 - (\EE\{Z_2 \Omega\})^2} + C_{K_1,K_2} \epsilon.
\]
Putting the pieces together, we see that on the event $\Ac_1\cap\Ac_2\cap\Ac_3$,
\[
\Bigl\lvert \frac{1}{\sqrt{n}}\|(\bI_n - \bA(\bA^{\top} \bA)^{-1}\bA^{\top}) \cdot \omega(\bX \btsharp, \by) \|_2 -  \sqrt{(\EE\{\Omega^2\} - (\EE\{Z_1 \Omega\})^2 - (\EE\{Z_2 \Omega\})^2} \Bigr\rvert \leq C_{K_1,K_2} \epsilon,
\]
as claimed. \qed

\paragraph{Proof of part (b):} From~\citet[Theorem 6.1]{wainwright2019high}, we directly have $\Pro\{ \Ac_1 \} \geq 1 - C e^{-cn\epsilon^2}$. Next, recall for convenience the other two events
\begin{align*}
	\Ac_2 = \Bigl\{\frac{1}{n^2}\lvert \langle \bg, \bomega_{\sharp} \rangle \rvert^2 \vee \Bigl\lvert \frac{1}{n^2}\langle \bz_1, \bomega_{\sharp}\rangle^2 - (\EE \{ Z_1 \Omega \})^2\Bigr\rvert \vee \Bigl\lvert \frac{1}{n^2}\langle \bz_2, \bomega_{\sharp} \rangle^2 - (\EE\{Z_2 \Omega \})^2 \Bigr \rvert \leq C_{K_1} \cdot \epsilon\Bigr\},
\end{align*}
and 
\begin{align*}
	\Ac_3 &= \Bigl\{\Bigl\lvert \frac{1}{n} \| \bomega_{\sharp} \|_2^2 - \EE\{\Omega^2\} \Bigr\rvert \leq C_{K_1} \cdot \epsilon \Bigr\}.
\end{align*}
Next, recall that by Assumption~\ref{ass:omega}, we have the bound $\| \bomega_{\sharp} \|_{\psi_2} \leq K_1$.  Thus, we apply Bernstein's inequality to obtain 
\[
\Pro\Bigl\{ \Bigl \lvert \frac{1}{n} \langle \bz_1, \bomega_{\sharp} \rangle - \EE\{ Z_1 \Omega\}\Bigr\rvert \geq C_{K_1} \cdot \epsilon \Bigr\} \leq 2e^{-cn\epsilon^2}.
\]
Consequently, with probability at least $1 - 2e^{-cn\epsilon^2}$,  
\[
\Bigl\lvert\frac{1}{n^2}\langle \bz_1, \bomega_{\sharp}\rangle^2 - (\EE \{ Z_1 \Omega \})^2\Bigr\rvert = \Bigl \lvert \frac{1}{n} \langle \bz_1, \bomega_{\sharp} \rangle - \EE\{ Z_1 \Omega\}\Bigr\rvert \cdot \Bigl \lvert \frac{1}{n} \langle \bz_1, \bomega_{\sharp} \rangle + \EE\{ Z_1 \Omega\}\Bigr\rvert \overset{\1}{\leq} C_{K_1} \cdot \epsilon,
\]
where step $\1$ additionally used the fact that $\EE\{Z_1 \Omega\} \leq C_{K_1}$.  Bounding the other terms similarly and applying a union bound, we obtain $\Pro\{\Ac_2\} \geq 1 - 6e^{-cn\epsilon^2}$.  Finally, we apply Bernstein's inequality once more to obtain the inequality $\Pro\{ \Ac_3 \} \geq 1 - 2e^{-cn\epsilon^2}$. \qed

\subsubsection{Proof of Lemma~\ref{lem:minimize-lbar-higher-order}}
\label{subsec:proof-lemma-minimize-lbar}
Set $C_{K_1}$ to be a sufficiently large positive constant depending only on $K_1$. For any $R \geq C_{K_1}$, that may only depend on $K_1$, we note the following characterization:
\begin{align}
	\label{eq:tau-characterization}
	\min_{\bxi \in \Pc(B_2(R))} \widebar{L}_n(\bxi; \btsharp) &= \min_{\bxi \in \Pc(B_2(R))} \inf_{\tau > 0}\; \frac{\tau}{2} + \frac{1}{2\tau n} \| \bA \bxi - \omega(\bX \btsharp, \by) \|_2^2 - \langle \bv_n, \bxi \rangle \nonumber \\
	&=  \inf_{\tau > 0}\min_{\bxi \in \Pc(B_2(R))}\; \frac{\tau}{2} + \frac{1}{2\tau n} \| \bA \bxi - \omega(\bX \btsharp, \by) \|_2^2 - \langle \bv_n, \bxi \rangle,
\end{align}
where the second equality follows on event $\Ac'_1$~\eqref{eq:prime-event1}, which holds with probability \\$1-C\exp(-c_{K_1,K_2}n)$ and implies that the infimum over $\tau$ is achieved. This proves the part (a) of the lemma.

Next, we prove part (b). To do this, consider the \emph{unconstrained} minimization in \eqref{eq:tau-characterization} over $\bxi\in\R^3$. Note that this admits the unique minimizer
\[
\bxi_n(\tau) = (\bA^{\top} \bA)^{-1} \bA^{\top} \omega(\bX\btsharp, \by) + \tau \cdot n \cdot (\bA^{\top} \bA)^{-1} \bv_n.
\]
Substituting this value into the RHS of the optimization problem~\eqref{eq:tau-characterization} yields
\begin{align}
	\label{eq:opt-tau}
	\min_{\bxi \in \R^3)} \widebar{L}_n(\bxi; \btsharp) &= \inf_{\tau > 0} \; \frac{\tau}{2} + \frac{1}{2\tau n} \| \bA \bxi_n(\tau) - \omega(\bX \btsharp, \by) \|_2^2 - \langle \bv_n, \bxi_n(\tau) \rangle\nonumber\\
	&=\inf_{\tau > 0}\;\frac{1}{2\tau n}\|(\bI_n-\bA(\bA^{\top}\bA)^{-1}\bA^{\top})\omega(\bX \btsharp, \by)\|_2^2 + \frac{\tau}{2}(1-n\cdot\bv_n^{\top}(\bA^{\top}\bA)^{-1}\bv_n)\nonumber\\
	&\qquad\;\;\;-\bv_n^{\top}(\bA^{\top}\bA)^{-1}\bA\omega(\bX \btsharp, \by).
\end{align}
Consider the events
\[
\Ac''_1 = \{ \| \proj_{S_{\sharp}}^{\perp}\bh \|_2^2 \leq 36d/25\} \qquad \text{ and } \qquad \Ac''_2 = \Bigl\{(1 - c') \bI_3 \preceq \frac{1}{n} \bA^{\top} \bA \preceq (1 + c') \bI_3 \Bigr\}.
\]
Apply Bernstein's inequality to obtain $\Pro\{\Ac''_1\} \geq 1 - 2e^{-cn}$ and apply~\citet[Theorem 6.1]{wainwright2019high} to obtain $\Pro\{\Ac''_2\} \geq 1 - 2e^{-cn}$.  On the event $\Ac''_1 \cap \Ac''_2$, note that 
\[
1 - n\cdot\bv_n^{\top}(\bA^{\top}\bA)^{-1}\bv_n \geq 1 - (1 + c'/2) \| \bv_n \|_2^2 \geq 1 - (1 + c'/2) \cdot \frac{6}{5 \kappa} > c'' >0,
\]
where in the last inequality we used $\kappa> C$. 
Thus, on the event $\Ac'_1 \cap \Ac''_1 \cap \Ac''_2$, the optimization problem~\eqref{eq:opt-tau} admits the unique minimizer
\[
\tau_n = \frac{\frac{1}{\sqrt{n}}\|(\bI_n - \bA(\bA^{\top} \bA)^{-1}\bA^{\top}) \cdot \omega(\bX \btsharp, \by) \|_2}{\sqrt{1 - n\cdot \bv_n^{\top} (\bA^{\top} \bA)^{-1} \bv_n}}.
\]
We have thus far shown that $\bxi_n(\tau_n)$ is the unique minimizer of 
\begin{align}\label{eq:uncon-ksi}
\min_{\bxi \in \R^3} \widebar{L}_n(\bxi; \btsharp).
\end{align}
 But, recalling the event~\eqref{eq:xi-concentration}, we note that $\|\bxi_n(\tau_n)-\bxi^\gor\|_2\leq C'_{K_1}\eps$ with probability at least $1-Ce^{-cn\eps^2}$. On this event, using triangle inequality in conjunction with the fact that $\|\bxi^\gor\|_2\leq C''_{K_1}$ by Assumption 1 and $\kappa>C$, yields the inequality $\|\bxi_n(\tau_n)\|_2\leq C_{K_1}$ for a positive constant $C_{K_1}$ depending only on $K_1$.  Therefore, the minimizer of equation~\eqref{eq:uncon-ksi} remains $\bxi_n(\tau_n)$ even if we constrain $\|\bxi\|_2\leq R$, provided $R \geq C_{K_1}$. To finish the proof, recall by Definition \ref{def:scalarized-set} that since $\bxi=\bxi(\bt)$, we have $\|\bxi\|_2=\|\bt\|_2$.
\qed

\subsection{Establishing growth conditions for first-order methods} \label{sec:growth-FO}
In this subsection, we prove Lemma~\ref{lem:ao-analysis-FO}. Note that the events in this section are unrelated to
events defined in Section~\ref{sec:growth-HO}. Also note that universal constants, as well as those depending on 
$K_1$ 
may change from line to line.
Now, recall from equation~\eqref{eq:Lbar-first-order} that $\widebar{L}_n$ can be written as
\begin{align*}
\widebar{L}_n(\parcomp, \perpone, \perptwo) = \;&\frac{\parcomp^2 + \perpone^2 + \perptwo^2}{2} - (\parcomp \parcomp_{\sharp} + \perpone \perpcomp_{\sharp}) + \frac{2\eta}{n} \cdot \langle \omega(\bX\btsharp, \by), \perptwo \bg + \parcomp \bz_1 + \perpone \bz_2 \rangle \nonumber\\
&- \frac{2 \eta \perptwo}{n} \cdot \| \proj_{S_{\sharp}}^{\perp} \bh\|_2 \cdot \| \omega(\bX\btsharp, \by) \|_2.
\end{align*}
Note that the $1$-strong convexity claimed in part (b) of the lemma is evident from the expression above, so we focus our attention on proving part (a) in the next subsection.

\subsubsection{Proof of Lemma~\ref{lem:ao-analysis-FO}(a)}
Evidently, $\widebar{L}_n$~\eqref{eq:Lbar-first-order} is strongly convex and continuously differentiable.  The optimizers are given by the first order conditions
\begin{subequations}\label{eq:first-order-minimizers}
	\begin{align}
		\parcomp_n &= \parcomp^{\sharp} - \frac{2\eta}{n} \cdot \langle \bz_1, \omega(\bX\btsharp, \by) \rangle \\
		\perpone_n &= \perpcomp^{\sharp} - \frac{2 \eta}{n} \cdot \langle \bz_2, \omega(\bX \btsharp, \by) \rangle \\
		\perptwo_n &= \frac{2\eta}{n} \cdot \langle \bg, \omega(\bX\btsharp, \by) \rangle + \frac{2\eta}{n} \cdot \| \proj_{S_{\sharp}}^{\perp} \bh\|_2 \cdot \| \omega(\bX \btsharp, \by) \|_2.
	\end{align}
\end{subequations}
Now, consider the events
\[
\Ac_1 = \Bigl\{\Bigl \lvert \frac{1}{n} \langle \bg, \omega(\bX \btsharp, \by) \rangle \Bigr \rvert  \vee \Bigl \lvert \frac{1}{n} \langle \bz_1, \omega(\bX \btsharp, \by) \rangle - \EE\{ Z_1 \Omega\} \Bigr \rvert \vee \Bigl \lvert \frac{1}{n} \langle \bz_2, \omega(\bX \btsharp, \by) \rangle - \EE\{ Z_2 \Omega\} \Bigr \rvert  \leq C_{K_1} \cdot \epsilon \Bigr\},
\]
and 
\[
\Ac_2 = \Bigl\{\Bigl \lvert \frac{1}{\sqrt{n}} \| \proj_{S_{\sharp}}^{\perp} \bh \|_2 - \kappa^{-1/2} \Bigr \rvert \vee \Bigl \lvert \frac{1}{\sqrt{n}} \| \omega(\bX\btsharp, \by) \|_2 - \sqrt{\EE\{\Omega^2\}} \Bigr\rvert \leq C_{K_1} \cdot \epsilon\}, 
\]
and apply Bernstein's inequality to obtain the bound $\Pro\{\Ac_1 \cap \Ac_2\} \geq 1 - 10e^{-cn\min\{\epsilon^2, \epsilon\}}$.  Consequently, on the event $\Ac_1 \cap \Ac_2$, we obtain 
\[
\Pro\{ \| [\parcomp_n, \perpone_n, \perptwo_n] - [\parcomp^{\gor}, \perpone^{\gor}, \perptwo^{\gor}]\|_{\infty} \leq C_{K_1} \cdot \epsilon\} \geq 1 - 20e^{-cn\min\{\epsilon^2, \epsilon\}}
\]
Additionally, substitute the empirical minimizers~\eqref{eq:first-order-minimizers} into the loss $\widebar{L}_n$~\eqref{eq:Lbar-first-order} to obtain
\[
\widebar{L}_n(\parcomp_n, \perpone_n, \perptwo_n) =\; - \frac{1}{2}\cdot (\parcomp_n^2 + \perpone_n^2 + \perptwo_n^2).
\]
Next, recall from Lemma~\ref{lem:ao-analysis-FO} the constant
\[
\bar{\mathsf{L}} = - \frac{1}{2}\cdot ((\parcomp^{\gor})^2 + (\perpone^{\gor})^2 + (\perptwo^{\gor})^2).
\]
We also claim that $\alpha^\gor \vee \mu^\gor \vee \nu^\gor \leq C_{K_1}$; this can be verified from Definition~\ref{def:expanded-gordon-FO}, and applying the Cauchy--Schwarz inequality in conjunction with the assumed bounds $\alpha^{\sharp}, \beta^{\sharp} \leq 3/2$.

Combining the pieces, note that on the event $\Ac_1 \cap \Ac_2$ we obtain
\[
\lvert \widebar{L}_n(\parcomp_n, \perpone_n, \perptwo_n) - \bar{L}\rvert \leq \frac{C_{K_1} \cdot \epsilon}{2}\Bigl(\lvert \parcomp_n + \parcomp^{\gor} \rvert + \lvert \perpone_n + \perpone^{\gor} \rvert + \lvert \perptwo_n + \perptwo^{\gor} \rvert \Bigr) \leq C'_{K_1} \cdot  \epsilon,
\]
as desired. \qed

%% file: sections/appendix/aux-random-init.tex

\section{Auxiliary proofs for general results, part (b)}
\label{sec:aux-random-init}
We state and prove two technical lemmas; the first is used throughout Section~\ref{sec:random-init-signal} and the second provides some basic properties about the initial point $\bt_0$.

\begin{lemma}
	\label{lem:wishart_mineig_moments}
	Let $n$ and $d$ be positive integers such that $n \geq 2d$.  Additionally, let $\bx_1, \dots, \bx_n \overset{\mathsf{iid}}{\sim} \NORMAL(0, \bI_d)$ and let $\bSig = \sum_{i=1}^{n} \bx_i \bx_i^{\top}$.  Then, there exists a universal positive constant $C$ such that for all integers $p$ where $1 \leq p < \frac{n - d - 1}{2}$, we have
	\[
	\E\left\{\|\bSig^{-1}\|_{\mathsf{op}}^p \right\} \leq \left(\frac{C}{n}\right)^p.
	\]
\end{lemma}
\begin{proof}
	We begin by noting that $\|\bSig^{-1}\|_{\mathsf{op}}^p = \lambda_{\min}(\bSig)^{-p}$.  
	Our strategy is to truncate $\lambda_{\min}(\bSig)$ at the level $An$, for a constant $A > 0$ to be chosen later, and decompose
	\begin{align}
		\label{decomposition:wishart_moments}
		\E\left\{\lambda_{\min}(\bSig)^{-p}\right\} = \E\left\{\lambda_{\min}(\bSig)^{-p} \mathbbm{1}\{\lambda_{\min}(\bSig) \leq An\}\right\} + \E\left\{\lambda_{\min}(\bSig)^{-p} \mathbbm{1}\{\lambda_{\min}(\bSig) \geq An\}\right\}.
	\end{align}
	For ease of notation, we will denote the two terms in the above decomposition by 
	\begin{align*}
		T_1 &= \E\left\{\lambda_{\min}(\bSig)^{-p} \mathbbm{1}\{\lambda_{\min}(\bSig) \leq An\}\right\}\\
		T_2 &= \E\left\{\lambda_{\min}(\bSig)^{-p} \mathbbm{1}\{\lambda_{\min}(\bSig) \geq An\}\right\},
	\end{align*}
	and handle each term in turn.
	
	\paragraph{Bounding the term $T_1$.}  We write explicitly:
	\begin{align*}
		\E\left\{\lambda_{\min}(\bSig)^{-p} \mathbbm{1}\{\lambda_{\min}(\bSig) \leq An\}\right\} &= \int_{0}^{An} t^{-p} f_{\lambda_{\min}}(t) \mathrm{d}t,
	\end{align*}
	\begin{align*}
		\int_{0}^{An} t^{-p} f_{\lambda_{\min}}(t) dt &\overset{\mathsf{(i)}}{\leq} \frac{2^{\frac{n - d - 1}{2}}\Gamma(\frac{n+1}{2})}{\Gamma(\frac{d}{2})\Gamma(n - d + 1)}\int_{0}^{An} t^{-p} t^{\frac{1}{2}(n - d - 1)}dt \\
		&\overset{\mathsf{(ii)}}{=} \frac{2^{\frac{n - d - 1}{2}}\Gamma(\frac{n+1}{2})}{\left(\frac{1}{2}(n - d +1) - p\right)\Gamma(\frac{d}{2})\Gamma(n - d + 1)} (An)^{\frac{1}{2}(n - d - 1) - p + 1} \\
		&\overset{\mathsf{(iii)}}{\leq} \frac{8(A)^{\frac{1}{2}(n - d + 1) - p}}{\Gamma(n - d + 1)} n^{n - d - p} \\
		&\overset{\mathsf{(iv)}}{\leq} \frac{8}{2e}\left(\frac{1}{4e^2 n }\right)^p.
	\end{align*}
Step $\1$ follows by applying~\citet[Lemma 3.3]{chen2005condition} to upper bound the density $f_{\lambda_{\min}}(t)$.  Step $\2$ follows by noting that $(n - d - 1)/2 - p > 0$ and evaluating the integral exactly.  Step $\3$ from the inequality $\Gamma(\frac{d}{2})\left(\frac{n}{2}\right)^{\frac{1}{2}(n - d + 1)} > \Gamma(\frac{n +1}{2})$ (see the proof of~\citet[Lemma 4.1]{chen2005condition}) and the fact that $n - d + 1 \geq n/2$.  Step $\4$ follows by utilizing Stirling's inequality for the Gamma function and setting $A = (4e^2)^{-1}$.
	Summarizing, we have shown 
	\begin{align}
		\label{ineq:t1_control_wishart_moments}
		\E\left\{\lambda_{\min}(\bSig)^{-p} \mathbbm{1}\{\lambda_{\min}(\bSig) \leq An\}\right\} \leq \frac{8}{2e}\left(\frac{1}{4e^2 n }\right)^p,
	\end{align} 
	where $A = (4e^2)^{-1}$.
	
	\paragraph{Bounding the term $T_2$.} Note that the function $t \mapsto t^{-p}$ is decreasing for $p > 0$.  Consequently, 
	\begin{align}
		\label{ineq:t2_control_wishart_moments}
		\E\left\{\lambda_{\min}(\bSig)^{-p} \mathbbm{1}\{\lambda_{\min}(\bSig) \geq An\}\right\} \leq \left(\frac{1}{An}\right)^{p} \Pr \left\{\lambda_{\min}(\bSig) \geq An\right\} \leq \left(\frac{4e^2}{n}\right)^p.
	\end{align}
	
	The result follows immediately upon combining the decomposition~\eqref{decomposition:wishart_moments} along with the  upper bound on term $T_1$~\eqref{ineq:t1_control_wishart_moments} the upper bound on term $T_2$~\eqref{ineq:t2_control_wishart_moments}.
\end{proof}

%% file: sections/appendix/aux_alt.tex

\section{Auxiliary technical results for specific models}

We begin by proving the four corollaries for one-step updates from Theorems~\ref{thm:one_step_main_HO} and~\ref{thm:one_step_main_FO}, and then proceed to proofs of Fact~\ref{prop:Pop-PR} and the technical lemmas stated in Section~\ref{sec:prelim-lemmas}.

\subsection{Proof of Corollary~\ref{lem:Gordon-SE-AM-PR}} \label{pf-cor1}

We evaluate the Gordon updates explicitly and verify Assumptions~\ref{ass:omega} and~\ref{ass:omega-lb}. The corollary then follows by invoking Theorem~\ref{thm:one_step_main_HO}. Note that in this case, we
have $\wfun( x, y) = \sign(x) \cdot y$, so that
\begin{align*}
\Omega = \sign(\parcomp Z_1 + \perpcomp Z_2) \cdot \left( |Z_1| + \noisestd Z_3 \right). 
\end{align*}

\subsubsection{Evaluating Gordon state evolution update} Let us begin by evaluating the three expectations that appear in the claimed Gordon update in Definition~\ref{def:starnext_main_HO}. Clearly, we have
\begin{align} \label{eq:exp-Omega-AMPR}
\EE[\Omega^2] = \EE \left( |Z_1| + \noisestd Z_3 \right) = 1 + \noisestd^2.
\end{align}
Since $Z_3$ is independent of the pair $(Z_1, Z_2)$, we also have
\begin{align*}
\EE[Z_1\Omega] = \EE [\sign(\parcomp Z_1 + \perpcomp Z_2) \cdot   \sign(Z_1) \cdot Z_1^2] \quad \text{ and } \quad \EE[Z_2\Omega] = \EE [\sign(\parcomp Z_1 + \perpcomp Z_2) \cdot  |Z_1| \cdot Z_2 ].
\end{align*}
We evaluate these expectations by first transforming into polar coordinates. Let $Z_1 = R \cos \Phi$ and $Z_2 = R \sin \Phi$, where $R$ is a $\chi$-random variable with $2$ degrees of freedom and the random variable $\Psi \sim \mathsf{Unif}([0, 2\pi])$ is drawn independently.
The first expectation can then be written as
\begin{align*}
\EE[Z_1\Omega] = \EE \left[ \sign(\parcomp \cos (\Psi) + \perpcomp \sin (\Psi) ) \sign(\cos (\Psi) ) R^2 \cos^2(\Psi) \right] = \EE \left[ \sign( \cos(\Psi) \cdot \cos(\Psi - \anglecurr)) R^2 \cos^2 (\Psi) \right],
\end{align*}
where we have used the fact that $\tan \anglecurr = \perpcomp/\parcomp$.
Evaluating the final expectation explicitly, we obtain
\begin{align}
\EE[Z_1\Omega] = \EE [R^2] \cdot \left( \frac{1}{2\pi} \int_{0}^{2\pi}  \sign( \cos(\psi) \cdot \cos(\psi - \anglecurr)) \cos^2 (\psi) d\psi \right) \notag \\
\overset{\1}{=} \frac{1}{\pi} \left( \pi - 4 \int_{\pi/2}^{\pi/2+\anglecurr} \cos^2 \psi d \psi \right) = 1 - \frac{1}{\pi} (2 \anglecurr - \sin(2\anglecurr) ), \label{eq:exp1-AMPR}
\end{align} 
where step $\1$ follows from noting that 
$\sign( \cos(\psi) \cdot \cos(\psi - \anglecurr)) = 1$  for all $\psi \in [0, \pi/2] \cup [\pi/2 + \anglecurr, 3\pi/2] \cup [3\pi/2 + \anglecurr, 2\pi]$.
Proceeding similarly for the second expectation, we have
\begin{align}
\EE[Z_2 \Omega] &= \EE \left[ \sign( \cos(\Psi) \cdot \cos(\Psi - \anglecurr)) R^2 \cos (\Psi) \sin(\Psi) \right] \notag \\
&= \frac{1}{\pi} \int_{0}^{2\pi}  \sign( \cos(\psi) \cdot \cos(\psi - \anglecurr)) \cos (\psi) \sin (\psi) d\psi
= \frac{2}{\pi} \sin^2 (\anglecurr). \label{eq:exp2-AMPR}
\end{align}
Putting together equations~\eqref{eq:exp-Omega-AMPR},~\eqref{eq:exp1-AMPR}, and~\eqref{eq:exp2-AMPR} with Definition~\ref{def:starnext_main_HO}, some straightforward calculation yields the Gordon state evolution update~\eqref{eq:Gordon-AM-PR}.

\subsubsection{Verifying assumptions} To verify Assumption~\ref{ass:omega}, note that 
$
\Omega^2 \leq 2 Z_1^2 + 2 \noisestd^2 Z_3^2, 
$
so that $\EE[\exp(\Omega^2 / (2 + 2\noisestd^2)] \leq 1$. Thus, we have $\| \Omega \|_{\psi_2} \leq 2 (1 + \noisestd^2)$. 
To verify Assumption~\ref{ass:omega-lb}, note that from the calculations above,
\[
\EE[\Omega^2] - (\EE[Z_1 \Omega])^2 - (\EE[Z_2 \Omega])^2 = 1 + \sigma^2 - \left( 1 - \frac{1}{\pi} (2 \anglecurr - \sin(2 \anglecurr)) \right)^2 - \frac{4}{\pi^2} \sin^4 \anglecurr \geq \sigma^2,
\]
where the final inequality can be verified for each $0 \leq \anglecurr \leq \pi/2$. \qed

\subsection{Proof of Corollary~\ref{lem:Gordon-SE-GD-PR}} \label{pf-cor2}

In this case, we evaluate the expectations and verify Assumption~\ref{ass:omega}, and the result follows by invoking Theorem~\ref{thm:one_step_main_FO}. We
have $\wfun = x - \sign(x) \cdot y$, so that
\begin{align*}
\Omega = \parcomp Z_1 + \perpcomp Z_2 - \sign(\parcomp Z_1 + \perpcomp Z_2) \cdot \left( |Z_1| + \noisestd Z_3 \right).
\end{align*}

\subsubsection{Evaluating Gordon state evolution update} Note the definition of $\Omega$ in conjunction with equations~\eqref{eq:exp1-AMPR} and~\eqref{eq:exp2-AMPR}, and Remark~\ref{rem:weight-rel}.
Performing some straightforward algebra using Definition~\ref{def:starnext_main_FO}(b) yields the Gordon updates~\eqref{eq:Gordon-GD-PR}.

\subsubsection{Verifying Assumption~\ref{ass:omega}} We have the upper bound 
\[
\Omega^2 \leq 2 (1 + \parcomp^2) Z_1^2 + 2 \perpcomp^2 Z_2 + 2 \noisestd^2 Z_3^2, 
\]
so that $\EE[\exp(\Omega^2 / (2 \parcomp^2 + 2 \perpcomp^2 + 2\noisestd^2 + 2)] \leq 1$. Thus, we have 
\[
\| \Omega \|_{\psi_2} \leq 2 (\parcomp^2 + \perpcomp^2  + \noisestd^2 + 1) \leq 2 (10 + \noisestd^2),
\]
where the final inequality is a consequence of the assumption $\alpha \vee \beta \leq 3/2$.

\subsection{Proof of Corollary~\ref{lem:Gordon-SE-AM-MLR}} \label{pf-cor3}

We evaluate the Gordon state evolution update explicitly, and verify Assumptions~\ref{ass:omega} and~\ref{ass:omega-lb}. The corollary then follows by invoking Theorem~\ref{thm:one_step_main_HO}.
In this case, we
have $\wfun(x, y) = \sign(yx) \cdot y$, so that
\begin{align*}
\Omega = \sign(\parcomp Z_1 + \perpcomp Z_2) \cdot \sign(Q \cdot Z_1 + \noisestd Z_3) \cdot \left( Q \cdot Z_1 + \noisestd Z_3 \right).
\end{align*}
Here $Q$ is a Rademacher random variable. 

\subsubsection{Evaluating Gordon state evolution update} An immediate calculation yields \mbox{$\EE[\Omega^2] = 1 + \noisestd^2$}.
We now claim that the following equalities characterize the remaining two expectations:
\begin{align}
\E\left[Z_1 \Omega \right] &= 1 - \frac{2}{\pi}\tan^{-1}\left(\sqrt{\rho^2+\noisestd^2+\rho^2\noisestd^2}\right) + \frac{2}{\pi}\frac{\sqrt{\rho^2+\noisestd^2+\rho^2\noisestd^2}}{1+\rho^2}, \text{ and }
\label{eq:mixtures_AM_G1Y}
\\
\E\left[Z_2 \Omega \right] &= \frac{2}{\pi}\frac{\rho\sqrt{\rho^2+\noisestd^2+\noisestd^2\rho^2}}{1+\rho^2},
\label{eq:mixtures_AM_G2Y}
\end{align}
where we recall the notation $\rho = \perpcomp/\parcomp$. Taking this claim as given, combining it with Definition~\ref{def:starnext_main_HO}, and performing some algebra yields the Gordon state evolution update~\eqref{eq:Gordon-AM-MLR}. It remains to establish the two equalities. We prove claim~\eqref{eq:mixtures_AM_G1Y} below; the proof of claim~\eqref{eq:mixtures_AM_G2Y} is similar and omitted for brevity.

\paragraph{Proof of equation~\eqref{eq:mixtures_AM_G1Y}:} Note that
\begin{align*}
\E\left[Z_1 \Omega \right]  &= \E\left[Z_1 \cdot  \sign(\parcomp Z_1 + \perpcomp Z_2) \cdot\sign(QZ_1+ \noisestd Z_3)\cdot (QZ_1+ \noisestd Z_3)\right] \\
&\overset{\1}{=} \E\left[Z_1 \cdot  \sign(\parcomp Z_1+\perpcomp Z_2) \cdot |Z_1+ \noisestd Z_3| \right].
\end{align*}
In step $\1$, 
we multiplied the expression by $Q^2 = 1$ and used the fact that $Z_3 \overset{(d)}{=} QZ_3$. To compute this expectation tractably, we use a change of variables. Let $Z' = \frac{\parcomp Z_1 + \perpcomp Z_2}{\sqrt{\parcomp^2 + \perpcomp^2}}$ and 
write 
\[
Z_1 = \frac{\parcomp}{\sqrt{\parcomp^2 + \perpcomp^2}} Z' + \frac{\perpcomp}{\sqrt{\parcomp^2 + \perpcomp^2}} \Ztil,
\]
where $\Ztil$ is a standard Gaussian independent of the tuple $(Z', Z_3)$. Finally, define the following standard Gaussian variate that is independent of $Z'$:
\[
Z'' = \left(\frac{\perpcomp^2}{\parcomp^2 + \perpcomp^2} + \noisestd^2 \right)^{-1/2} \left( \frac{\perpcomp}{\sqrt{\parcomp^2 + \perpcomp^2}} \cdot \Ztil + \noisestd Z_3 \right),
\]
and use $\sbar = \sqrt{ \frac{\perpcomp^2}{\parcomp^2 + \perpcomp^2} + \noisestd^2}$ to denote the normalization constant. 
Substituting above, we obtain
\begin{align*}
\E\left[Z_1 \Omega \right] &= \frac{\parcomp}{\sqrt{\parcomp^2 + \perpcomp^2}} \cdot \E \left[ |Z'| \cdot \left| \frac{\parcomp}{\sqrt{\parcomp^2 + \perpcomp^2}} Z' + \sbar Z'' \right| \right] \\
&\qquad +  \frac{\perpcomp}{\sqrt{\parcomp^2 + \perpcomp^2}} \cdot \E \left[ \Ztil \cdot \sign(Z') \cdot \left| \frac{\parcomp}{\sqrt{\parcomp^2 + \perpcomp^2}} Z' + \sbar Z'' \right| \right] \\
&=  \frac{\parcomp}{\sqrt{\parcomp^2 + \perpcomp^2}} \cdot \underbrace{ \E \left[ |Z'| \cdot \left| \frac{\parcomp}{\sqrt{\parcomp^2 + \perpcomp^2}} Z' + \sbar Z'' \right| \right]}_{T_1} \\
&\qquad +  \frac{\perpcomp^2}{\sbar \cdot (\parcomp^2 + \perpcomp^2)} \cdot \underbrace{\E \left[ Z'' \cdot \sign(Z') \cdot \left| \frac{\parcomp}{\sqrt{\parcomp^2 + \perpcomp^2}} Z' + \sbar Z'' \right| \right]}_{T_2}
\end{align*}
We may now use polar coordinates to compute the two expectations; write $Z' = R \cos \Psi$ and $Z'' = R \sin \Psi$.  Let $\lambda = \tan^{-1} (\sbar \sqrt{1 + \rho^2} )$ for convenience, so that 
\[
\frac{\parcomp}{\sqrt{\parcomp^2 + \perpcomp^2}} Z' + \sbar Z'' = \sqrt{1 + \sigma^2} \cdot R \cos(\Psi - \lambda).
\]

\noindent \underline{Evaluating term $T_1$:} We have
\begin{align*}
T_1 = \sqrt{1 + \sigma^2} \cdot \E[R^2] \cdot \left( \frac{1}{2\pi} \int_{0}^{2\pi} |\cos (\psi) \cdot \cos(\psi - \lambda)| d \psi \right) \overset{\1}{=} \frac{\sqrt{1 + \sigma^2} }{\pi} \cdot \left( (\pi - 2\lambda) \cos \lambda + 2 \sin \lambda \right),
\end{align*}
where step $\1$ follows from evaluating the integral explicitly, noting that $\cos (\psi) \cdot \cos(\psi - \lambda)$ is nonnegative except in the range $(\pi/2, \pi/2+ \lambda) \cup (3\pi/2, 3\pi/2 + \lambda)$.

\noindent \underline{Evaluating term $T_2$:} We have
\begin{align*}
T_2 = \sqrt{1 + \sigma^2} \cdot \E[R^2] \cdot \left( \frac{1}{2\pi} \int_{0}^{2\pi} \sin (\psi) \cdot \sign(\cos (\psi)) \cdot |\cos(\psi - \lambda)| d \psi \right) \overset{\1}{=} \sqrt{1 + \sigma^2} \cdot \left( \frac{\pi - 2\lambda}{\pi} \right) \cdot \sin \lambda,
\end{align*}
where once again, step $\1$ follows from evaluating the integral explicitly, noting that $\cos (\psi) \cdot \cos(\psi - \lambda)$ is nonnegative except in the range $(\pi/2, \pi/2+ \lambda) \cup (3\pi/2, 3\pi/2 + \lambda)$.

\noindent \underline{Putting together the pieces:} Since $\lambda = \tan^{-1} (\sbar \sqrt{1 + \rho^2} )$, we have 
\[
\sin \lambda = \frac{\sqrt{\rho^2 + \sigma^2 + \rho^2 \sigma^2}}{\sqrt{1 + \sigma^2} \cdot \sqrt{1 + \rho^2}} \qquad \text{ and } \qquad  \cos \lambda = \frac{1}{\sqrt{1 + \sigma^2} \cdot \sqrt{1 + \rho^2}}.
\]
Consequently, 
\begin{align*}
\E[Z_1 \Omega] &= \left( \frac{\pi - 2 \lambda }{\pi} \right) \cdot \frac{1}{1 + \rho^2} + \frac{2}{\pi} \cdot \frac{\sqrt{\rho^2 + \sigma^2 + \rho^2 \sigma^2}}{1 + \rho^2} + \left( \frac{\pi - 2\lambda}{\pi}\right) \cdot \frac{\rho^2}{1 + \rho^2} \\
&= 1 - \frac{2\lambda}{\pi} + \frac{2}{\pi} \cdot \frac{\sqrt{\rho^2 + \sigma^2 + \rho^2 \sigma^2}}{1 + \rho^2},
\end{align*}
as claimed.

\subsubsection{Verifying assumptions} To verify Assumption~\ref{ass:omega}, note that $\Omega^2 \leq 2 Z_1^2 + 2 \noisestd^2 Z_3^2$, so that $\EE[\exp(\Omega^2 / (2 + 2\noisestd^2)] \leq 1$. Consequently $\| \Omega \|_{\psi_2} \leq 2 (1 + \sigma^2)$. To verify Assumption~\ref{ass:omega-lb}, note from the calculations above that
\[
\EE[\Omega^2] - (\EE[Z_1 \Omega])^2 -  (\EE[Z_2 \Omega])^2 = 1 - \left(1 - A_{\noisestd}(\rho) + B_{\noisestd} (\rho) \right)^2 - [\rho B_{\noisestd}(\rho)]^2 + \noisestd^2 := H(\rho),
\]
where $A_{\noisestd}(\rho)$ and $B_{\noisestd}(\rho)$ were defined in equation~\eqref{eq:AB-shorthand}. Note that $H'(\rho) = \frac{4 \rho (\rho^2 (\rho^2 + 2) + \noisestd^2) (\pi - 2\tan^{-1} (\sqrt{\rho^2 (\rho^2 + 2) + \noisestd^2})}{\pi^2 (\rho^2 + 1)^2 \sqrt{\rho^2 (\rho^2 + 2) + \noisestd^2}} \geq 0$, so that $H(\rho) \geq H(0) = \noisestd^2$.

\subsection{Proof of Corollary~\ref{lem:Gordon-SE-GD-MLR}} \label{pf-cor4}

We evaluate the various expectations and verify Assumption~\ref{ass:omega}. The corollary then follows by invoking Theorem~\ref{thm:one_step_main_FO}.
In this case, we
have $\wfun(x, y) = x - \sign(yx) \cdot y$, so that
\begin{align*}
\Omega = \parcomp Z_1 + \perpcomp Z_2 - \sign(\parcomp Z_1 + \perpcomp Z_2) \cdot \sign(Q \cdot Z_1 + \noisestd Z_3) \cdot \left( Q \cdot Z_1 + \noisestd Z_3 \right).
\end{align*}
Here $Q$ is a Rademacher random variable. 

\subsubsection{Evaluating Gordon state evolution update} Note the definition of $\Omega$ in conjunction with equations~\eqref{eq:mixtures_AM_G1Y} and~\eqref{eq:mixtures_AM_G2Y}, and Remark~\ref{rem:weight-rel}.
Performing some straightforward algebra using Definition~\ref{def:starnext_main_FO}(b) yields the Gordon updates~\eqref{eq:Gordon-GD-MLR}.

\subsubsection{Verifying Assumption~\ref{ass:omega}} As in the previous subgradient update, we have the upper bound $\Omega^2 \leq 2 (1 + \parcomp^2) Z_1^2 + 2 \perpcomp^2 Z_2 + 2 \noisestd^2 Z_3^2$, so that $\EE[\exp(\Omega^2 / (2 \parcomp^2 + 2 \perpcomp^2 + 2\noisestd)] \leq 1$. Thus, we have $\| \Omega \|_{\psi_2} \leq 2 (\parcomp^2 + \perpcomp^2  + \noisestd^2)$.

\subsection{Proof of Fact~\ref{prop:Pop-PR}} \label{pf-fact1}

Let us begin by restating the population update for alternating minimization as applied to phase retrieval:
\begin{align*}
\parcomppop = 1 - \frac{1}{\pi}(2\anglecurr - \sin(2\anglecurr))  \qquad \text{ and } \qquad \perpcomppop = \frac{2}{\pi}\sin^2(\anglecurr). 
\end{align*}
We refer to these as the $\Fbar$ and $\Gbar$ maps respectively and let $\mathcal{S}_{\pop}$ denote the population state evolution operator. In order to prove the desired fact, it suffices to verify that $\mathcal{S}_{\pop}$ is $\Goodset$-faithful, and to prove upper and lower bounds on its one-step convergence.

\paragraph{Verifying $\Goodset$-faithfulness:} This follows directly from the $\Goodset$-faithfulness of the $\Sopgordon$ update, since $\Fbar = F$ and $\Gbar \leq G$.

\paragraph{Upper bound on one-step convergence:}
First, note that $\Gbar(\parcomp, \perpcomp) \leq \frac{4}{\pi^2} \anglecurr^4$.
On the other hand, equation~\eqref{eq:xi-ineq} yields
\begin{align*}
(1 - \Fbar(\parcomp, \perpcomp))^2 &\leq \frac{16}{9 \pi^2} \anglecurr^6.
\end{align*}
Putting together the pieces, we have
\begin{align*}
[\DeltaSE(\mathcal{S}_{\pop}(\bzeta))]^2 = [\Gbar(\parcomp, \perpcomp)]^2 + (1 - \Fbar(\parcomp, \perpcomp))^2 \leq \perpcomp^4 \leq \left\{ \perpcomp^2 + (1 - \parcomp)^2 \right\}^{2},
\end{align*}
where the first inequality is a result of noting that $\anglecurr \leq \perpcomp / \parcomp$ and $\parcomp \geq 0.55$.

\paragraph{Lower bound on two-step convergence:}
Moving now to the lower bound, let us compute two steps of the Gordon update, letting $F_{+} = F^2(\parcomp, \perpcomp)$ and $G_+ = G^2(\parcomp, \perpcomp)$. Analogously, we let $F = F(\parcomp, \perpcomp)$ and $G = G(\parcomp, \perpcomp)$, and use $\anglecurr_+ = \tan^{-1} (G/F)$ to denote the angle after one step of the Gordon update. Recall that $\tan \anglecurr = \perpcomp/\parcomp$. We have
\begin{align*}
G_{+}^2 = \frac{4}{\pi^2} \cdot \sin^4 \anglecurr_+ = \frac{4G^4}{\pi^2(F^2 + G^2)^2} \geq \frac{4 \pi^2}{(\pi^2 + 4)^2} \cdot G^4 > \frac{G^4}{5},
\end{align*}
where the penultimate inequality uses the fact that $F \leq 1$ and $G \leq 2/\pi$, guaranteed by one step of the Gordon update. For the $F$ component, equation~\eqref{eq:Fmap-stuff} yields $\pi^2 (1 - F_+)^2  \geq \frac{G^6}{8}$. Furthermore, we have
$
(1 - F)^4 \lesssim \perpcomp^{12} \lesssim G^6,
$
so that putting together the pieces yields
\begin{align*}
G_{+}^2 + (1 - F_+)^2 \geq \frac{G^4}{5}  + c (1 - F)^4  \geq  c' \cdot \left\{ G^2 + (1 - F)^2 \right\}^{2},
\end{align*}
where the last step follows because $(A + B)^{\kappa} \leq 2^{\kappa} (A^\kappa + B^{\kappa})$ for any positive scalars $(A, B)$ and $\kappa \geq 1$. Taking square roots completes the proof.
\qed

\begin{remark} \label{lem:quad-angle-PR}
The claimed quadratic convergence holds in a region much larger than the good region $\GPR$. Indeed, it is straightforward to show that $\tan^{-1} (G/F) \leq \frac{2}{\pi} \cdot \left( \tan^{-1} (\perpcomp / \parcomp) \right)^2$, showing that the \emph{angle} converges quadratically fast, globally for any $\anglecurr < \pi/2$.
\end{remark}

\subsection{Proof of Lemma~\ref{lem:map-ineqs}} \label{app:pf-maps}

Since the lemma consists of several parts, we prove each in turn.

\subsubsection{Proof of part (a)}
Since the maps $F_0$ and $F$ coincide, we have $F_0(\parcomp, \perpcomp) = 1 - \frac{2\anglecurr - \sin(2\anglecurr)}{\pi}$, which is a non-increasing function of $\anglecurr$. Evaluating it at $\anglecurr = \{0, \pi/2\}$, we obtain $0 \leq F_0(\parcomp, \anglecurr) \leq 1$. Next, note that $\sin x \geq x - \frac{x^3}{3!}$ for $x \geq 0$ to obtain
\begin{align*}
F_0(\parcomp, \perpcomp) = 1 - \frac{2\anglecurr - \sin(2\anglecurr)}{\pi} \geq 1 - \frac{4}{3\pi} \anglecurr^3.
\end{align*}
Finally, using the fact that $\sin x \leq x - \frac{x^3}{3!} + \frac{x^5}{5!}$ for $x \geq 0$, we have that for $\rho \leq 1/5$,
\begin{align*}
1 - F_0(\parcomp, \perpcomp) = \frac{2\anglecurr - \sin(2\anglecurr)}{\pi} \geq \frac{2}{5} \anglecurr^3.
\end{align*}
Here the final inequality uses the fact that $2\anglecurr \leq 2\rho \leq 2/5$. 
\qed

\subsubsection{Proof of part (b)}
Introduce the change of variables $\zeta = \pi/2 - \anglecurr = \tan^{-1} (\parcomp/\perpcomp)$.  We have 
\begin{align*}
F_0(\parcomp, \perpcomp) = F(\parcomp, \perpcomp) = \frac{2 \zeta + \sin 2\zeta }{\pi}
\stackrel{\1}{\geq} \frac{4\zeta}{\pi} - \frac{4 \zeta^3}{3\pi}
\stackrel{\2}{\geq} \frac{4}{\pi} \left( \frac{\parcomp}{\perpcomp} \right) - \frac{8}{3\pi} \left( \frac{\parcomp}{\perpcomp} \right)^3
\end{align*}
where in step $\1$, we have used the fact that $\sin x \geq x - \frac{x^3}{3!}$ and in step $\2$ we have used the Taylor expansion of the $\tan^{-1}$ function to conclude that $x - \frac{x^3}{3} \leq \tan^{-1} x \leq x$.
Now using the facts that $\parcomp/\perpcomp \leq 1/2$ and $\perpcomp \leq 1$, respectively, we have
\begin{align*}
F(\parcomp, \perpcomp) \geq \frac{10}{3\pi} \left( \frac{\parcomp}{\perpcomp} \right) > 1.06 \cdot \parcomp
\end{align*}
as desired.
\qed

\subsubsection{Proof of part (c)}
Consider the map $f: (\rho, \noisestd) \mapsto 1 - A_{\noisestd}(\rho) + B_{\noisestd}(\rho)$, and note that $\frac{\partial f}{\partial \noisestd} = \frac{1}{\pi} \cdot \frac{\noisestd}{(2 \noisestd^2 + 2) \sqrt{\rho^2 + \noisestd^2 + \rho^2 \noisestd^2}}$, which is non-negative for each $\rho, \noisestd \geq 0$. Thus $F_{\noisestd}$ is non-decreasing in $\noisestd$ for each $\rho$, i.e., for each $(\parcomp, \perpcomp)$ pair. 

Also note that $\frac{\partial f}{\partial \rho} = - \frac{1}{\pi} \cdot \frac{(\noisestd^2 + 2) + \noisestd^2/\rho^2}{2(\rho^2+1)^2 \sqrt{\rho^2 + \noisestd^2 + \rho^2 \noisestd^2}}$ which is non-positive for each $\rho, \noisestd \geq 0$. Thus, 
\begin{align*}
F_{\noisestd}(\parcomp, \perpcomp) \leq 1 - A_{\noisestd}(0) + B_{\noisestd}(0) = 1 + \frac{2}{\pi} ( \noisestd - \tan^{-1} (\noisestd)) \leq  1 + \frac{2\noisestd^3}{3\pi},
\end{align*}
where the final inequality uses $\tan^{-1} x \geq x - x^3 / 3$.
\qed

\subsubsection{Proof of part (d)}
Note that we can simplify $G_{\noisestd}(\parcomp, \perpcomp) = \sqrt{[\rho B_{\noisestd}(\rho)]^2 \cdot \frac{\kappa - 2}{\kappa - 1} + \frac{1}{\kappa - 1}\left(1 + \noisestd^2 - [F_{\noisestd}(\parcomp, \perpcomp)]^2 \right) }$, from which the lower bound follows immediately. To prove the upper bound, note that
\begin{align}
G_{\noisestd}(\parcomp, \perpcomp) &\leq \sqrt{[\rho B_{\noisestd}(\rho)]^2 \cdot \frac{\kappa - 2}{\kappa - 1} + \frac{\noisestd^2}{\kappa - 1} + \frac{(1 + F_{\noisestd}(\parcomp, \perpcomp))}{\kappa - 1}  } \\
&\leq \sqrt{[\rho B_{\noisestd}(\rho)]^2 + \frac{\noisestd^2}{\kappa - 1} + \frac{1}{\kappa - 1} \cdot \left( 1 + \frac{\noisestd^3}{3\pi} \right) },
\end{align}
where the second step follows since $F_{\noisestd}(\parcomp, \perpcomp) \leq 1 + \frac{2\noisestd^3}{3\pi}$, as proved in the previous part. Now note that a straightforward calculation yields that for all $\rho \geq 0$, we have
\begin{align*}
\rho B_{\noisestd}(\rho) \leq \frac{2}{\pi} \sqrt{1 + \noisestd^2}.
\end{align*} 
Noting that $\noisestd \leq 1/2$, and choosing $\kappa \geq C$ for a large constant $C$, we have
\begin{align*}
G_{\noisestd}(\parcomp, \perpcomp) &\leq \sqrt{\frac{5}{\pi^2} + 0.2} \leq 0.8.
\end{align*}
\qed

\subsubsection{Proof of part (e)}
We have
\begin{align}
[G_0(\parcomp, \perpcomp)]^2 &= \frac{\kappa - 2}{\kappa - 1} \cdot \frac{4}{\pi^2} \cdot \sin^4 \anglecurr + \frac{1}{\kappa - 1} \cdot (1 - [F_0(\parcomp, \perpcomp)]^2 )  \notag \\
&\leq \frac{4}{\pi^2} \anglecurr^4 + \frac{1}{\kappa - 1} (1 - F(\parcomp, \perpcomp)) (1 + F(\parcomp, \perpcomp))  \notag \\
&\leq \frac{4}{\pi^2} \anglecurr^4 + \frac{8\anglecurr^3}{3\pi(\kappa - 1)}   \notag \\
&\leq \frac{\anglecurr^3}{10}. \label{eq:sigma-ineq}
\end{align}
Here the penultimate inequality makes use of parts (a) and (c) of the lemma to conclude that $1 - F(\parcomp, \perpcomp) \leq \frac{4\anglecurr^3}{3\pi}$ and $1 + F(\parcomp, \perpcomp) \leq 2$, respectively.
The last line follows since $\kappa \geq C$ and $\anglecurr \leq 1/5$. 
\qed

\subsubsection{Proof of part (f)}
By definition of the maps, performing some algebra yields
\begin{align}
[g_{\noisestd}(\parcomp, \perpcomp)]^2 &= \frac{\kappa - 2}{\kappa - 1} \cdot [\rho B_{\noisestd}(\rho)]^2 + \frac{1}{\kappa} \left\{ (\parcomp - F_{\noisestd}(\parcomp, \perpcomp))^2 + (\perpcomp - \rho B_{\noisestd}(\rho) )^2 + 1 - [F_{\noisestd}(\parcomp, \perpcomp)]^2 + \noisestd^2 \right\} \notag \\
&\leq [G_{\noisestd}(\parcomp, \perpcomp)]^2 + \frac{1}{\kappa} \left\{ (\parcomp - F_{\noisestd}(\parcomp, \perpcomp))^2 + (\perpcomp - \rho B_{\noisestd}(\rho) )^2 \right\}, \label{eq:keygG}
\end{align}
Using the numeric inequality $(a + b)^2 \leq 2a^2 + 2b^2$ twice, we obtain
\begin{align*}
(\parcomp - F_{\noisestd}(\parcomp, \perpcomp))^2 + (\perpcomp - \rho B_{\noisestd}(\rho) )^2 &\leq 2 (1 - \parcomp)^2 +  2\perpcomp^2 + 2 (1 - F_{\noisestd}(\parcomp, \perpcomp))^2 + 2 [\rho B_{\noisestd}(\rho)]^2,
\end{align*}
and combining the pieces completes the proof of the upper bound.

To prove the lower bound, note that
\[
\frac{\kappa - 2}{\kappa - 1} \cdot [\rho B_{\noisestd}(\rho)]^2 + \frac{1}{\kappa} \left\{ (\parcomp - F_{\noisestd}(\parcomp, \perpcomp))^2 + (\perpcomp - \rho B_{\noisestd}(\rho) )^2 + 1 - [F_{\noisestd}(\parcomp, \perpcomp)]^2 + \noisestd^2 \right\} \geq \frac{\kappa - 1}{\kappa} \cdot [G_{\noisestd}(\parcomp, \perpcomp)]^2.
\]
\qed

\subsubsection{Proof of part (g)}
Combining the fact that $G_{\noisestd}(\parcomp, \perpcomp) = \sqrt{\rho^2 B_{\noisestd}(\rho)^2 \cdot \frac{\kappa - 2}{\kappa - 1} + \frac{1}{\kappa - 1}\left(1 + \noisestd^2 - [F_{\noisestd}(\parcomp, \perpcomp)]^2 \right) }$ with the inequalities $F_{\noisestd}(\parcomp, \perpcomp) \leq 1 + \frac{2\noisestd^3}{3\pi}$ and $\noisestd \leq 1/2$, the lower bound on $\frac{G_{\noisestd}}{F_{\noisestd}}$ follows from straightforward calculation.

To prove the upper bound, begin by noting that since $1 - F_{\noisestd}(\parcomp, \perpcomp) \leq \frac{4\anglecurr^3}{3\pi}$, we have
\begin{align*}
G_{\noisestd}(\parcomp, \perpcomp) \leq \sqrt{\rho^2 B_{\noisestd}(\rho)^2 + \frac{\noisestd^2}{\kappa - 1} + \frac{8 \anglecurr^3}{3 \pi(\kappa - 1)} \cdot \left( 1 + \frac{\noisestd^3}{3\pi} \right) }.
\end{align*}
We now use the fact that $\anglecurr \leq \rho$,
and that $\phi \leq 1.12$ for all $\rho \leq 2$.
In conjunction with the assumption $\noisestd \leq 0.5$, we obtain
\begin{align*}
G_{\noisestd}(\parcomp, \perpcomp) &\leq \rho B_{\noisestd}(\rho) + \frac{1.05\rho}{\sqrt{\kappa - 1}} + \frac{1.1\noisestd}{\sqrt{\kappa - 1}}
\end{align*}
where we have also used the inequality $\sqrt{a + b + c} \leq \sqrt{a} + \sqrt{b} + \sqrt{c}$, valid for any three non-negative scalars $(a, b, c)$. 
Using the definition $F_{\noisestd}(\parcomp, \perpcomp) = 1 - A_{\noisestd}(\rho) + B_{\noisestd}(\rho)$, we have
\begin{align*}
\frac{G_{\noisestd}}{F_{\noisestd}} &\leq \rho \left( 1 - \frac{1 - A_{\noisestd}(\rho)}{1 + B_{\noisestd}(\rho) - A_{\noisestd}(\rho)} + \frac{1.05}{F_{\noisestd} \cdot \sqrt{\kappa - 1}}\right) + \frac{1.1 \noisestd}{F_{\noisestd} \cdot \sqrt{\kappa - 1}} \\
&\leq \rho \left( 1 - \frac{1 - A_{\noisestd}(\rho)}{1 + B_{\noisestd}(\rho) - A_{\noisestd}(\rho)} + \frac{2.1}{\sqrt{\kappa - 1}}\right) + \frac{2\noisestd}{\sqrt{\kappa - 1}},
\end{align*}
where in the final inequality, we have used the fact that $A_{\noisestd}(\rho) \geq 0$ for all $\rho$ and 
$F_{\noisestd} \geq F_0 \geq 0.56$ for all $\rho \leq 2$.
Now note that if $\noisestd \leq 0.5$ and $\rho \leq 2$, then a straightforward computation yields that $A_{\noisestd}(\rho) \leq 0.74$ and $B_{\noisestd}(\rho) \leq 0.4$. Since $1 + B_{\noisestd}(\rho) - A_{\noisestd} (\rho) \geq 0$, we have $\frac{1 - A_{\noisestd}(\rho)}{1 + B_{\noisestd}(\rho) - A_{\noisestd}(\rho)} \geq 1/4$. Putting together the pieces yields
\begin{align*}
\frac{G_{\noisestd}}{F_{\noisestd}} &\leq \rho \left( 0.75 + \frac{2.1}{\sqrt{\kappa - 1}}\right) + \frac{2\noisestd}{\sqrt{\kappa - 1}} \leq \frac{4}{5} \cdot \frac{\perpcomp}{\parcomp} + \frac{2\noisestd}{\sqrt{\kappa - 1}},
\end{align*}
where the final inequality uses the fact that $\kappa \geq C$. 
\qed

\subsection{Proof of Lemma~\ref{lem:grad-map-ineqs}} \label{app:pf-gradmaps}
Once again, we prove each part separately.

\subsubsection{Proof of part (a)}
By definition, we have $F_{\noisestd}(\parcomp, \perpcomp) = \Phi(\rho(\parcomp, \perpcomp))$ for a univariate. $\noisestd$-dependent function $\Phi$. By chain rule,
$\nabla F_{\noisestd}(\parcomp, \perpcomp) = \Phi'(\rho) \cdot \nabla \rho(\parcomp, \perpcomp)$. In addition, $\nabla \rho(\parcomp, \perpcomp) = \parcomp^{-1} \cdot (1, - \rho)$, so that $\| \nabla \rho(\parcomp, \perpcomp) \|_1 = \parcomp^{-1} (1 + \rho)$. 
Differentiating the univariate function $\Phi$, we obtain $| \Phi'(\rho)| = \frac{2\rho}{\pi} \cdot \frac{\rho^2 (2 + \noisestd^2) + \noisestd^2}{(1 + \rho^2)^2 \cdot \sqrt{\rho^2(1 + \noisestd^2) + \noisestd^2}}$. Thus, we have
\begin{align*}
\|\nabla F_{\noisestd}(\parcomp, \perpcomp) \|_1 = \frac{2\rho}{\pi \parcomp} \cdot \frac{\rho^2 (2 + \noisestd^2) + \noisestd^2}{\sqrt{\rho^2(1 + \noisestd^2) + \noisestd^2}} \cdot \frac{(1 + \rho)}{(1 + \rho^2)^2} =: \parcomp^{-1} \cdot f(\rho).
\end{align*}
We now claim that $f(\rho)$ is non-decreasing in the interval $[0, 0.25]$. This claim directly yields $\|\nabla F_{\noisestd}(\parcomp, \perpcomp) \|_1 \leq \parcomp^{-1} \cdot f(0.25) \leq 1/2$ for all $\parcomp \geq 0.5$, $\rho \leq 0.25$, and $\noisestd \leq 1/2$.

To prove that $f$ is non-decreasing, note that a straightforward calculation yields
\begin{align}
f'(\rho) &= \frac{2(\noisestd^2 + (2+ \noisestd^2)\rho^2)}{\pi (1 + \rho^2)^2\sqrt{\noisestd^2 + (1 + \noisestd^2) \rho^2}} \cdot \left[1 + 2\rho - \frac{\rho^2(1 + \rho)}{\sqrt{\noisestd^2 + (1 + \noisestd^2)\rho^2}} - \frac{4\rho^2(1 + \rho)}{1 + \rho^2}\right] \notag \\
&\qquad \qquad + \frac{8(2 + \noisestd^2)\rho^2(1 + \rho)}{\pi(1 + \rho^2)^2\sqrt{\noisestd^2 + (1 + \noisestd^2)\rho^2}}. \label{eq:derivative_f}
\end{align}
Since
\[
\frac{\rho^2 (1 + \rho)}{\sqrt{\noisestd^2 + (1 + \noisestd^2) \rho^2}} \leq \rho(1 + \rho),
\]
we have
\begin{align}
\label{ineq:penultimate_f_increasing}
1 + 2\rho - \frac{\rho^2(1 + \rho)}{\sqrt{\noisestd^2 + (1 + \noisestd^2)\rho^2}} - \frac{4\rho^2(1 + \rho)}{1 + \rho^2} \geq \frac{1 + \rho - 4\rho^2 - 3\rho^3 - \rho^4}{1 + \rho^2} \geq 0.
\end{align}
Here, the final inequality follows from the following argument: note that $\widetilde{f}: \rho \mapsto 1 + \rho - 4\rho^2 -3\rho^3 -\rho^4$ is concave and thus on the interval $[0, 0.25]$, it takes its minimizer at one of the endpoints. Furthermore, $\min \{ \widetilde{f}(0), \widetilde{f}(0.25) \} > 0$.
\qed

\subsubsection{Proof of part (b)}
Writing $G_{\noisestd}(\parcomp, \perpcomp) = \Gamma(\rho)$, note that
\begin{align*}
\Gamma(\rho) &= \sqrt{\rho^2 B_{\noisestd}(\rho)^2 \cdot \frac{\kappa - 2}{\kappa - 1} + \frac{1}{\kappa - 1}\left(1 + \noisestd^2 - [\Phi(\rho)]^2 \right) }
\end{align*}
By Lemma~\ref{lem:map-ineqs}(a), we have
$\Phi(\rho) \leq 1 + \frac{2\noisestd^3}{3\pi}$,
and so we obtain
\begin{align}
\label{lb:G}
\Gamma(\rho) &\geq \sqrt{\rho^2B_{\noisestd}(\rho)^2 \frac{\kappa - 2}{\kappa - 1} + \frac{1}{\kappa - 1}\left(1 + \noisestd^2 - \left(1 + \frac{2\noisestd^3}{3\pi}\right)^2\right)} \notag \\
& \geq \sqrt{\rho^2B_{\noisestd}(\rho)^2 \frac{\kappa - 2}{\kappa - 1} + \frac{\noisestd^2}{2(\kappa - 1)}} \geq \rho B_{\noisestd}(\rho) \sqrt{\frac{\kappa - 2}{\kappa - 1}},
\end{align}
where the first inequality holds since $\noisestd \leq 0.5$.  Now, define the function
\[
S(\rho) = [\rho B_{\noisestd}(\rho)]^2 + \frac{1}{\kappa - 1}\left(1 + \noisestd^2 - (1 - A_{\noisestd}(\rho) + B_{\noisestd}(\rho))^2 - \rho^2 B_{\noisestd}(\rho)^2\right),
\]
and note that $\Gamma(\rho) = \sqrt{S(\rho)}$ and $\Gamma'(\rho) = \frac{S'(\rho)}{2\sqrt{S(\rho)}}$.
We then have 
\begin{align}
\label{ineq:G_l1_bound}
\| \nabla G_{\noisestd}(\parcomp, \perpcomp) \|_1 = \frac{\parcomp^{-1}}{2} \cdot \frac{\lvert S'(\rho) \rvert}{\sqrt{S(\rho)}} (1 + \rho)  \overset{\1}{\leq} \frac{\lvert S'(\rho) \rvert}{\sqrt{S(\rho)}} (1 + \rho) =:  T_1 + T_2,
\end{align}
where step $\1$ follows since $\parcomp \geq 1/2$, and we have let 
\begin{subequations}
\begin{align}
T_1 &= \frac{2(\kappa - 2)}{(\kappa - 1)\sqrt{S(\rho)}} \rho(1+ \rho) B_{\noisestd}(\rho) \left(B_{\noisestd}(\rho) + \rho B_{\noisestd}'(\rho)\right) \label{eq:T1_appendix_convergence}\\
T_2 &= \frac{2(1 + \rho)}{(\kappa - 1)\sqrt{S(\rho)}}(1 - A_{\noisestd}(\rho) + B_{\noisestd}(\rho))(B_{\noisestd}'(\rho) - A_{\noisestd}'(\rho)). \label{eq:T2_appendix_convergence}
\end{align}
\end{subequations}
We will bound each of these terms in turn.

\paragraph{Bounding the term $T_1$~\eqref{eq:T1_appendix_convergence}.} An explicit computation yields
\begin{align}
\label{ub:first_ub_T1}
T_1 = \frac{1}{\sqrt{S(\rho)}} \frac{8 (\kappa - 2)}{\pi^2 (\kappa - 1)} \rho (1 + \rho) \frac{2\rho^2 + \noisestd^2 + \noisestd^2 \rho^2}{(1 + \rho^2)^3 } \overset{\1}{\leq} \frac{4}{\pi} \frac{(1 + \rho)(2 \rho^2 + \noisestd^2 (1 + \rho^2))}{(1 + \rho^2)^2 \sqrt{\rho^2 + \noisestd^2(1  + \rho^2)}},
\end{align}
where step $\1$ follows by using the inequality~\eqref{lb:G}.  Now, define the function
\[
a(\rho) :=  \frac{(1 + \rho)(2 \rho^2 + \noisestd^2 (1 + \rho^2))}{(1 + \rho^2)^2 \sqrt{\rho^2 + \noisestd^2(1  + \rho^2)}},
\]
so that the inequality~\eqref{ub:first_ub_T1} is equivalent to the inequality
\[
T_1 \leq \frac{4}{\pi} a(\rho).
\]
Then, note that
\begin{align*}
a'(\rho) &= \frac{\noisestd^4(1 - 2\rho^6 - 3\rho^5) + \rho^4(3 - 4\rho^2 - 6\noisestd^2\rho^2 - 6\rho -9\noisestd^2\rho) + \rho^4(1 - 3\noisestd^4)}{(1 + \rho^2)^3 (\rho^2 + \noisestd^2(1 + \rho^2))^{3/2}}\\
&\qquad \qquad + \frac{\rho^3(2 - 6\noisestd^2 -6\noisestd^4) + \rho^2(6\noisestd^2) + \rho(3\noisestd^2 - 3\noisestd^4)}{(1 + \rho^2)^3 (\rho^2 + \noisestd^2(1 + \rho^2))^{3/2}} \geq 0,
\end{align*}
where the last inequality can be verified for all $0 \leq \rho \leq 1/4$ and $0 \leq \noisestd \leq 1/2$.  Thus, for all $0 \leq \noisestd \leq 0.5$, $a(\rho)$ is an increasing function of $\rho$ on the interval $[0, 0.25]$.  Note also that the function $\noisestd \mapsto (a\noisestd^2 + 2d)/(c\sqrt{a\noisestd^2 + d})$ is increasing for $\noisestd \geq 0$.  Combining these pieces implies that
\begin{align}
	\label{ub:final_T1}
	T_1 \leq \frac{4}{\pi} a(\rho) \overset{\1}{\leq} 0.962,
\end{align}
where step $\1$ evaluated $a$ with $\rho = 0.25$ and $\noisestd = 0.5$.

\paragraph{Bounding the term $T_2$~\eqref{eq:T2_appendix_convergence}.} First, note that $1 - A_{\noisestd}(\rho) + B_{\noisestd}(\rho) = \Phi(\rho)$ and $B_{\noisestd}'(\rho) - A_{\noisestd}'(\rho) = \Phi'(\rho)$, so we have
\begin{align}
T_2 = \frac{2\Phi(\rho) \lvert \Phi'(\rho)\rvert(1 + \rho)}{\sqrt{S(\rho)}(\kappa - 1)} \overset{\1}{\leq} 2\frac{\left(1 + \frac{2\noisestd^3}{3\pi}\right)}{\kappa - 1} \cdot \frac{\lvert \Phi'(\rho) \rvert(1 + \rho)}{\sqrt{S(\rho)}} &\overset{\2}{\leq} 4\frac{\left(1 + \frac{2\noisestd^3}{3\pi}\right)}{\kappa - 1}\frac{2(2 \rho^2 + \noisestd^2(1 + \rho^2))(1 + \rho)}{\pi (1 + \rho^2)(\rho^2 + \noisestd^2(1 + \rho^2))}\nonumber\\
&\leq \frac{20\left(1 + \frac{2\noisestd^3}{3\pi}\right)}{\pi (\kappa - 1)},
\end{align}
where step $\1$ follows from Lemma~\ref{lem:map-ineqs}(a), step $\2$ follows by computing $\Phi'(\rho)$ explicitly and using the fact that the inequality~\eqref{lb:G} implies $\sqrt{S(\rho)} \geq \rho B_{\noisestd}(\rho) \sqrt{\frac{\kappa - 2}{\kappa - 1}} \geq \frac{\rho B_{\noisestd}(\rho)}{2}$.  Under the assumption $\kappa \geq C$, we thus see that 
\begin{align}
\label{ub:final_T2}
T_2 \leq 0.018.
\end{align}
Combining the upper bound on $T_1$~\eqref{ub:final_T1}, the upper bound on $T_2$~\eqref{ub:final_T2}, and the decomposition~\eqref{ineq:G_l1_bound} yields
\[
\| \nabla G_{\noisestd}(\parcomp, \perpcomp) \|_1 \leq 0.98
\]
for all $\parcomp \geq 1/2$ and $\rho \leq 1/4$, as desired.
\qed

\subsubsection{Proof of part (c)}
Let us establish the lemma for the map $G$; an identical argument also holds for the map $g$. Recall that for all $(\parcomp, \perpcomp)$ we have $[G(\parcomp, \perpcomp)]^2 = [G_0(\parcomp,\perpcomp)]^2 + \frac{\noisestd^2}{\kappa - 1}$. Taking a gradient of both sides of the equation with respect to $(\parcomp, \perpcomp)$ yields
\begin{align*}
2 \cdot [G(\parcomp, \perpcomp)] \cdot [\nabla G (\parcomp, \perpcomp)] = 2 \cdot [G_0(\parcomp, \perpcomp)] \cdot [\nabla G_0 (\parcomp, \perpcomp)],
\end{align*}
so that $\nabla G(\parcomp, \perpcomp) = \frac{G_0(\parcomp, \perpcomp)}{G (\parcomp, \perpcomp)} \cdot \nabla G_0 (\parcomp, \perpcomp)$. Taking $\ell_1$ norms on both sides and noting that $G(\parcomp, \perpcomp) \geq G_0(\parcomp, \perpcomp)$, we have $\| \nabla G(\parcomp, \perpcomp) \|_1 \leq \| \nabla G_0(\parcomp, \perpcomp) \|_1$, as desired.
\qed

\subsubsection{Proof of part (d)}
From equation~\eqref{eq:keygG} we see that for each $(\parcomp, \perpcomp)$ pair, the following relation holds:
\[
[g_{\noisestd}(\parcomp, \perpcomp)]^2 = [G_{\noisestd}(\parcomp, \perpcomp)]^2 + \frac{1}{\kappa} \left\{ (\parcomp - F_{\noisestd}(\parcomp, \perpcomp))^2 + (\perpcomp - \rho B_{\noisestd}(\rho) )^2 \right\}.
\]
Taking gradients and then $\ell_1$ norms on both sides, we have
\begin{align*}
&2 [g_{\noisestd}(\parcomp, \perpcomp)] \cdot \| \nabla g_{\noisestd}(\parcomp, \perpcomp) \|_1 \\
&\leq 2 [G_{\noisestd}(\parcomp, \perpcomp)] \cdot \| \nabla G_{\noisestd}(\parcomp, \perpcomp) \|_1 \\
& \qquad + \frac{2}{\kappa} \cdot \Big( | \parcomp -  F_{\noisestd}(\parcomp, \perpcomp)| \cdot (1 + \| \nabla F_{\noisestd}(\parcomp, \perpcomp)\|_1) + | \perpcomp -  \rho B_{\noisestd}(\rho)| \cdot (1 + \| \nabla_{\parcomp, \perpcomp} \rho B_{\noisestd}(\rho) \|_1) \Big)
\end{align*}
Noting that $g_{\noisestd}(\parcomp, \perpcomp) \geq G_{\noisestd}(\parcomp, \perpcomp) \lor \frac{| \parcomp -  F_{\noisestd}(\parcomp, \perpcomp)|}{\sqrt{\kappa}} \lor \frac{| \perpcomp -  \rho B_{\noisestd}(\rho)|}{\sqrt{\kappa}}$, we have
\begin{align*}
\| \nabla g_{\noisestd}(\parcomp, \perpcomp) \|_1 &\leq \| \nabla G_{\noisestd}(\parcomp, \perpcomp) \|_1 + \frac{1}{\sqrt{\kappa}} \left( 1 + \| \nabla F_{\noisestd}(\parcomp, \perpcomp)\|_1 + 1 + \| \nabla \rho B_{\noisestd}(\rho) \|_1 \right) \\
&\leq \| \nabla G_{\noisestd}(\parcomp, \perpcomp) \|_1 + \frac{1}{\sqrt{\kappa}} \left( 3 + \| \nabla F_{\noisestd}(\parcomp, \perpcomp)\|_1 \right),
\end{align*}
where we have used the shorthand $\nabla \rho B_{\noisestd}(\rho) \equiv \nabla_{\parcomp, \perpcomp} [\rho B_{\noisestd}(\rho)]$, and
the final inequality holds for $\rho \leq 1/4$ and $\noisestd \leq 1/2$, since
\begin{align*}
\| \nabla \rho B_{\noisestd}(\rho) \|_1 &=  (1 + \rho) \cdot \left( \frac{\sqrt{\rho^2 + \noisestd^2 + \noisestd^2 \rho^2}}{1 + \rho^2} + \frac{\rho^2 (1 + \noisestd^2)}{(1 + \rho^2) \sqrt{\rho^2 + \noisestd^2 + \noisestd^2 \rho^2}} - \frac{2\rho^2 \sqrt{\rho^2 + \noisestd^2 + \noisestd^2 \rho^2} }{(1 + \rho^2)^2}\right)\\
&= \frac{(1 + \rho) \cdot 2\rho^2}{(1 + \rho^2)^2\sqrt{\rho^2 + \noisestd^2 + \noisestd^2 \rho^2}} + \frac{(1 + \rho) \cdot \sigma^2}{(1 + \rho^2)\sqrt{\rho^2 + \noisestd^2 + \noisestd^2 \rho^2}}\\
&\overset{\1}{\leq} 2\rho + \sigma \leq 1,
\end{align*}
where step $\1$ follows by using the simple bounds $(1 + \rho) \leq (1 + \rho^2)^2$ and $\sqrt{\rho^2 + \sigma^2 + \sigma^2\rho^2} \geq \rho$ to upper bound the first term and $\sqrt{\rho^2 + \sigma^2 + \sigma^2\rho^2} \geq \sigma \sqrt{1 + \rho^2}$ to bound the second term.  The final inequality follows by using the fact that $\rho \leq 1/4$ and $\sigma \leq 1/2$. 
\qed

\subsection{Proof of Lemma~\ref{lem:generalb}} \label{sec:bproof}

Let $(\parcompbar_t, \perpcompbar_t) \defn \Sopbar^t (\parcomp_0, \perpcomp_0)$.
Owing to the $\Goodset$-faithfulness of $\Sopbar = (\Fbar, \Gbar)$, we have $(\parcompbar_t, \perpcompbar_t) \in \Goodset$ for all $t \geq 0$. Furthermore, by assumption, the sequence $\{\bzeta_t = (\parcomp_t, \perpcomp_t)\}_{t \geq 0}$
satisfies, for all $t \geq 0$,
\begin{align*}
\| \bzeta_{t+1} - \Sopbar(\parcomp_t, \perpcomp_t) \|_\infty \leq \Delta.
\end{align*}

We now claim that for each $t \geq 0$, we have
\begin{align} \label{eq:key-claim-prob}
|\parcomp_{t+1} - \parcompbar_{t+1}| \lor |\perpcomp_{t+1} - \perpcompbar_{t+1}| \leq \Delta + (1 - \tau) \cdot \left( |\parcomp_t - \parcompbar_t| \lor |\perpcomp_t - \perpcompbar_t| \right).
\end{align}
Note that this claim immediately yields the desired result, since applying it iteratively for $t = 0, 1, \ldots, k - 1$ and using the fact that $\parcompbar_0 = \parcomp_0$ and $\perpcompbar_0 = \perpcomp_0$ yields
\begin{align*}
| \parcompbar_k - \parcomp_k | \lor |\perpcompbar_k - \perpcomp_k| \leq \sum_{i = 0}^{k - 1} (1 - \tau)^i \cdot \Delta \leq \frac{\Delta}{\tau}.
\end{align*}
It remains to prove claim~\eqref{eq:key-claim-prob}, and we do so by induction. 

\paragraph{Base case:} The case $t = 0$ is clearly true, since $\parcompbar_0 = \parcomp_0$ and $\perpcompbar_0 = \perpcomp_0$.

\paragraph{Induction step:} Suppose that the claim is true for all $t \leq k - 1$, so that 
\begin{align*}
| \parcompbar_k - \parcomp_k | \lor |\perpcompbar_k - \perpcomp_k| \leq \sum_{i = 0}^{k - 1} (1 - \tau)^i \cdot \Delta \leq \frac{\Delta}{\tau}.
\end{align*}
We must show that it holds for $t = k$. By triangle inequality, we have
\begin{align*}
|\parcompbar_{k+1} - \parcomp_{k + 1}| &\leq | \Fbar(\parcompbar_k, \perpcompbar_k) - \Fbar(\parcomp_k, \perpcomp_k)| + | \Fbar(\parcomp_k, \perpcomp_k) - \parcomp_{k+1}|.
\end{align*}
Let $\parcomp^u := u \parcompbar_k + (1 - u) \parcomp_k$ and define $\perpcomp^u$ analogously. Let $\bzeta_k = (\parcomp_k, \perpcomp_k)$ and $\bzetabar_k = (\parcompbar_k, \perpcompbar_k)$. Since $\Fbar$ is a continuously differentiable function of its arguments, a first order Taylor expansion at $\bzeta_k$ yields
\begin{align*}
| \Fbar(\parcompbar_k, \perpcompbar_k) - \Fbar(\parcomp_k, \perpcomp_k)| &\leq \sup_{0 \leq u \leq 1} | \langle \nabla \Fbar(\parcomp^u, \perpcomp^u), \bzetabar_k - \bzeta_k \rangle | \\
&\leq  \sup_{0 \leq u \leq 1} \| \nabla \Fbar(\parcomp^u, \perpcomp^u) \|_{1} \cdot \| \bzetabar_k - \bzeta_k \|_\infty \\
&\leq (1 - \tau) \cdot \left( |\parcomp_k - \parcompbar_k| \lor |\perpcomp_k - \perpcompbar_k| \right),
\end{align*}
where the final inequality follows by using the fact that by the inductive hypothesis, $|\parcomp_k - \parcompbar_k| \lor |\perpcomp_k - \perpcompbar_k| \leq \Delta / \tau$, whence $(\parcomp^u, \perpcomp^u) \in \mathbb{B}_{\Delta/\tau} (\Goodset)$ for each $0 \leq u \leq 1$ in conjunction with 
the assumption $\| \nabla \Fbar(\parcomp, \perpcomp) \|_{1} \leq 1 - \tau$ for all $(\parcomp, \perpcomp) \in \mathbb{B}_{\Delta/\tau} (\Goodset)$. At the same time, we have by assumption that $\bar{F}(\alpha_k, \beta_k) - \alpha_{k+1} \leq \Delta$.  An identical argument holds for $\perpcomp_{k + 1}$, and so this completes the inductive step.
\qed


\subsection{Proof of Lemma~\ref{lem:generalc}} \label{sec:cproof}

Recall that $\parcomp_0 / \perpcomp_0 \geq \frac{1}{50{\sqrt{d}}}$ by assumption, and that
\[
t_0 = \log_{1.05} (50 \sqrt{d}) + \log_{55/54} (10) + 2.
\] 
Also define the scalar $\underline{t} := \log_{1.05} (50 \sqrt{d}) + 1$, 
and given the iterates $\{ \parcomp_t, \perpcomp_t \}_{t = 0}^{t_0}$, define
\begin{align*}
\underline{T} = \inf \left\{ t \; \Big| \; \frac{\parcomp_t}{\perpcomp_t} \geq 1/2 \right\} \qquad \text{ and } \qquad  
T_0 = \inf \left\{ t \; \Big|\; \frac{\parcomp_t}{\perpcomp_t} \geq 5 \right\},
\end{align*}
with the convention that each quantity is set to $\infty$ if the condition is not met. Conditions~\eqref{eq:par-condition} and~\eqref{eq:perp-condition} with $c = 1/100$ yield
\begin{align} \label{two-sandwiches}
\max_{0 \leq t \leq t_0} \; |\parcomp_{t+1} - \Fbar(\parcomp_{t}, \perpcomp_{t})| \leq \frac{1}{100\sqrt{d}} \qquad \qquad \text{ and } \qquad \qquad \max_{0 \leq t \leq t_0} \; |\perpcomp_{t + 1} - \Gbar(\parcomp_{t}, \perpcomp_{t})| \leq \frac{1}{100}.
\end{align}
We begin by noting that $\parcomp_t \leq 3/2$ for all $1\leq t \leq t_0$. This is straightforward to establish: note that it follows directly from condition C1 and the bound~\eqref{two-sandwiches} for $t = 1$, and from that point onward, by induction over $t$, applying condition C4 and the bound~\eqref{two-sandwiches}. 
The crux of the lemma is the following claim.
\begin{claim} \label{clm:t-bound}
Under both settings (a) and (b) of the lemma, we have $\underline{T} \overset{(i)}{\leq} \underline{t}$ and $T_0 \overset{(ii)}{\leq} t_0 - 1$. \\
\end{claim}
Taking the claim as given for the moment, we note that it suffices to show that \mbox{$(\parcomp_{T_0 + 1}, \perpcomp_{T_0 + 1}) \in \Goodset$}. 
Since $T_0 + 1 \leq t_0$, we have
\begin{align} \label{eq:pert-bound}
|\parcomp_{T_0 + 1} - \Fbar(\parcomp_{T_0}, \perpcomp_{T_0})| \lor |\perpcomp_{T_0 + 1} - \Gbar(\parcomp_{T_0}, \perpcomp_{T_0})| \leq \frac{1}{100},
\end{align}
and by definition of $T_0$, we have
$\parcomp_{T_0} / \perpcomp_{T_0} \geq 5$. 
Condition C4 yields the bounds
\begin{align*}
0.56 \leq \Fbar(\parcomp_{T_0}, \perpcomp_{T_0}) \leq 1.04, \quad \text{ and } \quad \frac{\Gbar(\parcomp_{T_0}, \perpcomp_{T_0})}{\Fbar(\parcomp_{T_0}, \perpcomp_{T_0}) } \leq 1/6.
\end{align*}
Combining the first bound with the perturbation bound~\eqref{eq:pert-bound} yields 
$0.55 \leq \parcomp_{T_0 + 1} \leq 1.05$.
Moreover, we have
\begin{align*}
\frac{\perpcomp_{T_0 + 1}}{\parcomp_{T_0 + 1}} \leq \frac{\Gbar(\parcomp_{T_0}, \perpcomp_{T_0}) + 0.01}{\Fbar(\parcomp_{T_0}, \perpcomp_{T_0}) - 0.01} \leq \frac{\Gbar(\parcomp_{T_0}, \perpcomp_{T_0})}{\Fbar(\parcomp_{T_0}, \perpcomp_{T_0})} + \frac{0.01}{0.55} \leq 1/5.
\end{align*}
Putting together the above two displays yields that $(\parcomp_{T_0 + 1}, \perpcomp_{T_0 + 1}) \in \Goodset$, as desired.
\noindent It remains to prove Claim~\ref{clm:t-bound}.

\paragraph{Proof of Claim~\ref{clm:t-bound}(i):}
Note that by condition C1 and the initialization condition (guaranteed by both Ia and Ib) that $\parcomp_0 / \perpcomp_0 \geq (50 \sqrt{d})^{-1}$, we have $\Fbar(\parcomp_0, \perpcomp_0) \geq (50 \sqrt{d})^{-1}$, so that 
\[
\parcomp_1 \geq (50\sqrt{d})^{-1} - (100\sqrt{d})^{-1} = (100\sqrt{d})^{-1}.
\]
We prove momentarily that under both settings (a) and (b) of the lemma and for all $1 \leq t \leq \underline{T}$,
\begin{align} \label{clm:ind}
\parcomp_{t} \geq (1.05)^{t-1} \cdot \parcomp_1 \qquad \text{ and } \qquad \perpcomp_{t} \leq 1.
\end{align}
Now suppose for the sake of contradiction that $\underline{T} > \underline{t}$, where we recall that $\underline{t} := \log_{1.05} (50 \sqrt{d}) + 1$. 
Putting together the above two displays,
we have 
\begin{align*}
\parcomp_{\underline{t}} \geq \frac{50 \sqrt{d}}{100 \sqrt{d}} = 1/2 \quad \text{ and } \quad \perpcomp_{\underline{t}} \leq 1,
\end{align*}
so that $\parcomp_{\underline{t}} / \perpcomp_{\underline{t}} \geq 1/2$. But this contradicts the fact that $\underline{T} > \underline{t}$ and proves the theorem. 

We prove claim~\eqref{clm:ind} under two distinct settings (a) and (b) of the lemma. Both proceed via induction.

\noindent \underline{Proof of claim~\eqref{clm:ind}, setting (a):}
To prove the base case $t = 1$, note that the $\parcomp$ component follows trivially. To handle $\perpcomp_1$, note that $\perpcomp_1 \leq \Gbar(\parcomp_0, \perpcomp_0) + 0.01 \leq 1$, where the final inequality follows from condition C5a. 

For the induction hypothesis, suppose that $\parcomp_t \geq (1.05)^t \cdot \parcomp_1 \geq \frac{1}{100 \sqrt{d}}$, and $\perpcomp_t \leq 1$. Since $t \leq \underline{T}$, we have $\parcomp_t / \perpcomp_t \leq 1/2$ by definition. Then condition C2 and equation~\eqref{eq:par-condition} together yield
\begin{align*}
\parcomp_{t + 1} \geq \Fbar(\parcomp_t, \perpcomp_t) - \frac{1}{100\sqrt{d}} \geq 1.06 \cdot \frac{\parcomp_t}{\perpcomp_t} - \frac{1}{100\sqrt{d}} \geq 1.05 \parcomp_t \geq (1.05)^t \cdot \parcomp_1.
\end{align*}
At the same time, condition C5a and equation~\eqref{eq:par-condition} together yield
\begin{align*}
\perpcomp_{t + 1} \leq \Gbar(\parcomp_t, \perpcomp_t) + \frac{1}{100} \leq 1.
\end{align*}
This completes the induction step.

\noindent \underline{Proof of claim~\eqref{clm:ind}, setting (b):}
Once again, for the base case $t = 1$, note that the $\parcomp$ component follows trivially. To handle $\perpcomp_1$, note that $\perpcomp_1 \leq \Gbar(\parcomp_0, \perpcomp_0) + 0.01 \leq 1$, where the final inequality follows from condition C5b and the initialization condition Ib. 

For the induction hypothesis, suppose that $\parcomp_t \geq (1.05)^t \cdot \parcomp_1 \geq \frac{1}{100 \sqrt{d}}$, and $\perpcomp_t \leq 1$. Since $t \leq \underline{T}$, we have $\parcomp_t / \perpcomp_t \leq 1/2$ by definition. Then condition C2 and equation~\eqref{eq:par-condition} together yield
\begin{align*}
\parcomp_{t + 1} \geq \Fbar(\parcomp_t, \perpcomp_t) - \frac{1}{100\sqrt{d}} \geq 1.06 \cdot \frac{\parcomp_t}{\perpcomp_t} - \frac{1}{100\sqrt{d}} \geq 1.05 \parcomp_t \geq (1.05)^t \cdot \parcomp_1.
\end{align*}
We also have $\Fbar(\parcomp_t, \perpcomp_t) \leq 1.04$ from condition C4, so that $\parcomp_{t + 1} \leq 1.05$.
Consequently, we may apply condition C5b and equation~\eqref{eq:par-condition} together, to yield
\begin{align*}
\perpcomp_{t + 1} \leq \Gbar(\parcomp_t, \perpcomp_t) + \frac{1}{100} \leq 1.
\end{align*}
This completes the induction step. 
\qed

\paragraph{Proof of Claim~\ref{clm:t-bound}(ii):}
Note that $\parcomp_{\underline{t}}/\perpcomp_{\underline{t}} \geq 1/2$ by part (a) of the claim. 
We will show momentarily that for each $\underline{t} + 1 \leq t < T_0$, we have
\begin{align} \label{eq:ind2}
\frac{\parcomp_t}{\perpcomp_t} \geq \left( \frac{55}{54} \right)^{t - \underline{t} - 1} \cdot \frac{\parcomp_{\underline{t}}}{\perpcomp_{\underline{t}}} \quad \text{ and } \quad 1/2 \leq \parcomp_t \leq 3/2.
\end{align}
Now suppose for the sake of contradiction that $T_0 > t_0 - 1$.
Setting $t = t_0 - 1$ in equation~\eqref{eq:ind2}, we obtain $\frac{\parcomp_t}{\perpcomp_t} \geq 10 \cdot \frac{1}{2} \geq 5$. But this contradicts the fact that $T_0 > t_0 - 1$.

\noindent \underline{Proof of claim~\eqref{eq:ind2}:}
To prove the base case $t = \underline{t} + 1$, 
note that condition C2 ensures that $\Fbar(\parcomp_{\underline{t}}, \perpcomp_{\underline{t}}) \geq 0.56$,
so that 
\[
\frac{\perpcomp_{\underline{t} + 1}}{\parcomp_{\underline{t} + 1}} \leq \frac{\Gbar(\parcomp_{\underline{t}}, \perpcomp_{\underline{t}}) + 0.01}{\Fbar(\parcomp_{\underline{t}}, \perpcomp_{\underline{t}}) - 0.01} \leq 1/2,
\]
where the last inequality uses condition C3.
Furthermore, condition C4 yields $\Fbar(\parcomp_{\underline{t}}, \perpcomp_{\underline{t}}) \leq 1.04$, so that $1/2 \leq \parcomp_{\underline{t} + 1} \leq 3/2$.

For the induction hypothesis, suppose that $\frac{\parcomp_t}{\perpcomp_t} \geq \left( \frac{55}{54} \right)^{t - \underline{t}} \cdot \frac{\parcomp_{\underline{t}}}{\perpcomp_{\underline{t}}} \geq 1/2$ and $1/2 \leq \parcomp_t \leq 3/2$. Since $t < T_0$, we also have $\parcomp_t / \perpcomp_t \leq 5$, and so condition C3 yields $\frac{\Gbar(\parcomp_t, \perpcomp_t)}{\Fbar(\parcomp_t, \perpcomp_t)} \leq \frac{7}{8} \cdot \frac{\perpcomp_t}{\parcomp_t}$. Therefore,
\[
\frac{\perpcomp_{t + 1}}{\parcomp_{t + 1}} \leq \frac{\Gbar(\parcomp_t, \perpcomp_t) + 0.01}{\Fbar(\parcomp_t, \perpcomp_t) - 0.01} \leq \frac{56}{55} \cdot \frac{7}{8} \cdot \frac{\perpcomp_t}{\parcomp_t} + \frac{0.01}{0.55} \leq \left( \frac{49}{55} + \frac{1}{11} \right) \cdot \frac{\perpcomp_t}{\parcomp_t},
\]
where in the last step, we have used the fact that $\frac{\perpcomp_t}{\parcomp_t} \geq 1/5$. At the same time, conditions C2 and C4 yield $0.56 \leq \Fbar(\parcomp_t, \perpcomp_t) \leq 1.04$, so that $1/2 \leq \parcomp_{t +1} \leq 3/2$.
This completes the inductive step.
\qed

%% file: sections/appendix/aux-elementary-lemmas.tex

\section{Some elementary lemmas}
In this section, we collect a few elementary lemmas that are used multiple times in the proof.

\begin{lemma}
	\label{lem:bound-operator-norm}
	Let $\mathcal{T}_n(\bt)$ denote the empirical operator corresponding to the higher order updates~\eqref{eq:sec_order_gen} and suppose that the weight function $\omega$ satisfies Assumption~\ref{ass:omega} with parameter $K_1$.  There exist universal, positive constants $c$ and $C$ such that
	\[
	\Pro\{\| \mathcal{T}_n(\bt) \|_2 \leq C\sqrt{K_1}\} \leq 2e^{-cn}.
	\]
\end{lemma}
\begin{proof}
	We write explicitly 
	\[
	\| \mathcal{T}_n(\bt) \|_2 = \| (\bX^{\top} \bX)^{-1}  \bX^{\top} \omega(\bX \bt, \by) \|_2 \leq \| (\bX^{\top} \bX)^{-1} \|_{\mathsf{op}} \| \bX\|_{\mathsf{op}} \| \omega(\bX\bt, \by) \|_2,
	\]
	where the final inequality is due to the sub-multiplicativity of the operator norm.  Now, we apply~\citet[Theorem 4.6.1]{vershynin2018high} to obtain the probabilistic inequality
	\[
	\Pro\{c\sqrt{n} \leq \| \bX \|_{\mathsf{op}} \leq C\sqrt{n}\} \geq 1 - 2e^{-c'n},
	\]
	which in turn implies that with probability at least $1 - 2e^{-c'n}$, 
	\[
	\| (\bX^{\top} \bX)^{-1} \|_{\mathsf{op}} \| \bX\|_{\mathsf{op}} \leq \frac{C}{\sqrt{n}}.
	\]
	Now, since $\omega$ satisfies Assumption~\ref{ass:omega}, the random vector $\omega(\bX\bt, \by)$ contains independent sub-Gaussian coordinates whence we apply~\citet[Theorem 3.1.1]{vershynin2018high} to obtain the probabilistic inequality
	\[
	\Pro\{\| \omega(\bX\bt, \by) \|_2 \leq C \sqrt{K_1 n}\} \geq 1 - 2e^{-cn}.
	\]
	Putting the pieces together, we obtain the result.
\end{proof}

\begin{lemma}
	\label{lem:truncation-tail-bound}
	Let $n$ be a positive integer and suppose that, for all positive integers $q \leq n/64$, the random variable $Z_n$ satisfies the inequality
	\begin{align}
		\label{ass:L-q-ineq}
		\|Z_n - \EE Z_n\|_{q} \leq \frac{A q^{\zeta}}{\sqrt{n}},
	\end{align}
	for constants $A > 0$ and $\zeta > 1$.  Further, suppose that there are positive constants $c_{\downarrow}$ and $C_{\downarrow}$ such that $\Pro\{\lvert Z_n \rvert \geq C_{\downarrow}\} \leq e^{-c_{\downarrow}n}$.  Then, there exist universal positive constants $c'$ and $C'$ such that for all $t \geq 0$, 
	\[
	\Pr\{\lvert Z_n - \EE Z_n \rvert \geq t\} \leq C'\exp\Bigl\{-c'\Bigl(\frac{t\sqrt{n}}{A}\Bigr)^{1/\zeta}\Bigr\} + e^{-c_{\downarrow}n}.
	\]
\end{lemma}
\begin{proof}
	We employ a truncation argument.  For some integer $M > 0$ to be specified later, we introduce the notation
	\begin{align*}
		\widebar{Z} = Z_n - \EE Z_n, \qquad \widebar{Z}^{\downarrow} = \widebar{Z} \mathbbm{1}\{\lvert \widebar{Z} \rvert \leq M\}, \qquad \text{ and } \qquad \widebar{Z}^{\uparrow} = \widebar{Z} \mathbbm{1}\{\lvert \widebar{Z} \rvert > M\},
	\end{align*}
	noting that
	\[
	\widebar{Z} = \widebar{Z}^{\downarrow}  + \widebar{Z}^{\uparrow}.
	\]
	Consequently, an application of the union bound yields the inequality
	\begin{align}
		\label{ineq:tail-bound-observation}
		\Pr\{\lvert \widebar{Z} \rvert \geq t\} \leq \Pr \Bigl\{\lvert \widebar{Z}^{\downarrow} \rvert \geq \frac{t}{2}\Bigr\} + \Pr \Bigl\{\lvert \widebar{Z}^{\uparrow}  \rvert \geq \frac{t}{2}\Bigr\}.
	\end{align}
	We control each of these terms in turn, beginning with the lower truncation.  First, for any $\lambda > 0$, applying Markov's inequality yields
	\begin{align}
		\label{ineq:Markov-MGF-Z-down}
		\Pr \Bigl\{\lvert \widebar{Z}^{\downarrow} \rvert \geq \frac{t}{2}\Bigr\} \leq \exp\bigl\{-(\lambda t/2)^{1/\zeta}\bigr\}\EE\bigl\{e^{(\lambda \lvert \bar{Z}^{\downarrow} \rvert)^{1/\zeta}}\bigr\}.  
	\end{align}
	Note that the expectation on the RHS of the display above exists since $\widebar{Z}^{\downarrow}$ is bounded.  Thus, the Taylor expansion of $x \mapsto e^{x}$ gives
	\begin{align}
		\label{ineq:MGF-Z-down-bound}
		\EE\bigl\{e^{(\lambda (\bar{Z}^{\downarrow})^{1/\zeta}}\bigr\} = 1 + \sum_{\ell = 1}^{\infty} \frac{\lambda^{\ell/\zeta} \EE \lvert \widebar{Z}^{\downarrow} \rvert^{\ell/\zeta}}{\ell!} &= 1 + \sum_{\ell = 1}^{n/64}  \frac{\lambda^{\ell/\zeta} \EE \lvert \widebar{Z}^{\downarrow} \rvert^{\ell/\zeta}}{\ell!}  + \sum_{\ell = n/64 + 1}^{\infty} \frac{\lambda^{\ell/\zeta} \EE \lvert \widebar{Z}^{\downarrow} \rvert^{\ell/\zeta}}{\ell!} \nonumber\\
		&\overset{\1}{\leq} 1 + \sum_{\ell = 1}^{n/64}  \frac{\lambda^{\ell/\zeta} \bigl(\EE \lvert \widebar{Z}^{\downarrow} \rvert^{\ell}\bigr)^{1/\zeta}}{\ell!}  + \sum_{\ell = n/64 + 1}^{\infty} \frac{\lambda^{\ell/\zeta} M^{\ell/\zeta}}{\ell!} \nonumber\\
		&\overset{\2}{\leq}1 + \sum_{\ell = 1}^{n/64}  \frac{(A\lambda/\sqrt{n})^{\ell/\zeta} \ell^{\ell}}{\ell!}  + \sum_{\ell = n/64 + 1}^{\infty} \frac{\lambda^{\ell/\zeta} M^{\ell/\zeta}}{\ell!},
	\end{align}
	where step $\1$ follows by applying Jensen's inequality to each summand of the first sum since the map $x \mapsto x^{1/\zeta}$ is concave on $\mathbb{R}_{\geq 0}$ and by noting that $| \widebar{Z}^{\downarrow} | \leq M$ pointwise.
	Step $\2$ follows by noting that 
	$\| \widebar{Z}^{\downarrow} \|_{q} \leq \| \widebar{Z} \|_q$ and using the assumption~\eqref{ass:L-q-ineq}.  Now when $\ell \geq n/64$ and $M \leq CAn^{(2\zeta - 1)/2}$, 
	\[
	M^{\ell/\zeta} \leq \Bigl(\frac{A}{\sqrt{n}}\Bigr)^{\ell/\zeta} \ell^{\ell}.
	\]
	Substituting the above inequality into the bound on the MGF~\eqref{ineq:MGF-Z-down-bound}, we obtain 
	\begin{align}
		\label{ineq:MGF-Z-down-bound-final}
		\EE\bigl\{e^{(\lambda (\bar{Z}^{\downarrow})^{1/\zeta}}\bigr\} \leq 1  + \sum_{\ell = 1}^{\infty} \frac{(A\lambda/\sqrt{n})^{\ell/\zeta} \ell^{\ell}}{\ell!} \overset{\1}{\leq} 1 + \sum_{\ell = 1}^{\infty} \Bigl(\frac{C A \lambda}{\sqrt{n}}\Bigr)^{\ell/\zeta} &= 1 + \frac{(CA\lambda/\sqrt{n})^{1/\zeta}}{1 - (CA\lambda/\sqrt{n})^{1/\zeta}} \nonumber\\
		&\overset{\2}{\leq} \exp\Bigl\{2\Bigl(\frac{CA\lambda}{\sqrt{n}}\Bigr)^{1/\zeta}\Bigr\},
	\end{align}
	where step $\1$ follows by the Stirling inequality $\ell! \geq \ell^{\ell}/e^{\ell}$ and step $\2$ follows for all \mbox{$\lambda \leq c/A \cdot (1/2)^{\zeta}\sqrt{n}$} by the elementary inequality $1 + x \leq e^{x}$.  Summarizing, we see that for $M \leq CAn^{(2\zeta - 1)/2}$ and $\lambda \leq  c/A \cdot \sqrt{n}$, plugging the inequality~\eqref{ineq:MGF-Z-down-bound-final} into the inequality~\eqref{ineq:Markov-MGF-Z-down} yields
	\[
	\Pr \Bigl\{\lvert \widebar{Z}^{\downarrow} \rvert \geq \frac{t}{2}\Bigr\} \leq \exp\Bigl\{-(\lambda t/2)^{1/\zeta} + 2\Bigl(\frac{CA\lambda}{\sqrt{n}}\Bigr)^{1/\zeta}\Bigr\}.
	\]
	Taking $\lambda$ as large as possible, we obtain
	\begin{align}
		\label{ineq:tail-bound-Z-down}
		\Pr \Bigl\{\lvert \widebar{Z}^{\downarrow} \rvert \geq \frac{t}{2}\Bigr\} \leq C\exp\{-(c/A \cdot t \sqrt{n})^{1/\zeta}\},
	\end{align}
	where we emphasize that the constants $c$ and $C$ changed from line to line.
	We turn now to bounding the upper truncation.  We have
	\begin{align}
		\label{ineq:tail-bound-Z-up}
		\Pr \Bigl\{\lvert \widebar{Z}^{\uparrow}  \rvert \geq \frac{t}{2}\Bigr\} = \Pr \Bigl\{\lvert \widebar{Z}^{\uparrow}  \rvert \geq \frac{t}{2}, \lvert \widebar{Z} \rvert > M\Bigr\} + \Pr \Bigl\{\lvert \widebar{Z}^{\uparrow}  \rvert \geq \frac{t}{2}, \lvert \widebar{Z} \rvert \leq M\Bigr\} &\overset{\1}{\leq} \Pr\{\lvert \widebar{Z} \rvert > M\} 
		\overset{\2}{\leq} e^{-c_{\downarrow}n},
	\end{align}
	where step $\1$ follows since by assumption $t > 0$, so the second term has zero-probability and step $\2$ follows by assumption as long as $M > C_{\downarrow}$.  We conclude by letting $M$ take any value between $C_{\downarrow}$ and $M \leq CAn^{(2\zeta - 1)/2}$ and then substituting the tail bound on the lower truncation~\eqref{ineq:tail-bound-Z-down} and the tail bound on the upper truncation~\eqref{ineq:tail-bound-Z-up} into the decomposition~\eqref{ineq:tail-bound-observation}.
\end{proof}

\begin{lemma} \label{lem:init-properties}
Let $\bu$ denote a random vector sampled uniformly from the unit sphere $\mathbb{S}^{d- 1}$. Suppose $(\bx_i, y_i)_{i = 1}^n$ are drawn i.i.d. from either the model~\eqref{eq:PR-model} or~\eqref{eq:MoR-model}, with $\thetastar \neq \bf{0}$ denoting an arbitrary vector (not necessarily unit norm). For any $\bt \in \real^d$, let $\alpha(\bt) = \frac{\inprod{\bt}{\thetastar}}{\| \thetastar \|_2^2}$ and $\beta(\bt) = \| \proj^{\perp}_{\thetastar} \bt \|_2$. Then the following statements are true. \\
(a) If $\bt_0 = \lambda \bu$ for an arbitrary positive scalar $\lambda$, then
\begin{align*}
\Pro \left\{ \frac{|\alpha(\bt_0)|}{\beta(\bt_0)} \leq \frac{\delta}{\sqrt{d - 1}} \right\} \leq \delta + \exp \left( - \frac{d - 1}{32}\right)
\end{align*}
(b) If $\bt_0 = \sqrt{\frac{1}{n} \sum_{i = 1}^n y_i^2} \cdot \bu$, then
\begin{align*}
\Pro \left\{ \alpha^2(\bt_0) + \beta^2(\bt_0) \geq  (\| \thetastar \|_2 + \sigma)^2 + C \log(1 / \delta) \right\} \leq \delta
\end{align*}
for an absolute constant $C > 0$.
\end{lemma}
\begin{proof}
Let $\| \thetastar \|_2 = \lambda^*$, noting that we may assume due to rotation invariance of $\bu$ that \sloppy \mbox{$\thetastar = \lambda^* \cdot \be_1$}. Furthermore, we may write $\bu = \bz / \| \bz \|_2$ for a random vector $\bz = (z_1, \ldots, z_d) \sim \NORMAL(0, \bI)$. Thus, we have $\frac{|\alpha(\bt_0)|}{\beta(\bt_0)} = |z_1| / \| \bz_{\setminus 1} \|_2$, where $\bz_{\setminus 1} = (z_2, \ldots, z_d)$. Part (a) then follows from the tail bounds
\begin{align*}
\Pro\{ |z_1 | \leq t_1 \} \leq \sqrt{ \frac{2}{\pi} } \cdot t_1 \quad \text{ and } \quad \Pro\{ \| \bz_{\setminus 1} \|_2 \geq \sqrt{d - 1} + t_2 \} \leq  e^{-t_2^2 / 2} \text{ for each } t_1, t_2 \geq 0.
\end{align*}
In particular, setting $t_1 = \sqrt{\pi/2} \cdot \delta$ and $t_2 = \frac{\sqrt{d - 1}}{4}$ and applying a union bound proves part (a).

Next, note that $\alpha^2(\bt_0) + \beta^2(\bt_0) = \| \bt_0 \|_2^2 = \frac{1}{n} \sum_{i = 1}^n y_i^2$. Furthermore, in both models of interest, we have $|y_i| \overset{(d)}{\leq} \overline{y}_i := \lambda^* |z_{1, i}| + \sigma |z_{2, i}|$, where $z_{1, i}$ and $z_{2, i}$ are independent Gaussians and $X\overset{(d)}{\leq}Y$ denotes that the random variable $X$ is stochastically dominated by $Y$. Furthermore, we have $\EE[\overline{y}^2_i] \leq (\lambda^* + \sigma)^2$, and also
\[
\Pro \left\{ \frac{1}{n} \sum_{i = 1}^n \overline{y}_i^2 - \EE[\overline{y}^2_i] \geq t \right\} \leq  e^{-ct},
\] 
since $\overline{y}^2_i$ is a subexponential random variable. Choosing $t = c^{-1} \cdot \log (1/\delta)$ completes the proof of part (b).
\end{proof}

\begin{lemma} \label{lem:angle-l2}
	(a) For all state evolution elements $\bzeta = (\parcomp, \perpcomp)$ with $\parcomp \geq 0$, we have
	\begin{align*}
		\DeltaSELtwo(\bzeta) \geq \frac{\perpcomp}{\sqrt{\parcomp^2 + \perpcomp^2}} = \sin (\DeltaSEangle(\bzeta)).
	\end{align*}
	(b) For all state evolution elements $\bzeta = (\parcomp, \perpcomp)$ with $1/2 \leq \parcomp \leq 3/2$ and $\rho \geq 1/5$, we have
	\begin{align*}
		\DeltaSELtwo(\bzeta) \leq 8\rho \leq 8 \tan (\DeltaSEangle(\bzeta)).
	\end{align*}
\end{lemma}
\begin{proof}
	To prove part (a), note that $\theta^*$ corresponds to the state evolution element $(\parcomp^*, \perpcomp^*) = (1, 0)$. All points $\bzeta = (\parcomp, \perpcomp)$ such that $\parcomp \geq 0$ and $\DeltaSEangle(\bzeta) = \angle( (\parcomp^*, \perpcomp^*), (\parcomp, \perpcomp)) = \anglecurr$ form a line in $\mathbb{R}^2$. The point on this line with smallest $\DeltaSELtwo$ is the projection of $(1, 0)$ onto this line. The length of this projection is, by definition, equal to $\sin \anglecurr$.
	
	To prove part (b), note that
	\[
	\DeltaSELtwo(\bzeta) = \sqrt{(1 - \parcomp)^2 + \perpcomp^2} = \parcomp \cdot \sqrt{(\parcomp^{-1} - 1)^2 + \left( \frac{\perpcomp}{\parcomp}\right)^2} \leq \frac{3}{2} \cdot \sqrt{ 1 + \left( \frac{\perpcomp}{\parcomp}\right)^2} \leq \frac{3}{2} \cdot \sqrt{26} \cdot \frac{\perpcomp}{\parcomp},
	\]
	where the final inequality follows since $\perpcomp/\parcomp \geq 1/5$.
\end{proof}

\begin{lemma} \label{lem:taninv}
	Let $a, b, x$ denote non-negative scalars. \\
	(a) For all $a < 1$, we have
	\begin{align*}
		\tan^{-1} (a x + b) \geq a \tan^{-1} x + b - \frac{b^3}{3(1 - a)^2}.
	\end{align*}
	(b) If $a \leq 1$ and $x \leq 1/5$, we have
	\begin{align*}
		\tan^{-1} (a x + b) \leq \frac{51}{50} \cdot a \tan^{-1} x + b.
	\end{align*}
\end{lemma}
\begin{proof}
	To prove part (a), note from the concavity of the $\tan^{-1}$ function on the positive reals that
	\[
	\tan^{-1} (a x + b) \geq a \tan^{-1} x + (1 - a) \tan^{-1} \left( \frac{b}{1 - a} \right).
	\]
	Using the inequality $\tan^{-1} x \geq x - \frac{x^3}{3}$ completes the proof.
	To prove part (b),
	first note from the subadditivity of the $\tan^{-1}$ function on the positive reals that 
	\[
	\tan^{-1} (a x + b) \leq \tan^{-1} (ax) + \tan^{-1} (b) \leq \tan^{-1} (ax) + b. 
	\]
	Next, let $h(a, x) = \frac{\tan^{-1} (ax)}{a\tan^{-1} (x)}$, and note that for each fixed $a \in [0, 1]$, $h$ is non-decreasing in $x \in [0, \infty)$. Similarly, for each fixed $x \geq 0$, $h$ is non-increasing in $a \in [0, 1]$. The proof is completed by noting that $\lim_{a \downarrow 0} h(a, 1/5) = [5 \tan^{-1} (1/5)]^{-1} < 1.02$.
\end{proof}